\documentclass{article} 
\usepackage{graphicx} 
\usepackage[utf8]{inputenc}
\usepackage[numbers]{natbib}
\usepackage{graphicx}
\usepackage{mathdots}
\usepackage{hyperref}
\usepackage{float}
\usepackage{booktabs}
\usepackage{adjustbox}
\usepackage[linesnumbered,ruled,vlined]{algorithm2e}
\usepackage{xltabular}
\usepackage{amssymb}

\usepackage{tabularx}       
\usepackage{siunitx}        
\usepackage{multirow}
\usepackage{amsmath, amssymb, bm}
\usepackage{bbm}  
\usepackage{amsmath,amsthm,amssymb,mathrsfs}
\usepackage{geometry} 
\usepackage{centernot} 
\usepackage{stmaryrd}
\usepackage{graphicx}
\usepackage{tikz} 
\usetikzlibrary{shapes,decorations}
\usepackage{caption} 
\usepackage{subcaption} 
\usepackage{mdframed} 
\usepackage{upquote}
\usepackage{booktabs}
\usepackage{xspace}
\usepackage{hyperref}
\usepackage[normalem]{ulem} 
\usepackage[title]{appendix}

\usepackage[ruled,vlined]{algorithm2e}
\usepackage{algorithmic}

\usepackage{makecell}   

\newcommand{\bena}{\begin{eqnarray}\begin{array}{l}}
\newcommand{\eena}{\end{array}\end{eqnarray}}
\newcommand{\benas}{\begin{eqnarray*}\begin{array}{l}}
\newcommand{\eenas}{\end{array}\end{eqnarray*}}
\newcommand{\benall}{\begin{eqnarray}\begin{array}{ll}}
\newcommand{\eenall}{\end{array}\end{eqnarray}}
\newcommand{\besn}{\begin{subnumcases}}
\newcommand{\eesn}{\end{subnumcases}}
\newcommand{\ben}{\begin{eqnarray}}
\newcommand{\een}{\end{eqnarray}}
\newcommand{\bea}{\begin{array}}
\newcommand{\eea}{\end{array}}
\newcommand{\bes}{\begin{subequations}}
\newcommand{\ees}{\end{subequations}}
\newcommand{\bec}{\begin{cases}}
\newcommand{\eec}{\end{cases}}
\newcommand{\bef}{\begin{figure}[H]\centering}
\newcommand{\eef}{\end{figure}}
\newcommand{\bet}{\begin{tikzpicture}}
\newcommand{\eet}{\end{tikzpicture}}
\newcommand{\beq}{\begin{equation}}
\newcommand{\eeq}{\end{equation}}
\newcommand{\bep}{\begin{proof}}
\newcommand{\eep}{\end{proof}}
\newcommand{\bei}{\begin{itemize}}
\newcommand{\eei}{\end{itemize}}
\newcommand{\beu}{\begin{enumerate}}
\newcommand{\eeu}{\end{enumerate}}
\def\beg#1\eeg{\begin{align}#1\end{align}}
\def\begd#1\eegd{\bena \begin{aligned}#1\end{aligned}\eena}
\def\besp#1\eesp{\begin{split}#1\end{split}}
\def\besl#1\eesl{\begin{subequations}\begin{align}#1\end{align}\end{subequations}} 

\def\inc(#1){\includegraphics[width=0.5\linewidth]{#1}}

\newcommand{\bigO}{\mathcal{O}}

\def\be{\ensuremath{{\bf e}}}

\def\bI{\ensuremath{{\bf I}}}

\def\b{\ensuremath{{\bf M}}}

\def\br{\ensuremath{{\bf r}}}

\def\bu{\ensuremath{{\bf u}}}
\def\bV{\ensuremath{{\bf V}}}

\def\bx{\ensuremath{{\bf x}}}
\def\by{\ensuremath{{\bf y}}}

\newtheorem{alg}{Algorithm}[section]
\newtheorem{theorem}{Theorem}[section]
\newtheorem{lemma}[theorem]{Lemma}

\theoremstyle{remark}
\newtheorem{remark}{Remark}[section]

\theoremstyle{definition}
\newtheorem{definition}{Definition}[section]
\newcommand{\la}{\langle}
\newcommand{\ra}{\rangle}

\DeclareMathOperator*{\argmin}{arg\,min}

\theoremstyle{remark}

\title{ Swarm-Based Inertial Methods for Optimization}
\author{Qiyu Wu, Kunhui Luan and Qi Wang\\\\
Department of Mathematics, 
 University of South Carolina, Columbia, SC 29208, USA}
\date{\today}

\begin{document}
\maketitle

\begin{abstract}

We introduce a new class of swarm-based inertial methods (SBIMs) for global minimization, formulated as coupled dissipative inertial dynamical systems derived from the generalized Onsager principle. The proposed framework identifies the friction operator and the scaling of the potential energy, namely the objective function to be minimized, as the key ingredients governing relaxation dynamics over the energy landscape. 
Within this framework, we propose a new underdamped inertial dynamics whose damping mechanisms incorporate both gradient and Hessian information, allowing the system to adjust damping or acceleration according to the agent trajectories and the curvature of the landscape.
Under suitable conditions, we prove that the underdamped system satisfies an energy dissipation law, from which we establish an upper bound on the asymptotic decay rate of the gap between the objective function and its global minimum, given by $O(1/\delta(t))$ (defined in \S 3). We further construct structure-preserving discretizations that retain both discrete energy dissipation and the convergence rate estimate, $O(1/\delta_k)$ (defined in \S3). In addition, we present several other efficient numerical algorithms for the dynamical system.
Numerical experiments for all proposed algorithms validate the theory on convex test problems and demonstrate convergence rates in function values that are substantially faster than the theoretical guarantees ($O(1/\delta_k)$). On nonconvex benchmark problems, the proposed methods achieve high success rates in reaching the global minimum, and exhibit more stable energy decay than swarm-based gradient descent and Nesterov methods. Overall, this work provides a systematic framework for the construction and analysis of SBIMs from an energy-dissipative perspective.

\end{abstract}
\section{Introduction}

\noindent \indent 
Optimization problems involve minimizing or maximizing an objective function. Since any maximization problem can be reformulated as a minimization problem by negating the objective function, we focus on global minimization problems.
Swarm-based minimization seeks to explore the minimum of an objective function or potential energy landscape, denoted here by 
$F(x)$, by evolving a collection of interacting agents rather than a single agent.  This multi-agent strategy increases the likelihood of reaching a global minimum in nonconvex landscapes. However, the general mathematical principles underlying the design of such methods remain incomplete. In this paper, we develop a systematic, dynamics-driven framework for constructing \emph{swarm-based inertial methods} (SBIMs) for minimization through a novel underdamped, dissipative, inertial dynamical system. The proposed methods are designed to improve the convergence rate of the objective function, guided by an energy dissipation law.

In general SBIMs, each agent $x_i$ with mass $m_i$, where $i=1, \cdots, N$ indexes the agents, evolves according to a coupled dynamical system whose relaxation dynamics can be derived from the generalized Onsager principle~\cite{Wang2020}: 
\begin{equation}\label{onsager:inrtial}
m_i \ddot{x}_i
+ \left(\tfrac{1}{2}\dot{m}_i + m_i R_i\right)\dot{x}_i
+ a_i \nabla F(x_i) = 0,\quad \dot{m}_i=\Psi(m_i, x_i),
\end{equation}
where $R_i$ is the friction operator, whose form depends on the specific model, $a_i$ is the scaling operator acting on the potential energy, and $\Psi(m_i, x_i)$ is a prescribed decay rate for the ith agent's mass $m_i$. 
By choosing  $R_i$ and $a_i$ and stipulating  the dynamics of $m_i$, this formulation provides a guiding framework for designing swarm-based inertial dynamics and subsequent numerical algorithms. This framework incorporates the influence of mass on each agent and extends the previous work by Lu et al.~\cite{lu2024swarm}.

In this study, we choose $R_i$ and $a_i$ in~\eqref{onsager:inrtial} to arrive at a novel family of inertial dynamical system for agent i as follows (note that we drop the index i for the sake of simplicity):
\begin{equation} \label{ODE}
\ddot{x}(t) + [\frac{\alpha}{t}\mathbf{I}
- \frac{\alpha}{t} \mathbf{1} \nabla F(x(t))^T
+ \gamma(t)\nabla^2 F(x(t))]\cdot \dot{x}(t)
+ \beta(t)\nabla F(x(t)) = 0,
\end{equation}
where $\alpha > 0$ is a constant, $\mathbf{I}$ is the $d\times d$ identity matrix, $\beta(t), \gamma(t) > 0$ are time-dependent functions that implicitly incorporate mass, $\mathbf{1}=(1, \cdots, 1)^T \in R^d$ is a column vector, and $\nabla^2=\nabla \nabla$. In particular, $\beta(t) = \frac{a}{m}$. 
The key feature of~\eqref{ODE} is that the damping mechanism takes into account \emph{both} the gradient and Hessian information. When the agent is in a basin with a local minimum, the friction is enhanced by a positive definite Hessian; whereas the friction is reduced when the agent is in a region with a negative definite Hessian \emph{et al.}~\cite{attouch2022first}. In addition, 
one of the main novelties of ~\eqref{ODE} is the new nonlinear "disturbing" term $- \frac{\alpha}{t} \nabla F(x(t)) \cdot \dot{x}(t) \mathbf{1}$, which extends the model of Attouch \emph{et al.}~\cite{attouch2022first}. The new term plays two distinctive roles: when the agent is in a descent trajectory in the energy landscape, this adds an additional friction force in all directions along the trajectory, whereas it adds an accelerating force in all directions when the agent is in an ascent trajectory, disturbing the system and promoting more vigorous exploration.

This new mechanism contributes to an acceleration when the agent is on a path where the potential energy increases and a deceleration while it is on a descent trajectory.
As a result, the evolution of the effective kinetic energy is more complex in system~\eqref{ODE}. 
This is detailed in Theorem~\ref{continuous:thm} in Section~\ref{section:new:inertial}.

In global minimization problems,  $F - F^*\geq 0$ is well defined. Moreover, $F - F^*$ can be interpreted as the effective potential energy or the optimality gap in function values, so it is appropriate to establish convergence rates in terms of $F - F^*$.
In Theorem~\ref{continuous:thm}, we show that  the quantity, $F(x(t)) - F^*$, converges to zero asymptotically at a rate of bounded by $\mathcal{O}\!\left(1/\delta(t)\right)$ in dynamical system ~\eqref{ODE}, where $\delta(t)\to \infty$  as $t\to \infty$, defined in Section~\ref{section:new:inertial}.
We then develop structure-preserving numerical algorithms for system \eqref{ODE} and show that numerical function value $F_k - F^* $ converges at least in the rate of $\mathcal{O}\!\left(1/\delta_k\right)$, as $k\to \infty$, where $\delta_k$ is the discrete analogue of $\delta(t)$ (see Theorem~\ref{fully-discretized:energystability} and~\ref{imexrb:energystability} in Section~\ref{section:new:inertial}). 

We then design our \emph{swarm-based inertial methods} (SBIMs) by incorporating the numerical algorithms for \eqref{ODE} into the swarm-based optimization framework.
The predicated convergence rates are confirmed in the first part of Section~\ref{numerical:exp} using four convex functions. The numerical results demonstrate much better performance than the theoretical estimates, including faster convergence rates and improved energy stability in higher-dimensional settings.
Finally, we compare the SBIMs with swarm-based gradient descent and Nesterov methods on a large class of nonconvex optimization problems. The results show improved overall performance in reaching the global minimum, substantiated by a wider range of exploration and greater stability when approaching the minimum.

This work is motivated by three main lines of research in the past: swarm-based optimization, dissipative inertial dynamics, and structure-preserving algorithms. We briefly review these areas next
 and explain how they motivate this work.

\subsection{Swarm-based optimization}
\noindent \indent Swarm-based optimization methods use a collection of interacting agents to minimize an objective function $F$ over a domain $\Omega$. Each agent explores the landscape locally, while information exchange among agents guides the swarm toward regions with lower values of the objective function. This multi-agent mechanism is particularly useful for landscapes of the objective function with multiple minima.

Particle swarm optimization (PSO), introduced by Kennedy and Eberhart~\cite{kennedy1995particle}, is one of the best-known swarm-based algorithms. 
In PSO, particles combine local exploration with social interactions (communication stage), updating their states based on their current positions, velocities, and historical best positions.
PSO methods have been studied from various perspectives, including parameter tuning~\cite{shi1998parameter,shi1999empirical,clerc2002particle}, exploration enhancement~\cite{clerc2002particle,eberhart2001tracking,bratton2007defining,pant2007simple,han2014diversity}, network design~\cite{mendes2004fully,liang2005dynamic,bird2006adaptively,kennedy2002population}, and hybridization with other optimization strategies~\cite{angeline1998using,lovbjerg2001hybrid,premalatha2009hybrid,chen2007particle}. Rigorous convergence analysis of PSO remains challenging. Existing approaches include linear stability analysis~\cite{clerc2002converge}, Markov-chain and martingale techniques~\cite{XU201865}, and stochastic differential equation and mean-field formulations~\cite{huang2023global, Ding2024ssde}.
These challenges motivate the development of swarm-based frameworks from a dynamical perspective.

\subsubsection{Gradient-driven swarm dynamical systems}\label{section:eitan}
\noindent\indent Recent work by Lu \emph{et al.}~\cite{lu2024swarm} and Tadmor and Zengino\u{g}lu~\cite{tadmor2024swarm} propose gradient-driven swarm dynamics in which each agent carries both a mass. In the models, the agent mass influences the effective step size of its motion, while mass exchange provides a mechanism for communication across the swarm of agents. 
As a result, agents with smaller masses explore more, whereas those in regions with lower objective function values tend to accumulate mass.

The swarm-based gradient descent model in~\cite{lu2024swarm} consists of the following equations. The mass conservation condition:
\begin{equation}
\sum_{i=1}^{N} m_i(t) = 1,
\end{equation}
where $m_i(t)\in(0,1]$ denotes the mass of agent $x_i$. During the motion, the agents follow the following gradient descent dynamics:
\begin{equation}\label{SBG:inertial}
\frac{d x_i}{dt} = - h_i(x_i,m_i)\nabla F(x_i), \qquad i=1,\dots,N,
\end{equation}
where $h_i$ is a mass-dependent mobility factor. During the communication stage, mass is transferred from agents with higher objective function values to those with lower values, while the total mass remains conserved. A representative communication rule is given by 
\begin{align}\label{masstrans}
\frac{d}{dt}m_i(t) &= -\phi_p(\eta_i(t))\,m_i(t), \qquad i \neq i_-(t),\\
m_{i_-(t)}(t) &= 1 - \sum_{j \neq i_-(t)} m_j(t),
\end{align}
where $i_-(t)\in\arg\min_{1\le j\le N} F(x_j(t))$ and
\[
\eta_i(t) =
\frac{F(x_i(t)) - F_{\min}(t)}{F_{\max}(t) - F_{\min}(t)}, 
\qquad
\phi_p(\eta)=\eta^p,\quad p>0,
\]
with $F_{\min}(t)=\min_j F(x_j(t))$ and $F_{\max}(t)=\max_j F(x_j(t))$. When $F_{\max}(t)=F_{\min}(t)$, $\eta_i(t)=0$.
This class of models provides deterministic dynamical-systems to study the agents' exploration and communication in swarm optimization. We also incorporate this communication mechanism into our swarm-based optimization framework.

\subsection{Dissipative inertial dynamical systems}\label{section:attouch}
\noindent \indent Continuous-time viewpoints play an important role in understanding acceleration and convergence in optimization. Early inertial methods employed in optimization algorithms include the heavy-ball method of Polyak~\cite{polyak1964some}, which introduces momentum into gradient-based optimization, and Nesterov's accelerated gradient method~\cite{Nesterov1983}, which achieves faster convergence for convex problems. Both admit continuous-time interpretations, viewed as second-order dissipative systems with friction~\cite{attouch2000heavy,su2016differential}.

Su, Boyd, and Cand\`es~\cite{su2016differential} showed that Nesterov's accelerated gradient method can be interpreted as the continuous-time limit of the second-order ODE
\begin{equation}\label{Nes:ODE}
\ddot{x}(t) + \frac{3}{t}\dot{x}(t) + \nabla F(x(t)) = 0.
\end{equation}
Subsequently, Attouch \emph{et al.}~\cite{attouch2018fast} considered the more general damping term $\frac{r}{t}\dot{x}(t)$ with $r>0$. This term acts as a time-dependent damping that vanishes as $t\to\infty$, admits a nonincreasing energy, and yields accelerated convergence rates when $r\ge 3$.

In~\cite{alvarez2002second}, Alvarez \emph{et al.} studied a second-order dissipative system including the Hessian in the friction:
\begin{equation}
\ddot{x}(t) + \alpha \dot{x}(t)+ \beta \nabla^2 F(x(t)) \dot{x}(t)+ \nabla F(x(t)) = 0.
\end{equation}
This framework was later extended in~\cite{attouch2022first}, where Attouch \emph{et al.} studied inertial dynamics with time-dependent Hessian-driven damping in the following form,
\begin{equation}\label{Attouch:Hessian:ODE}
\ddot{x}(t) + \frac{\alpha}{t}\dot{x}(t)+ \beta(t)\nabla^2 F(x(t))\dot{x}(t)+ b(t)\nabla F(x(t)) = 0.
\end{equation}
For~\eqref{Attouch:Hessian:ODE}, they established convergence of the continuous-time trajectories via a Lyapunov analysis. Moreover, by viewing $\nabla^2 F(x(t))\dot{x}(t)$ as the time derivative of $\nabla F(x(t))$, they derived first-order discrete algorithms that incorporate gradient-difference corrections. These algorithms preserve key features of the continuous-time model and retain accelerated convergence rates in discrete time under suitable assumptions and parameter choices.

\subsection{Structure-preserving discretizations}
\noindent \indent Bridging continuous-time dynamics and numerical algorithms has motivated substantial work on inertial ODEs, Hessian-driven damping, and related numerical schemes; see, for example,~\cite{attouch2022first,attouch2022fast,attouch2020newton,aujol2023fast} for more details.In general, many optimization algorithms can be understood as time discretizations of continuous dynamical systems. 
This viewpoint has led to growing interest in optimization schemes that retain key features of the continuous dynamics, such as stability, energy dissipation, and convergence behavior; see, for example,~\cite{muehlebach2021optimization,shi2019acceleration}. Related methods for accelerating optimization by using previous iterates can be found in~\cite{scieur2016regularized}.

\subsection{Outline of the paper}
\noindent \indent The works above motivated us to develop a numerical optimization framework that combines swarm interaction, inertial dynamics, and structure-preserving discretization. In this paper, we develop a framework for relaxation dynamics through the Onsager principle~\eqref{onsager:inrtial}, propose new inertial dynamics~\eqref{ODE}, and integrate them into a swarm-based optimization framework with corresponding structure-preserving numerical algorithms. Our main contribution is to formulate this framework and analyze both the continuous-time system and the discrete algorithms via an energy or Lyapunov approach.

The remainder of this paper is organized as follows. In Section~\ref{section:sbim}, we introduce relaxation dynamics within the swarm-based optimization framework using the generalized Onsager principle, and use the Nesterov accelerated dynamics~\eqref{Nes:ODE} to illustrate parameter selection. In Section~\ref{section:new:inertial}, we present the proposed swarm-based inertial dynamics~\eqref{ODE} and analyzes its asymptotic behavior via a Lyapunov functional in Theorem~\ref{continuous:thm}; we also develop discrete algorithms that preserve the accelerated convergence rates in Theorems~\ref{fully-discretized:energystability} and~\ref{imexrb:energystability}. In Section~\ref{numerical:exp}, we conduct numerical experiments to confirm convergence rates, to demonstrate performance of the swarm-based algorithms on nonconvex problems, and to compare the results with various available swarm-based methods. Finally, Section~\ref{section:conclude} concludes the paper and discusses future directions. Additional proofs and numerical results are provided in the appendix.



\section{Swarm-based inertial algorithms}\label{section:sbim}
\noindent\indent In the swarm-based optimization with $N$ agents, optimizing the objective function $F(x_i)$ for each agent $x_i$ requires a specific optimization algorithm to reduce the value of $F(x_i)$, $i = 1, \dots, N$ during the process. In classical swarm-based optimization methods, the steepest gradient descent is commonly used. However, a major drawback of the gradient descent approach is that some agents may be trapped in local extrema. To address this limitation and improve the optimization process, we introduce an underdamping mechanism to allow the system to escape from local extrema and thereby to improve the chance of reaching the global extremum within a non-equilibrium thermodynamic framework. 

Specifically, we replace the overdamped gradient descent method with an underdamped  approach with inertia. We call this method the Swarm-Based Inertial Method (SBIM). The introduction of inertia into the optimization process or dynamics strengthens the coupling between each agent’s dynamics and its mass, which effectively enhances the activity (or vitality) of the agents during the optimization process. If the inertia and damping are properly balanced, this approach can increase the likelihood that one of the agents reaches the global extremum. As the result, the method could significantly improve the swarm-based global optimization process when optimizing non-convex objective functions in practice.

\subsection{Relaxation dynamics with inertia }\label{Relaxation:dynamics}

\noindent \indent We first discuss how to design the under-damped dynamics systematically. We note that an optimization algorithm is in fact a discrete dynamical system. Its convergence could be analyzed by connecting the discrete iteration with a corresponding continuous-time dynamical system. Lyapunov functions are then  called upon to understand the long-time behavior of the dynamical system, which often sheds light on convergence of the algorithm.
Motivated by this connection, we begin with the total mechanical energy of a continuous dynamical system, consisting of the kinetic energy and the potential energy (i.e., the objective function). We then formulate a relaxation dynamics that dissipate the energy of the system in a certain way. Discretizing the dynamical system yields an optimization algorithm.  

Specifically, the mechanical energy, or Lyapunov function, for each agent $x_i$ is defined as follows:
\ben\label{def:sb:energy}
E_i=E(x_i)=\frac{1}{2}m(x_i,t)\|\dot{x_i}(t)\|^2+a_iF(x_i),\quad i=1, \cdots, N,
\een
 $a_i > 0$ is a scaling constant, $F(x_i)$ and $\frac{1}{2} m(x_i,t)\|\dot{x}_i(t)\|^2$ denote the potential energy and kinetic energy of the agent, $x_i(t)$, respectively.

In a dissipative dynamical system, the kinetic energy should eventually vanish while the potential energy reaches its minimum at the value $\min_{x \in \Omega} F(x)$.  To design such relaxation dynamics, we compute the energy dissipation rate directly from \eqref{def:sb:energy}:
\begin{align*}
    {\frac{d}{dt}}E_i &= \frac{1}{2} \bigg(\frac{\partial m}{\partial x_i}\dot{x_i}+\frac{\partial m}{\partial t}\bigg)\|\dot{x_i}\|^2+m(x_i,t)\ddot{x_i}\dot{x_i}+a_i\nabla F(x_i)\cdot \dot{x_i}\\
    &= \frac{1}{2}\frac{dm}{dt}\dot{x_i}\dot{x_i}+\bigg(m_i\ddot{x_i}+a_i\nabla F(x_i)\bigg)\dot{x_i}\\
    &=\bigg(\frac{1}{2}\dot{m_i}\dot{x_i}+m_i\ddot{x_i}+a_i\nabla F(x_i)\bigg)\dot{x_i}.
\end{align*}

To ensure energy dissipation, we consider a quadratic form and set
\[
\frac{d}{dt} E_i = - m_i \dot{x}_i^{\top} R_i \dot{x}_i,
\]
with $R_i \ge 0$, where $m_i R_i$ denotes the friction matrix or operator for the $i$-th agent. The prefactor $m_i$ in front of the friction operator reflects that the energy dissipation rate scales with the mass of the agent, which couples mass to the swarm-based dynamics.
Under this assumption, we obtain the momentum balance equation for the $i$-th agent as follows:
 \begin{equation}\label{NM}
     m_i\ddot{x_i}=-\bigg(\frac{1}{2}\dot{m_i}+m_iR_i\bigg)\dot{x_i}-a_i\nabla F(x_i).
 \end{equation}

 We denote $\dot{x}_i(t) = y_i(t)$. Together with a postulated evolutionary equation of mass, complete governing system of equations for the N-agent system is then rewritten into the following system:
\ben
\left\{
\bea{l}
\dot{x}_i = y_i,\\
\dot{y}_i = -\bigg(\dfrac{\dot{m}_i}{2m_i} + R_i\bigg) y_i - \dfrac{a_i}{m_i} \nabla F(x_i)=-(-\frac{1}{2}\phi_p+{R_i})y_i-\frac{a_i}{m_i}\nabla F_i,\\
\dot{m_i}=-\phi_p(m_i,t)m_i, \quad i=1,\cdots, N.
\eea
\right.\label{N-system}
\een
 This is an under-damped dynamical system, where the total energy decreases with a proper choice of the friction matrix, $R_i$. The total energy is given by
\ben
E=\sum_{i=1}^NE_i(x_i)=\sum_{i=1}^N[\frac{1}{2}m(x_i,t)\|y_i(t)\|^2+a_iF(x_i)].
\een

If we ignore the inertial term in the momentum balance equation, it reduces to the over-damped dynamical system, on which the gradient descent algorithm is based. We remark that different choices of $R_i$ yields different descent algorithms. For example, $R_i=r \bI$ leads to the steepest descent method while $R_i=\nabla F(\bx_i)$ yields the Newton's method.

In summary, the dynamical system of a compressible system is given by \eqref{N-system}, in which mass may not be conserved.
To add an incompressible constraint to the system, the dynamical system must be modified as follows:
\ben
\left\{
\bea{l}
\dot{x_i}=y_i,\\
\dot{y_i}=-(-\frac{1}{2}\phi_p+{R_i})y_i-\frac{a_i}{m_i}\nabla F_i,\ i\neq i_-,\\
\dot{m_i}=-\phi_p(m_i,t)m_i,\ i\neq i_-,\\
\sum_{i=1}^N m_i=1.
\eea\right.
\label{gf-2}
\een
We will comment on the performance of swarm-based inertial algorithms designed based on the incompressible dynamics later.

Notice that \eqref{gf-2} is not a gradient flow system. However, 
the following subsystem when $m_i$ is viewed as a prescribed is a gradient flow:
\ben
\left(
\bea{l}
\dot{x_i}\\
\dot{y_i}
\eea\right)
=\left(
\bea{cc}
0 & 1\\
 -1&-(R_i-\phi_p/2)\eea
\right)
\cdot \left(
\bea{l}
\frac{\delta e_i}{\delta x_i}\\
\frac{\delta e_i}{\delta y_i}
\eea
\right),\label{GF-3}
\een
where $e_i=\frac{E_i}{m_i}$. We define
\ben
Q_i=R_i-\phi_p/2, i=1,\cdots, N.
\een
The dynamical system is written into the following form
\ben
\left(
\bea{l}
\dot{x_i}\\
\dot{y_i}
\eea\right)
=\left [\left(
\bea{cc}
0 & 0\\
0&-Q_i
\eea
\right)+\left (
\bea{cc}
0& 1\\
-1 & 0
\eea\right)\right]
\cdot \left(
\bea{l}
\frac{\delta e_i}{\delta x_i}\\
\frac{\delta e_i}{\delta y_i}
\eea
\right),
\een
where the mobility matrix can be decomposed into a symmetric and an anti-symmetric part. The symmetric part contributes to energy dissipation provided $Q_i>0$, while the anti-symmetric part does not. In this study, we assume $Q_i\geq 0$ in addition to $R_i\geq 0$ for $i=1,\cdots, N$. 

In general, we can introduce a general dynamical system to guide the relaxation of the agents over the landscape defined by the energy as follows:
\ben
\left(
\bea{l}
\dot{x_i}\\
\dot{y_i}
\eea\right)
=-\left(
\bea{cc}
M_{11}& M_{12}\\
M_{21}& M_{22}
\eea
\right)
\cdot \left(
\bea{l}
\frac{\delta e_i}{\delta y_i}\\
\frac{\delta e_i}{\delta x_i}
\eea
\right),\label{GF-30}
\een
where $(M_{ij})\geq 0$ is known as the mobility operator. The selection of an effective mobility operator using asymptotic convergence and convergence rates as design criteria is systematically studied in~\cite{fazlyab2018analysis}. Here, we will not pursue this general dynamics in this study.

\subsection{Nesterov accelerated dynamics within the SBIM framework}


\noindent \indent For optimization problems, the classical Nesterov accelerated gradient method is given as follows. Given an initial point $x_0$ and $y_0 = x_0$ and step size $s$, the iterative scheme generates the following sequence:
\ben
\left\{
\bea{l}
x^{k}=y^{k-1}-s\nabla F(y^{k-1}),\\
y^{k}=x^{k}+\frac{k-1}{k+2}(x^{k}-x^{k-1}),\quad k=1, 2, \cdots.
\eea\right.
\een
According to \cite{su2016differential}, the Nesterov accelerated gradient method can be obtained from a proper discretization of the following ODE:
\ben
\ddot{X}+\frac{3}{t}\dot{X}+\nabla F(X)=0,\label{Nesterov-ode}
\een
where $\Delta t=\sqrt{s}$ and $t=k\Delta t$.

Compared with ODE \eqref{NM} in SBIM, we have the following mass dynamic equation for the i-th agent:
\ben\label{continuous:Riai}
\left\{
\bea{l}
\dot{m_i}=(-2R_i+\frac{6}{t})m_{i},\\
a_i=m_i,
\eea\right.
\een
where $t=k\sqrt{s},\ k\in \mathbb{N^+}$. 

We discretize \eqref{GF-3} using the backward Euler method on the first and the forward Euler on the second equation to obtain the Nesterov method:
\ben
\bea{l}
x_i^{n+1}-x_i^{n}=\Delta t y_i^{n+1},\\
y_i^{n+1}-y_i^{n}=\Delta t[(-R_i^{n}+\frac{1}{2}\phi_p(x_i^{n}))y_i^{n}-\frac{a_i^{n}}{m_i^{n}}\nabla F_i(x_i^{n})],
\eea
\een
where $x_i^{n+1}$ denotes the state of the $i$-th agent at the $(n+1)$-th time step.
In addition, we discretize the dynamical equation of mass using the forward Euler method:
\ben
m_i^{n+1}-m_i^{n}=-\Delta t \phi_p(x_i^{n}) m_i^{n}, i\neq i_-,
\een
where $i_-=arg \min F(x_i^{n}).$ 
To enforce the mass conservation law at every time step, the mass of the agent corresponding to the minimum energy is given by:
\ben
m_{i_-}^{n+1}=1-\sum_{i\neq i_-} m_i^{n+1}.
\een
The above discrete equations give the Nesterov scheme within the SBIM framework.
Moreover, based on \eqref{continuous:Riai}, we have the following requirements on the discretized forms $R_i$ and $a_i$ at the $n$-th time step:

If $i\neq i^-_{n}$
\ben\label{riai:Nes}
\left\{
\bea{l}
-\frac{1}{2}\phi_p(x_i^{n-1})m_{i}^{n-1}+m_{i}^{n} R_i^{n}=\frac{3}{n\sqrt{s}}m_i^{n}\\
a_i^{n}=m_i^{n}
\eea\right.
\een

If $i= i^-_{n}$
\ben
\left\{
\bea{l}
-\frac{1}{2}[(1-\sum_{i\neq i^-_{n}}m_i^{n})-m_{i^-}^{n-1}]+m^n_{i^-} R_{i^-}^n=\frac{3}{n\sqrt{s}}m^n_{i^-}\\
a_{i^-}^{n}=m_{i^-}^{n}
\eea\right.
\een

We summarize the above into the following algorithm.
\begin{alg}[Nesterov Accelerated Gradient SBIM]\label{classical:Nesterov Algorithm}
\ben
\left\{
\bea{l}
x_i^{n+1}=x_i^{n}+\Delta t y_i^{n+1},\\
m_i^{n+1}=m_i^{n}-\Delta t \phi_p(x_i^{n}) m_i^{n}, i\neq i^-_n,\\
m_{i_-}^{n+1}=1-\sum_{i\neq i_-} m_i^{n+1},\ i= i^-_n\\
y_i^{n+1}=y_i^{n}+\Delta t[(-R^n_i+\frac{1}{2}\phi_p(x_i^{n}))y_i^{n}-\frac{a^n_i}{m_i^{n}+\epsilon}\nabla F_i(x_i^{n})],\ i\neq i^-_n,\\
y_i^{n+1}=y_i^{n}+\Delta t[(-R^n_i+\frac{1}{2}\frac{(1-\sum_{i\neq i^-_{n}}m_i^{n+1}-m_{i^-_{n}}^{n})}{m^n_{i^-_{n}}})y_i^{n}-\frac{a^n_i}{m_i^{n}+\epsilon}\nabla F_i(x_i^{n})],\ i= i^-_n,
\eea\right.
\een
where $\epsilon>0$ is a small positive regularization parameter to ensure there is no zeros in denominators.
\end{alg}

The convergence property of the classical Nesterov method has been well studied through analysis of the objective function (potential energy) \cite{aujol2019optimal} as well as  from an ODE perspective \cite{su2016differential}. 
Here, we analyze the convergence property of Algorithm~\eqref{classical:Nesterov Algorithm} in the SBIM setting. Instead of analyzing the error in  $\|x_n - x^*\|$, we focus on a discrete potential energy of the discrete system, $|F_n - F^*|$. We refer to this convergence property as the convergence of function values. We will show next that energy dissipation implies the convergence of $|F_n - F^*|$, which in turn indicates the convergence of $x - x^*$ when $F$ is strongly convex. 

The reason for focusing on the convergence of $F_n - F^*$ is that, in practice, especially for large-scale optimization problems, the distance $\|x_n - x^*\|$ may vary significantly depending on the choice of norm due to the high dimension of the problem. Moreover, for highly oscillatory objective functions, even when $\|x_n - x^*\|$ is small, the value of $F_n - F^*$ can still be large, which is undesirable in minimizing residual problems. 
For this type of optimization problem, the primary goal is to reduce the objective function value rather than finding the exact minimizer. More importantly, when the objective function $F$ is strongly convex, convergence of $F_n - F^*$ implies convergence of $x_n - x^*$. From this perspective, studying the convergence of $F_n - F^*$ is more meaningful than studying $\|x_n - x^*\|$ in some optimization problems

\begin{definition}
The discrete energy for each agent is defined by
\[E_i^{n} = a_i \bigl(F_i^{n} - F(x^*)\bigr) + \frac{1}{2} m_i^{n} \|y_i^{n}\|^2,\]
where $x^*$ is a minimizer of $F$, and $i$ denotes the $i$-th agent. We denote $i_n^- = \arg\min_{1 \le i \le N} F(x_i^{n})$. The total discrete energy of the system is defined as follows
\[E^{n} = \sum_{i=1}^N E_i^{n}.\]
\end{definition}

\begin{theorem}[Discrete Energy Dissipation of the SBIM--Nesterov Scheme]\label{thm:Nes}
	
Let $\epsilon = 0$ in Algorithm~\eqref{classical:Nesterov Algorithm}. Assume for all $i$, $F \in C^1$, $m_i^{n} > 0$ for all $n \in \mathbb{N}$, and the updates of $R_i^{n}$ and $a_i^{n}$ follow \eqref{riai:Nes}.
	
Then, discrete energy in Algorithm~\eqref{classical:Nesterov Algorithm} satisfies
\[E_i^{n+1} - E_i^{n}= -\, m_i^{n}\, \Delta t\, R_i^{n}\, \|y_i^{n}\|^2 \le 0\]
for all $n \in \mathbb{N}$ and all $i \neq i_n^-$, neglecting higher-order terms of order $\mathcal{O}(\Delta t^2)$.
\end{theorem}
The proof is given in Appendix ~\ref{proof:NM}.

\begin{remark}\label{remark:Nes}
	Note that in Theorem~\ref{thm:Nes}, we can only prove that the discrete energy corresponding to agents with $i \neq i_n^{-}$ is non-increasing. In general, it is difficult to conclude that the total discrete energy of the system is decreasing. This is because when $i = i_n^{-}$, the discrete energy of that agent may increase as its mass increases. As a result, the total discrete energy of the system may either increase or decrease over time, which leads to oscillatory behavior in the discrete energy plot.
	
	The theorem shows the asymptotic behavior of $y_i^{n}$ for all agents with $i \neq i_n^{-}$. In particular, it implies that the dynamics of the Swarm-Based Nesterov algorithm approach stationary points of the objective function $F$. The limit of the sequence generated by the Swarm-Based Nesterov algorithm for agent i with $i \neq i_n^{-}$ is a stationary point of $F$. If the objective function $F$ is strongly convex, then this stationary point is unique. Moreover, if there exist an index $i^*$ and an integer $N$ such that $i_n^- = i^*$ for all $n > N$, then under the strong convexity of $F$ locally, the trajectory $x_{i^*,n}$ converges to the unique minimizer of $F$.
\end{remark}

The practical implementation of the algorithm is given as follows: 

\begin{algorithm}[H]
\caption{Swarm-Based Nesterov Algorithm}
\label{sbnm:alg}
\DontPrintSemicolon
\SetKwInOut{Input}{Input}
\SetKwInOut{Output}{Output}

\Input{
Number of agents $N$, objective function $F(x)$, time step $h$,\\
parameters $\alpha$,\\
tolerances $\mathrm{tol}_{\mathrm{res}}, \mathrm{tol}_m, \mathrm{tol}_{\mathrm{merge}}$
}

\Output{Final agent set and best solution}

\BlankLine

Initialize agents $\{x_i^{0}\}_{i=1}^N$ randomly.
Set masses $m_i^{0} = 1/N$.
Generate $\{x_i^{1}\}_{i=1}^N$ if needed.

\For{$n=2,\dots$ until the stopping criterion is satisfied}{

  \tcp{1. Communication (mass update)}
  \ForEach{agent $i$}{
    \uIf{$m_i^{n} < \mathrm{tol}_m$}{
      Remove agent $i$\;
    }
    \Else{
      Update $m_i^{n+1}$ using the mass transition rule~\eqref{masstrans}\;
    }
  }

  \tcp{2. Nesterov inertial update}
  \ForEach{remaining agent $i$}{
    Compute $x_i^{n+1}$ using~\eqref{classical:Nesterov Algorithm}\;
    \begin{align*}    
    y_i^{n+1}&=y_i^{n}+h[(-R^n_i+\frac{1}{2}\phi_p(x_i^{n}))y_i^{n}-\frac{a^n_i}{m_i^{n}+\epsilon}\nabla F_i(x_i^{n})],\ i\neq i^-_n,\\
    y_i^{n+1}&=y_i^{n}+h[(-R^n_i+\frac{1}{2}\frac{(1-\sum_{i\neq i^-_{n}}m_i^{n+1}-m_{i^-_{n}}^{n})}{m^n_{i^-_{n}}})y_i^{n}-\frac{a^n_i}{m_i^{n}+\epsilon}\nabla F_i(x_i^{n})],\ i= i^-_n,\\
    x_i^{n+1}&=x_i^{n}+h y_i^{n+1},
    \end{align*}
    where $\epsilon>0$ is a small positive parameter
  }

  \tcp{3. Merging}
  \For{each pair of agents $(i,j)$}{
    \If{$\|x_i^{n+1} - x_j^{n+1}\| < \mathrm{tol}_{\mathrm{merge}}$}{
      Merge agents $i$ and $j$\;
    }
  }

}

\Return best agent and corresponding objective function value\;

\end{algorithm}

	
	
		
		

\section{New underdamped inertial dynamics}\label{section:new:inertial}

\subsection{Underdamped ODEs with damping coupled with the Hessian}

\noindent \indent Recently, a new underdamped dynamics, inspired by the continuous-time ODE formulation of the Nesterov accelerated gradient method,  was introduced, leading to an inertial system with Hessian-driven damping \cite{attouch2022first},
\[\ddot{x}(t) + \frac{\alpha}{t}\dot{x}(t) + \beta(t)\nabla^2 F(x(t))\dot{x}(t) + \nabla F(x(t)) = 0.\]
Compared with the continuous-time Nesterov dynamics, we note that this system adds the curvature effect to the damping dynamics via the Hessian of potential function $F(x)$. 

To promote interactions between the curvature and the energy gradient, we propose a new dynamics for each agent as follows,
\begin{equation}
	\ddot{x}(t)+\frac{\alpha}{t}\dot{x}(t) - \frac{\alpha}{t}\langle \nabla F(x(t)) , \dot{x}(t)\rangle \mathbf{1}
	+ \gamma(t)\nabla^2 F(x(t))\dot{x}(t)+\beta(t)\nabla F(x(t)) = 0,\label{gen-ode}
\end{equation}
where $\alpha > 0, \beta>0,$ and $\gamma(t) > 0$ are control parameters,  $x \in \mathbb{R}^d$, and $\mathbf{1}$ denotes the vector in $\mathbb{R}^d$ whose entries are all equal to one.


For this ODE, we have the following theorem.
\begin{theorem}\label{continuous:thm}
	Let $F : \mathbb{R} \to \mathbb{R}$ be a convex function and in $C^2$, $\nabla F$ is $L$-Lipschitz continuous with Lipschitz constant $L$,  $\alpha \ge 1$,   $x : [t_0, +\infty) \to \mathbb{R}$ a solution trajectory of \eqref{gen-ode},  $x^* := \argmin_{x\in\mathbb{R}} F(x)$, and $F^* := F(x^*)$. We define
	\begin{align}
		\omega(t) &:= \frac{\gamma(t)}{t} + \dot{\gamma}(t) - \beta(t), 		\qquad 
		\delta(t) := -t^2 \omega(t), \\ 
		E(t) &:= \delta(t)\bigl(F(x(t)) - F(x^*)\bigr) + \frac{1}{2}\|v(t)\|^2,\label{eq:lyapunov} \\
		v(t) &:= (\alpha - 1)(x(t) - x^*) 
		- \alpha\bigl(F(x(t)) - F(x^*)\bigr)\mathbf{1}
		+ t\bigl(\dot{x}(t) + \gamma(t)\nabla F(x(t))\bigr).
	\end{align}
	Suppose that the following conditions hold:
	\begin{enumerate}
		\item $\dfrac{\gamma(t)}{t} + \dot{\gamma}(t) - \beta(t) < 0$;
		\item $\dot{\delta}(t) + (\alpha - 1)t\omega(t) - \alpha t \omega(t)\sqrt{2L} \le 0$;
		\item $\gamma(t) \ge 0$.
	\end{enumerate}
	Then,
	\[\frac{d}{dt} E(t) \le 0,\]
	and
    \[F(x(t)) - \min_{x} F \leq \frac{E(t_0)}{\delta (t)}.\]
\end{theorem}

\begin{proof}
	We consider the energy function, $E(t)$, which consists of a potential energy term and a kinetic energy term:
	\begin{equation}\label{energy:func}
	E(t) := \delta(t)(F(x(t))-F(x^*))+\frac{1}{2}\|v(t)\|^2 .
	\end{equation}
	$\delta(t)$ will be determined later. Note that $v(t)$ is no longer given by the simple expression $\dot{x}(t)$. Nevertheless, as $t \to +\infty$, we still have $v(t) \to 0$. The function $v(t)$ is defined as
	\begin{equation}\label{velocity}
	v(t) := (\alpha-1)(x-x^*)-\alpha(F(x(t))-F(x^*))\mathbf{1}+t(\dot{x}+\gamma(t)\nabla F(x(t))).
    \end{equation}
   We now consider the behavior of the energy function over time. Differentiate $E(t)$ yields
    \begin{equation}
    	\dfrac{d}{dt} E(t) = \dot{\delta}(t)(F(x(t))-F(x^*))+\delta(t)\langle\nabla F(x(t)), \dot{x}(t) \rangle + \la v(t),\dot{v}(t)\ra.
    \end{equation}
    Then, from \eqref{velocity} we observe
    \begin{align*}
    	\dot{v}(t) &= (\alpha-1) \dot{x}(t)-\alpha \la \nabla F(x(t)), \dot{x}(t)\ra \mathbf{1}+\dot{x}(t)+\gamma(t)\nabla F(x(t))+ t\bigg(\Ddot{x}(t)+\dot{\gamma}(t)\nabla F(x(t))+\gamma(t) \nabla^2 F(x(t)) \dot{x}(t) \bigg)\\
    	&= \alpha \dot{x}(t)-\alpha \la \nabla F(x(t)),\dot{x}(t)\ra \mathbf{1}+\gamma(t)\nabla F(x(t))+t\bigg(\Ddot{x}(t)+\dot{\gamma}(t)\nabla F(x(t))+\gamma(t)\nabla^2F(x(t))\dot{x}(t)\bigg)\\
    	&= \gamma(t)\nabla F(x(t))-t\beta(t)\nabla F(x(t))+t\dot{\gamma}(t)\nabla F(x(t)) \\
    	&=t\omega(t)\nabla F(x(t)). 
    \end{align*}
    Here we denote $\omega(t):= \frac{\gamma(t)}{t}+\dot{\gamma(t)}-\beta(t)$. Then,
    \begin{align*}
    	\la v(t), \dot{v}(t) \ra&= \la (\alpha-1)(x(t)-x^*)-\alpha (F(x(t))-F^*)\mathbf{1}+t(\dot{x}(t)+\gamma(t)\nabla F(x(t)), t\omega(t)\nabla F(x(t)) \ra \\
    	&= (\alpha-1)t\omega(t) \la(x(t)-x^*), \nabla F(x(t)) \ra-\alpha t\omega(t)(F(x(t))-F^*)\la \mathbf{1},\nabla F(x(t))\ra\\ &+t^2 \omega(t) \la \dot{x}(t), \nabla F(x(t))\ra +t^2 \omega(t)\gamma(t)\|\nabla F(x(t))\|^2
    \end{align*}
     The term $\langle \nabla F(x(t)), \dot{x}(t) \rangle$ is difficult to control. To eliminate this term in $\frac{d}{dt}E(t)$, we define $\delta(t) := -t^2 \omega(t)$.
     
    \begin{align*}
    	\frac{d}{dt}E(t)  &= \dot{\delta}(t)(F(x(t))-F^*)+\delta(t)\la \nabla F(x(t)), \dot{x}(t)\ra  + (\alpha-1)t\omega(t) \la(x(t)-x^*), \nabla F(x(t)) \ra \\ &-\alpha t\omega(t)(F(x(t))-F^*)\la \nabla F(x(t)), \mathbf{1} \ra+t^2 \omega(t) \la \dot{x}(t), \nabla F(x(t))\ra +t^2 \omega(t)\gamma(t)\|\nabla F(x(t))\|^2 \\
    	&= \dot{\delta}(t)(F(x(t))-F^*)+ (\alpha-1)t\omega(t) \la(x(t)-x^*), \nabla F(x(t)) \ra\\&-\alpha t\omega(t)(F(x(t))-F^*)\la \nabla F(x(t)), \mathbf{1} \ra+t^2 \omega(t)\gamma(t)\|\nabla F(x(t))\|^2 .\\ 
    \end{align*}
    Since $F$ is convex, $\omega(t)<0$, and $\alpha>1$, we have
    \begin{align*}
    	F^* - F(x(t)) \geq \la \nabla F(x(t)),x^*-x\ra.
    \end{align*}
    It then follows that
    \begin{align*}
    	(\alpha-1)t\omega(t) \la x-x^*, \nabla F(x(t))\ra \leq (\alpha-1)t\omega(t)(F(x(t))-F^*).
    \end{align*}
    The above inequality implies that:
    \begin{align*}
    \frac{d}{dt}E(t) &\leq (\dot{\delta}(t)+(\alpha-1)t\omega(t))(F(x(t))-F^*)- \alpha t\omega(t)(F(x(t))-F^*)\la \nabla F(x(t)), \mathbf{1}\ra + t^2\gamma(t)\omega(t)\|\nabla F(x(t))\|^2.
    \end{align*}
    By Cauchy-Schwarz, we have that
    \begin{equation*}
    	\la \nabla F(x(t)), \mathbf{1}\ra \leq \|\mathbf{1}\|\|\nabla F(x(t))\| = \sqrt{d}\|\nabla F(x(t))\|.
    \end{equation*}
    In addition, by the $L$-Lipschitz condition on $\nabla F$, $\|\nabla^2 F\|\leq L$ holds.
    Define $\phi(y):= F(x)+\nabla F(x)(y-x)+\frac{L}{2}\|y-x\|^2$.
    Hence,
    \begin{align*}
    	F^*\leq \phi(y^*) = F(x) - \frac{1}{2L}\|\nabla F(x(t))\|^2.
    \end{align*}
    Hence,
    \begin{align*}
    	\|\nabla F(x(t))\| \leq \sqrt{2L (F(x(t))-F^*)}.
    \end{align*}
    We now can demonstrate energy dissipation.
    \begin{align*}
    	\frac{d}{dt} E(t) &\leq \bigg(\dot{\delta}(t)+(\alpha-1)t\omega(t)\bigg)(F(x(t))-F^*)-\alpha\omega(t)t\sqrt{2Ld}(F(x(t))-F^*)^{\frac{3}{2}}
        +t^2\gamma(t)\omega(t)\|\nabla F(x(t))\|^2 \\
    	&\leq \bigg(\dot{\delta}(t)+(\alpha-1)t\omega(t)-\alpha\omega(t)t\sqrt{2Ld}\bigg)\max \bigg\{(F(x(t))-F^*), (F(x(t))-F^*)^{\frac{3}{2}}\bigg\} \\&+ t^2\gamma(t)\omega(t)\|\nabla F(x(t))\|^2.
    \end{align*}
    With $\dot{\delta}(t)+(\alpha-1)t\omega(t)-\alpha\omega(t)t\sqrt{2L} \leq 0$ and $t^2\gamma(t)\omega(t) \leq 0$,
    \[\frac{d}{dt} E(t) \le 0,\]
 	Using the energy dissipation property and the non-negativity of the terms in the energy function $E(t)$, we obtain
	\[\delta(t)(F(x(t))-F^*)\leq E(t)\leq E(t_0).\]
	This implies
    \begin{equation}\label{con:F:con}
    (F(x(t))-F^*)\leq E(t)/\delta(t)\leq E(t_0)/\delta(t).
    \end{equation}

\end{proof}
\begin{remark}
It follows from the result of the theorem, 
\[0\leq F(x(t)) - \min F \leq E(t)/\delta(t)\leq E(t_0)/\delta(t),\]
that  $F(x(t)) - \min F $ is bounded by quantity of order $\mathcal{O}\!\left(\frac{1}{\delta(t)}\right)$ for $t\in [t_0, \infty)$. Namely, $F(x(t)) - \min F$ converges to zero at least at the rate of $\mathcal{O}\!\left(\frac{1}{\delta(t)}\right)$ as $t\to \infty$. In the numerical section, we will show the actual "convergence rate of $F(x(t))-minF$" is better than $\mathcal{O}\!\left(\frac{1}{\delta(t)}\right)$ in all the cases tested.

\end{remark}

\subsection{A fully discretized algorithm (FD algorithm)}

\noindent \indent Guided by the theorem, we next develop a set of discrete algorithms that preserve both the energy dissipation and the convergence rate of function values by properly discretizing the ODE with proper control parameters.  
  In the first scheme, we apply the backward Euler method to discretize $\dot{x}(t)$, $\langle \nabla F(x(t)), \dot{x}(t) \rangle$, and $\nabla^2 F(x(t)) \dot{x}(t)$, and use a central difference to discretize $\ddot{x}(t)$. Let $h > 0$ denote the time step. The discretized  ODE~\eqref{ODE} is given as follows: 
\begin{equation*}
	\frac{x^{k+1}-2x^{k}+x^{k-1}}{h^2}+(\frac{\alpha}{kh}-\frac{\alpha}{kh}\cdot\frac{F^{k+1}-F^{k}}{x^{k+1}-x^{k}}+\gamma_k \frac{\nabla F^{k+1}-\nabla F^{k}}{x^{k+1}-x^{k}})\cdot \frac{x^{k+1}-x^{k}}{h}+\beta_k \nabla F^{k+1} = 0,
\end{equation*}
where $F^{k}=F(x^{k})$.

Multiplying both sides by $kh^2$, we have
\begin{align}\label{fully_discretized_system}
	&k(x^{k+1}-2x^{k}+x^{k-1})+\alpha(x^{k+1}-x^{k})-\alpha(F^{k+1}-F^{k})+\gamma_k kh(\nabla F^{k+1}-\nabla F^{k})
    +\beta_kkh^2 \nabla F^{k+1} = 0.
\end{align}
Although this scheme involves implicit terms in the update of $x^{k+1}$, it preserves the energy dissipation and the convergence rate of $\{F^{k}\}$ stated in Theorem~\ref{continuous:thm} at the discrete level.

\begin{algorithm}[H] 
	\caption{Swarm-Based Fully Discretized Algorithm}
	\label{sbfd:alg}
	\DontPrintSemicolon
	\SetKwInOut{Input}{Input}
	\SetKwInOut{Output}{Output}
	
	\Input{
		Number of agents $N$, objective function $F(x)$, time step $h$,\\
		parameters $\alpha$, $\gamma_k$, $\beta_k$,\\
		tolerances $\mathrm{tol}_{\mathrm{res}}, \mathrm{tol}_m, \mathrm{tol}_{\mathrm{merge}}$
	}
	
	\Output{Final agent set and best solution}
	
	\BlankLine
	
	Initialize agents $\{x_i^{0}\}_{i=1}^N$ randomly.\;
	Set masses $m_i^{0} = 1/N$.\;
	Generate $\{x_i^{1}\}_{i=1}^N$ if needed.\;
	
	\For{$n=2,\dots$ until the stopping criterion is satisfied}{
		
		\tcp{1. Communication (mass update)}
		\ForEach{agent $i$}{
			\uIf{$m_i^{n} < \mathrm{tol}_m$}{
				Remove agent $i$\;
			}
			\Else{
				Update $m_i^{n+1}$ using the mass transition rule~\eqref{masstrans}\;
			}
		}
		
		\tcp{2. Fully discretized inertial update}
		\ForEach{remaining agent $i$}{
			Compute $x_i^{n+1}$ by solving the fully discretized inertial system:
			\begin{align*}
				k(x^{k+1}-2x^{k}+x^{k-1})
				+\alpha(x^{k+1}-x^{k})
				-\alpha\big(F^{k+1}-F^{k}\big)
				+\gamma_k k h(\nabla F^{k+1}-\nabla F^{k})
				+\beta_k k h^2 \nabla F^{k+1}
				= 0
			\end{align*}
			where $x^{k} = x_i^{n}$ and $x^{k-1} = x_i^{n-1}$\;
		}
		
		\tcp{3. Merging}
		\For{each pair of agents $(i,j)$}{
			\If{$\|x_i^{n+1} - x_j^{n+1}\| < \mathrm{tol}_{\mathrm{merge}}$}{
				Merge agents $i$ and $j$\;
			}
		}
		
	}
	\Return best agent and corresponding objective function value\;
	
\end{algorithm}

\begin{theorem}\label{fully-discretized:energystability}
	Let $F: \mathbb{R}\rightarrow\mathbb{R}$ to be a convex function in $C^2$ with $\nabla F$ being $L$-Lipschitz continuous,  $\alpha\geq 1$. Let $x^* \in \arg\min_{\mathbb{R}} F$, and  $F^* := F(x^*)$. 	
We define
	\begin{align*}
		C_k &:= \gamma_{k+1}+k(\gamma_{k+1}-\gamma_k)-\beta_k kh, \quad \delta_k:= -C_k h (k+1), \\
		E^{k}& : = \delta_k (F^{k}-F^*) +\frac{1}{2} \|v_k\|^2, \\
		v_k &:=(\alpha - 1)(x^{k}-x^*) - \alpha (F^{k} - F^*)+k(x^{k}-x^{k-1}+\gamma_k h \nabla F^{k}).
	\end{align*}
	Suppose the following conditions hold:
	\begin{enumerate}
		\item  $C_k = \gamma_{k+1}+k(\gamma_{k+1}-\gamma_k)-\beta_k kh < 0$;
		\item  $\delta_{k+1}-\delta_k + (\alpha-1)C_k h - \alpha C_k h \sqrt{2L} \leq 0$;
		\item  $\gamma_k \geq 0$, $\forall k$.
	\end{enumerate}
	Then, with solutions of (\ref{fully_discretized_system}), we have:
	\[E^{k+1}-E^{k}\leq 0,\ \forall k\] 
	and
    \[0\leq F(x^{k})-\min F\leq  E^{k-l}/\delta_k,\]
where $0\leq l\leq k.$
\end{theorem}
\begin{proof}
	Following the proof of Theorem~\ref{continuous:thm}, we define the discrete energy and velocity as follows:
	\begin{align*}
		E^{k}:& = \delta_k (F^{k}-F^*) +\frac{1}{2} \|v_k\|^2. \\
		v_k :
		&=(\alpha - 1)(x^{k}-x^*) - \alpha (F^{k} - F^*)+k(x^{k}-x^{k-1}+\gamma_k h \nabla F^{k}).
	\end{align*}
	Our goal is to show that $\Delta E^{k} := E^{k+1} - E^{k}$ is non-positive.
	\begin{align}
		\Delta E^{k} 
		&= (\delta_{k+1}-\delta_k)(F^{k+1}-F^*)+\delta_k (F^{k+1}-F^{k})+\frac{1}{2}(\|v_{k+1}\|^2-\|v_k\|^2).
	\end{align}
	Note that
	\begin{equation*}
		\frac{1}{2}(\|v_{k+1}\|^2 - \|v_k\|^2) = \la v_{k+1}-v_k, v_{k+1}\ra - \frac{1}{2}(\|v_{k+1}-v_k\|^2).
	\end{equation*}
	For this term, we only need to compute $\Delta v_k := v_{k+1}-v_k$:
	\begin{align}
		\Delta v_k 
		&= \bigg(   \alpha (x^{k+1}-x^{k}) -\alpha (F^{k+1}-F^{k}) + k(x^{k+1}-2x^{k} +x^{k-1}) +kh \gamma_k (\nabla F^{k+1}-\nabla F^{k})    \bigg) \nonumber \\
		&+\gamma_{k+1}h \nabla F^{k+1}+kh (\gamma_{k+1}-\gamma_k)\nabla F^{k+1} \nonumber \\
		 &=h \bigg( \gamma_{k+1}+k(\gamma_{k+1}-\gamma_k) -\beta_k kh \bigg) \nabla F^{k+1}.
	\end{align}
	Here, we set 
	\begin{equation}
		C_k := \gamma_{k+1} + k(\gamma_{k+1}-\gamma_k)-\beta_k kh.
	\end{equation}
	Then $\Delta v_k = C_k h \nabla F^{k+1}$. We now observe
	\begin{align}
		\frac{1}{2} (\|v_{k+1}\|^2 - \|v_k\|^2) 
		&=  \bigg( C_k(k+1)\gamma_{k+1}h^2 -\frac{h^2}{2}C_k^2 \bigg ) \|\nabla F^{k+1}\|^2 + C_kh(k+1)\la \nabla F^{k+1}, x^{k+1}-x^{k}\ra \nonumber \\
		&- (\alpha-1)C_kh \la \nabla F^{k+1}, x^*-x^{k+1}\ra  - \alpha C_k h \la \nabla F^{k+1}, F^{k+1} -F^*\ra.
	\end{align}
	Since  $C_k = \gamma_{k+1}+k(\gamma_{k+1}-\gamma_k)-\beta_k kh < 0,$, then we have
	\begin{align}
		\frac{1}{2} (\|v_{k+1}\|^2 - \|v_k\|^2)  &\leq -C_k h (k+1)\la \nabla F^{k+1}, x^{k}-x^{k+1}\ra  - (\alpha -1 )C_k h \la \nabla F^{k+1}, x^* - x^{k+1}\ra \nonumber \\
		&- \alpha C_k h \la \nabla F^{k+1}, F^{k+1}-F^*\ra. 
	\end{align}
	By convexity of $F$, two of the three terms above can be controlled in the following inequalities:
	\begin{equation}
		\begin{cases}
			-C_kh(k+1)\la \nabla F^{k+1}, x^{k} - x^{k+1}\ra \leq -C_k h (k+1)(F^{k}-F^{k+1}),\\
			-(\alpha - 1)C_k h \la \nabla F^{k+1},x^*- x^{k+1}\ra \leq -(\alpha-1)C_k h (F^*- F^{k+1}).
		\end{cases}
	\end{equation}
	We then set $\delta_k:= -C_k h (k+1)$. This leads to
	\begin{align}
		\Delta E^{k} &\leq (\delta_{k+1}-\delta_k)(F^{k+1}-F^*) - (\alpha-1)C_kh (F^*- F^{k+1}) \nonumber
		-\alpha C_k h \la \nabla F^{k+1}, F^{k+1} - F^* \ra.
	\end{align}
	With $\nabla F$ $L$-Lipschitz, this gives the following:
	\begin{align}
		\la \nabla F^{k+1}, F^{k+1}-F^*\ra \leq \sqrt{2L}\|F^{k+1}-F^*\|^{\frac{3}{2}}.
	\end{align}
	Altogether, we have
	\begin{align}
		\Delta E^{k} \leq \bigg ((\delta_{k+1}-\delta_k)+(\alpha-1)C_k h - \alpha C_k h \sqrt{2L} \bigg)\max \{F^{k+1}-F^*,\|F^{k+1}-F^*\|^{\frac{3}{2}}\}.
	\end{align}
	With $\delta_{k+1}-\delta_k + (\alpha-1)C_k h - \alpha C_k h \sqrt{2L} \leq 0$ and $\gamma_k \geq 0$, $\forall k$,
	\[E^{k+1}-E^{k}\leq 0,\ \forall k.\]
	Similarly, with the energy dissipation, and the positivity of each term in the energy function, we have
	\[\delta_k (F^{k}-F^*)\leq E^{k}\leq E^{k-l}.\]

\end{proof}
	The theorem implies $0\leq F(x^{k})-\min F$ converges at least in the order of $\bigO(\frac{1}{\delta_k})$ as $k\to \infty$.	
\subsection{The IMEX-RB algorithm}

\noindent \indent In addition to algorithms that preserve energy dissipation and function-value convergence rates, other methods also maintain energy dissipation while generating a convergent sequence of solution values.
The IMEX-RB algorithm is a self-adaptive implicit-explicit time integrator for systems of ODEs. The method is first-order accurate in time.  
For the following initial value problem with $\by: I\subset \mathbb{R}\rightarrow\mathbb{R}^{N}$ and $\by\in C^2(I)$,
\begin{equation}\label{ode:cauchy}
    \begin{cases}
\by'(t)=f(t,\by(t)),\,\forall t\in I,\\[8pt]
\by(0)=\by_0,
\end{cases}
\end{equation}
where $I=(0,T]$ is the time interval and $f :I\times\mathbb{R}^{N}\rightarrow\mathbb{R}^{N}$.
This method was first introduced in \cite{bassanini2025imex}, which generates a sequence $\{\bu^{n}\}_{n=0}^{N_t}$ to approximate $\by(t_n)$.

\begin{alg}(IMEX-RB Algorithm)\label{IMEXRB-single}

\noindent\textbf{Step 1.} Initialize $\by_0$, the time step $\Delta t$, the number of time steps $N_t$, the absolute stability parameter $\epsilon$, the maximum number of inner iterations $M$, and a lower bound $\delta$ to avoid quasi-collinearity at each step of the QR factorization. Set $\bu_0 = \by_0$.\\[0.5em]

\noindent\textbf{Step 2.} While $0 \le n \le N_t - 1$, compute the QR factorization
\[ [\bu^{n}, \dots, \bu^{n-\min\{N,n\}+1}] = \bV^n_{0} \br_n,\]
with $\delta$ imposed as a lower bound to avoid quasi-collinearity.\\[0.5em]

\noindent For $k = 0, 1, \dots, M-1$, do:
\begin{enumerate}
\item Apply an implicit Euler step by solving
\[\bar{\bu}^{n+1}_N=\Delta t\, {\bV^n_{k}}^{T}\bf\bigl(t_{n+1}, \bV^n_{k} \bar{\bu}^{n+1}_N + \bu^{n}\bigr).\]

\item Update
\[\bu^{n+1}_{k}=\bu^{n}+\Delta t\, \bf\bigl(t_{n+1},\bV^n_{k} \bar{\bu}^{n+1}_N + \bu^{n}\bigr).\]

\item Compute
\[\br^{n+1}_{k}=(\bI - \bV^n_{k} {\bV^n_{k}}^{T})\bu^{n+1}_{k}.\]
If \[\|\br^{n+1}_{k}\| / \|\bu^{n+1}_{k}\| \le \epsilon,\]
stop the inner iteration. Otherwise, update
\[\bV^n_{k+1}=\bigl[\bV^n_{k}, \br^{n+1}_{k} / \|\br^{n+1}_{k}\|\bigr].\]
\end{enumerate}

\noindent\textbf{Step 3.} Set $\bu^{n+1} = \bu^{n+1}_{k}$, update $n \leftarrow n+1$, and return to Step $2$. Repeat until $n = N_t$.

\end{alg}

\begin{algorithm}[H]
\caption{Swarm-Based IMEX-RB Algorithm}
\label{sbimex:alg}
\DontPrintSemicolon
\SetKwInOut{Input}{Input}
\SetKwInOut{Output}{Output}

\Input{
Number of agents $N$, objective function $F(x)$, IMEX-RB time step $\Delta t$,\\
parameters $\alpha$, $\gamma_k$, $\beta_k$, IMEX-RB parameters $(\epsilon, \delta, M)$,\\
tolerances $\mathrm{tol}_{\mathrm{res}}, \mathrm{tol}_m, \mathrm{tol}_{\mathrm{merge}}$
}

\Output{Final agent set and best solution}

\BlankLine

Initialize agents $\{x_i^{0}\}_{i=1}^N$ randomly.
Set masses $m_i^{0} = 1/N$.
Generate $\{x_i^{1}\}_{i=1}^N$ if needed.

\For{$n=2,\dots$ until the stopping criterion is satisfied}{

  \tcp{1. Communication (mass update)}
  \ForEach{agent $i$}{
    \uIf{$m_i^{n} < \mathrm{tol}_m$}{
      Remove agent $i$\;
    }
    \Else{
      Update $m_i^{n+1}$ using the mass transition rule~\eqref{masstrans}\;
    }
  }
  \tcp{2. IMEX-RB inertial update}
   \ForEach{remaining agent $i$}{
   Compute $x_i^{n+1}$ by applying the IMEX-RB Algorithm~\eqref{IMEXRB-single} to the ODE system ~\eqref{ode:system}
  }

  \tcp{3. Merging}
  \For{each pair of agents $(i,j)$}{
    \If{$\|x_i^{n+1} - x_j^{n+1}\| < \mathrm{tol}_{\mathrm{merge}}$}{
      Merge agents $i$ and $j$\;
    }
  }

}

\Return best agent and corresponding objective function value\;

\end{algorithm}

In \cite{bassanini2025imex}, the convergence of this method was proved, given in the following theorem.
\begin{theorem}\label{imex:convergence:thm}
    Consider the Cauchy problem of \ref{ode:cauchy}. Let $\by\in C^2(I)$, and suppose that $f$ is Lipschitz continuous in its second argument, with Lipschitz constant $L>0$. Then the IMEX-RB method is convergent of order 1, i.e. $\exists C>0$ such that
    \[\|\be^{n+1}\|=\|\by(t_{n+1})-\bu^{n+1}\|<C\Delta t\].
\end{theorem}

To show that the sequence generated by the IMEX-RB method converges to the solution of our new inertial dynamics~\eqref{gen-ode},
we recast it into the following system:
\begin{equation}\label{ode:system}
\begin{cases}
\dot{\bx}(t) = \bV(t), \\[8pt]
\dot{\bV}(t) = -\dfrac{\alpha}{t}\,\bV(t)
+ \dfrac{\alpha}{t}\,\big\langle \nabla F(\bx(t)),\, \bV(t) \big\rangle \,\mathbf{1}
- \gamma(t)\,\nabla^{2} F(\bx(t))\,\bV(t)
- \beta(t)\,\nabla F(\bx(t)).
\end{cases}
\end{equation}
Here, we denote
\begin{equation}
\by(t) =
\begin{bmatrix}
\bx(t) \\[4pt] \bV(t)
\end{bmatrix},
\end{equation}
and
\begin{equation}
    g(t,\by(t))=
    \begin{bmatrix}
        \bV(t)\\[4pt]
         -\dfrac{\alpha}{t}\,\bV(t)
+ \dfrac{\alpha}{t}\,\big\langle \nabla F(\bx(t)),\, \bV(t) \big\rangle \,\mathbf{1}
- \gamma(t)\,\nabla^{2} F(\bx(t))\,\bV(t)
- \beta(t)\,\nabla F(\bx(t))
    \end{bmatrix}
\end{equation}

We need to verify $\by\in C^2(I)$ and $g$ is Lipschitz continuous respect to $\by$. Then by \ref{imex:convergence:thm}, we have that the IMEX-RB scheme for our ODE also converges.
In order to show $g$ is Lipschitz continuous, we need an extra assumption that $\nabla^2f$ is Lipschitz continuous.

\begin{theorem}\label{imexrb:thm}
	Under the same assumptions as Theorem~\ref{continuous:thm} together with that $\nabla^2 F$ is Lipschitz continuous, the IMEX-RB scheme generates a convergent sequence $\{\bu^{n}\}$:
	\[\bu^{n} \to 
	\begin{bmatrix}
		\bx^* \\[4pt] \b0
	\end{bmatrix}\].
\end{theorem}

From Theorem~\ref{imexrb:thm} , we learn that sequence $\{u^n\}$ generated by the IMEX-RB method converges. We next show that the IMEX-RB method preserves the energy dissipation property and the convergence rate of $F^{k}$ established in Theorem~\ref{fully-discretized:energystability}. The following lemmas and the main theorem establish this result. The detailed proofs are provided in Appendix~\ref{proof:imexrb}.

\begin{lemma}\label{imex:E:lip}
	We assume $F$ satisfies the conditions in Theorem~\ref{imexrb:thm}, the discrete sequence generated remain in a bounded domain $\mathcal B$,  parameters $\alpha$, $\beta$, and $\gamma$ in the discrete energy $E$ defined in~\ref{fully-discretized:energystability} are uniformly bounded. Then, the discrete energy is Lipschitz continuous on $\mathcal B\times\mathcal B$ with Lipschitz constant $L_E$.
\end{lemma}

\begin{lemma}\label{imexrb:lemma3}
	Under the same assumptions as Lemma~\ref{imex:E:lip}, let $\bu^{n+1}$ and $\by^{n+1}$ be the one-step updates of the IMEX-RB method and the backward Euler method starting from the same point $\bu^n$, respectively.
	Assume that the stopping condition in the IMEX-RB method satisfies
	\[\|(I-V^n{V^n}^\top)\bu^{n+1}\|/\|\bu^{n+1}\|\leq \varepsilon.\]
	If $\Delta t\leq \frac{1}{2L_g}$, where $L_g$ is the Lipschitz constant of $g$, then
	\[\|\bu^{n+1}-\by^{n+1}\|\leq C\,\Delta t\,\varepsilon,\]
	where $C$ is independent of $\Delta t$.
\end{lemma}


\begin{theorem}\label{imexrb:energystability}
	Under the assumptions of Theorem~\ref{fully-discretized:energystability},  Lemmas~\ref{imex:E:lip}, and \ref{imexrb:lemma3}, let $\{E^k_{\mathrm{IMEX}}\}$ denote the discrete energy in the IMEX--RB method, and $\{E^k_{\mathrm{BE}}\}$ be the discrete energy in the backward Euler method in Theorem~\ref{fully-discretized:energystability}.
	We define
	\[\Delta_k := E^k_{\mathrm{BE}} - E^{k+1}_{\mathrm{BE}} \ge 0,\]
	and assume that there exists a constant $c_0 > 0$ independent of $\Delta t$, such that for all $k$ with $t_k \le T$,
	\[\Delta_k \ge c_0\, \Delta t.\]	
	Assume that the tolerance $\varepsilon$ in the IMEX-RB method's stopping condition satisfies
	\[\varepsilon \le C_0 \Delta t,\]
	where $C_0 = \frac{c_0}{2 L_E M}$, $L_E$ denotes the Lipschitz constant of the discrete energy, and $M > 0$ is a constant depending only on $T$, $L_g$, $C$, and $L_E$.
	Then, for all $k$ with $t_k \le T$,
	\[E^{k+1}_{\mathrm{IMEX}} - E^k_{\mathrm{IMEX}} \le 0\]
	and
	\[F(\bu^k_{\mathrm{IMEX}}) - \min_{x} F\leq E^k_{\mathrm{IMEX}}/\delta_k. \]
\end{theorem}

\subsection{Additional algorithms}

\noindent \indent Finally, we present two additional explicit optimization algorithms, which solve the subproblems at each iteration using optimization approaches, rather than relying on Newton-type methods or contraction mappings to handle the implicit steps used in the previous schemes.

We first present a semi-discretization, keeping $\nabla F$ unchanged. Then, we discretize the semi-discrete equation by a forward-backward algorithm to address the internal iteration to decouple the implicit step. The main difference between these two methods lies in the representation of $\nabla F^{k}$ in the damping term.

\subsubsection{The semi-discretized algorithm (semi algorithm)}

\noindent \indent In the semi-discrete algorithm, we  discretize the ODE as follows:

\begin{align}\label{semi:dis:ode}
      &k(x^{k+1}-2x^{k}+x^{k-1})+\alpha (x^{k+1}-x^{k})-\alpha \la\nabla F^{k}, x^{k}-x^{k-1}\ra \mathbf{1}+ \gamma_k k h (\nabla F^{k+1}-\nabla F^{k})+\beta_k kh^2 \nabla F^{k+1} = 0.
 \end{align}
Here, the velocity term $\dot{x}$ in $\frac{\alpha}{t}\langle \nabla F(x(t)), \dot{x}(t) \rangle$ is discretized using the forward Euler method.
Applying the proximal gradient method, which is
\[x^{k+1} = y^{k} - \mu_{k} \nabla F^{k+1},\]
we obtain
\begin{align*}
    &(k+\alpha)(y^{k} - \mu_{k} \nabla F^{k+1}) - (2k+\alpha)x^{k} + kx^{k-1}
    - \alpha \langle \nabla F^{k}, x^{k} - x^{k-1} \rangle \mathbf{1} \\
    &\qquad + \gamma_k k h (\nabla F^{k+1} - \nabla F^{k})
    + \beta_k k h^2 \nabla F^{k+1} = 0.
\end{align*}
We choose $\mu_{k}$ such that the coefficient of the implicit term $\nabla F^{k+1}$ vanishes, $ \mu_{k} = \dfrac{k h}{k+\alpha}(\gamma_k + \beta_k h).$
Rearranging the terms and solving $y^{k}$ lead to the following explicit optimization algorithm:
\begin{equation}\label{semi:alg}
    \begin{cases}
        y^{k} = x^{k} + \dfrac{k}{k+\alpha}(x^{k} - x^{k-1})
        + \dfrac{\alpha}{k+\alpha}\langle \nabla F^{k}, x^{k} - x^{k-1} \rangle \mathbf{1}
        + \dfrac{k h}{k+\alpha}\gamma_k \nabla F^{k}, \\[0.5em]
        \mu_{k} = \dfrac{k h}{k+\alpha}(\gamma_k + \beta_k h), \\[0.5em]
        x^{k+1} = \mathrm{prox}_{\mu_{k} F}(y^{k}).
    \end{cases}
\end{equation}

\subsubsection{The Forward-backward algorithm (FB algorithm)}
\noindent \indent We now present the forward-backward semi-discrete algorithm, where $\nabla F^{k}$ in the damping term is discretized using the forward Euler method, 
\begin{align}\label{fb:dis:ode}
	k(x^{k+1}-2x^{k}+x^{k-1})+\alpha (x^{k+1}-x^{k})-\alpha (F^{k} - F^{k-1})\mathbf{1}+ \gamma_k kh(\nabla F^{k+1}-\nabla F^{k})+\beta_k kh^2 \nabla F^{k+1} = 0.
\end{align}
Similarly, applying the proximal gradient method, substituting the expression for $x^{k+1}$ into the discretized system, assigning the coefficient of the implicit term into zero, and rearranging the terms, we obtain
\begin{align*}
	[(k+\alpha)y^{k} - (2k+\alpha)x^{k}+kx^{k-1}-\alpha(F^{k}-F^{k-1}) \mathbf{1}-\gamma_k kh \nabla F^{k}]+[\gamma_k kh +\beta_k kh^2-\mu_{k} (k+\alpha)]\nabla F^{k+1} = 0.
\end{align*}
This leads to the following explicit optimization method:
\begin{equation}\label{fb:alg}
	\begin{cases}
		y^{k} = x^{k} + \frac{k}{k+\alpha}(x^{k}-x^{k-1})+\frac{\alpha}{k+\alpha}(F^{k} - F^{k-1})\mathbf{1} + \frac{kh}{k+\alpha}\gamma_k \nabla F^{k}, \\
		\mu_{k} = \frac{kh}{k+\alpha}(\gamma_k + \beta_k h), \\
		x^{k+1} = \text{prox}_{\mu_{k} F}(y^{k}).
	\end{cases}
\end{equation}

Both algorithms leverage the proximal gradient method, providing alternative ways to solve the new underdamped inertial dynamics introduced in \eqref{gen-ode} explicitly. In each step, the proximal operator $x^{k+1} = \mathrm{prox}_{\mu_{k} F}(y^{k})$
is applied when updating $x^{k+1}$. In general, 
the update $x^{k+1}$ can be computed using other optimization methods, such as quasi-Newton methods, trust-region methods, etc. 

This approach can alternate between a classical time-space discretization, where $x^{k+1}$ is updated implicitly through \eqref{semi:dis:ode} or \eqref{fb:dis:ode}, and an optimization step, where $x^{k+1}$ is obtained by solving a subproblem at each iteration through the proximal operator. This gives more flexibility in practice.
\begin{remark}\label{semi:energy:remark}
	We note that, for these two methods derived from the Euler method, we have yet able to establish energy stability for them.
\end{remark}

Applying the above algorithms, to the multi-agent system,  we end up with the Swarm-based inertial algorithms:

\begin{algorithm}[H]
\caption{Swarm-Based Semi/FB Algorithm}
\label{sbsemifb:alg}
\DontPrintSemicolon
\SetKwInOut{Input}{Input}
\SetKwInOut{Output}{Output}

\Input{
Number of agents $N$, objective function $F(x)$, time step $h$,\\
parameters $\alpha$, $\gamma_k$, $\beta_k$\\
tolerances $\mathrm{tol}_{\mathrm{res}}, \mathrm{tol}_m, \mathrm{tol}_{\mathrm{merge}}$
}

\Output{Final agent set and best solution}

\BlankLine

Initialize agents $\{x_i^{0}\}_{i=1}^N$ randomly.
Set masses $m_i^{0} = 1/N$.
Generate $\{x_i^{1}\}_{i=1}^N$ if needed.

\For{$n=2,\dots$ until the stopping criterion is satisfied}{

  \tcp{1. Communication (mass update)}
  \ForEach{agent $i$}{
    \uIf{$m_i^{n} < \mathrm{tol}_m$}{
      Remove agent $i$\;
    }
    \Else{
      Update $m_i^{n+1}$ using the mass transition rule~\eqref{masstrans}\;
    }
  }

  \tcp{2. Semi/FB algorithm inertial update}
  \ForEach{remaining agent $i$}{
    Compute $x_i^{n+1}$  by solving~\eqref{semi:alg} for Semi algorithm,~\eqref{fb:alg} for FB algorithm, where $x^{k} = x_i^{n}$ and $x^{k-1} = x_i^{n-1}$\;
  }

  \tcp{3. Merging}
  \For{each pair of agents $(i,j)$}{
    \If{$\|x_i^{n+1} - x_j^{n+1}\| < \mathrm{tol}_{\mathrm{merge}}$}{
      Merge agents $i$ and $j$\;
    }
  }

}

\Return best agent and corresponding objective function value\;

\end{algorithm}

	
	
		


\section{Numerical results}\label{numerical:exp}


\noindent \indent  In this section, we first present numerical experiments to examine convergence properties of the new inertial algorithms in terms of function values, including the Forward-Backward (FB) method, the semi-discretized (Semi) method, the fully backward-discretized (FD) method, and the IMEX-RB method. We then apply these inertial algorithms, together with the classical Nesterov method, to the swarm-based optimization problem and benchmark their performance with the swarm-based gradient descent method proposed in \cite{lu2024swarm}.

\subsection{Convergence-rate test}

\noindent \indent We use four convex test functions of various dimensions to examine convergence rates of the IMEX-RB method and the FD method, respectively. Although we do not have theoretical results for the energy stability of the FB method and the Semi-discrete method, we examine their energy evolution and convergence behavior through numerical experiments. 

 Theorem~\ref{fully-discretized:energystability} and Theorem~\ref{imexrb:energystability} show that the FB method and the IMEX-RB method have a convergence rate of function values  as follows
\[0\leq F(x^{k}) - \min F \leq E^{k}/\delta_k.\]
This indicates that the worst convergence rate of function values is $O(E^{k}/\delta_k)$. We would like to see if actual convergence rates of the two algorithms and others do better than the upper bound.
We test this using the following four convex functions:
\begin{equation}
	\begin{aligned}
		\text{Sphere Function:}\qquad &
		F_{\mathrm{S}}(x)
		= \sum_{i=1}^d x_i^2, 
		\quad x_i \in [-5.12, 5.12],\; 1 \le i \le d, 
		\\[0.3em]
		\text{Modified Sphere Function:} \qquad &
		F_{\mathrm{MS}}(x)
		= \frac{1}{899}(\sum_{i=1}^d x_i^22^i-1745), 
		\quad x_i \in [-5.12, 5.12],\; 1 \le i \le d, 
		\\[0.3em]
		\text{Sum Squares Function:} \qquad &
		F_{\mathrm{SS}}(x)
		= \sum_{i=1}^d i\, x_i^2, 
		\quad x_i \in [-10,10],\; 1 \le i \le d,
		\\[0.3em]
		\text{Rotated Hyper-Ellipsoid Function:} \qquad &
		F_{\mathrm{RHE}}(x)
		= \sum_{i=1}^d \sum_{j=1}^i x_j^2,
		\quad x_i \in [-65.536, 65.536],\; 1 \le i \le d.
	\end{aligned}
\end{equation}
Note that all the functions have their global minima at $x^* = (0, \dots, 0)$. 

We use the forward-backward (FB), semi-discretized (Semi), fully backward-discretized (FD), and IMEX-RB method to examine the convergence rate of function values next.
For comparison, we also include the inertial proximal algorithm with Hessian damping (IPAHD)\cite{attouch2022first}, the classical Nesterov method (NM), and the gradient descent method (GD) in the investigation. These three methods provide reference convergence rates for each test function and serve as baselines.

The Inertial Proximal Algorithm with Hessian Damping (IPAHD) in \cite{attouch2022first} is given as follows:
\begin{equation}
	\begin{cases}
		\mu_{k} = \frac{k}{k + \alpha } (\beta_k h + h^2b_k  ), \\
		y^{k}=   x^{k} + \left( 1- \frac{\alpha}{k + \alpha}\right) ( x^{k}  - x^{k-1}) + \beta_k h   \left( 1- \frac{\alpha}{k + \alpha}\right) \nabla F (x^{k}), \\
		x^{k+1} = \text{prox}_{\mu_{k} F}(y^{k}).
	\end{cases}
\end{equation}

The following conditions are applied to all tested inertial algorithms in all test cases.
\begin{itemize}
\item {\textbf{Stopping Criteria:}}
All algorithms are terminated when both conditions listed below are satisfied:
\[\|F(x^{n+1}) - F(x^{n})\| \le 10^{-6},\quad\|x^{n+1} - x^{n}\| \le 10^{-6}.\]
\item 
{\textbf{Success Criterion:}}
To determine whether the optimal point is successfully reached, we use $1$ to denote success and $0$ to denote failure. The condition for success is defined as:
\[\| F(x^{n+1}) - F(x^*) \| \le 10^{-4}.\]
\item 
{\textbf{Parameter Settings:}}
Since all the new inertial algorithms require a time step size $h$ in the iteration, we set the time step size $h$ accordingly to the values listed in the tables.
\end{itemize}
For new ODE~\eqref{gen-ode}, we set the control parameters as follows:
\[\alpha = 2, \quad \gamma(t) = 200, \quad \beta(t) = \frac{500}{t},\]
to ensure that the conditions in Theorem~\ref{fully-discretized:energystability} are met.
In the resulting algorithms, the parameter values  are
\[\alpha = 2, \quad \gamma_k = 200, \quad \beta_k = \frac{500}{k h}.\]

\noindent {\textbf{Convergence rates of function values:}}

Next, we test the convergence rate of $F^{k}-F^*$. First recall the following from Theorem~\ref{fully-discretized:energystability} and Theorem~\ref{imexrb:energystability}:
\[ \delta_k := -C_k h (k+1)>0,\quad 
C_k = \gamma_{k+1} + k(\gamma_{k+1} - \gamma_k) - \beta_k k h<0.\]
Here, $\delta_k$ is the discretized version of $\delta(t)$ in \eqref{eq:lyapunov}. From the definition of $\delta_k$, we have
$\delta_k \to \infty$ as $k\to \infty$.
 Since $F^{k} - F^* = O(E^{k}/\delta_k)$,
we assume there exists $C>0$ and $K>0$ such that $|F^{k}-F^*|\approx C/\delta_k^p$ for all $k\ge K$, where $p>0$.   

The goal is to estimate the convergence rate, $p$. We name it the convergence rate of function values.  To this end, we have
\[\frac{F^{k} - F^*}{F^{k+1} - F^*}\approx\left(\frac{\delta_{k+1}}{\delta_k}\right)^p \Longrightarrow p \approx \frac{\log\!\left(\frac{F^{k} - F^*}{F^{k+1} - F^*}\right)
}{\log\!\left(\frac{\delta_{k+1}}{\delta_k}\right)}.\]
Hence, the local convergence rate can be approximated by $p_k$ for a sufficiently large $k$: 
\[p_k=\frac{\log\!\left(\frac{F^{k} - F^*}{F^{k+1} - F^*}\right)}{\log\!\left(\frac{\delta_{k+1}}{\delta_k}\right)}.\]
Since the local slope $p_k$ may be affected by individual iterations, we report the following average convergence rate in the tables define by 
\[\bar{p}=\frac{\sum_{k} p_k}{\text{iterations number}}.\]
This averaged $\bar{p}$ reduces the influence of local oscillations and provides a more stable estimate of the asymptotic behavior of the convergence rate.
In a set of numerical experiments, we decrease time step $h$ monotonically from 1. This leads to a slight, monotonic decrease in the average convergence rate $\bar{p}$ in most algorithms as shown in Tables 1, suggesting that one should not choose a small $h$ while using the algorithms.

Next, we present numerical results for optimizations of the Rotated Hyper-Ellipsoid function using $h=1/16$. All new inertial algorithms successfully reach the global optimum for this test function. The following plots include three parts for each iteration of each optimization method: the function value or the potential energy $F - F^*$ in Figure \ref{fig:ffstar_comparison}, the convergence rate in Figure \ref{fig:convergence_comparison}, and the evolution of the discrete energy $E^k$ in Figure \ref{fig:energy_all}. These results provide a comprehensive view of the methods’ accuracy, convergence behavior, and stability. The convergence rates of function values all approach a value slightly larger than 4 when iteration number k is large. The NM and GD give smaller $p$ values at large $k$ however. 
\begin{figure}[H]
  \centering
  \captionsetup[subfigure]{justification=centering}

  \begin{subfigure}{0.24\textwidth}\centering
    \includegraphics[width=\textwidth]{"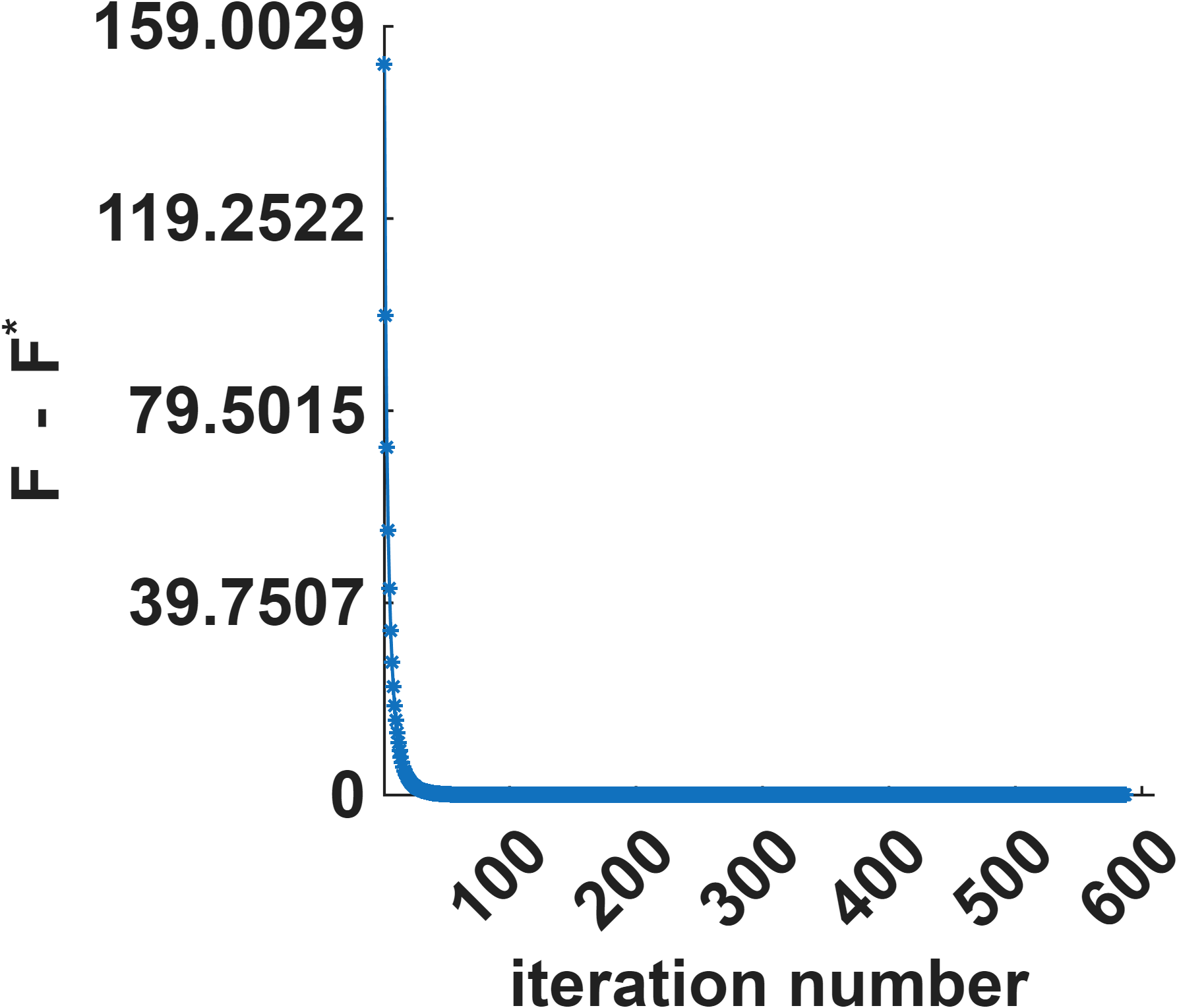"}
    \caption{FD}
  \end{subfigure}\hfill
  \begin{subfigure}{0.24\textwidth}\centering
    \includegraphics[width=\textwidth]{"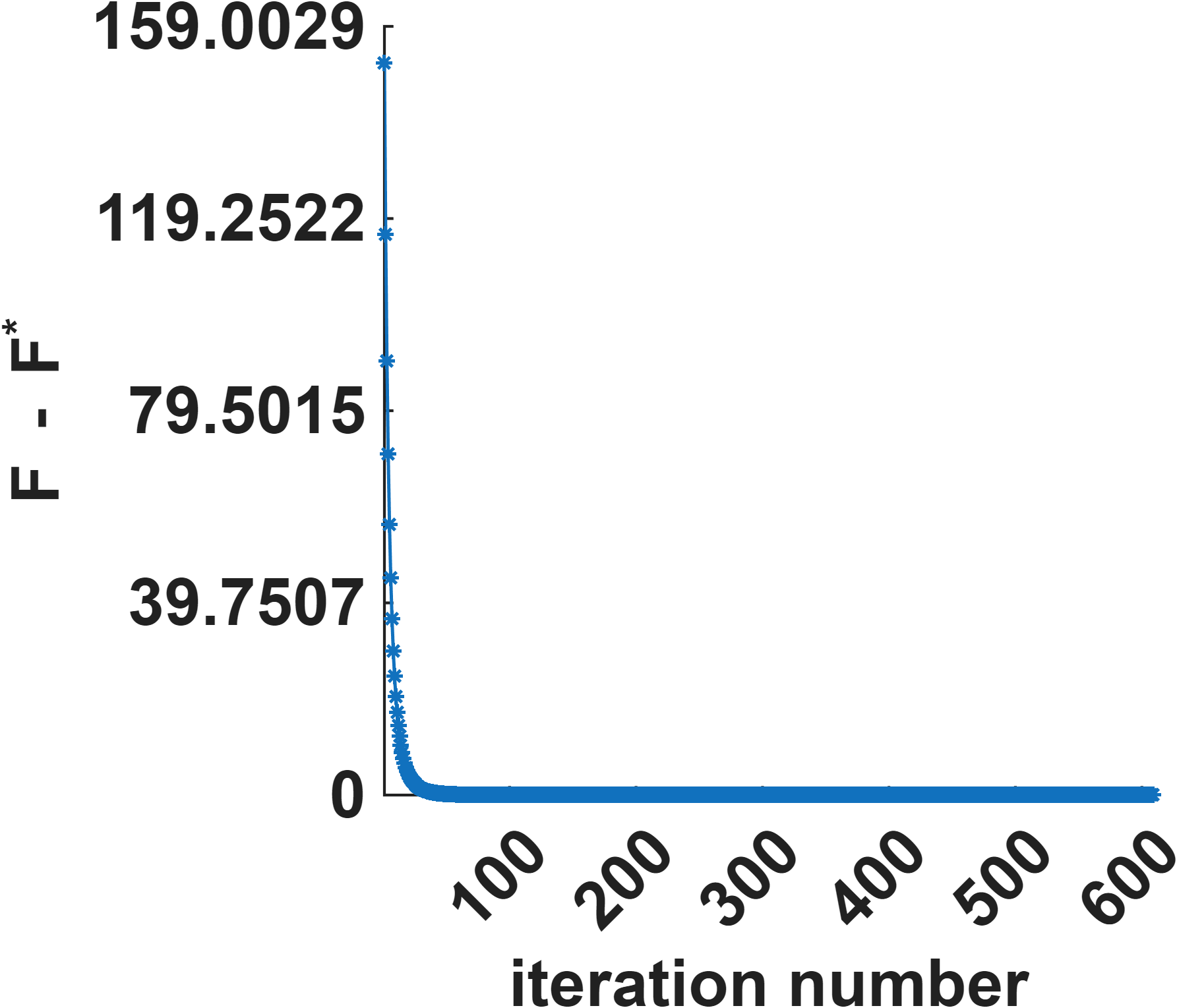"}
    \caption{IMEX-RB}
  \end{subfigure}\hfill
   \begin{subfigure}{0.24\textwidth}\centering
    \includegraphics[width=\textwidth]{"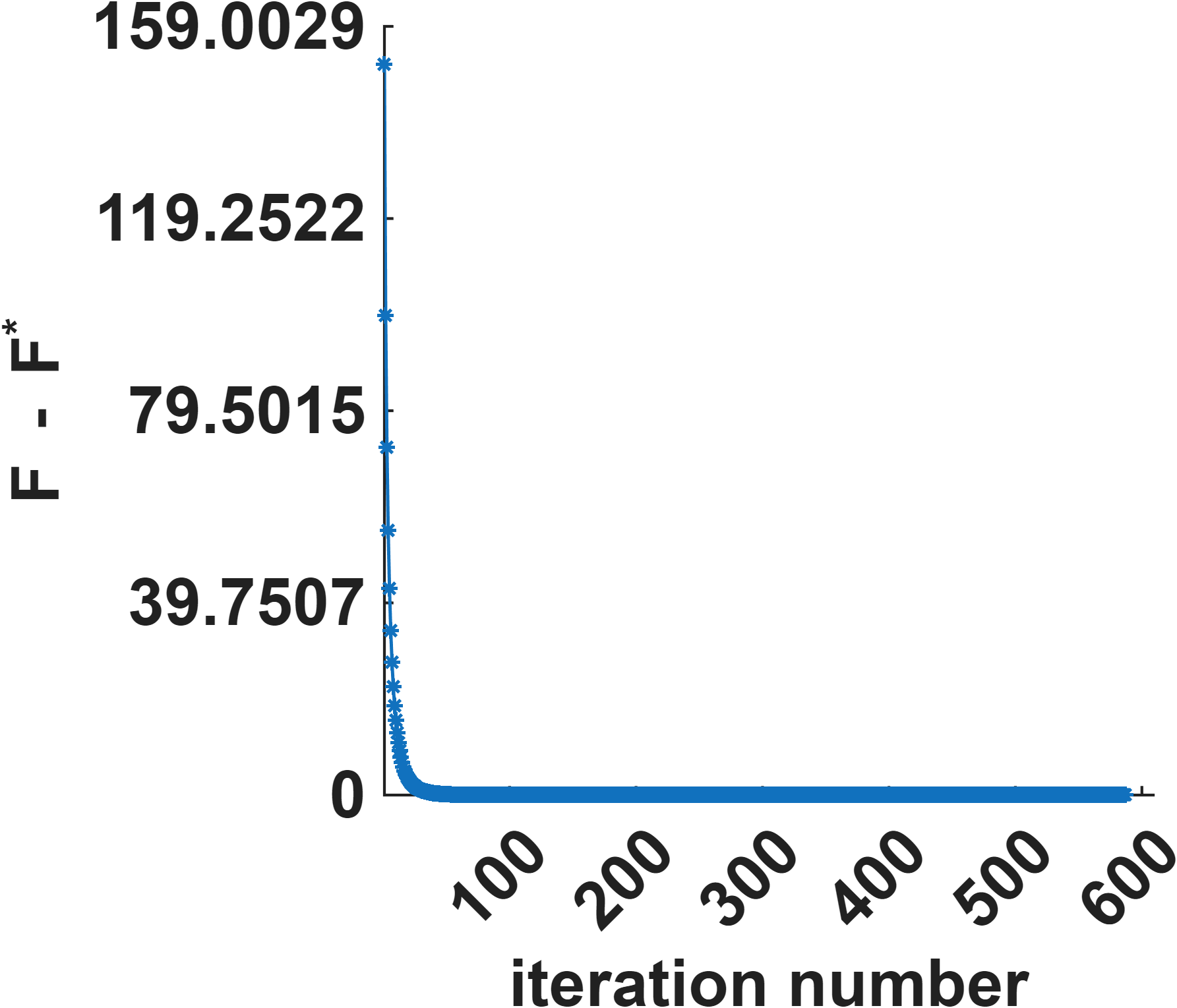"}
    \caption{FB}
  \end{subfigure}\hfill
\begin{subfigure}{0.24\textwidth}\centering
    \includegraphics[width=\textwidth]{"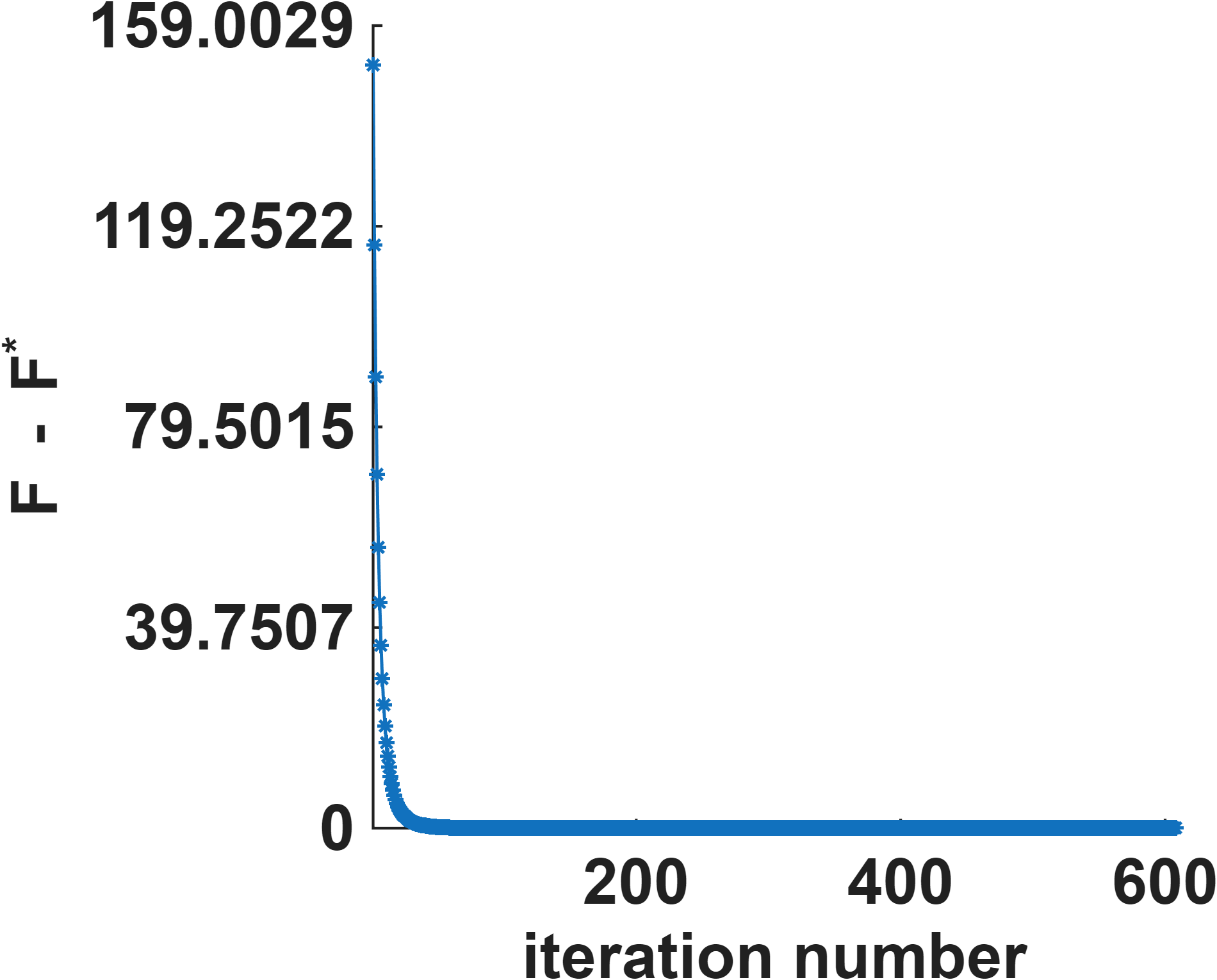"}
    \caption{Semi}
  \end{subfigure}

  \vspace{0.45em}

  \begin{subfigure}{0.24\textwidth}\centering
    \includegraphics[width=\textwidth]{"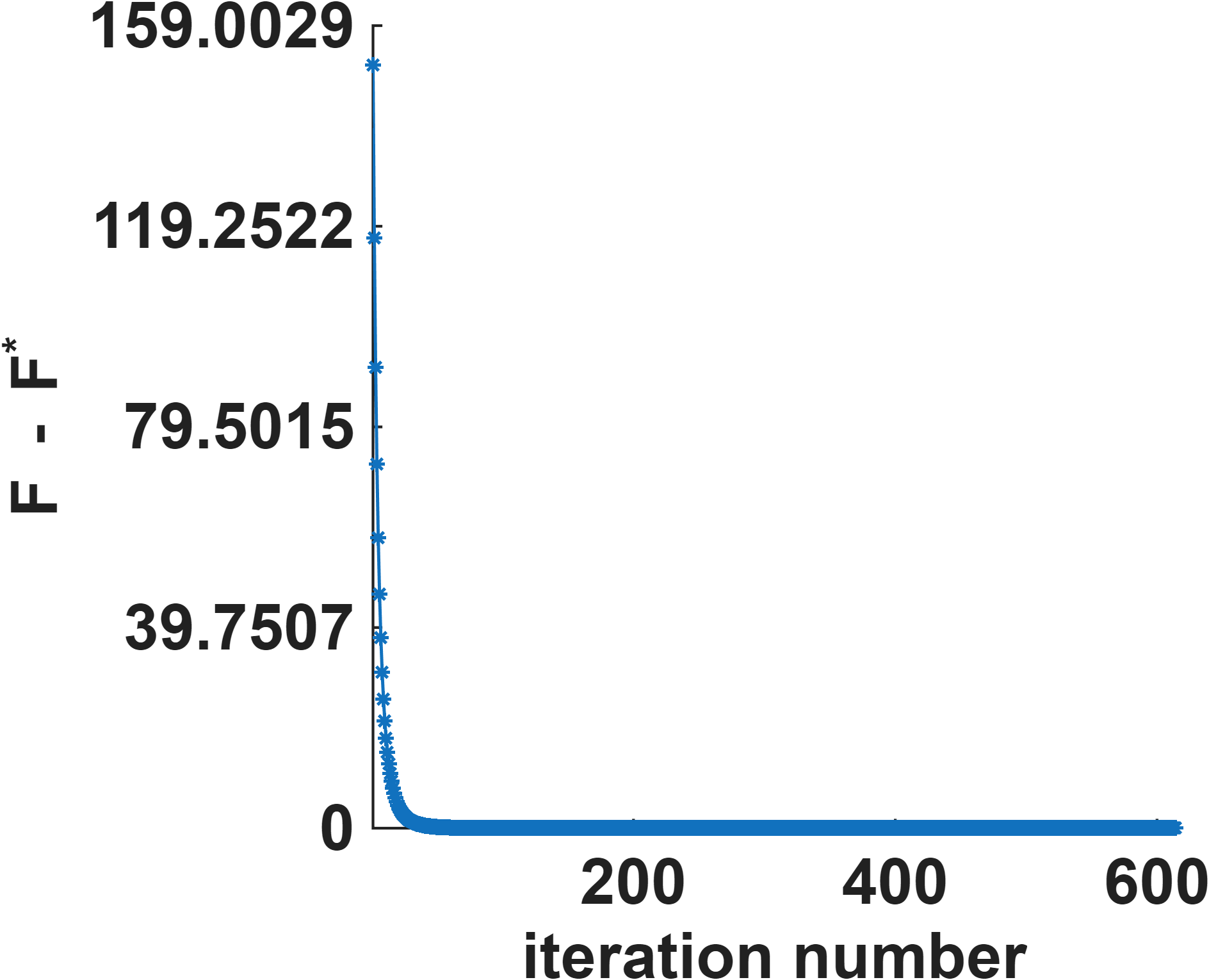"}
    \caption{IPAHD}
  \end{subfigure}\hfill
  \begin{subfigure}{0.24\textwidth}\centering
    \includegraphics[width=\textwidth]{"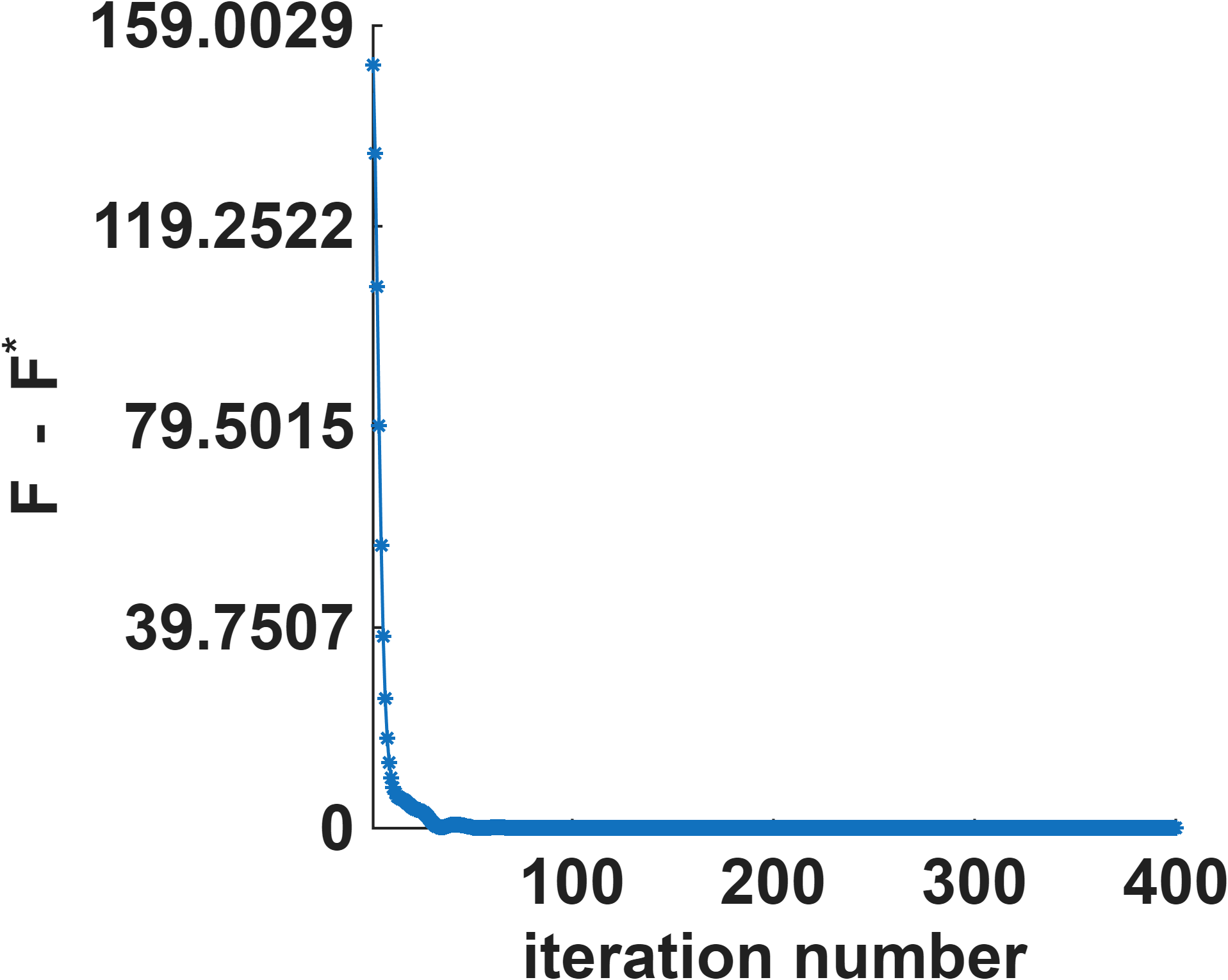"}
    \caption{Nesterov}
  \end{subfigure}\hfill
  \begin{subfigure}{0.24\textwidth}\centering
    \includegraphics[width=\textwidth]{"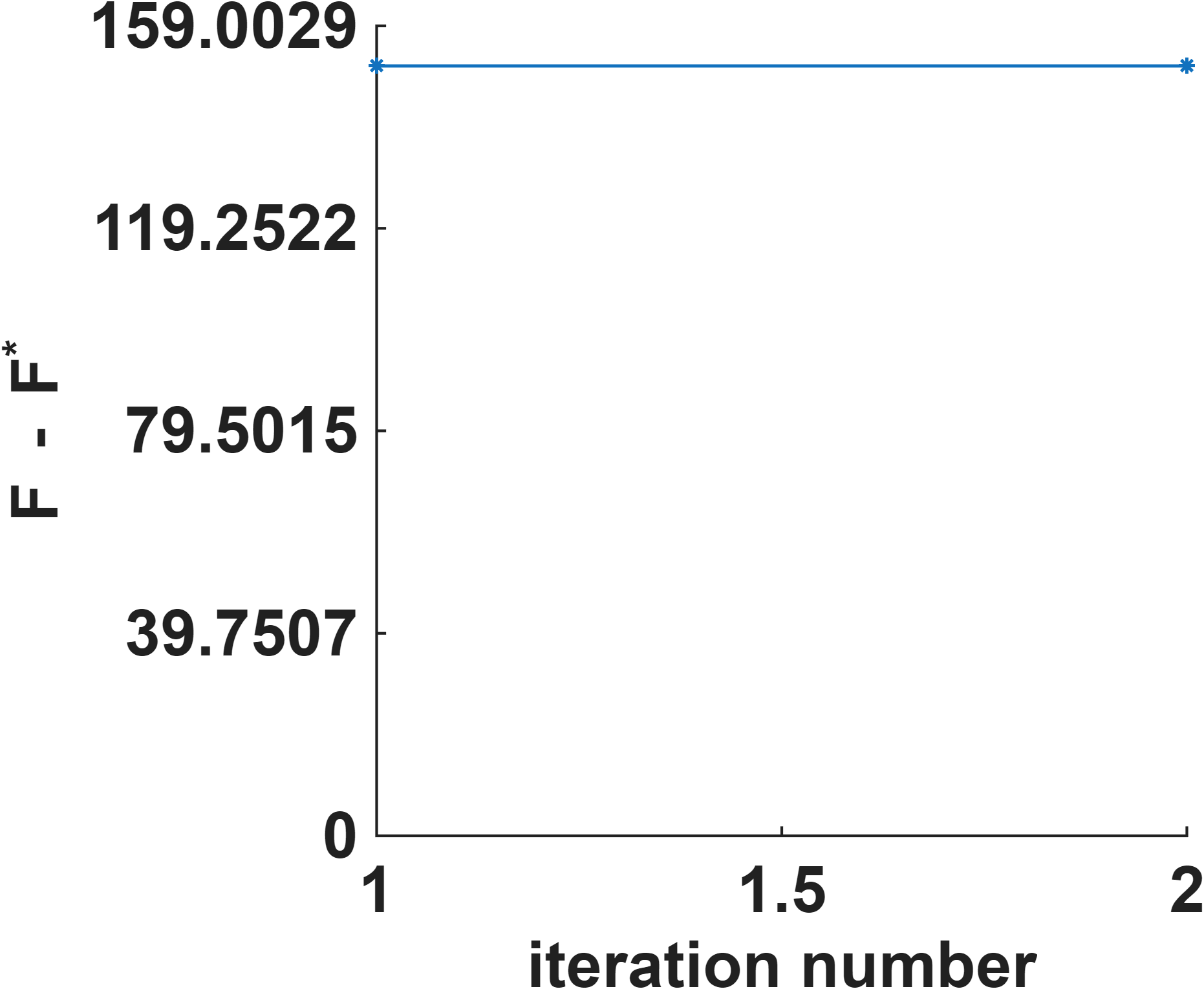"}
    \caption{GD}
  \end{subfigure}

  \caption{Comparison of \(F-F^{*}\) across methods for the 10D Rotated Hyper-Ellipsoid Function (\(h=1/16\)). Each subplot is one method (left-to-right, top-to-bottom). The potential energy decays monotonically during the  iteration for all methods except for the Nesterov method, which exhibits some slight oscillations. The GD method converges in two steps to the minimum value of $F(x^*)$.}
  \label{fig:ffstar_comparison}
\end{figure}

\begin{figure}[H]
  \centering
  \captionsetup[subfigure]{justification=centering}

  \begin{subfigure}{0.24\textwidth}\centering
    \includegraphics[width=\textwidth]{"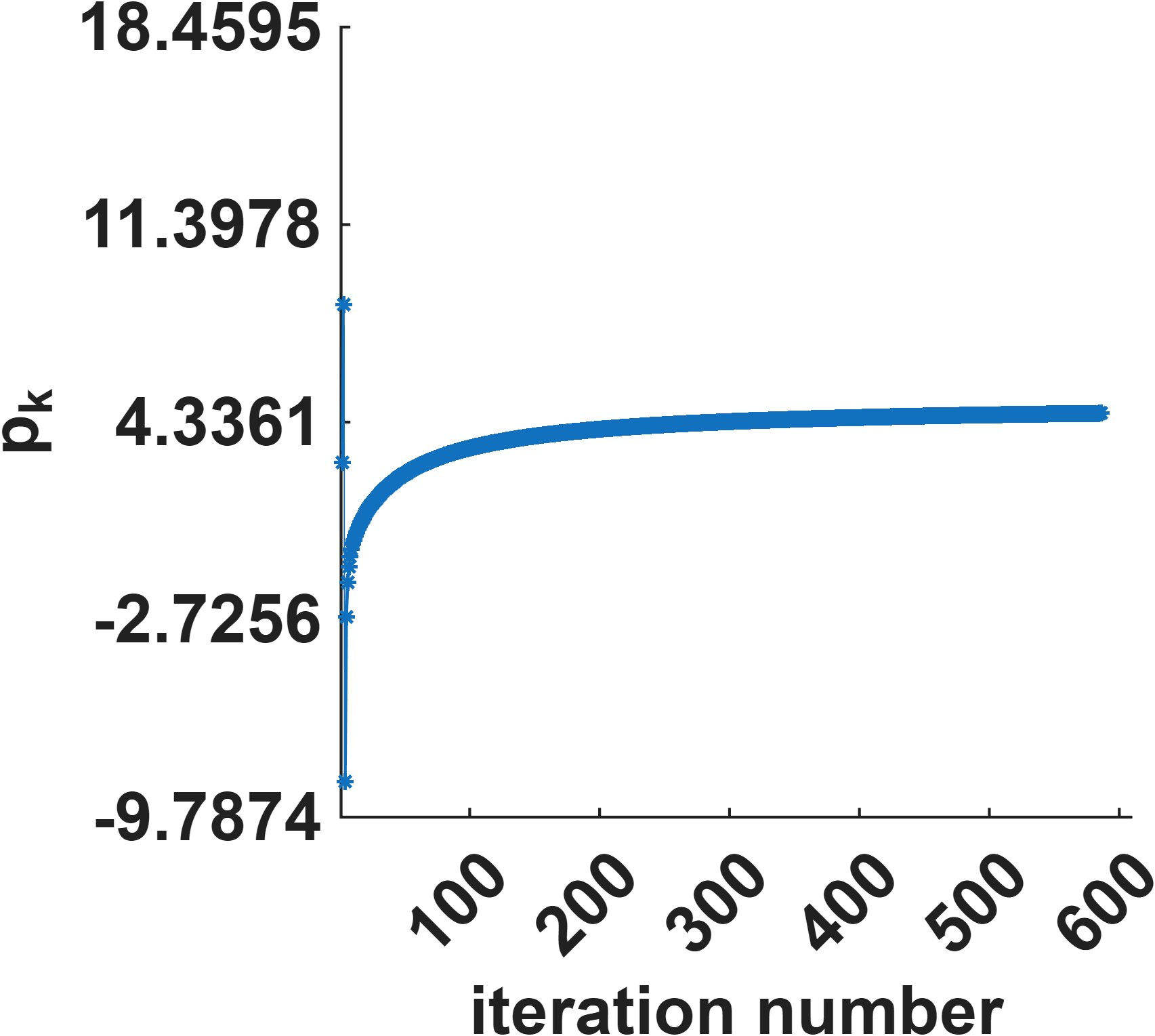"}
    \caption{FD}
  \end{subfigure}\hfill
  \begin{subfigure}{0.24\textwidth}\centering
    \includegraphics[width=\textwidth]{"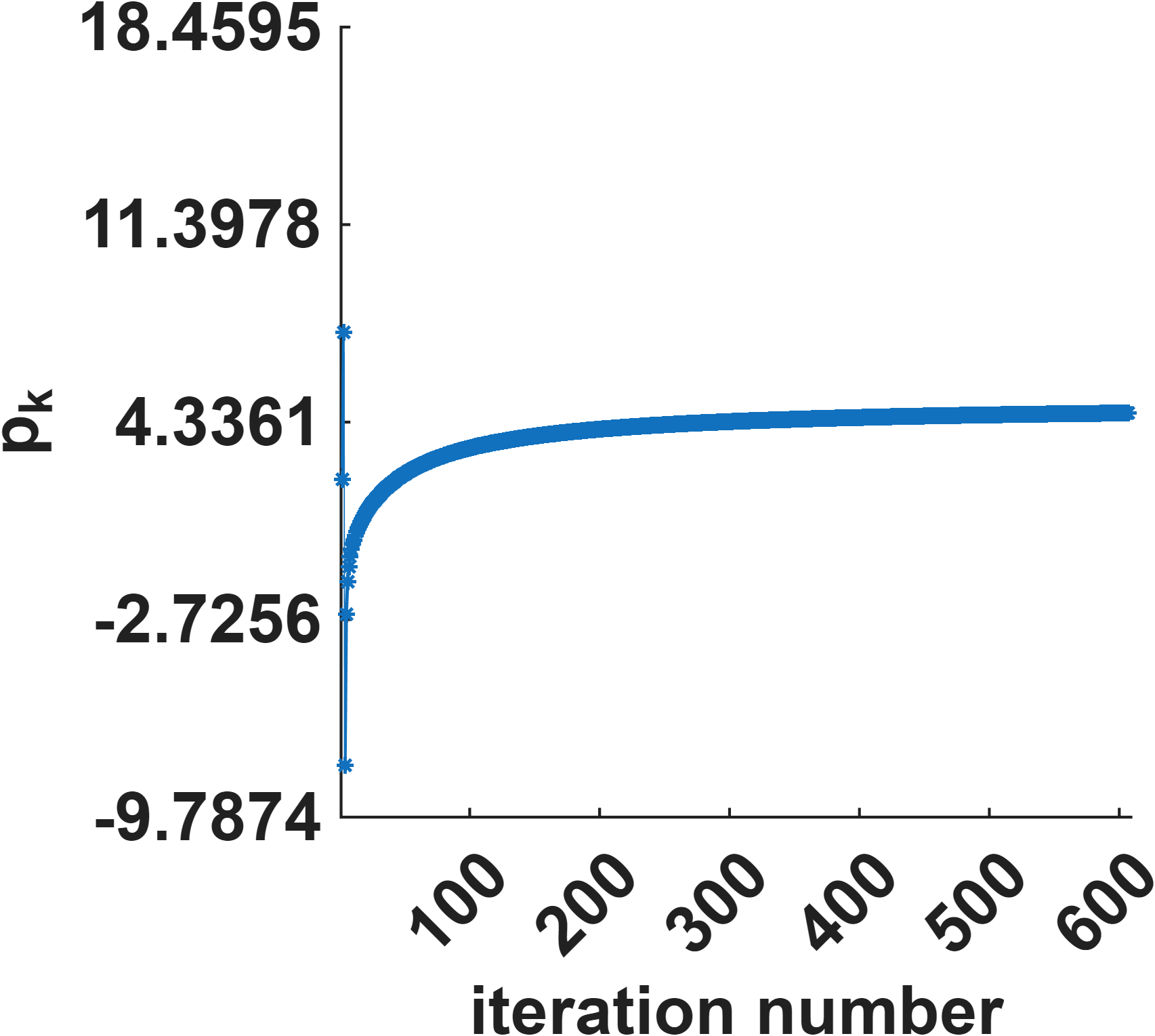"}
    \caption{IMEX-RB}
  \end{subfigure}\hfill
  \begin{subfigure}{0.24\textwidth}\centering
    \includegraphics[width=\textwidth]{"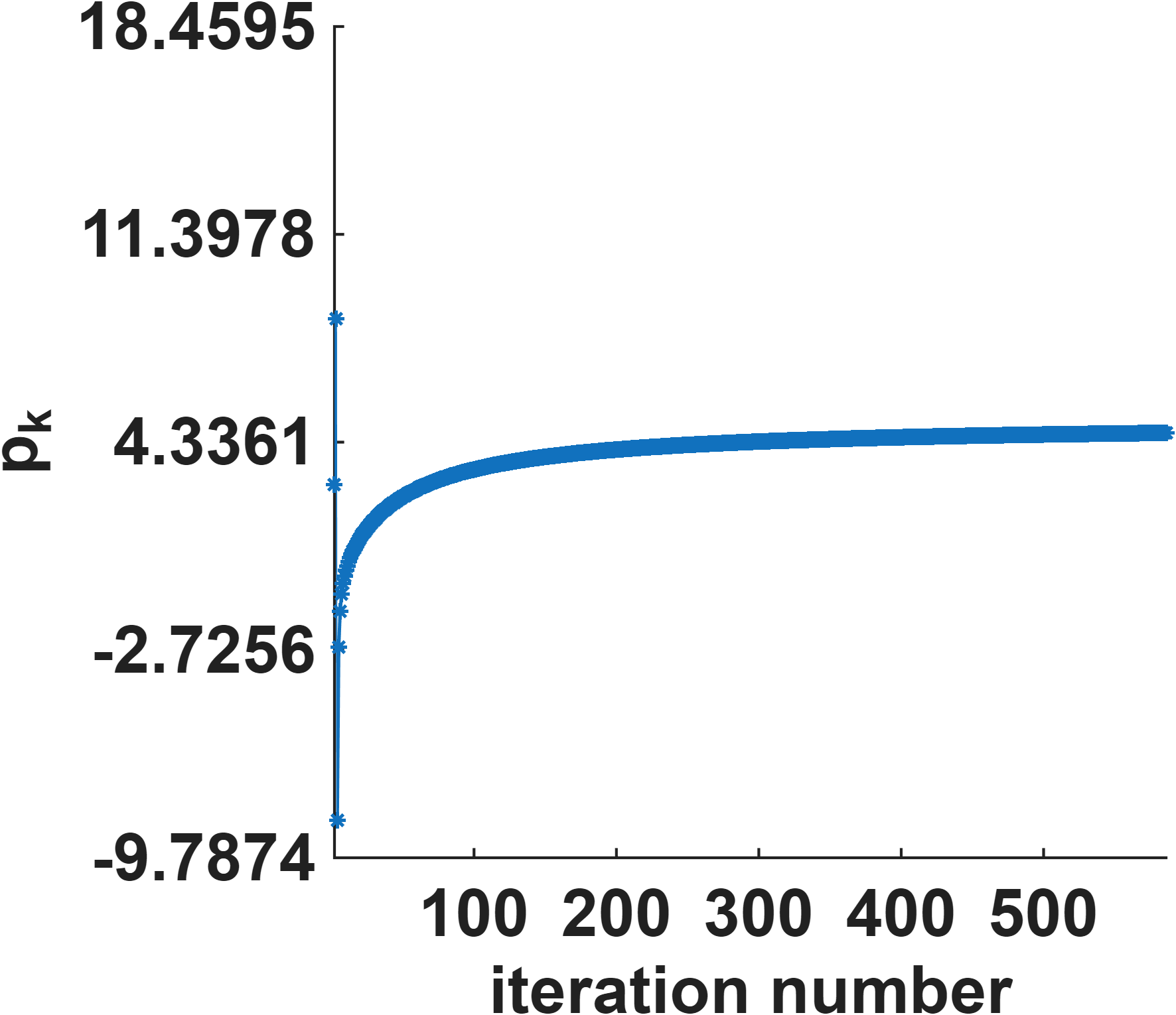"}
    \caption{FB}
  \end{subfigure}\hfill
  \begin{subfigure}{0.24\textwidth}\centering
    \includegraphics[width=\textwidth]{"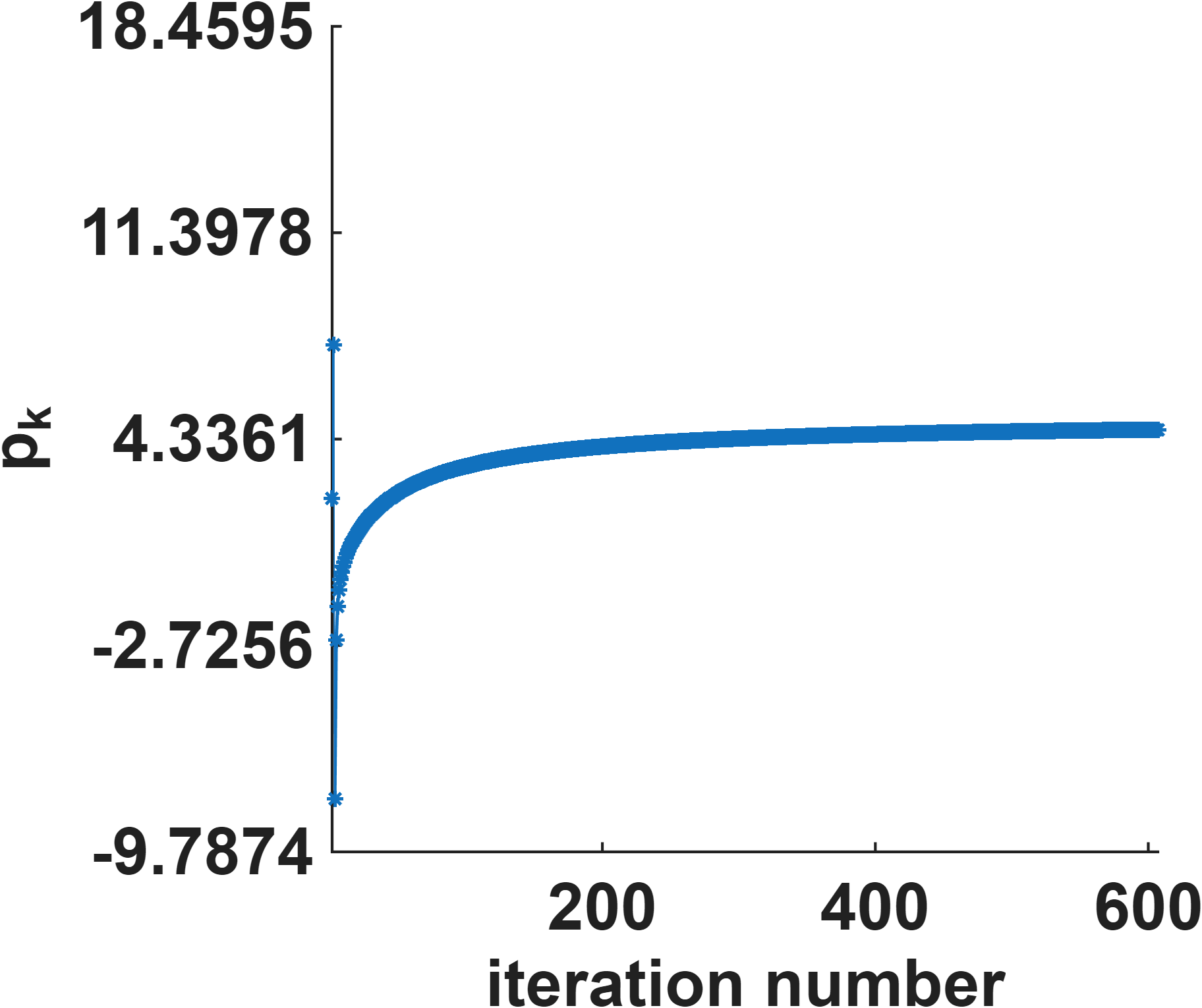"}
    \caption{Semi}
  \end{subfigure}

  \vspace{0.45em}

  \begin{subfigure}{0.24\textwidth}\centering
    \includegraphics[width=\textwidth]{"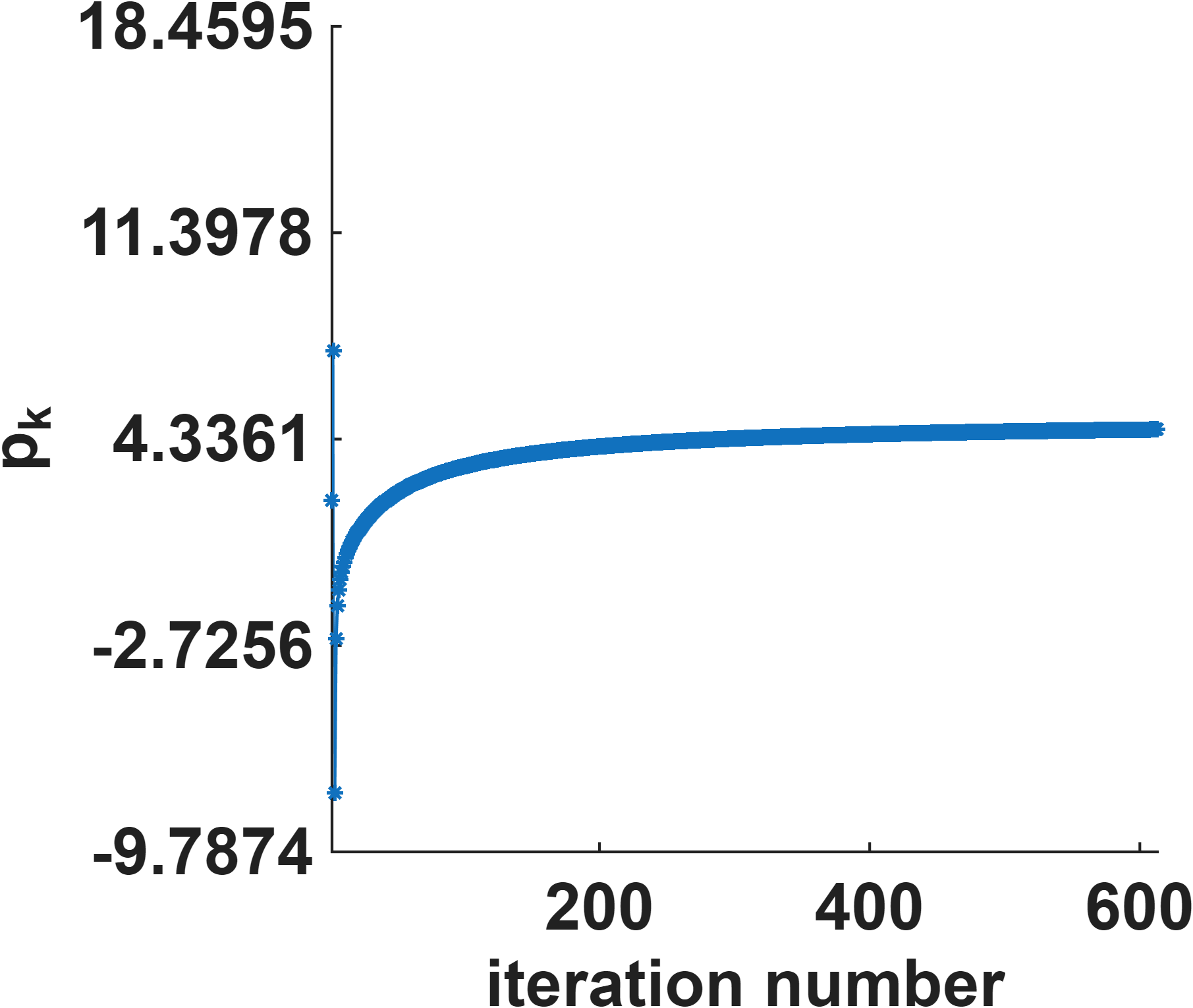"}
    \caption{IPAHD}
  \end{subfigure}\hfill
  \begin{subfigure}{0.24\textwidth}\centering
    \includegraphics[width=\textwidth]{"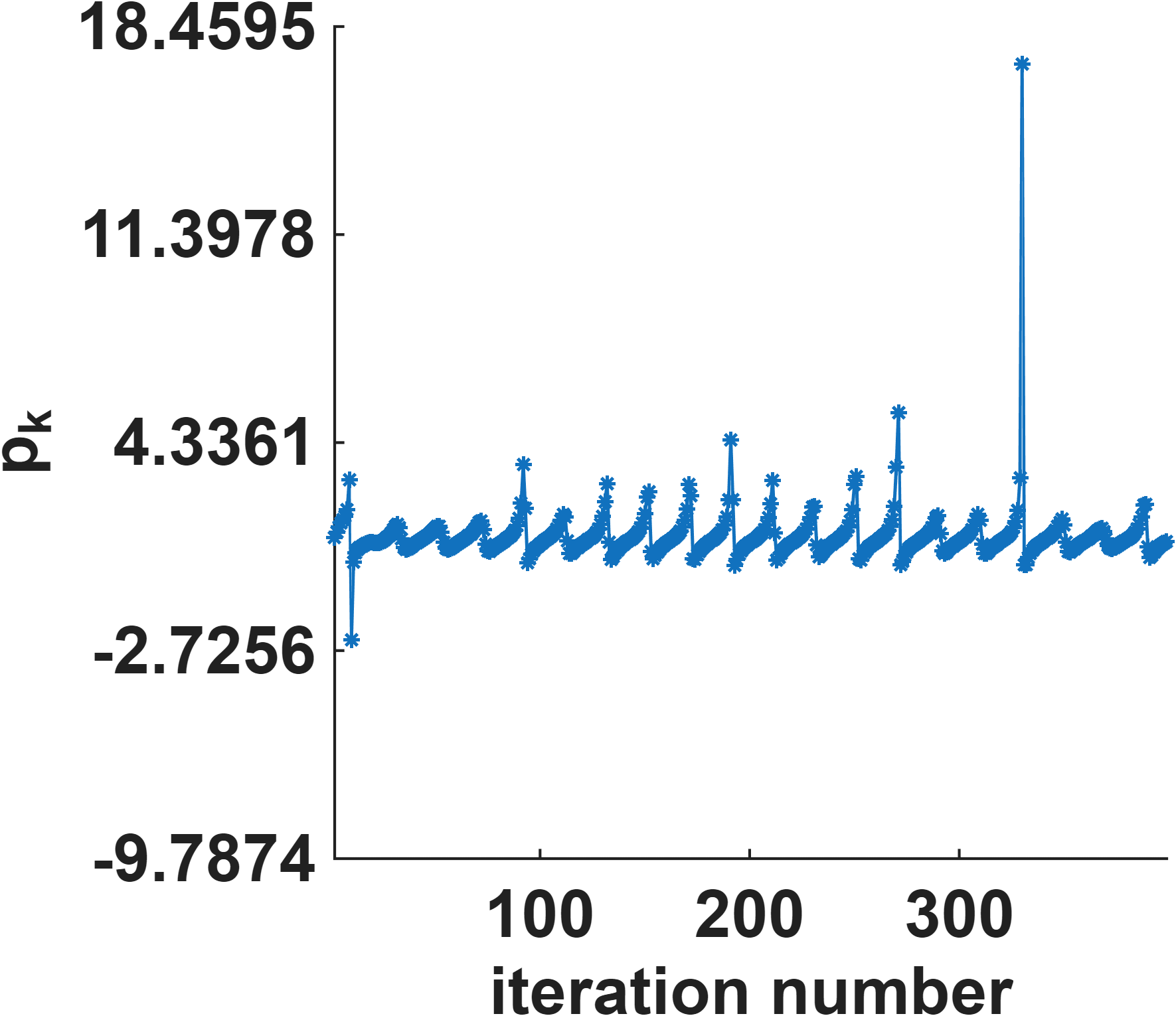"}
    \caption{Nesterov}
  \end{subfigure}\hfill
  \begin{subfigure}{0.24\textwidth}\centering
    \includegraphics[width=\textwidth]{"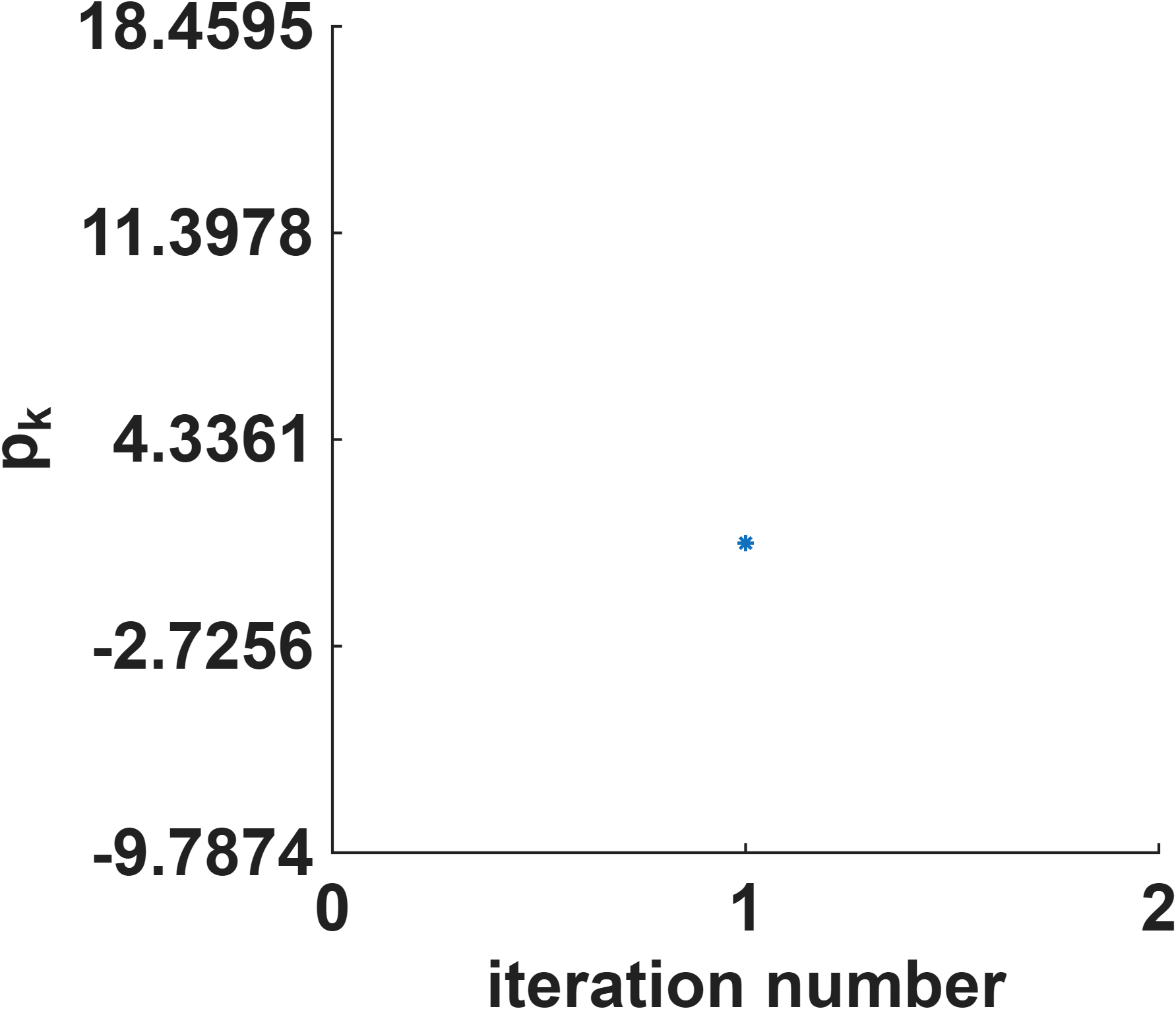"}
    \caption{GD}
  \end{subfigure}

  \caption{Local convergence rate exponent $p_k$ across methods (10D Rotated Hyper-Ellipsoid Function, $h = 1/16$). The first iteration is initialized by $x_1 = x_0 - 10^{-4} \nabla F(x_0)$, and the subsequent iterations follow the corresponding numerical schemes. All inertial methods demonstrate a uniform plateau except for the Nesterov method, which shows violent oscillations during the iteration.}
  \label{fig:convergence_comparison}
\end{figure}

\begin{figure}[H]
  \centering
  \captionsetup[subfigure]{justification=centering}

  \begin{subfigure}{0.24\textwidth}\centering
    \includegraphics[width=\textwidth]{"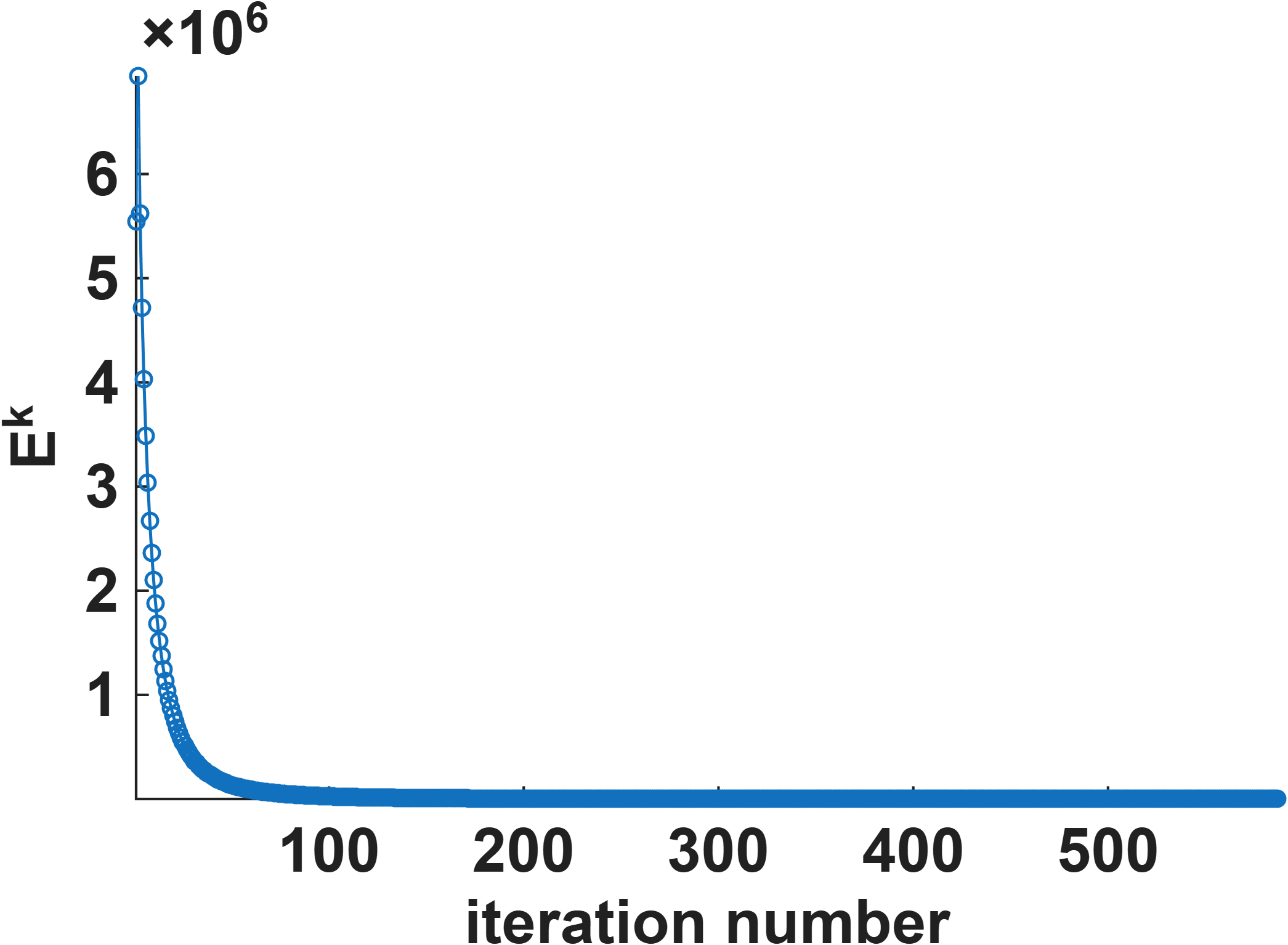"}
    \caption{FD}
  \end{subfigure}\hfill
  \begin{subfigure}{0.24\textwidth}\centering
    \includegraphics[width=\textwidth]{"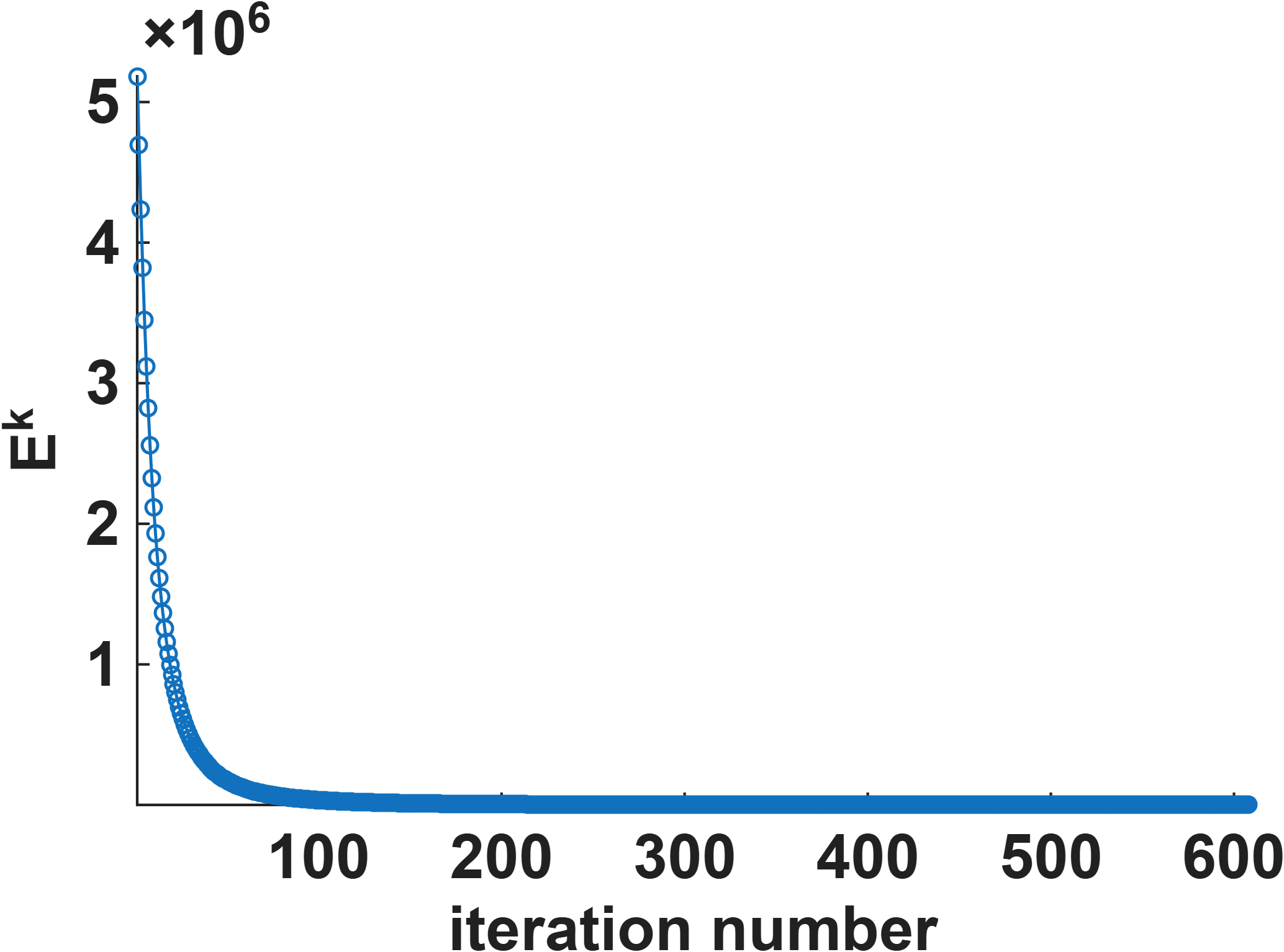"}
    \caption{IMEX-RB}
  \end{subfigure}\hfill
  \begin{subfigure}{0.24\textwidth}\centering
    \includegraphics[width=\textwidth]{"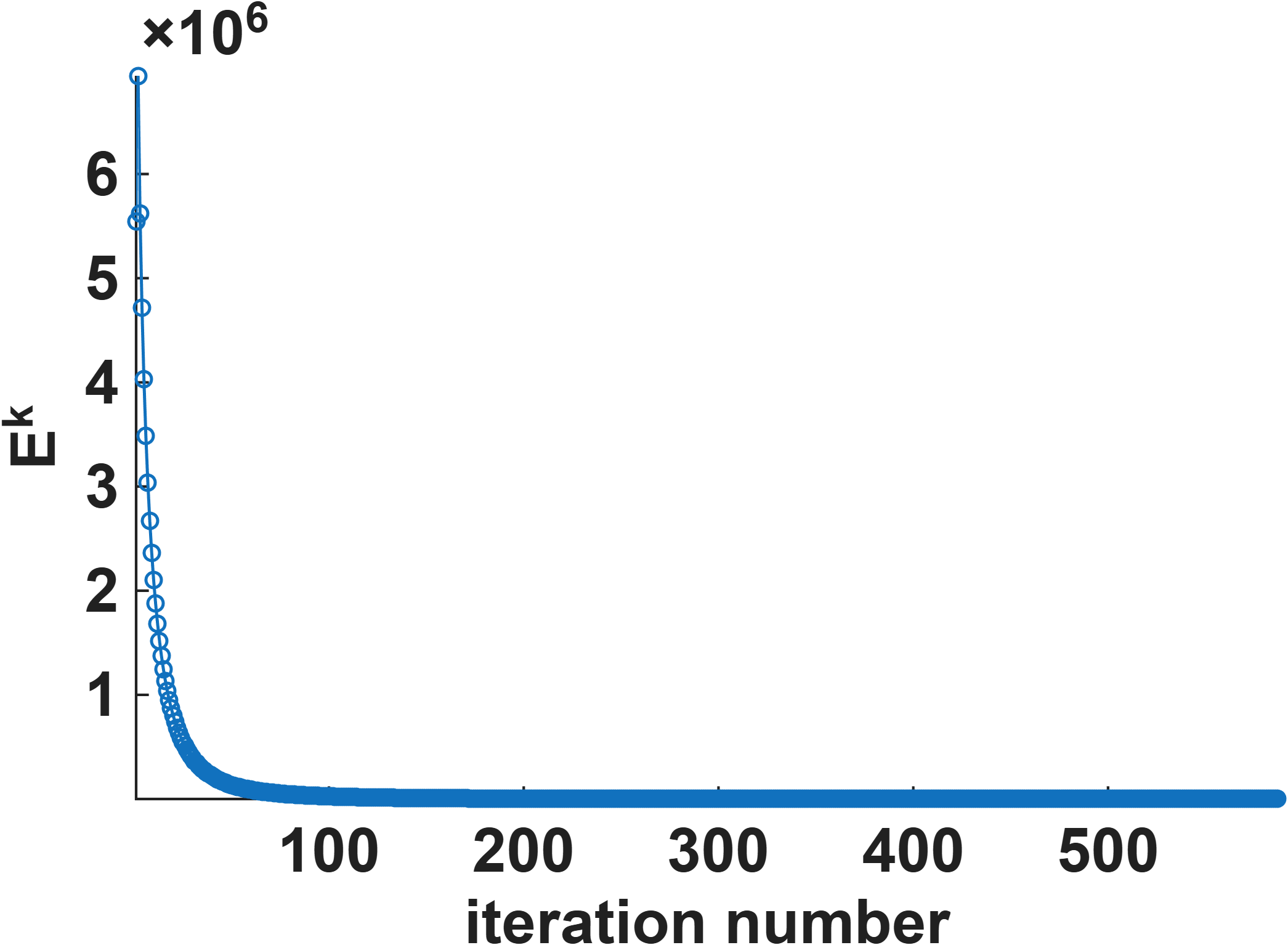"}
    \caption{FB}
  \end{subfigure}\hfill
  \begin{subfigure}{0.24\textwidth}\centering
    \includegraphics[width=\textwidth]{"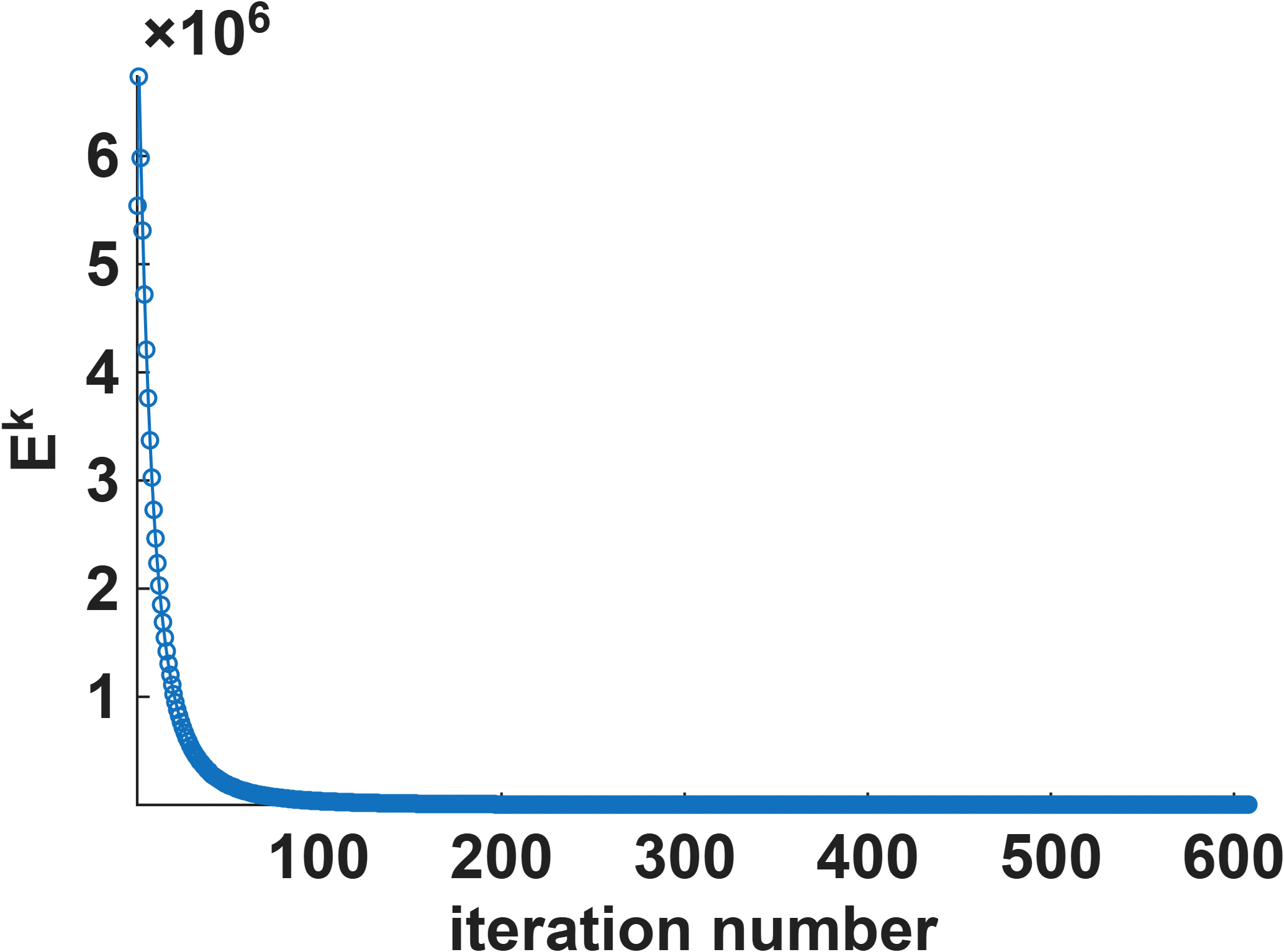"}
    \caption{Semi}
  \end{subfigure}

  \vspace{0.45em}

  \begin{subfigure}{0.24\textwidth}\centering
    \includegraphics[width=\textwidth]{"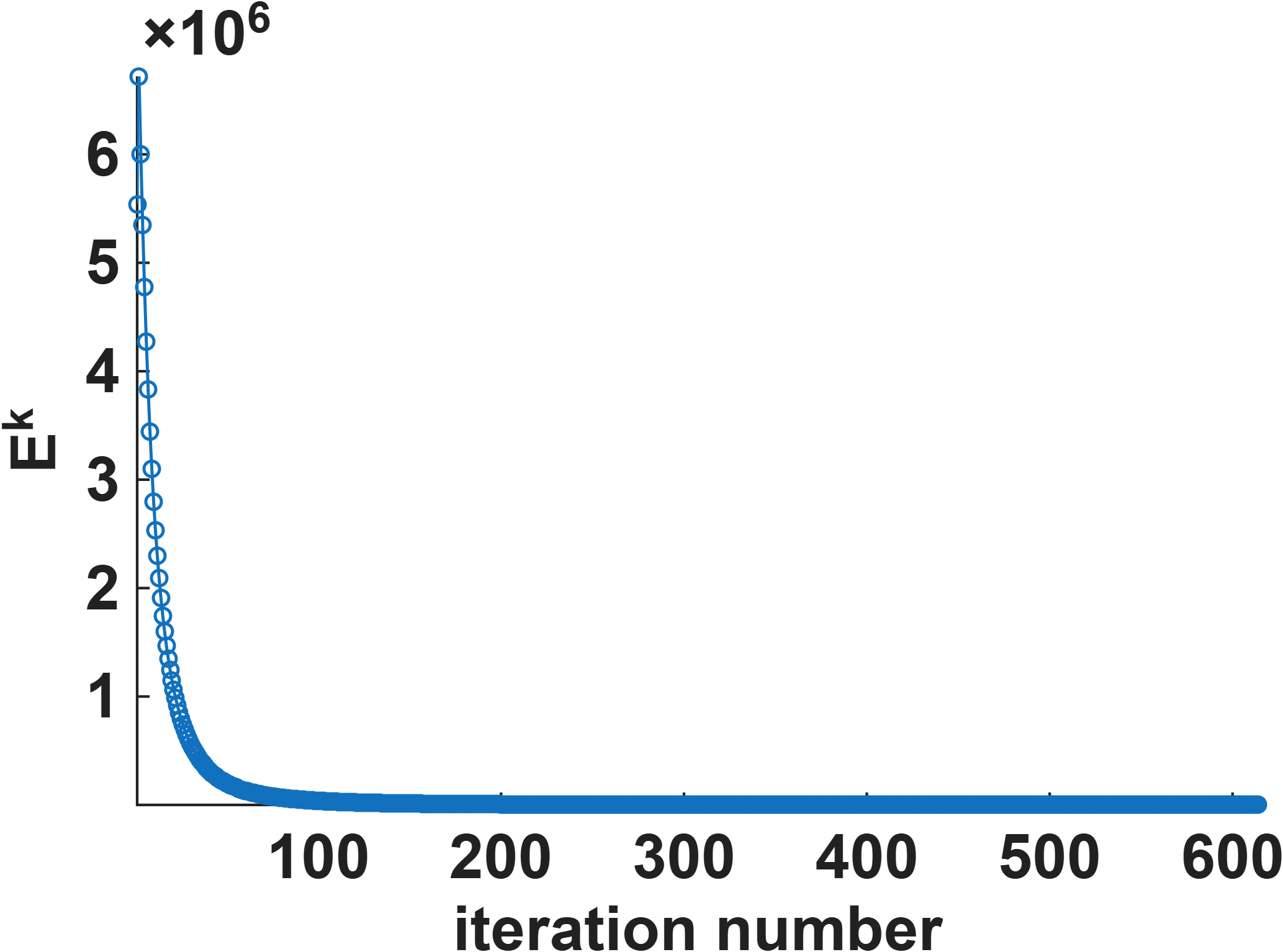"}
    \caption{IPAHD}
  \end{subfigure}\hfill
  \begin{subfigure}{0.24\textwidth}\centering
    \includegraphics[width=\textwidth]{"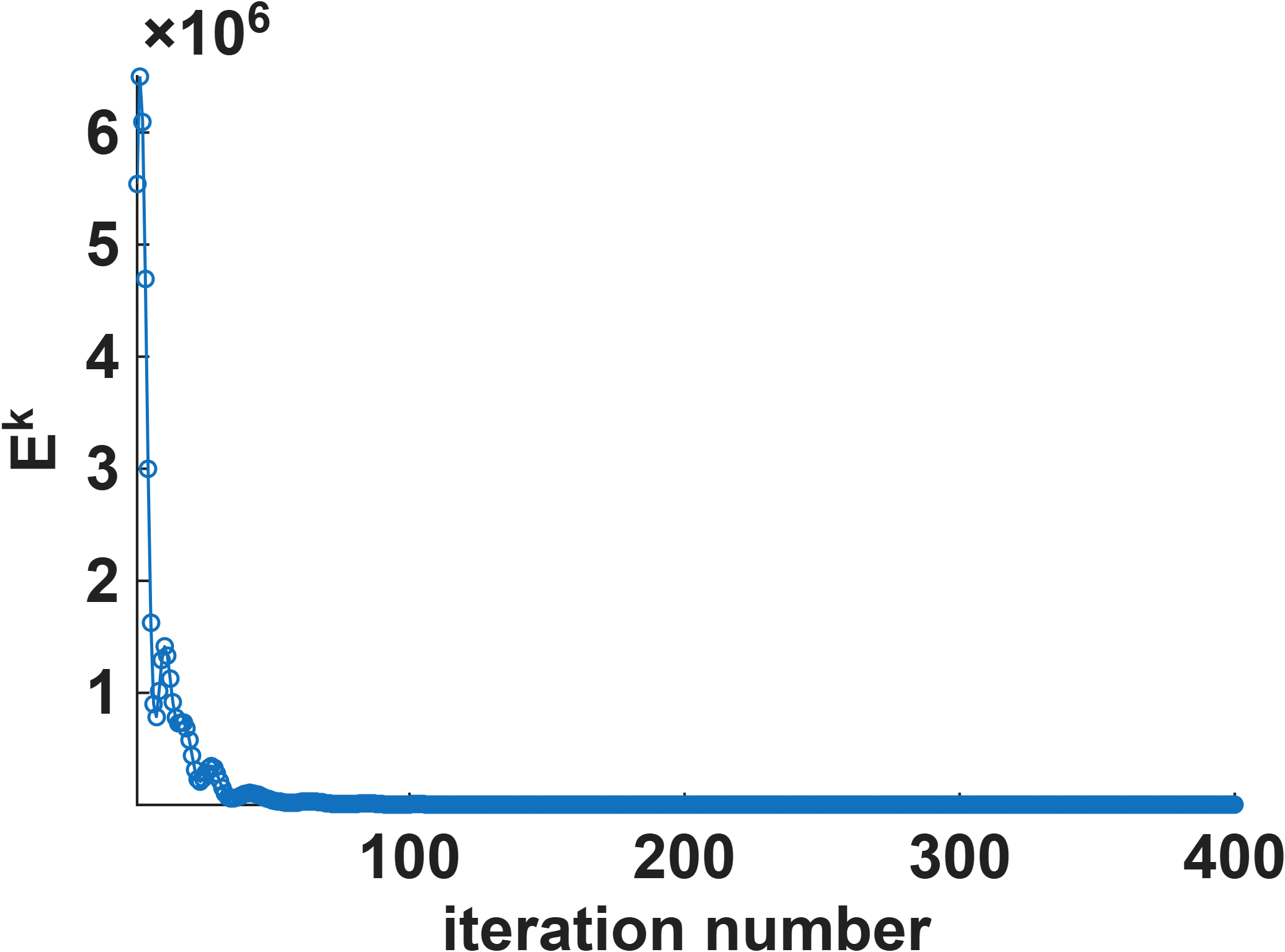"}
    \caption{Nesterov}
  \end{subfigure}\hfill
  \begin{subfigure}{0.24\textwidth}\centering
    \includegraphics[width=\textwidth]{"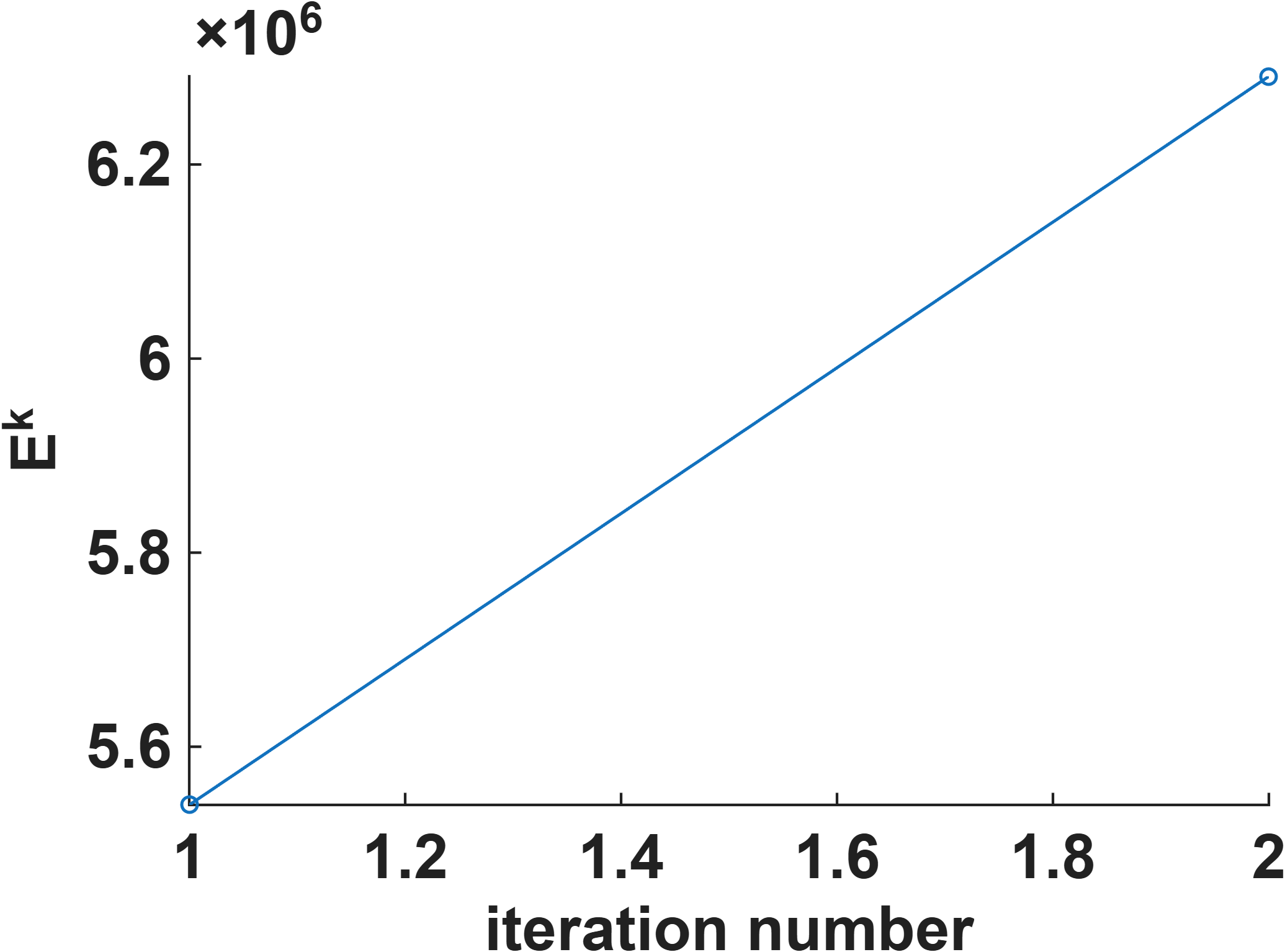"}
    \caption{GD}
  \end{subfigure}

  \caption{Comparison of $E^k$ across methods for the 10D Rotated Hyper-Ellipsoid Function (\(h=1/16\)). Each subplot is one method (left-to-right, top-to-bottom). The discrete energy oscillates wildly when using the Nesterov method. }
  \label{fig:energy_all}
\end{figure}

In Tables~\ref{tab:rothe1d_compact}--\ref{tab:rothe50d_compact}, we present the average convergence rate, $\bar{p}$, for the Rotated Hyper-Ellipsoid function in different dimensions.
Numerical results on convergence rates of other test functions are provided in the appendix~\ref{num:append:con}.

All new inertial dynamics based methods have $\bar{p} > 1$ and successfully converge to the minimum of  the test function's. This means that $F^{k} - F^*$ decays at least at the rate, $\mathcal{O}\!\left(\frac{1}{\delta_k^p}\right)$. 
In these cases, the value of $\bar{p}$ ranges from $3$ to $5$, depending on $h$.  
These results are consistent with the result of Theorem~\ref{continuous:thm}. In fact, they show that the actual convergence rate is much faster than $\mathcal{O}\!\left(\frac{1}{\delta_k}\right)$. The rate $\mathcal{O}\!\left(\frac{1}{\delta_k}\right)$ is therefore a "worst-case" upper bound. From the results, we notice a phenomenon: the larger $h$ is, the higher $p$ is. We think this applies only in a finite range of $h$ bounded by an upper bound that depends on the test function.

In contrast, the classical Nesterov method and the gradient descent method do not always reach the optimal convergence rate. In particular, the gradient descent method reaches the minimum only in the 1D case and fails in the higher dimensional cases.
For these two methods, $\bar{p}$ is sometimes negative. This indicates that $F^{k} - F^*$ does not decay monotonically and it increases more often than it decreases during the iterations, which is consistent with the failure to reach the minimum in higher dimensional cases. This also suggests that the dynamic paths in these two methods in the optimization process are quite different from the rest. 

Although $\bar{p}$ for these two methods can be quite large in the 1D case (For "NaN" in the table, it indicates that $F^{k} - F^*$ decreases to 0 extremely fast), it drops rapidly in higher-dimensional cases to around $1$ for the classical Nesterov method. This is much smaller than $\bar{p}$ obtained by the new inertial dynamics based methods with the same parameter values. 
This shows that these two methods do not achieve convergence rates as high and stable as the new inertial dynamics–based methods. Moreover, in higher dimensions, their convergence to the minimum is less stable than that of the new methods.

The value of $\bar{p}$ for the classical Nesterov method indicates a decay rate of approximately $\mathcal{O}\!\left(\frac{1}{\delta_k}\right)$, which is equivalent to $\mathcal{O}\!\left(\frac{1}{t_k^2}\right)$. This is consistent with the theoretical results reported in \cite{aujol2019optimal}.
For the gradient descent method, $\bar{p}$ is generally between $0.5$ and $1$, corresponding to a decay rate of approximately $\mathcal{O}\!\left(\frac{1}{t_k}\right)$, and in some cases approaching $\mathcal{O}\!\left(\frac{1}{t_k^2}\right)$. This behavior agrees with the classical convergence analysis in \cite{nocedal2006numerical}.


\begin{table}[H]
  \centering
  \caption{1D Rotated Hyper-Ellipsoid Function}
  \label{tab:rothe1d_compact}
  {\fontsize{9}{11}\selectfont
  \setlength{\tabcolsep}{4pt}
  \begin{adjustbox}{max width=\textwidth}
  \begin{tabular}{l
                  r c
                  r c
                  r c
                  r c
                  r c
                  r c
                  r c}
    \toprule
    \makecell[l]{h} &
    \makecell{FB\\$\bar{p}$} & \makecell{FB\\S/F} &
    \makecell{Semi\\$\bar{p}$} & \makecell{Semi\\S/F} &
    \makecell{FD\\$\bar{p}$} & \makecell{FD\\S/F} &
    \makecell{IMEX\\$\bar{p}$} & \makecell{IMEX\\S/F} &
    \makecell{IPAHD\\$\bar{p}$} & \makecell{IPAHD\\S/F} &
    \makecell{NM\\$\bar{p}$} & \makecell{NM\\S/F} &
    \makecell{GD\\$\bar{p}$} & \makecell{GD\\S/F} \\
    \midrule
    1          & 4.7236 & 1 & 4.7236 & 1 & 4.7236 & 1 & 4.7229 & 1 & 4.7230 & 1 & 241.7813 & 1 & NaN & 1 \\
    $1/2^{1}$  & 4.6111 & 1 & 4.6111 & 1 & 4.6111 & 1 & 4.6111 & 1 & 4.6103 & 1 & 10.2989  & 1 & NaN & 1 \\
    $1/2^{2}$  & 4.4735 & 1 & 4.4745 & 1 & 4.4735 & 1 & 4.4735 & 1 & 4.4717 & 1 & 27.6623  & 1 & NaN & 1 \\
    $1/2^{3}$  & 4.3428 & 1 & 4.3427 & 1 & 4.3428 & 1 & 4.3428 & 1 & 4.3400 & 1 & 83.2620  & 1 & NaN & 1 \\
    $1/2^{4}$  & 4.1744 & 1 & 4.1744 & 1 & 4.1744 & 1 & 4.1744 & 1 & 4.1697 & 1 & 10.4080  & 1 & NaN & 1 \\
    $1/2^{5}$  & 4.0498 & 1 & 4.0503 & 1 & 4.0498 & 1 & 4.0498 & 1 & 4.0453 & 1 & 4.1936   & 1 & NaN & 1 \\
    $1/2^{6}$  & 3.8585 & 1 & 3.8588 & 1 & 3.8585 & 1 & 3.8584 & 1 & 3.8518 & 1 & 4183.8   & 1 & NaN & 1 \\
    $1/2^{7}$  & 3.7620 & 1 & 3.7619 & 1 & 3.7620 & 1 & 3.7617 & 1 & 3.7567 & 1 & 5.6631   & 1 & NaN & 1 \\
    \bottomrule
  \end{tabular}
  \end{adjustbox}
  }
\end{table}

\begin{table}[H]
  \centering
  \caption{2D Rotated Hyper-Ellipsoid Function}
  \label{tab:rothe2d_compact}
  {\fontsize{9}{11}\selectfont
  \setlength{\tabcolsep}{4pt}
  \begin{adjustbox}{max width=\textwidth}
  \begin{tabular}{l
                  r c
                  r c
                  r c
                  r c
                  r c
                  r c
                  r c}
    \toprule
    \makecell[l]{h} &
    \makecell{FB\\$\bar{p}$} & \makecell{FB\\S/F} &
    \makecell{Semi\\$\bar{p}$} & \makecell{Semi\\S/F} &
    \makecell{FD\\$\bar{p}$} & \makecell{FD\\S/F} &
    \makecell{IMEX\\$\bar{p}$} & \makecell{IMEX\\S/F} &
    \makecell{IPAHD\\$\bar{p}$} & \makecell{IPAHD\\S/F} &
    \makecell{NM\\$\bar{p}$} & \makecell{NM\\S/F} &
    \makecell{GD\\$\bar{p}$} & \makecell{GD\\S/F} \\
    \midrule
    1          & 4.8494 & 1 & 4.8494 & 1 & 4.8494 & 1 & 4.8501 & 1 & 4.8517 & 1 & 9.3867  & 1 & 0.25    & 0 \\
    $1/2^{1}$  & 4.7888 & 1 & 4.7861 & 1 & 4.7888 & 1 & 4.7864 & 1 & 4.7894 & 1 & 5.1663  & 1 & -0.4444 & 0 \\
    $1/2^{2}$  & 4.7189 & 1 & 4.7084 & 1 & 4.7189 & 1 & 4.7067 & 1 & 4.7116 & 1 & 7.6625  & 1 & 2.0000  & 0 \\
    $1/2^{3}$  & 4.6577 & 1 & 4.6348 & 1 & 4.6577 & 1 & 4.6313 & 1 & 4.6335 & 1 & 2.6310  & 1 & 0.9630  & 0 \\
    $1/2^{4}$  & 4.5745 & 1 & 4.5381 & 1 & 4.5745 & 1 & 4.5344 & 1 & 4.5358 & 1 & 1.6467  & 1 & 0.7873  & 0 \\
    $1/2^{5}$  & 4.5036 & 1 & 4.4604 & 1 & 4.5036 & 1 & 4.4580 & 1 & 4.4561 & 1 & 1.3303  & 1 & 0.7224  & 0 \\
    $1/2^{6}$  & 4.4091 & 1 & 4.3491 & 1 & 4.4091 & 1 & 4.3474 & 1 & 4.3425 & 1 & 1.1623  & 1 & 0.6936  & 0 \\
    $1/2^{7}$  & 4.3320 & 1 & 4.2746 & 1 & 4.3320 & 1 & 4.2740 & 1 & 4.2686 & 1 & 1.1269  & 1 & 0.6799  & 0 \\
    \bottomrule
  \end{tabular}
  \end{adjustbox}
  }
\end{table}


\begin{table}[H]
  \centering
  \caption{10D Rotated Hyper-Ellipsoid Function}
  \label{tab:rothe10d_compact}
  {\fontsize{9}{11}\selectfont
  \setlength{\tabcolsep}{4pt}
  \begin{adjustbox}{max width=\textwidth}
  \begin{tabular}{l
                  r c
                  r c
                  r c
                  r c
                  r c
                  r c
                  r c}
    \toprule
    \makecell[l]{h} &
    \makecell{FB\\$\bar{p}$} & \makecell{FB\\S/F} &
    \makecell{Semi\\$\bar{p}$} & \makecell{Semi\\S/F} &
    \makecell{FD\\$\bar{p}$} & \makecell{FD\\S/F} &
    \makecell{IMEX\\$\bar{p}$} & \makecell{IMEX\\S/F} &
    \makecell{IPAHD\\$\bar{p}$} & \makecell{IPAHD\\S/F} &
    \makecell{NM\\$\bar{p}$} & \makecell{NM\\S/F} &
    \makecell{GD\\$\bar{p}$} & \makecell{GD\\S/F} \\
    \midrule
    1          & 4.6463 & 1 & 4.6518 & 1 & 4.6463 & 1 & 4.6518 & 1 & 4.6529 & 1 & 0.9986 & 0 & 0.2500 & 0 \\
    $1/2^{1}$  & 4.4999 & 1 & 4.5119 & 1 & 4.4999 & 1 & 4.5119 & 1 & 4.5133 & 1 & 1.7811 & 1 & -0.4444 & 0 \\
    $1/2^{2}$  & 4.3215 & 1 & 4.3403 & 1 & 4.3215 & 1 & 4.3419 & 1 & 4.3434 & 1 & 1.6257 & 1 & 2.0000  & 0 \\
    $1/2^{3}$  & 4.1734 & 1 & 4.1835 & 1 & 4.1734 & 1 & 4.1847 & 1 & 4.1879 & 1 & 1.2726 & 1 & 0.9630 & 0 \\
    $1/2^{4}$  & 3.9520 & 1 & 3.9744 & 1 & 3.9520 & 1 & 3.9755 & 1 & 3.9808 & 1 & 1.1790 & 1 & 0.7873 & 0 \\
    $1/2^{5}$  & 3.8103 & 1 & 3.8318 & 1 & 3.8103 & 1 & 3.8325 & 1 & 3.8375 & 1 & 1.0921 & 1 & 0.7224 & 0 \\
    $1/2^{6}$  & 3.5715 & 1 & 3.5965 & 1 & 3.5715 & 1 & 3.5971 & 1 & 3.6054 & 1 & 1.0450 & 1 & 0.6936 & 0 \\
    $1/2^{7}$  & 3.4781 & 1 & 3.5029 & 1 & 3.4781 & 1 & 3.5032 & 1 & 3.5090 & 1 & 1.0201 & 1 & 0.6799 & 0 \\
    \bottomrule
  \end{tabular}
  \end{adjustbox}
  }
\end{table}

\begin{table}[H]
  \centering
  \caption{50D Rotated Hyper-Ellipsoid Function}
  \label{tab:rothe50d_compact}
  {\fontsize{9}{11}\selectfont
  \setlength{\tabcolsep}{4pt}
  \begin{adjustbox}{max width=\textwidth}
  \begin{tabular}{l
                  r c
                  r c
                  r c
                  r c
                  r c
                  r c
                  r c}
    \toprule
    \makecell[l]{h} &
    \makecell{FB\\$\bar{p}$} & \makecell{FB\\S/F} &
    \makecell{Semi\\$\bar{p}$} & \makecell{Semi\\S/F} &
    \makecell{FD\\$\bar{p}$} & \makecell{FD\\S/F} &
    \makecell{IMEX\\$\bar{p}$} & \makecell{IMEX\\S/F} &
    \makecell{IPAHD\\$\bar{p}$} & \makecell{IPAHD\\S/F} &
    \makecell{NM\\$\bar{p}$} & \makecell{NM\\S/F} &
    \makecell{GD\\$\bar{p}$} & \makecell{GD\\S/F} \\
    \midrule
    1          & 3.7996 & 1 & 3.7990 & 1 & 3.7996 & 1 & 3.7990 & 1 & 3.7990 & 1 & 0.9987 & 0 & 0.2500 & 0 \\
    $1/2^{1}$  & 3.3899 & 1 & 3.3893 & 1 & 3.3899 & 1 & 3.3893 & 1 & 3.3893 & 1 & 0.9985 & 0 & -0.4444 & 0 \\
    $1/2^{2}$  & 2.9133 & 1 & 2.9130 & 1 & 2.9133 & 1 & 2.9130 & 1 & 2.9130 & 1 & 0.9984 & 0 & 2.0000  & 0 \\
    $1/2^{3}$  & 2.6119 & 1 & 2.6090 & 1 & 2.6119 & 1 & 2.6255 & 1 & 2.6090 & 1 & 1.2954 & 1 & 0.9630 & 0 \\
    $1/2^{4}$  & 2.0507 & 1 & 2.0499 & 1 & 2.0507 & 1 & 2.0499 & 1 & 2.0499 & 1 & 1.0820 & 1 & 0.7873 & 0 \\
    $1/2^{5}$  & 1.9610 & 1 & 1.9568 & 1 & 1.9610 & 1 & 1.9640 & 1 & 1.9567 & 1 & 1.0578 & 1 & 0.7224 & 0 \\
    $1/2^{6}$  & 1.3213 & 1 & 1.3258 & 1 & 1.3213 & 1 & 1.3258 & 1 & 1.3258 & 1 & 1.0141 & 1 & 0.6936 & 0 \\
    $1/2^{7}$  & 1.7860 & 1 & 1.7786 & 1 & 1.7860 & 1 & 1.7798 & 1 & 1.7784 & 1 & 1.0135 & 1 & 0.6936 & 0 \\
    \bottomrule
  \end{tabular}
  \end{adjustbox}
  }
\end{table}

In summary, for the new inertial dynamics based algorithms, $F^{k} - F^*$ decreases monotonically after the first few iterations. This is consistent with the energy stability property proved in Theorem~\ref{continuous:thm}, Theorem~\ref{fully-discretized:energystability}, and Theorem~\ref{imexrb:energystability}. In contrast, NM and GD show oscillations or increase in the energy during the iteration. 
This difference highlights why the Onsager principle and structure-preserving numerical schemes tend to be more stable: they provide better control over the inertia in the system and help avoid the unnecessary oscillations that occur in some traditional momentum methods.

\subsection{Test of swarm-based algorithms}
\noindent\indent In many practical applications, optimization problems are non-convex, and the objective functions possess multiple local minima. This makes finding the global optimum using classical optimization methods more challenging. Swarm-based algorithms often perform better in such settings because the underlying mass-transport dynamics promote exploration and help the method escape local minima.

In the previous sections, we focused on the construction of inertial algorithms that preserve discrete energy dissipation. In this section, we apply these inertial algorithms to several representative non-convex test functions to evaluate their performance.

\noindent \textbf{Test Functions:}

We consider the Ackley and Rastrigin functions as our test functions and examine the performance of the swarm-based algorithms over the two functions  in various dimensions. The definitions of these two functions in d-dimension are given as follows:
\begin{equation}
	\begin{aligned}
		&F_{\mathrm{Ackley}}(\mathbf{x})= -20 \exp\!\left(-\frac{0.2}{\sqrt{d}}\left(\sum_{i=1}^{d} (\mathbf{x}_B)_i^{2}\right)^{\frac{1}{2}}\right)- \exp\!\left(-\frac{1}{\sqrt{d}}\sum_{i=1}^{d} \cos\!\bigl(2\pi (\mathbf{x}_B)_i\bigr)\right)+ 20 + e + C, \\
		&F_{\mathrm{Rastrigin}}(\mathbf{x})= \frac{1}{d} \sum_{i=1}^{d}\left( (\mathbf{x}_B)_i^{2}- 10 \cos\!\bigl(2\pi (\mathbf{x}_B)_i\bigr)+ 10 \right)+ C,
	\end{aligned}
\end{equation}
where $\mathbf{x}\in\mathbb{R}^d$, $d$ denotes the dimension, $\mathbf{x}_B = \mathbf{x} - B$, and $B$ and $C$ are parameters. These parameters are adjusted according to the experiments and will be specified. We fix $C=0$ in the following numerical experiments.

Both the Ackley and Rastrigin function are non-convex and have many local minima, making them challenging in global optimizations.  As a result, they are commonly used as representative objective functions for testing optimization methods. Although these functions have many local minima, they share the same unique global minimum at
\[x^* = (B, \dots, B).\]

The plots of the 1D test functions with shifting parameters $B = 0$ and $C = 0$ are shown below:
\begin{figure}[H]
  \centering
  \captionsetup{font=small, skip=4pt} 

  \begin{subfigure}[b]{0.45\textwidth}
    \centering
    \includegraphics[width=\textwidth]{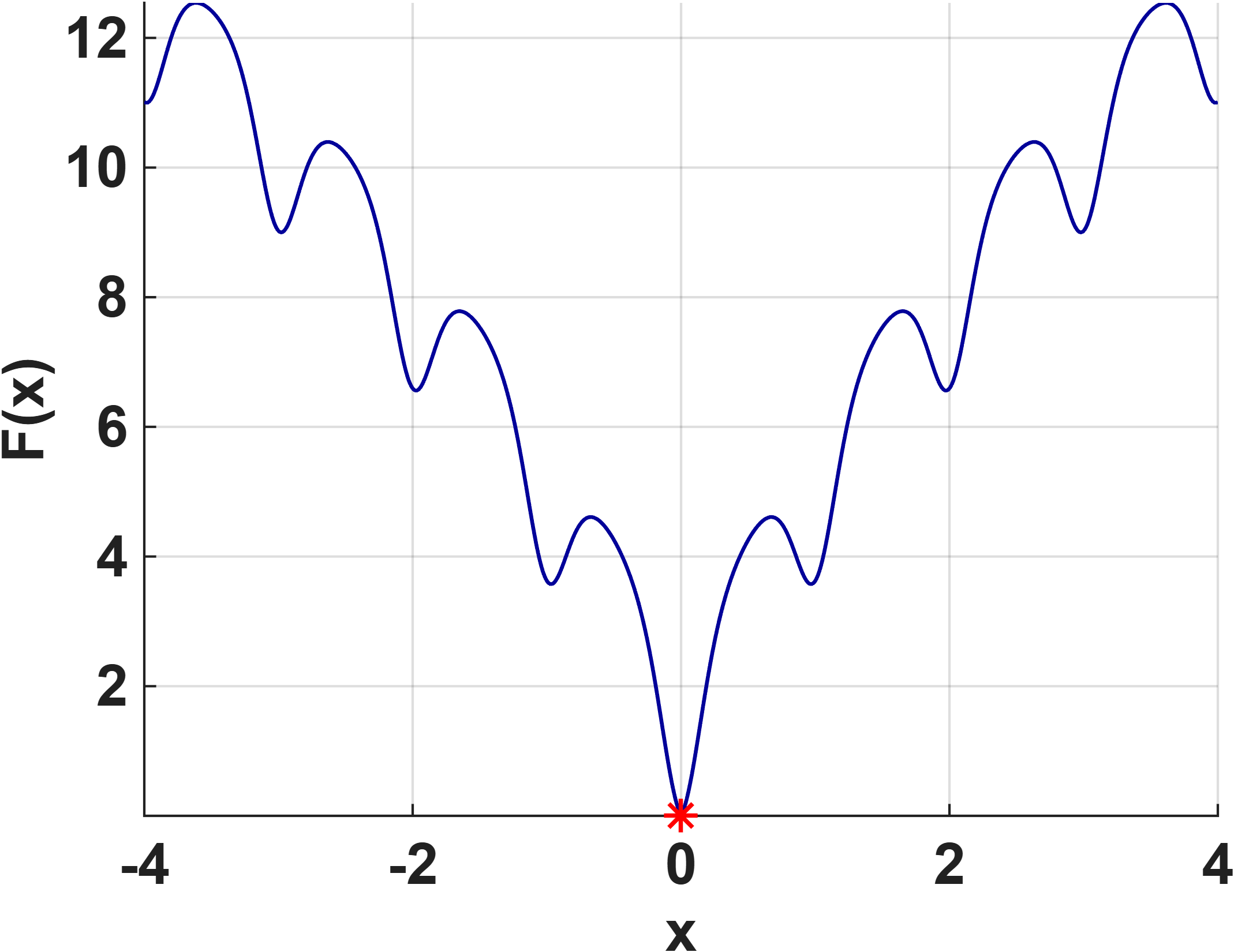}
    \caption{1D Ackley}
    \label{fig:1d_ackley}
  \end{subfigure}\hfill
  \begin{subfigure}[b]{0.45\textwidth}
    \centering
    \includegraphics[width=\textwidth]{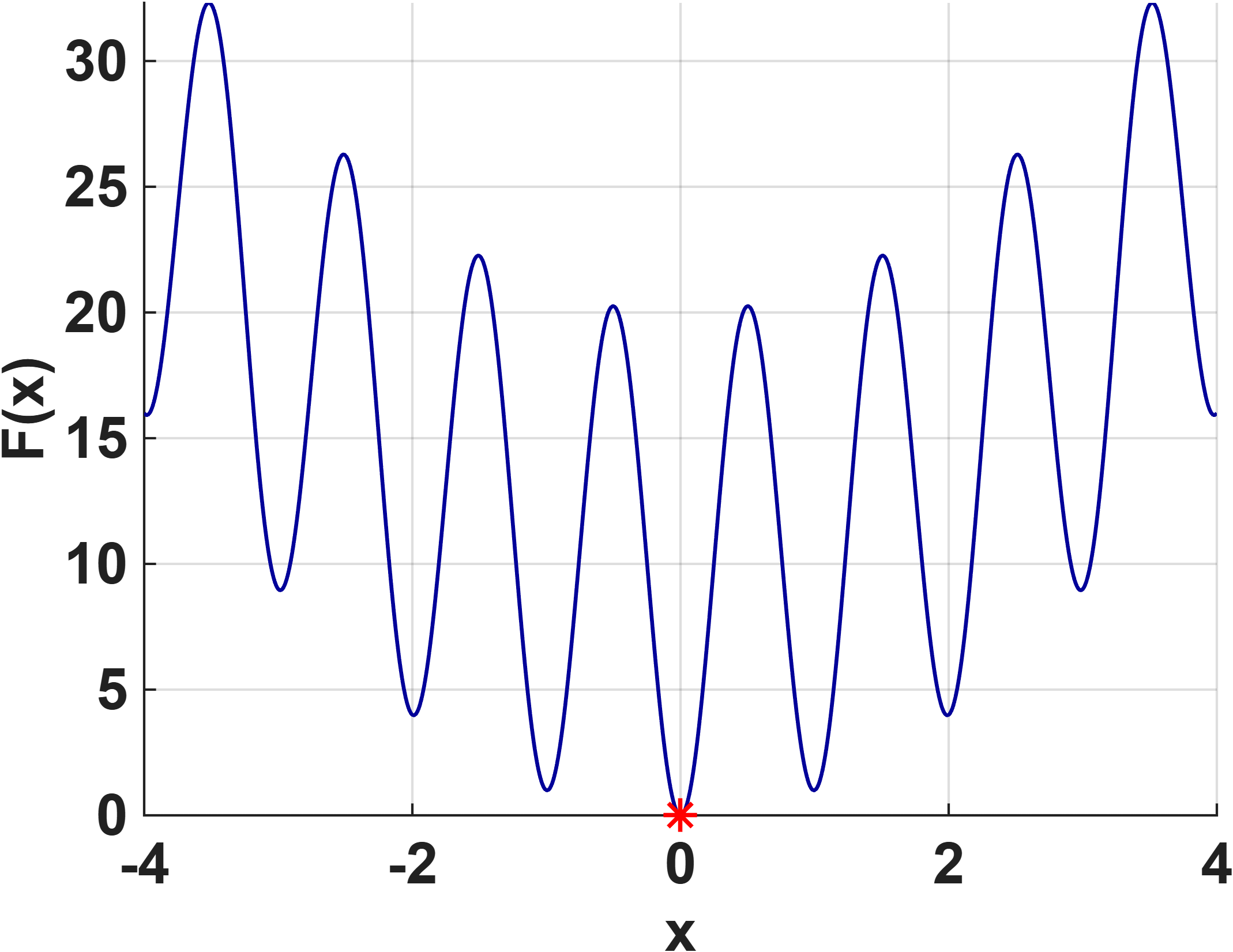}
    \caption{1D Rastrigin}
    \label{fig:1d_rastrigin}
  \end{subfigure}

  \caption{1D test functions ($B=0$).}
  \label{fig:1d_test_functions}
\end{figure}

The Ackley function is smooth and symmetric, with a global minimum at the origin. It has a deep and narrow central valley surrounded by small oscillations caused by the $\cos$ term. These oscillations decrease away from the origin, and the function becomes nearly flat in the outer region. This combination of flat regions and a sharp valley makes the problem challenging for optimization methods.

The Rastrigin function is highly oscillatory over the whole domain due to the $\cos$ term. It has many local minima, and the global minimum is at the origin. Since the oscillations remain strong everywhere, the problem is strongly nonconvex and optimization methods can easily be trapped in local minima. It is therefore widely used to test the stability of optimization algorithms.

\noindent \textbf{Parameters Settings:}

For the swarm-based algorithms, we generate $N=10$ agents randomly with uniform distribution over the search domain, and each agent is assigned an initial mass $m_i = 1/N,\quad i=1,\dots,N$. For the inertial algorithms, we use a time-step size of $\Delta t = 10^{-2}$ for the IMEX-RB method, while a fixed step size $h = 1$ for the FB, Semi and FD algorithms, respectively. These choices ensure numerical stability and computational efficiency in the experiments.

\noindent \textbf{Stopping Criteria:}

In addition to the stopping criteria used in the convex experiments, we introduce one additional stopping condition for the swarm-based methods:
\[\text{number of agents} = 1.\]
This condition forces the swarm-based methods converge to one single agent eventually. The full stopping criteria here is therefore given by
\[\|F(x^{n+1}) - F(x^{n})\| \le 10^{-6}, \quad\|x^{n+1} - x^{n}\| \le 10^{-6},\quad\text{number of agents} = 1.\]

\noindent \textbf{Performance Evaluation:}

To evaluate the performance of the swarm-based methods, we define several quantities that are recorded in the numerical experiments. 

If $\|F(x_\text{output}) - F(x^*)\| \le 10^{-4}$, then we call this is a success run.
We compute the success rate by repeating each experiment $1000$ times using randomly generated initial agents. The success rate is defined as
\[\text{Success rate}=\frac{\text{Number of successful runs}}{1000}.\]

We also record the number of iterations in each run and report the averages over the $1000$ runs, together with the corresponding average CPU time. The success rate, average iteration number, and average CPU time provide a set of metrics for overall performance of the swarm-based methods without being influenced by a few poor initializations.

We use the same parameters as in the convergence rate tests to ensure a fair comparison. This allows us to compare swarm-based methods with different inertial algorithms under the same communication dynamics and to evaluate their performance in high-dimensional cases.

The numerical results consist of two parts. First, 
we provide Tables~\ref{tab:1DAck:part1}--\ref{tab:10DRas:part2}, which report the success rate, the average number of iterations, and the average CPU time for each method. These results demonstrate the efficiency, stability, and computational cost of the swarm-based algorithms.
Second, we present plots of the energy evolution for swarm-based methods with different inertial algorithms. The discrete energy of the $i$-th agent at the $k$-th iteration is defined in Definition~\ref{def:sb:energy} as
\[E_i^k = E(x_i^k)= \frac{1}{2} m(x_i^k, t_k) \|\dot{x}_i^k(t)\|^2+ a_i F^k(x_i^k).\]
Figures~\ref{fig:swar:1DAck}--\ref{fig:swar:10DRas} in Appendix~\ref{Appendix:sb:energy:plot} present the evolution of the total discrete energy $E^k=\sum_i E_i^k$ in the swarm-based framework and illustrate the behavior of the SBIMs from an ODE dynamics perspective.

\subsubsection{1D cases}
\noindent \indent In the 1D experiments, we set the dimension $d = 1$ to clearly show the performance of the SBIMs. For the mass transport parameter $p$ in \eqref{masstrans}, we fix $p = 1$, corresponding to the standard linear interaction.
Although different values of $p$ may influence performance, our main goal is to study the effect of inertial algorithms. The role of $p$ has already been discussed in \cite{lu2024swarm}. Therefore, we keep $p = 1$ in all experiments.

All agents are initialized independently with a uniform distribution on $[B-4,\, B+4]$. This interval is large enough to include local minima and the global minimum for both test functions.
This setting allows us to examine how the SBIMs escape local minima and converge to the global optimum in these nonconvex problems.

The parameters $R_i$ describe the relaxation dynamics in the SBIMs. 
We choose these parameters to match the corresponding inertial algorithms, such as the swarm-based Nesterov method in Algorithm~\ref{sbnm:alg} and the new underdamped inertial algorithms in Algorithms~\ref{sbfd:alg}--\ref{sbsemifb:alg}.
This comparison shows how different inertial algorithms affect the performance of SBIMs in non-convex optimization problems for simplicity.

\paragraph{1D Ackley Function}

 Tables~\ref{tab:1DAck:part1}--\ref{tab:1DAck:part2} report results in terms of the success rate, the average number of iterations, and the average CPU time for SBIM with different inertial algorithms on the 1D Ackley function with various shift parameter values $B$.
First, the success rates show that SB-FB, SB-Semi, and SB-IPAHD perform better than the other methods. The main difference between these new inertial schemes (together with IPAHD) and the classical optimization methods is that the former include Hessian information, while the latter rely only on first-order gradient information. Although the SB-IMEX-RB and SB-FD methods fail to reach the global minimum in several cases, as shown in the discrete energy plots in Figure~\ref{fig:swar:1DAck}, their agents are often trapped in local minima close to the global minimum. In contrast, for SB-NM and SB-GD, the agents tend to be trapped in local minima near the boundary.
These results suggest that including Hessian information into the inertial dynamics improves landscape exploration, helps the agents escape local minima, and increases the likelihood of approaching the global minimum in the non-convex optimization problem.

Second, the average number of iterations and the CPU times show that SB-FB, SB-Semi, and SB-IPAHD reach the global minimum efficiently. They require relatively few iterations and therefore incur low computational cost. Although SB-IMEX-RB and SB-FD do not always reach the global minimum, their iteration counts and CPU times indicate that they remain effective in practice. Their larger average iteration numbers, compared with the classical optimization methods, suggest that they spend more time exploring the search space. Even when they are eventually trapped in local minima, the final agents are often close to the global minimum.

In contrast, SB-NM and SB-GD behave quite differently. They stop much sooner before reaching the global minimum, so the agents have limited ability to explore the landscape. As a result, the methods may converge to local minima that are far from the global minimum. 

\begin{table}[H]
\centering
\fontsize{9}{11}\selectfont
\caption{1D Ackley Function: part 1 - successful methods in reaching the global minimum.}
\label{tab:1DAck:part1}
\begin{adjustbox}{width=0.75\textwidth,totalheight=0.28\textheight,keepaspectratio}
\begin{tabular}{l
                S[table-format=1.0] S[table-format=1.0] S[table-format=1.4]
                S[table-format=1.0] S[table-format=1.0] S[table-format=1.4]
                S[table-format=1.0] S[table-format=1.0] S[table-format=1.4]}
\toprule
Parameter
  & \multicolumn{3}{c}{SB-FB}
  & \multicolumn{3}{c}{SB-Semi}
  & \multicolumn{3}{c}{SB-IPAHD} \\
  & {\shortstack{Success\\Rate}} & {\shortstack{Avg\\Number}} & {\shortstack{Avg\\CPU\\(s)}}
  & {\shortstack{Success\\Rate}} & {\shortstack{Avg\\Number}} & {\shortstack{Avg\\CPU\\(s)}}
  & {\shortstack{Success\\Rate}} & {\shortstack{Avg\\Number}} & {\shortstack{Avg\\CPU\\(s)}} \\
\midrule
B=0   & 1 & 2 & 0.0715 & 1 & 2 & 0.0703 & 1 & 2 & 0.0702 \\
B=15  & 1 & 2 & 0.0738 & 1 & 2 & 0.0744 & 1 & 2 & 0.0754 \\
B=25  & 1 & 2 & 0.0785 & 1 & 2 & 0.0789 & 1 & 2 & 0.0778 \\
\bottomrule
\end{tabular}
\end{adjustbox}
\end{table}

\begin{table}[H]
\centering
\fontsize{9}{11}\selectfont
\caption{1D Ackley Function: part 2 - unsuccessful methods in reaching the global minimum}
\label{tab:1DAck:part2}
\begin{adjustbox}{width=0.98\textwidth,totalheight=0.28\textheight,keepaspectratio}
\begin{tabular}{l
                S[table-format=1.0] S[table-format=6.0] S[table-format=2.4]
                S[table-format=1.0] S[table-format=5.0] S[table-format=1.4]
                S[table-format=1.0] S[table-format=5.3] S[table-format=1.4]
                S[table-format=1.4] S[table-format=3.4] S[table-format=1.4]}
\toprule
Parameter
  & \multicolumn{3}{c}{SB-IMEX-RB}
  & \multicolumn{3}{c}{SB-FD}
  & \multicolumn{3}{c}{SB-NM}
  & \multicolumn{3}{c}{SB-GD} \\
  & {\shortstack{Success\\Rate}} & {\shortstack{Avg\\Number}} & {\shortstack{Avg\\CPU\\(s)}}
  & {\shortstack{Success\\Rate}} & {\shortstack{Avg\\Number}} & {\shortstack{Avg\\CPU\\(s)}}
  & {\shortstack{Success\\Rate}} & {\shortstack{Avg\\Number}} & {\shortstack{Avg\\CPU\\(s)}}
  & {\shortstack{Success\\Rate}} & {\shortstack{Avg\\Number}} & {\shortstack{Avg\\CPU\\(s)}} \\
\midrule
B=0
  & 0 & 1750 & 9.0100
  & 0 & 583  & 0.2479
  & 0 & 13.1150 & 0.0003
  & 0.0060 & 8.2060 & 0.0006 \\

B=15
  & 0 & 1780 & 8.4299
  & 0 & 580  & 0.2188
  & 0 & 13.1150 & 0.0003
  & 0.0070 & 8.2180 & 0.0006 \\

B=25
  & 0 & 1810 & 8.0871
  & 0 & 594  & 0.2312
  & 0 & 13.1690 & 0.0004
  & 0.0080 & 8.2450 & 0.0008 \\

\bottomrule
\end{tabular}
\end{adjustbox}
\end{table}

We show the evolution of the discrete energy $E^k$ for a representative run on the 1D Ackley function (see Appendix, Figure~\ref{fig:swar:1DAck}). From Figure~\ref{fig:swar:1DAck}, we see that when a method converges to the global minimum, the discrete energy decays to zero. 
If the method becomes trapped in a local or boundary minimum, the energy converges to a positive constant or oscillates.

For the new underdamped inertial methods, the discrete energy decreases rapidly at the beginning and remains stable, even when the method is trapped in a local minimum. In these cases, although the energy approaches a positive constant, the function values still remain close to the global minimum.
In contrast, for the SB-NM and SB-GD, the discrete energy shows persistent oscillations when they do not reach the global minimum. 
This indicates weaker stability and less predictability for their large-time behavior.

Overall, the results show that the swarm-based methods derived from the new underdamped inertial dynamics proposed in this study are more stable than those based on classical inertial schemes. The discrete energy $E^k$ provides a useful tool to analyze and compare stability and convergence behavior from an energy viewpoint.

\paragraph{1D Rastrigin Function}

Tables~\ref{tab:1DRas:part1} and~\ref{tab:1DRas:part2} summarize the performance of the swarm-based methods with various inertial algorithms on the 1D Rastrigin function.

From the success rates, SB-FD and SB-IMEX-RB perform better than the others, which is different from the 1D Ackley case. Since the new inertial algorithms include both Hessian and gradient information in the damping term, this turns out to be  particularly useful for the highly oscillatory landscape, and makes global exploration more effective. Therefore, SB-FD and SB-IMEX-RB achieve consistently high success rates for all tested shift parameters.
We also note that SB-FB and SB-Semi always perform poorly, and SB-FD and SB-IMEX-RB are not always the best choices. In higher-dimensional tests, different outcomes show up in some cases.

The discretization of the inertial dynamics also plays an important role. FD and IMEX-RB preserve the energy structure, shown in Theorems~\ref{fully-discretized:energystability} and~\ref{imexrb:energystability}. In contrast, the FB and Semi schemes do not have a theoretical guarantee of energy stability. This may have something to do with their reduced success rates with the shift parameter, even though all methods are derived from the same inertial dynamics. 

Nevertheless, there is no general rule for selecting the best discretization and the resulting algorithm. Because the Rastrigin function is highly oscillatory, its gradient can be large in many regions, and different schemes (explicit or implicit) may behave differently. Implicit schemes may improve stability, while explicit schemes may be more efficient in some cases. This helps explain why methods derived from the same dynamics can perform differently in different dimensions.   Overall, the numerical evidence does how that combining gradient and Hessian information in the damping term improves performance for oscillatory problems.

From the computational perspective, SB-FD and SB-IMEX-RB often require more iterations, but their CPU times remain moderate. Thus, the overall cost is reasonable.
In summary, for the 1D Rastrigin function, swarm-based methods with the new inertial algorithms perform better in both success rate and computational cost.

\begin{table}[H]
\centering
\fontsize{9}{11}\selectfont
\caption{1D Rastrigin Function: part 1 - partially success methods.}
\label{tab:1DRas:part1}
\begin{adjustbox}{width=0.75\textwidth,totalheight=0.28\textheight,keepaspectratio}
\begin{tabular}{l
                S[table-format=1.0] S[table-format=3.4] S[table-format=1.4]
                S[table-format=1.0] S[table-format=3.0] S[table-format=1.4]
                S[table-format=1.0] S[table-format=3.4] S[table-format=1.4]}
\toprule
Parameter
  & \multicolumn{3}{c}{SB-FB}
  & \multicolumn{3}{c}{SB-Semi}
  & \multicolumn{3}{c}{SB-IPAHD} \\
  & {\shortstack{Success\\Rate}}
  & {\shortstack{Avg\\Number}}
  & {\shortstack{Avg\\CPU\\(s)}}
  & {\shortstack{Success\\Rate}}
  & {\shortstack{Avg\\Number}}
  & {\shortstack{Avg\\CPU\\(s)}}
  & {\shortstack{Success\\Rate}}
  & {\shortstack{Avg\\Number}}
  & {\shortstack{Avg\\CPU\\(s)}} \\
\midrule
B=0   & 1 & 23.5040 & 0.1371 & 1 & 138 & 0.7760 & 1 & 23.4910 & 0.1364 \\
B=15  & 0 & 23.4900 & 0.1657 & 0 & 137 & 0.8311 & 0 & 23.5160 & 0.1538 \\
B=25  & 0 & 23.5320 & 0.1544 & 0 & 137 & 0.7878 & 0 & 23.5640 & 0.1481 \\
\bottomrule
\end{tabular}
\end{adjustbox}
\end{table}

\begin{table}[H]
\centering
\fontsize{9}{11}\selectfont
\caption{1D Rastrigin Function: part 2 - partially successful and unsuccessful methods.}
\label{tab:1DRas:part2}
\begin{adjustbox}{width=0.98\textwidth,totalheight=0.28\textheight,keepaspectratio}
\begin{tabular}{l
                S[table-format=1.3] S[table-format=4.0] S[table-format=2.4]
                S[table-format=1.3] S[table-format=3.0] S[table-format=1.4]
                S[table-format=1.0] S[table-format=2.3] S[table-format=1.4]
                S[table-format=1.3] S[table-format=1.4] S[table-format=1.4]}
\toprule
Parameter
  & \multicolumn{3}{c}{SB-IMEX-RB}
  & \multicolumn{3}{c}{SB-FD}
  & \multicolumn{3}{c}{SB-NM}
  & \multicolumn{3}{c}{SB-GD} \\
  & {\shortstack{Success\\Rate}}
  & {\shortstack{Avg\\Number}}
  & {\shortstack{Avg\\CPU\\(s)}}
  & {\shortstack{Success\\Rate}}
  & {\shortstack{Avg\\Number}}
  & {\shortstack{Avg\\CPU\\(s)}}
  & {\shortstack{Success\\Rate}}
  & {\shortstack{Avg\\Number}}
  & {\shortstack{Avg\\CPU\\(s)}}
  & {\shortstack{Success\\Rate}}
  & {\shortstack{Avg\\Number}}
  & {\shortstack{Avg\\CPU\\(s)}} \\
\midrule
B=0
  & 1.000 & 2000 & 8.0200
  & 0.999 & 220 & 0.0648
  & 0 & 10.1590 & 0.0002
  & 0.135 & 4.3920 & 0.0003 \\

B=15
  & 0.999 & 2000 & 8.3840
  & 0.998 & 230 & 0.0859
  & 0 & 10.3410 & 0.0002
  & 0.118 & 4.5320 & 0.0004 \\

B=25
  & 0.996 & 2000 & 8.1469
  & 0.996 & 227 & 0.0817
  & 0 & 9.9830 & 0.0001
  & 0.143 & 4.3600 & 0.0003 \\

\bottomrule
\end{tabular}
\end{adjustbox}
\end{table}

Similar to the 1D Ackley case, we present a representative run of the energy evolution plots in Figure~\ref{fig:swar:1DRas}.
For the 1D Rastrigin function, all swarm-based new inertial algorithms reach the global minimum, and the discrete energy converges to $0$. As in the 1D Ackley case, the energy decreases rapidly during the first few iterations when using these new inertial algorithms. This behavior is consistent with our previous observations and demonstrates better stability from the energy perspective.

In contrast, SB-NM and SB-GD do not reach the global minimum. Their discrete energy oscillates and increases during iterations. Compared with the new inertial algorithms, these classical methods are less stable within the swarm-based framework. As a result, their large-time behavior is more difficult to predict and control, especially for non-convex optimization problems.
Overall, the results for the 1D Rastrigin function lead to the same conclusions regarding stability and convergence as those for the 1D Ackley case.

\subsubsection{High-dimension cases}
\noindent \indent In the higher-dimensional experiments, we test the Ackley and Rastrigin functions with dimensions $d = 2$ and $d = 10$, respectively. Although these functions are difficult to visualize in higher dimensions, their landscapes and oscillatory behavior are similar to the 1D case. Therefore, we focus on the performance of the optimization methods.

As in the 1D case, the search domain is set to $[B-4,\, B+4]^d$. This domain is large enough to include many local minima and the unique global minimum of both functions. By fixing the search domain, we can focus on the effect of the inertial algorithms rather than on parameter values. Then the results illustrate how inertial algorithms influence energy convergence and the final positions of the agents in higher-dimensional non-convex optimization problems.

In all experiments, $N=10$ agents are initialized independently with a uniform distribution over the domain. The mass transport parameter is fixed at $p = 1$, as in the 1D experiments. The inertial parameter $R_i$ is chosen according to the corresponding SBIMs and is consistent with the 1D settings. Using the same parameter values allows us to study how the swarm-based methods perform as the dimension increases.


\paragraph{Higher Dimension Ackley Function Case}

The Tables~\ref{tab:2DAck:part1}-~\ref{tab:10DAck:part2} report the success rates, average iteration numbers, and average CPU times of the swarm-based methods for the 2D and 10D Ackley functions. From the success rates, we see that SB-FB, SB-Semi, and SB-IPAHD perform best under all shift parameters. In the 2D case, SB-IMEX-RB also shows competitive performance.

These results suggest that including Hessian information in damping improves the landscape exploration ability of the swarm-based methods. This helps the agents navigate through relatively flat local minima in the Ackley function. This result is consistent with the 1D Ackley case. The different discretization of the inertial dynamics plays an important role here  since different new inertial algorithms lead to different success rates.  

In terms of the computational cost, SB-FD requires more iterations on average than SB-NM and SB-GD do. This suggests that SB-FD spends more time exploring the search space. In contrast, SB-NM and SB-GD rely only on first-order information. They show some difficulties in choosing a suitable step size at each iteration in high-dimensional cases. 
Since the gradient has many components, its norm can be large, which requires a very small step size to maintain stability, as explained in detail in the convergence test section.
As a result, the methods converge slowly and are more likely to be trapped in local minima, limiting their ability to explore the search space globally.

These results explain why swarm-based methods with the new inertial algorithms generally perform better for non-convex optimization problems. This conclusion is consistent with the 1D experiments.

\begin{table}[H]
\centering
\fontsize{9}{11}\selectfont
\caption{2D Ackley Function: part 1-successful methods.}
\label{tab:2DAck:part1}
\begin{adjustbox}{width=0.75\textwidth,totalheight=0.28\textheight,keepaspectratio}
\begin{tabular}{l
                S[table-format=1.0] S[table-format=1.3] S[table-format=1.4]
                S[table-format=1.0] S[table-format=1.3] S[table-format=1.4]
                S[table-format=1.0] S[table-format=1.3] S[table-format=1.4]}
\toprule
Parameter
  & \multicolumn{3}{c}{SB-FB}
  & \multicolumn{3}{c}{SB-Semi}
  & \multicolumn{3}{c}{SB-IPAHD} \\
  & {\shortstack{Success\\Rate}}
  & {\shortstack{Avg\\Number}}
  & {\shortstack{Avg\\CPU\\(s)}}
  & {\shortstack{Success\\Rate}}
  & {\shortstack{Avg\\Number}}
  & {\shortstack{Avg\\CPU\\(s)}}
  & {\shortstack{Success\\Rate}}
  & {\shortstack{Avg\\Number}}
  & {\shortstack{Avg\\CPU\\(s)}} \\
\midrule
B=0   & 1 & 2      & 0.0707 & 1 & 2      & 0.0686 & 1 & 2      & 0.0681 \\
B=15  & 1 & 2      & 0.1759 & 1 & 2      & 0.1724 & 1 & 2      & 0.1651 \\
B=25  & 1 & 2.2050 & 0.1566 & 1 & 2.2110 & 0.1650 & 1 & 2.1910 & 0.1545 \\
\bottomrule
\end{tabular}
\end{adjustbox}
\end{table}
\begin{table}[H]
\centering
\fontsize{9}{11}\selectfont
\caption{2D Ackley Function: part 2- partially successful methods.}
\label{tab:2DAck:part2}
\begin{adjustbox}{width=0.98\textwidth,totalheight=0.28\textheight,keepaspectratio}
\begin{tabular}{l
                S[table-format=1.3] S[table-format=4.0] S[table-format=3.4]
                S[table-format=1.0] S[table-format=4.0] S[table-format=1.4]
                S[table-format=1.0] S[table-format=2.3] S[table-format=1.4]
                S[table-format=1.0] S[table-format=1.3] S[table-format=1.4]}
\toprule
Parameter
  & \multicolumn{3}{c}{SB-IMEX-RB}
  & \multicolumn{3}{c}{SB-FD}
  & \multicolumn{3}{c}{SB-NM}
  & \multicolumn{3}{c}{SB-GD} \\
  & {\shortstack{Success\\Rate}}
  & {\shortstack{Avg\\Number}}
  & {\shortstack{Avg\\CPU\\(s)}}
  & {\shortstack{Success\\Rate}}
  & {\shortstack{Avg\\Number}}
  & {\shortstack{Avg\\CPU\\(s)}}
  & {\shortstack{Success\\Rate}}
  & {\shortstack{Avg\\Number}}
  & {\shortstack{Avg\\CPU\\(s)}}
  & {\shortstack{Success\\Rate}}
  & {\shortstack{Avg\\Number}}
  & {\shortstack{Avg\\CPU\\(s)}} \\
\midrule
B=0
  & 0.000 & 4070 & 186.0000
  & 0     & 974  & 0.6056
  & 0     & 25.8220 & 0.0006
  & 0     & 8.8930  & 0.0013 \\

B=15
  & 0.300 & 5140 & 22.9461
  & 0     & 713  & 0.3985
  & 0     & 15.9900 & 0.0023
  & 0     & 8.9100  & 0.0015 \\

B=25
  & 0.221 & 6710 & 15.7393
  & 0     & 681  & 0.3196
  & 0     & 16.6830 & 0.0005
  & 0     & 8.8870  & 0.0008 \\

\bottomrule
\end{tabular}
\end{adjustbox}
\end{table}

\begin{table}[H]
\centering
\fontsize{9}{11}\selectfont
\caption{10D Ackley Function: part 1- successful methods. }
\label{tab:10DAck:part1}
\begin{adjustbox}{width=0.75\textwidth,totalheight=0.28\textheight,keepaspectratio}
\begin{tabular}{l
                S[table-format=1.0] S[table-format=1.3] S[table-format=1.4]
                S[table-format=1.0] S[table-format=1.3] S[table-format=1.4]
                S[table-format=1.0] S[table-format=1.3] S[table-format=1.4]}
\toprule
Parameter
  & \multicolumn{3}{c}{SB-FB}
  & \multicolumn{3}{c}{SB-Semi}
  & \multicolumn{3}{c}{SB-IPAHD} \\
  & {\shortstack{Success\\Rate}}
  & {\shortstack{Avg\\Number}}
  & {\shortstack{Avg\\CPU\\(s)}}
  & {\shortstack{Success\\Rate}}
  & {\shortstack{Avg\\Number}}
  & {\shortstack{Avg\\CPU\\(s)}}
  & {\shortstack{Success\\Rate}}
  & {\shortstack{Avg\\Number}}
  & {\shortstack{Avg\\CPU\\(s)}} \\
\midrule
B=0   & 1 & 2.0000 & 0.0707 & 1 & 2.0000 & 0.0686 & 1 & 2.0000 & 0.0681 \\
B=15  & 1 & 2.5040 & 0.3878 & 1 & 2.4690 & 0.4032 & 1 & 2.4830 & 0.3672 \\
B=25  & 1 & 3.1650 & 0.3656 & 1 & 3.1400 & 0.3813 & 1 & 3.1410 & 0.3478 \\
\bottomrule
\end{tabular}
\end{adjustbox}
\end{table}
\begin{table}[H]
\centering
\fontsize{9}{11}\selectfont
\caption{10D Ackley Function: part 2- unsuccessful methods.}
\label{tab:10DAck:part2}
\begin{adjustbox}{width=0.98\textwidth,totalheight=0.28\textheight,keepaspectratio}
\begin{tabular}{l
                S[table-format=1.0] S[table-format=4.0] S[table-format=3.4]
                S[table-format=1.0] S[table-format=4.0] S[table-format=1.4]
                S[table-format=1.0] S[table-format=2.3] S[table-format=1.4]
                S[table-format=1.0] S[table-format=1.3] S[table-format=1.4]}
\toprule
Parameter
  & \multicolumn{3}{c}{SB-IMEX-RB}
  & \multicolumn{3}{c}{SB-FD}
  & \multicolumn{3}{c}{SB-NM}
  & \multicolumn{3}{c}{SB-GD} \\
  & {\shortstack{Success\\Rate}}
  & {\shortstack{Avg\\Number}}
  & {\shortstack{Avg\\CPU\\(s)}}
  & {\shortstack{Success\\Rate}}
  & {\shortstack{Avg\\Number}}
  & {\shortstack{Avg\\CPU\\(s)}}
  & {\shortstack{Success\\Rate}}
  & {\shortstack{Avg\\Number}}
  & {\shortstack{Avg\\CPU\\(s)}}
  & {\shortstack{Success\\Rate}}
  & {\shortstack{Avg\\Number}}
  & {\shortstack{Avg\\CPU\\(s)}} \\
\midrule
B=0
  & 0 & 4070 & 186.0000
  & 0 & 974  & 0.6056
  & 0 & 25.8220 & 0.0006
  & 0 & 8.8930  & 0.0013 \\

B=15
  & 0 & 5300 & 216.0000
  & 0 & 980  & 0.6162
  & 0 & 25.5990 & 0.0007
  & 0 & 8.9200  & 0.0011 \\

B=25
  & 0 & 5430 & 221.0000
  & 0 & 981  & 0.5884
  & 0 & 25.7880 & 0.0007
  & 0 & 8.9170  & 0.0011 \\

\bottomrule
\end{tabular}
\end{adjustbox}
\end{table}

We present a representative run of the discrete energy evolution plots for the 2D Ackley function in Figure~\ref{fig:swar:2DAck} and for the 10D Ackley function in Figure~\ref{fig:swar:10DAck}.
For both higher-dimensional cases, the energy of the swarm-based new inertial dynamics algorithms decreases rapidly at the beginning and then converges. For methods that reach the global minimum, such as SB-FB, SB-Semi, and SB-IPAHD, the energy converges to zero. For methods that are trapped in a local minimum, such as SB-IMEX-RB and SB-FD, the energy converges to a positive value. This behavior is consistent with the 1D Ackley case.

In contrast, the SB-NM and SB-GD show oscillatory energy evolution. The energy increases during some iterations, and the overall trend does not show convergence. As in the 1D case, this makes the final agent positions difficult to predict.

\paragraph{Higher Dimension Rastrigin Function Case}

Tables~\ref{tab:2DRas:part1}-~\ref{tab:10DRas:part2} report the success rates, average iteration numbers, and average CPU times of the swarm-based methods for the 2D and 10D Rastrigin functions with different shift parameters.

In the 2D case, SB-IMEX-RB and SB-FD show relatively high success rates for all shift parameter values. In contrast, SB-FB, SB-Semi, and SB-IPAHD have a perfect success rate when $B=0$, but their success rates decrease to $25\%$ when $B=15$ and $B=25$. In the 10D case, SB-FB, SB-Semi, and SB-IPAHD perform better than SB IMEX-RB and SB-FD in the success rate. The SB-GD method performs moderately in 2D case, while the SB-NM method consistently fails to reach the global minimum.

These results indicate that including the Hessian and interaction terms into the inertial dynamics improves the ability of the agents to escape local minima in highly oscillatory problems such as the Rastrigin function. In contrast, SB-GD uses only first-order gradient  information, which limits its performance on non-convex landscapes. The consistent failure of SB-NM shows that adding momentum alone is not sufficient to navigate through highly oscillatory landscapes to search for global minima. The new inertial algorithms are more effective because they combine gradient and curvature information within the damping term, leading to better performance in the non-convex optimization.


In the computational cost, SB-IMEX-RB requires the largest number of iterations and the highest CPU time. However, the average CPU time is about 11 seconds, which is acceptable. SB-FD has similar success rates with roughly $1/10$ of the iterations of SB-IMEX-RB, leading to a much lower computational cost. SB-FB, SB-Semi, and SB-IPAHD also show differences in CPU time. These results show once again that the specific discretization of the new ODE dynamics has a significant impact on the computational cost. Overall, the new inertial algorithms show strong performance on the minimization of the Rastrigin function, consistent with the 1D experiments.

\begin{table}[H]
\centering
\fontsize{9}{11}\selectfont
\caption{2D Rastrigin Function: part 1}
\label{tab:2DRas:part1}
\begin{adjustbox}{width=0.75\textwidth,totalheight=0.28\textheight,keepaspectratio}
\begin{tabular}{l
                S[table-format=1.3] S[table-format=3.4] S[table-format=2.4]
                S[table-format=1.3] S[table-format=3.0] S[table-format=2.4]
                S[table-format=1.3] S[table-format=3.4] S[table-format=2.4]}
\toprule
Parameter
  & \multicolumn{3}{c}{SB-FB}
  & \multicolumn{3}{c}{SB-Semi}
  & \multicolumn{3}{c}{SB-IPAHD} \\
  & {\shortstack{Success\\Rate}}
  & {\shortstack{Avg\\Number}}
  & {\shortstack{Avg\\CPU\\(s)}}
  & {\shortstack{Success\\Rate}}
  & {\shortstack{Avg\\Number}}
  & {\shortstack{Avg\\CPU\\(s)}}
  & {\shortstack{Success\\Rate}}
  & {\shortstack{Avg\\Number}}
  & {\shortstack{Avg\\CPU\\(s)}} \\
\midrule
B=0   & 1.000 & 30.2250 & 0.3523 & 1.000 & 184 & 1.5815 & 1.000 & 30.2470 & 0.2816 \\
B=15  & 0.214 & 30.4280 & 0.3347 & 0.271 & 185 & 1.4031 & 0.271 & 30.4260 & 0.2672 \\
B=25  & 0.215 & 30.2560 & 0.3317 & 0.283 & 183 & 1.3733 & 0.283 & 30.2660 & 0.2647 \\
\bottomrule
\end{tabular}
\end{adjustbox}
\end{table}

\begin{table}[H]
\centering
\fontsize{9}{11}\selectfont
\caption{2D Rastrigin Function: part 2}
\label{tab:2DRas:part2}
\begin{adjustbox}{width=0.98\textwidth,totalheight=0.28\textheight,keepaspectratio}
\begin{tabular}{l
                S[table-format=1.3] S[table-format=4.0] S[table-format=2.4]
                S[table-format=1.3] S[table-format=3.0] S[table-format=2.4]
                S[table-format=1.0] S[table-format=4.0] S[table-format=2.4]
                S[table-format=1.3] S[table-format=1.4] S[table-format=1.4]}
\toprule
Parameter
  & \multicolumn{3}{c}{SB-IMEX-RB}
  & \multicolumn{3}{c}{SB-FD}
  & \multicolumn{3}{c}{SB-NM}
  & \multicolumn{3}{c}{SB-GD} \\
  & {\shortstack{Success\\Rate}}
  & {\shortstack{Avg\\Number}}
  & {\shortstack{Avg\\CPU\\(s)}}
  & {\shortstack{Success\\Rate}}
  & {\shortstack{Avg\\Number}}
  & {\shortstack{Avg\\CPU\\(s)}}
  & {\shortstack{Success\\Rate}}
  & {\shortstack{Avg\\Number}}
  & {\shortstack{Avg\\CPU\\(s)}}
  & {\shortstack{Success\\Rate}}
  & {\shortstack{Avg\\Number}}
  & {\shortstack{Avg\\CPU\\(s)}} \\
\midrule
B=0
  & 0.943 & 2200 & 11.1000
  & 0.940 & 369  & 0.1266
  & 0     & 3550 & 0.0380
  & 0.419 & 6.0950 & 0.0006 \\

B=15
  & 0.947 & 2190 & 10.5000
  & 0.944 & 373  & 0.1501
  & 0     & 5000 & 0.0552
  & 0.403 & 6.2790 & 0.0007 \\

B=25
  & 0.946 & 2200 & 10.3000
  & 0.951 & 372  & 0.1509
  & 0     & 4630 & 0.0521
  & 0.430 & 6.0480 & 0.0007 \\
\bottomrule
\end{tabular}
\end{adjustbox}
\end{table}

\begin{table}[H]
\centering
\fontsize{9}{11}\selectfont
\caption{10D Rastrigin Function: part 1}
\label{tab:10DRas:part1}
\begin{adjustbox}{width=0.75\textwidth,totalheight=0.28\textheight,keepaspectratio}
\begin{tabular}{l
                S[table-format=1.3] S[table-format=3.4] S[table-format=2.4]
                S[table-format=1.3] S[table-format=3.0] S[table-format=2.4]
                S[table-format=1.3] S[table-format=3.4] S[table-format=2.4]}
\toprule
Parameter
  & \multicolumn{3}{c}{SB-FB}
  & \multicolumn{3}{c}{SB-Semi}
  & \multicolumn{3}{c}{SB-IPAHD} \\
  & {\shortstack{Success\\Rate}}
  & {\shortstack{Avg\\Number}}
  & {\shortstack{Avg\\CPU\\(s)}}
  & {\shortstack{Success\\Rate}}
  & {\shortstack{Avg\\Number}}
  & {\shortstack{Avg\\CPU\\(s)}}
  & {\shortstack{Success\\Rate}}
  & {\shortstack{Avg\\Number}}
  & {\shortstack{Avg\\CPU\\(s)}} \\
\midrule
B=0   & 0.971 & 46.3830 & 1.2927 & 0.642 & 289 & 4.7608 & 0.642 & 46.4370 & 1.0786 \\
B=15  & 0.744 & 48.3930 & 1.4658 & 0.620 & 290 & 4.6385 & 0.620 & 48.3500 & 1.0673 \\
B=25  & 0.752 & 48.2340 & 1.4405 & 0.620 & 282 & 4.3262 & 0.620 & 48.1940 & 1.0346 \\
\bottomrule
\end{tabular}
\end{adjustbox}
\end{table}

\begin{table}[H]
\centering
\fontsize{9}{11}\selectfont
\caption{10D Rastrigin Function: part 2}
\label{tab:10DRas:part2}
\begin{adjustbox}{width=0.98\textwidth,totalheight=0.28\textheight,keepaspectratio}
\begin{tabular}{l
                S[table-format=1.3] S[table-format=4.0] S[table-format=3.4]
                S[table-format=1.3] S[table-format=3.0] S[table-format=2.4]
                S[table-format=1.0] S[table-format=5.0] S[table-format=2.4]
                S[table-format=1.3] S[table-format=1.3] S[table-format=1.4]}
\toprule
Parameter
  & \multicolumn{3}{c}{SB-IMEX-RB}
  & \multicolumn{3}{c}{SB-FD}
  & \multicolumn{3}{c}{SB-NM}
  & \multicolumn{3}{c}{SB-GD} \\
  & {\shortstack{Success\\Rate}}
  & {\shortstack{Avg\\Number}}
  & {\shortstack{Avg\\CPU\\(s)}}
  & {\shortstack{Success\\Rate}}
  & {\shortstack{Avg\\Number}}
  & {\shortstack{Avg\\CPU\\(s)}}
  & {\shortstack{Success\\Rate}}
  & {\shortstack{Avg\\Number}}
  & {\shortstack{Avg\\CPU\\(s)}}
  & {\shortstack{Success\\Rate}}
  & {\shortstack{Avg\\Number}}
  & {\shortstack{Avg\\CPU\\(s)}} \\
\midrule
B=0
  & 0.009 & 2930 & 182.0000
  & 0.004 & 582  & 0.2802
  & 0     & 8450 & 0.0928
  & 0.003 & 8.9320 & 0.0008 \\

B=15
  & 0.008 & 3030 & 157.0000
  & 0.006 & 582  & 0.3318
  & 0     & 9880 & 0.1320
  & 0.005 & 8.9300 & 0.0013 \\

B=25
  & 0.005 & 3100 & 158.0000
  & 0.004 & 583  & 0.3331
  & 0     & 9880 & 0.1330
  & 0.005 & 8.9420 & 0.0011 \\
\bottomrule
\end{tabular}
\end{adjustbox}
\end{table}

We present a representative run of the discrete energy evolution for the 2D Rastrigin function in Figure~\ref{fig:swar:2DRas} and for the 10D Rastrigin function in Figure~\ref{fig:swar:10DRas}.
For the 2D Rastrigin function, all new inertial algorithms reach the global minimum. In the 10D case, however, SB-IMEX-RB and SB-FD fail to do so. With all new inertial algorithms, the discrete energy decreases rapidly at the beginning and then converges. This behavior is consistent with the results from the 1D Rastrigin and Ackley cases.

In contrast, the SB-NM and SB-GD show energy oscillations during iterations. Although SB-GD reaches the global minimum in the 2D Rastrigin case, its oscillatory energy makes it difficult to predict whether the energy converges to $0$ or to a small positive constant in other cases, such as the 10D Rastrigin problem. 
Overall, these two classical methods lack certain stability from the perspective of the energy evolution. As in the 1D case, this makes the final positions of the agents difficult to predict.

\section{Conclusion}\label{section:conclude}

\noindent\indent In this work, we develop a framework for SBIMs by formulating the multi-agent optimization as a class of coupled dissipative inertial dynamical systems derived from the generalized Onsager principle. Within this framework, the friction operator $R$ and the scaling operator of the potential energy $a$ act as control parameters that determine the dissipation mechanism and the relaxation behavior over the landscape of the objective function. This approach provides a systematic way to construct inertial swarm dynamics beyond the standard swarm-based gradient descent.
Based on this framework, we propose underdamped inertial model~\eqref{ODE}, whose dissipation is influenced by both gradient and Hessian information. This yields a deceleration and acceleration mechanism biased on the direction of the agent trajectory beyond the standard swarm dynamics. Using a energy approach, we establish the energy dissipation property and objective function decay estimate of order $\bigO(1/\delta(t))$ for the continuous system. We then construct structure-preserving discretizations and prove analogous discrete function decay estimate of order $\bigO(1/\delta_k)$.

The numerical results in Section~\ref{numerical:exp} confirm the theoretical estimates. In convex test problems, the proposed inertial schemes exhibit energy stable behavior and better than expected convergence rates. In particular, the observed objective function decay rate is faster than the theoretical guarantee. 
For nonconvex test functions, including the Ackley and Rastrigin families, the proposed schemes~\ref{sbfd:alg},~\ref{sbimex:alg},~\ref{sbsemifb:alg} achieve high success rates and obtain more stable energy than the standard methods. These results suggest that combining swarm interaction with dissipation informed by momentum, gradients, and curvature information can improve robustness and range of exploration in challenging landscapes of objective functions.

Several directions remain for future work on this new framework. First, it would be valuable to extend the analysis beyond the current convex and smooth settings to broader nonconvex problems. Second, adaptive choices of the damping and control parameters may further improve robustness across challenging landscapes of objective functions.
Third, although our numerical schemes avoid explicit computing the Hessian by using gradient differences, they may still be computationally expensive for large-scale problems. Fourth, there are rooms for improvements in how to design more efficient implementations and reduce computational cost while preserving the dissipative structure. Fifth, it would be interesting to consider a stochastically forced version of \eqref{ODE}, where one could leverage the vanishing noise convergence results established in \cite{MR4408018}.
Finally, a deeper understanding of the interaction between communication dynamics and inertial motion may lead to improved mass-transfer rules and stronger convergence guarantees for the full swarm system, especially in high-dimensional settings.

\section*{Acknowledgments}

\noindent \indent Qi Wang and Qiyu Wu's research is partially supported by NSF award OIA-2242812 and DOE award DE-SC0025229. Qi Wang's research is also partially
supported by NSF award DMS-2038080.

\section*{Statements and Declarations}
\subsection*{Conflicts of interest} 
\noindent \indent The authors have no relevant financial or non-financial interests to disclose.
\bibliographystyle{unsrtnat}

\bibliography{references.bib}

\begin{appendices}
\section{Proof of the SBIM Nesterov Scheme}\label{proof:NM}
{\bf Proof of Theorem~\ref{thm:Nes}}
\begin{proof}
	By the definition of discrete energy, we have
	\begin{align*}
		E_i^{n+1}-E_i^{n}
		&=a_i^{n}(F_i^{n+1}-F_i^{n})
		+\frac{1}{2}m_i^{n+1}\|y_i^{n+1}\|^2
		-\frac{1}{2}m_i^{n}\|y_i^{n}\|^2 .
	\end{align*}
	
	For the first term $a_i^{n}(F_i^{n+1}-F_i^{n})$, since $F\in C^1$,
	\begin{align*}
		F_i^{n+1}-F_i^{n}
		&=\langle \nabla F(x_i^{n}),x_i^{n+1}-x_i^{n}\rangle
		+\mathcal O(\|x_i^{n+1}-x_i^{n}\|^2) \\
		&=\Delta t\langle \nabla F(x_i^{n}),y_i^{n+1}\rangle
		+\mathcal O(\Delta t^2).
	\end{align*}
	Since $y_i^{n+1}=y_i^{n}+\mathcal O(\Delta t)$, this gives
	\[
	a_i^{n}(F_i^{n+1}-F_i^{n})
	= a_i^{n}\Delta t\langle \nabla F(x_i^{n}),y_i^{n}\rangle
	+\mathcal O(\Delta t^2).
	\]
	
	For the kinetic energy part, from the update rule,
	\[
	y_i^{n+1}
	= y_i^{n}
	+\Delta t\Big[(-R_i^{n}+\tfrac12\phi_p(x_i^{n}))y_i^{n}
	-\frac{a_i^{n}}{m_i^{n}}\nabla F(x_i^{n})\Big],
	\]
	we obtain
	\begin{align*}
		\|y_i^{n+1}\|^2
		&=\|y_i^{n}\|^2
		+2\Delta t\Big\langle y_i^{n},
		(-R_i^{n}+\tfrac12\phi_p)y_i^{n}
		-\frac{a_i^{n}}{m_i^{n}}\nabla F(x_i^{n})
		\Big\rangle
		+\mathcal O(\Delta t^2).
	\end{align*}
	
	Since $m_i^{n+1}=m_i^{n}(1-\Delta t\phi_p(x_i^{n}))$ for all $i\neq i_n^-$, we have
	\begin{align*}
		\frac12 m_i^{n+1}\|y_i^{n+1}\|^2
		-\frac12 m_i^{n}\|y_i^{n}\|^2
		&= m_i^{n}\Delta t
		\Big\langle y_i^{n},
		(-R_i^{n}+\tfrac12\phi_p)y_i^{n}
		-\frac{a_i^{n}}{m_i^{n}}\nabla F(x_i^{n})
		\Big\rangle \\
		&\quad
		-\frac12\Delta t\phi_p(x_i^{n})\,m_i^{n}\|y_i^{n}\|^2
		+\mathcal O(\Delta t^2) \\
		&=
		- m_i^{n}\Delta t R_i^{n}\|y_i^{n}\|^2
		- a_i^{n}\Delta t\langle y_i^{n},\nabla F(x_i^{n})\rangle
		+\mathcal O(\Delta t^2).
	\end{align*}
	
	Combining these two parts, we have:
	\[
	E_i^{n+1}-E_i^{n}
	= -\,m_i^{n}\Delta t R_i^{n}\|y_i^{n}\|^2
	+\mathcal O(\Delta t^2).
	\]
	
	By the definition of $R_i^{n}$ and $a_i^{n}$ in \eqref{riai:Nes}, and since
	$a_i^{n}=m_i^{n}>0$ and $R_i^{n}>0$, neglecting higher-order terms
	$\mathcal O(\Delta t^2)$, we obtain
	\[
	E_i^{n+1}-E_i^{n}
	= -\,m_i^{n}\Delta t R_i^{n}\|y_i^{n}\|^2 \le 0,
	\]
	for all $i\neq i_n^-$.
\end{proof}

\section{Proofs of the Lemmas and Theorems for the IMEX-RB Algorithm}\label{proof:imexrb}

{\bf Proof of Theorem~\ref{imexrb:thm}}
\begin{proof} 
\noindent\indent
	From Theorem~\ref{continuous:thm}, the continuous solution $\by(t)=(\bx(t),\bV(t))^\top$ converges to $(\bx^*,\vec{0})$. Since $F\in C^2$ and $\nabla^2F$ is Lipschitz continuous, and $g(t,\by)\in C^1(I) $ is with respect to $\by$, then $\by\in C^2(I)$.
	
	We now verify that $g$ is Lipschitz continuous with respect to $\by$. Denote $\by_i=(\bx_i,\bV_i)$ for $i=1,2$. The first component of $g$ satisfies
	\[\|\bV_1-\bV_2\| \le \|\by_1-\by_2\|.\]
	
	For the second component, define
	\[s(t,\bx,\bV)= -\frac{\alpha}{t}\bV+ \frac{\alpha}{t}\langle\nabla F(\bx),\bV\rangle\mathbf{1}- \gamma\nabla^2F(\bx)\bV- \beta\nabla F(\bx).\]
	Then
	\[\mathrm{Term\ II} := s(t,\bx_1,\bV_1)-s(t,\bx_2,\bV_2).\]
	
	By adding and subtracting intermediate terms, we obtain
	\begin{align*}
		\mathrm{Term\ II}&= -\frac{\alpha}{t}(\bV_1-\bV_2)+ \frac{\alpha}{t}\bigl(\langle\nabla F(\bx_1),\bV_1-\bV_2\rangle+ \langle\nabla F(\bx_1)-\nabla F(\bx_2),\bV_2\rangle\bigr)\mathbf{1} \\
		&\quad -\gamma\nabla^2F(\bx_1)(\bV_1-\bV_2)-\gamma(\nabla^2F(\bx_1)-\nabla^2F(\bx_2))\bV_2-\beta(\nabla F(\bx_1)-\nabla F(\bx_2)).
	\end{align*}
	
	Since $\nabla F$ and $\nabla^2F$ are Lipschitz continuous and bounded on boundedsets, there exist constants $L_1,L_2>0$ such that
	\[\|\nabla F(\bx_1)-\nabla F(\bx_2)\|\le L_1\|\bx_1-\bx_2\|,\]
	\[\|\nabla^2F(\bx_1)-\nabla^2F(\bx_2)\|\le L_2\|\bx_1-\bx_2\|.\]
	Because $\bV$ is bounded on $I$, there exists $M>0$ such that $\|\bV\|_{L^\infty}\le M$.
	Therefore,
	\[\|\mathrm{Term\ II}\|\le C_1\|\bV_1-\bV_2\| + C_2\|\bx_1-\bx_2\|\le C\|\by_1-\by_2\|,\]
	for some constant $C>0$.
	
	Hence $g(t,\by)$ is Lipschitz continuous with respect to $\by$ for each fixed $t$.All assumptions of Theorem~\ref{imex:convergence:thm} in\cite{bassanini2025imex} are satisfied. Therefore, the sequence $\{\bu^{n}\}$ generated by the IMEX-RB scheme converges to $(\bx^*,\vec{0})$.
\end{proof}

{\bf Proof of Lemma~\ref{imex:E:lip}}
\begin{proof}
	Since $\nabla^2F$ is Lipschitz on the bounded invariant set $\mathcal B$, then $g$ is Lipschitz continuous, from Theorem~\ref{imexrb:thm}. We denote its Lipschitz constant by $L_g$.
	
	Since $F\in C^2(\mathcal B)$, we have
	\[\sup_{x\in\mathcal B}\|\nabla F(x)\|<\infty.\]
	Since $\alpha$, $\beta$, and $\gamma$ are uniformly bounded, it follows that $\delta_k$ and $C_k$ are uniformly bounded on any finite time interval. Moreover, $v_k$ is uniformly bounded on $\mathcal B$.
	We denote
	\[\|v_k\|\le M,\]
	where $M$ is independent of $k$.
	
	Let $(x^{k},x^{k-1})$ and $(\tilde x^{k},\tilde x^{k-1})$ are in $\mathcal B\times\mathcal B$, then
	\begin{align*}
		|E(x^{k},x^{k-1})-E(\tilde x^{k},\tilde x^{k-1})|
		&\le |\delta_k(F(x^{k})-F(\tilde x^{k}))|+ \frac12\big|\|v_k\|^2-\|\tilde v_k\|^2\big| \\
		&\le \delta_{\max}\sup_{x\in\mathcal B}\|\nabla F(x)\|\,\|x^{k}-\tilde x^{k}\|+ \frac12\big|\|v_k\|^2-\|\tilde v_k\|^2\big|.
	\end{align*}
	
	For the second term,
	\[\frac12\big|\|v_k\|^2-\|\tilde v_k\|^2\big|\le \frac12(\|v_k\|+\|\tilde v_k\|)\,\|v_k-\tilde v_k\|\le M\|v_k-\tilde v_k\|.\]
	
	From the definition of $v_k$ and the Lipschitz continuity of $\nabla F$, there exists a constant $L_1>0$ such that
	\begin{align*}
		\|v_k-\tilde v_k\|&\le\big(|\alpha-1| + \alpha\sup_{x\in\mathcal B}\|\nabla F(x)\| + k + \gamma_k k h L_1\big)\|x^{k}-\tilde x^{k}\| \\
		&\quad + k\,\|x^{k-1}-\tilde x^{k-1}\|.
	\end{align*}
	
	Combining the estimates above, we obtain
	\begin{align*}
		|E(x^{k},x^{k-1})-E(\tilde x^{k},\tilde x^{k-1})|&\le\big(\delta_{\max}\sup_{x\in\mathcal B}\|\nabla F(x)\|+ M L_v\big)\|x^{k}-\tilde x^{k}\| \\
		&\qquad + kM\|x^{k-1}-\tilde x^{k-1}\|,
	\end{align*}
	where $L_v$ denotes the constant multiplying $\|x^{k}-\tilde x^{k}\|$ above.
	
	Therefore, there exists a constant $L_E>0$ such that
	\[|E(x^{k},x^{k-1})-E(\tilde x^{k},\tilde x^{k-1})|\le L_E\big(\|x^{k}-\tilde x^{k}\|+\|x^{k-1}-\tilde x^{k-1}\|\big),\]
	which proves that $E$ is Lipschitz on $\mathcal B\times\mathcal B$.
\end{proof}

{\bf Proof of Lemma~\ref{imexrb:lemma3}}
\begin{proof}
	Since both methods start from the same point $\bu^n$, we have
	\begin{align*}
		\|\bu^{n+1}-\by^{n+1}\|&= \|\Delta t\,g(t_{n+1},V^n{V^n}^\top \bu^{n+1})-\Delta t\,g(t_{n+1},\by^{n+1})\| \\
		&\leq \Delta t\,L_g\|V^n{V^n}^\top \bu^{n+1}-\by^{n+1}\| \\
		&\leq \Delta t\,L_g\bigl(\|(I-V^n{V^n}^\top)\bu^{n+1}\|+\|\bu^{n+1}-\by^{n+1}\|\bigr) \\
		&\leq \Delta t\,L_g\bigl(\varepsilon\|\bu^{n+1}\|+\|\bu^{n+1}-\by^{n+1}\|\bigr),
	\end{align*}
	where the last inequality follows from the stopping condition of the IMEX-RB method.
	
	Rearranging the terms gives
	\[\|\bu^{n+1}-\by^{n+1}\|\leq \frac{L_g\varepsilon\Delta t}{1-L_g\Delta t}\,\|\bu^{n+1}\|.
	\]
	
	Since the iterates remain in the bounded domain $\mathcal B$,there exists a constant $M>0$, independent of $n$, such that $\|\bu^{n+1}\|\leq M$. Moreover, if $\Delta t\leq \frac{1}{2L_g}$, then
	\[\|\bu^{n+1}-\by^{n+1}\|\leq 2L_gM\,\varepsilon\,\Delta t.\]
	
	Taking $C=2L_gM$ completes the proof.
	
\end{proof}

{\bf Proof of Theorem~\ref{imexrb:energystability}}
\begin{proof}
	Let $\{\bu^k_{\mathrm{IMEX}}\}$ and $\{y^{k}_{\mathrm{BE}}\}$ be the sequences generated by the IMEX-RB method and the backward Euler method respectively, starting from the same initial value.
	The main idea is to compare $E^{\mathrm{IMEX}}_{k+1}$ with $E^{\mathrm{BE}}_{k+1}$. We then apply Theorem~\ref{fully-discretized:energystability} to conclude the result for the IMEX-RB method.
	
	Denote \[E^{k} := \bu^k_{\mathrm{IMEX}} - \by^{k}_{\mathrm{BE}}.\]
	
	Consider $E^{k+1}_{\mathrm{IMEX}} - E^k_{\mathrm{IMEX}}$,
	\begin{align*}
		E^{k+1}_{\mathrm{IMEX}} - E^k_{\mathrm{IMEX}}&=(E^{k+1}_{\mathrm{IMEX}} - E^{k+1}_{\mathrm{BE}})+ (E^{k+1}_{\mathrm{BE}} - E^k_{\mathrm{BE}})+ (E^k_{\mathrm{BE}} - E^k_{\mathrm{IMEX}}).
	\end{align*}
	
	From Theorem~\ref{fully-discretized:energystability}, $E^{k+1}_{\mathrm{BE}} - E^k_{\mathrm{BE}} = -\Delta_k \le 0$. From Lemma~\ref{imex:E:lip}, we obtain Lipschitz continuity for the discrete energy. Therefore,
	\beq\label{imexeq:1}
	E^{k+1}_{\mathrm{IMEX}} - E^k_{\mathrm{IMEX}}\le L_E(\|E^{k+1}\| + \|E^{k}\|) - \Delta_k.
	\eeq
	
	To estimate $\|E^{k+1}\|$, note that
	\[E^{k+1}=\phi_{\mathrm{IMEX}}(x^{\mathrm{IMEX}}_{k})-\phi_{\mathrm{IMEX}}(x^{\mathrm{BE}}_{k})+\phi_{\mathrm{IMEX}}(x^{\mathrm{BE}}_{k})-\phi_{\mathrm{BE}}(x^{\mathrm{BE}}_{k}).\]

	From Theorem~\ref{imex:convergence:thm}, we know that $g(t,y)$ is $C^1$ and $L_g$-Lipschitz. Therefore, the one-step IMEX-RB operator $\phi_{\mathrm{IMEX}}^{n}$, which is
	\[\phi_{\mathrm{IMEX}}(\bu^{k+1}) = \bu^{k} + \Delta t\, g(t_{n+1}, P \bu^{k+1}),\]
	is $C^1$ on $\mathcal{B}$. Moreover, the operator is uniformly bounded by $C_\phi$ in $\mathcal{B}$:
	\[\|\phi_{\mathrm{IMEX}}\| \le \frac{1}{1 - \Delta t L_g} = C_\phi \le 2,\]
	since the projector $P = V^n {V^n}^\top$ is fixed at each iteration and $\Delta t \le \frac{1}{2L_g}$.

	Then 
	\[\|E^{k+1}\|\le C_\phi \|E^{k}\|+\|\phi_{\mathrm{IMEX}}(y^{k}_{\mathrm{BE}}) -\phi_{\mathrm{BE}}(y^{k}_{\mathrm{BE}})\|.\]
	
	By Lemma~\ref{imexrb:lemma3},
	\[\|\phi_{\mathrm{IMEX}}(y^{k}_{\mathrm{BE}}) - \phi_{\mathrm{BE}}(y^{k}_{\mathrm{BE}})\|\le C\,\varepsilon\,\Delta t.\]
	
	Hence,
	\beq\label{imexeq:2}
	\|E^{k+1}\| \le C_\phi\|E^{k}\| + C\,\varepsilon\,\Delta t.
	\eeq
	
	On finite time interval $[0,T]$, with the same initial value, $e_0=0$, and
	\[\|E^{k}\| \le C\,\varepsilon\,\Delta t\sum_{j=0}^{k-1}C_\phi^j= C\,\varepsilon\,\Delta t\cdot\frac{C_\phi^k-1}{C_\phi-1}.\]
	
	Since  $t_k\le T$, then there exists an integer $N>0$ such that $k\le N$, and 
	\[\|E^{k}\| \le C\,\varepsilon\,\Delta t\frac{C_\phi^N-1}{C_\phi-1}\le C\,\varepsilon\,\Delta t\frac{e^{T L_g}-1}{\Delta t L_g} = C\,\varepsilon\frac{e^{T L_g}-1}{L_g}=C_e\,\varepsilon,\]
	where $C_e= C \frac{e^{T L_g}-1}{L_g}$.

	Substituting into \eqref{imexeq:1} yields
	\begin{align*}
		E^{k+1}_{\mathrm{IMEX}} - E^k_{\mathrm{IMEX}}&\le L_E\big((1+C_\phi)\|E^{k}\| + C\,\varepsilon\,\Delta t\big) - \Delta_k\\
		&\le L_E\big((1+C_\phi)C_e\,\varepsilon + C\,\varepsilon\,\Delta t\big) - \Delta_k\\
		&\le L_E\big((1+C_\phi)C_e\,\varepsilon + C\,\varepsilon\,\Delta t\big) - c_0\Delta t\\
		&\le L_EM\varepsilon- c_0\Delta t
	\end{align*}
	where $M=(1+C_\phi)C_e+ C/(2L_g)$.
	
	Since $	\varepsilon\le \frac{c_0}{2 L_E M}\Delta t$, then for all $t_k\le T$
	\[E^{k+1}_{\mathrm{IMEX}} - E^k_{\mathrm{IMEX}}\le -\frac{1}{2}c_0\Delta t\le 0.\]
	
	Then we conclude
	\[F(\bu^k_{\mathrm{IMEX}})-\min F = \mathcal O\!\left(\frac{1}{\delta_k}\right).\]
\end{proof}


\section{Numerical Results on Convergence Rate test}\label{num:append:con}
\subsection{Kinetic Energy $\|v_k\|_2$ Plots for the 10D Rotated Hyper-Ellipsoid Function with Time Step $h = 1/16$}
\begin{figure}[H]
  \centering
  \captionsetup[subfigure]{justification=centering}

  \begin{subfigure}{0.24\textwidth}\centering
    \includegraphics[width=0.8\textwidth]{"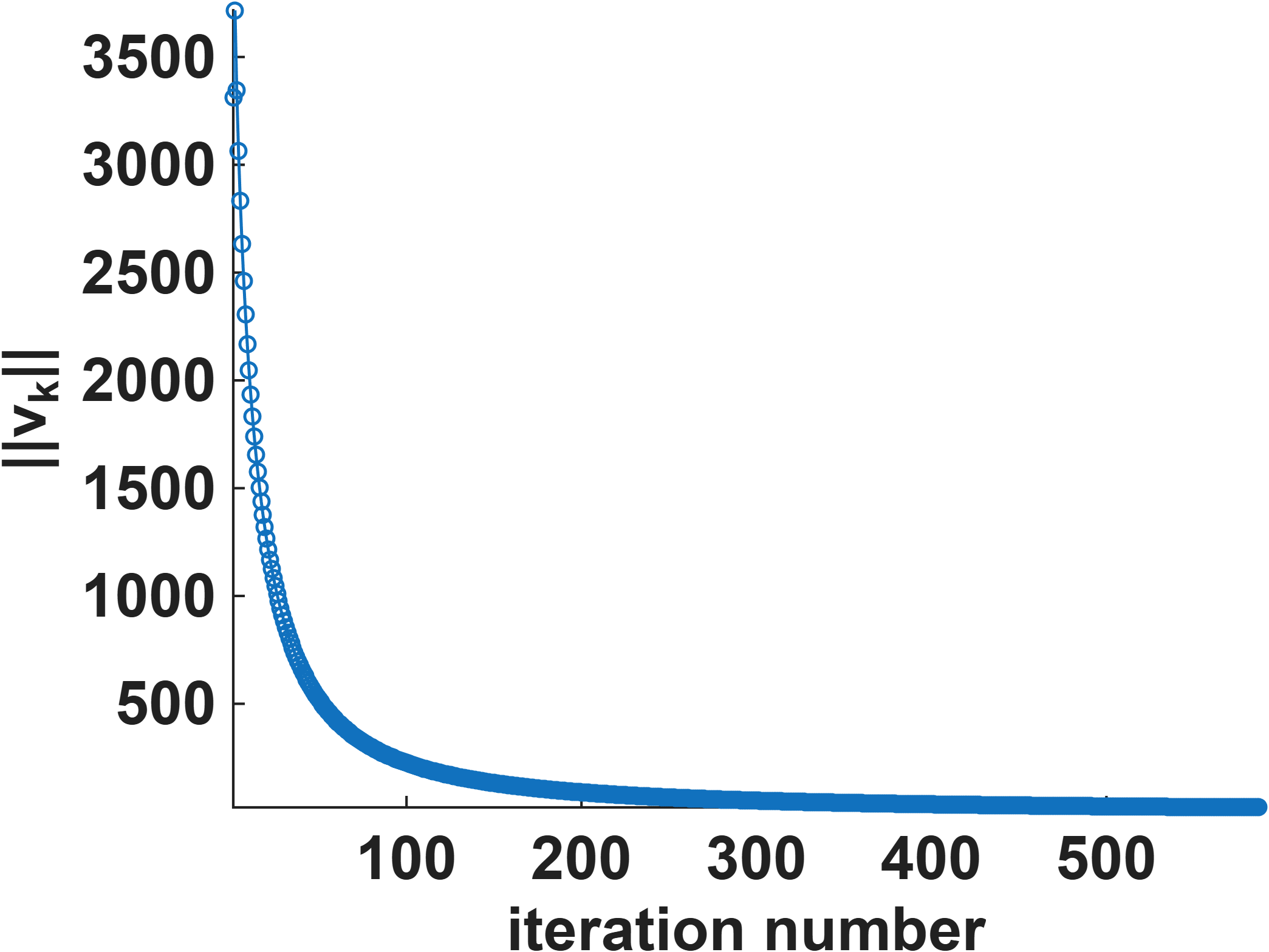"}
    \caption{FD}
  \end{subfigure}\hfill
  \begin{subfigure}{0.24\textwidth}\centering
    \includegraphics[width=0.8\textwidth]{"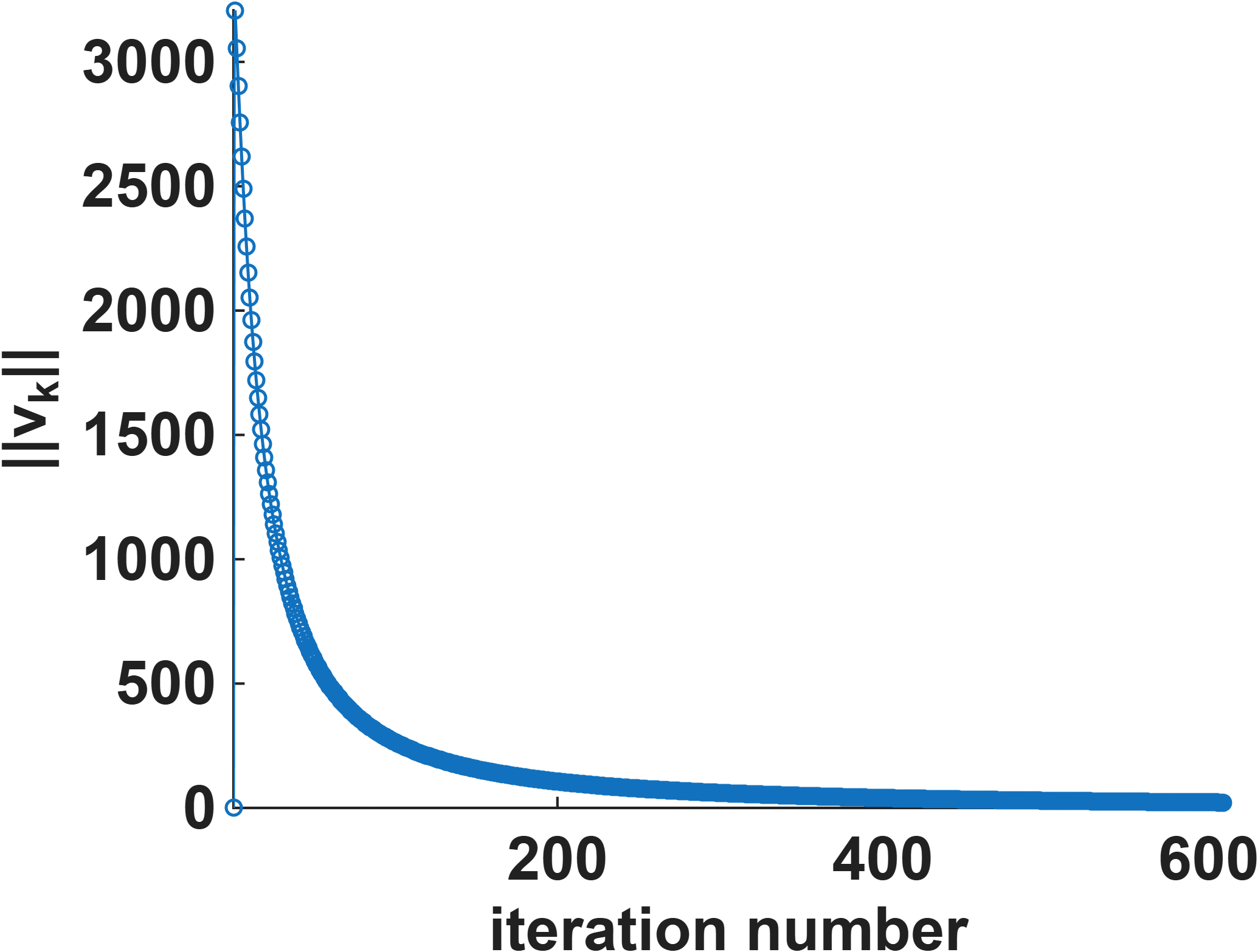"}
    \caption{IMEX-RB}
  \end{subfigure}\hfill
  \begin{subfigure}{0.24\textwidth}\centering
    \includegraphics[width=0.8\textwidth]{"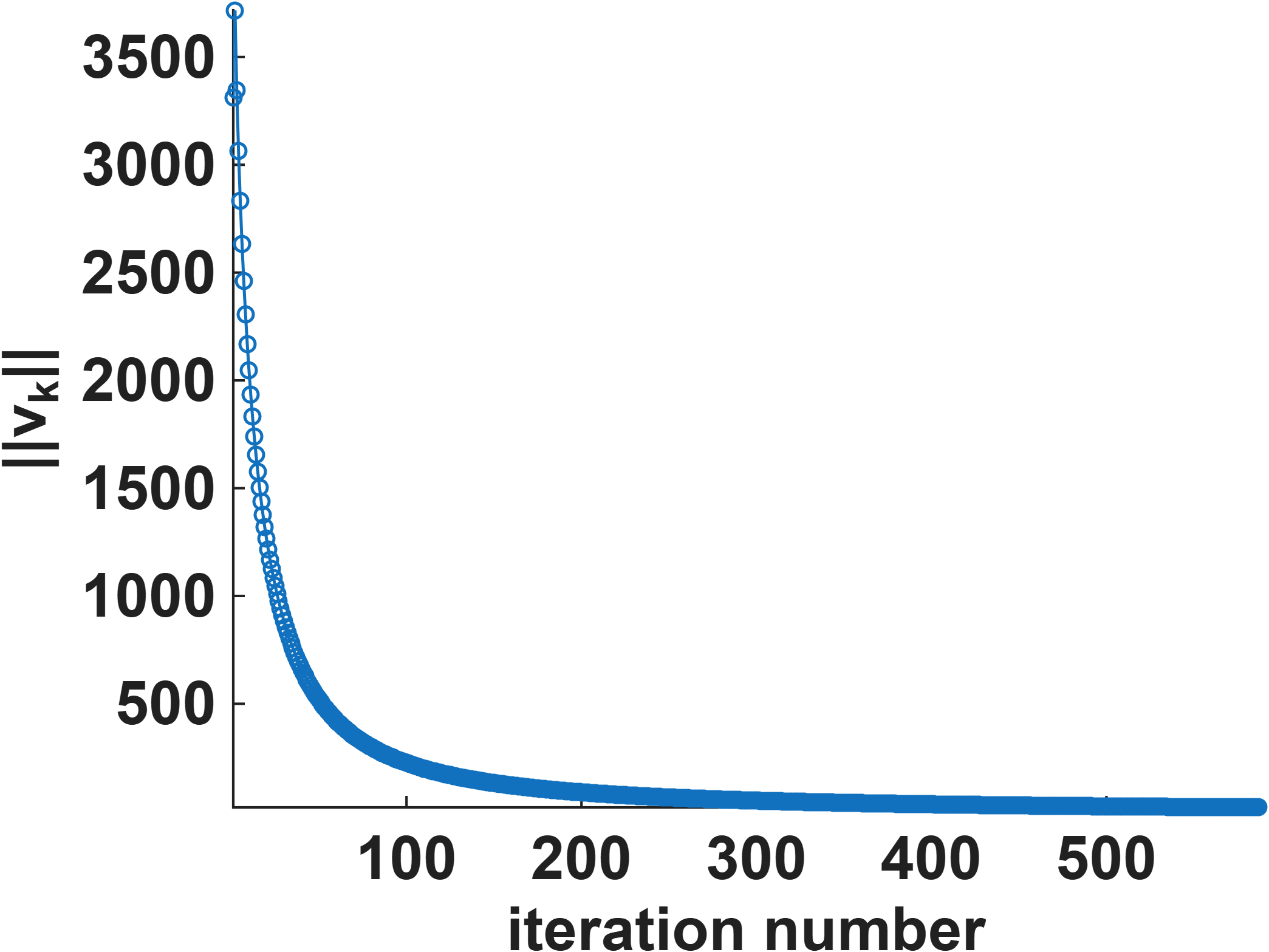"}
    \caption{FB}
  \end{subfigure}\hfill
  \begin{subfigure}{0.24\textwidth}\centering
    \includegraphics[width=0.8\textwidth]{"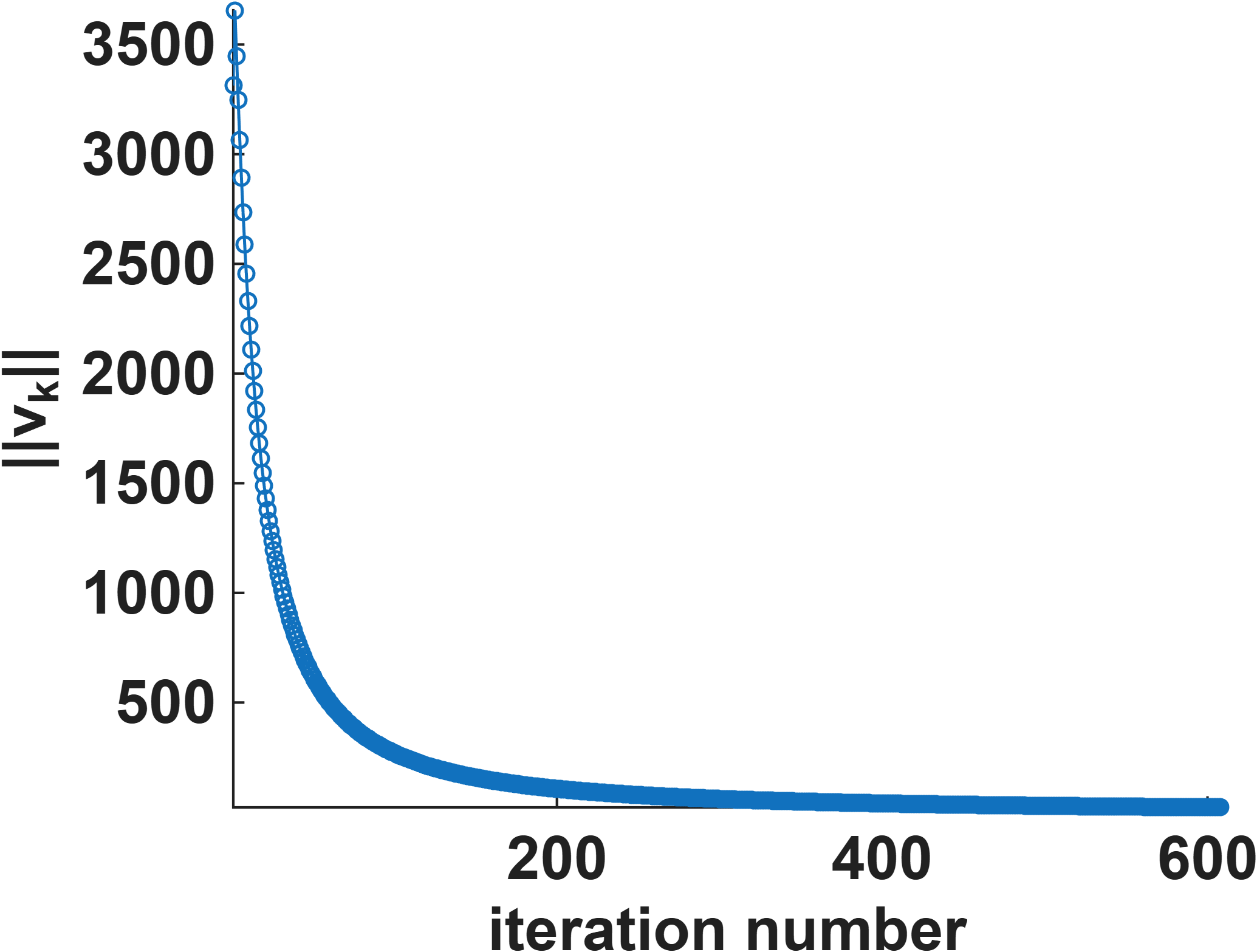"}
    \caption{Semi}
  \end{subfigure}

  \vspace{0.45em}

 \begin{subfigure}{0.24\textwidth}\centering
    \includegraphics[width=0.8\textwidth]{"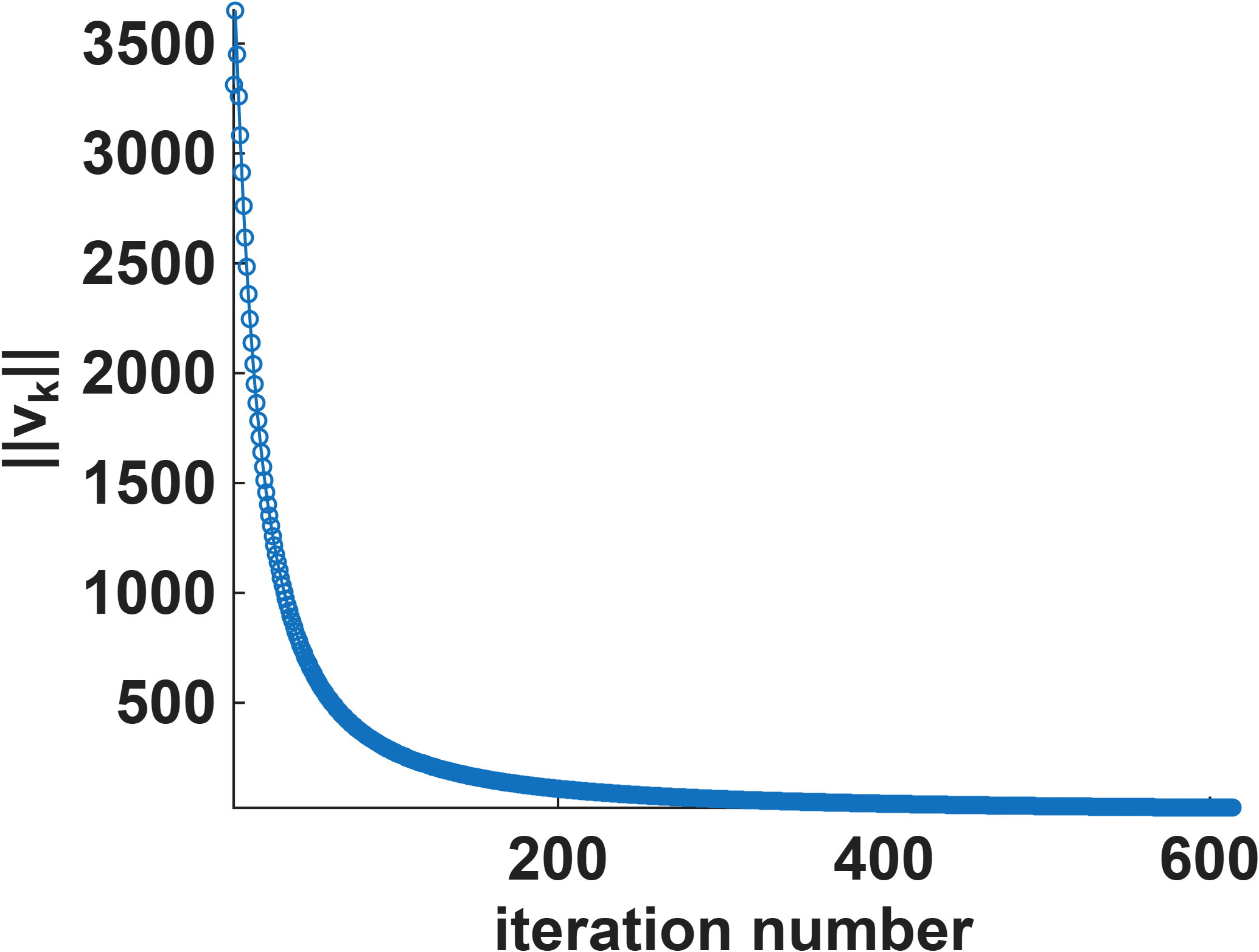"}
    \caption{IPAHD}
  \end{subfigure}\hfill
 \begin{subfigure}{0.24\textwidth}\centering
    \includegraphics[width=0.8\textwidth]{"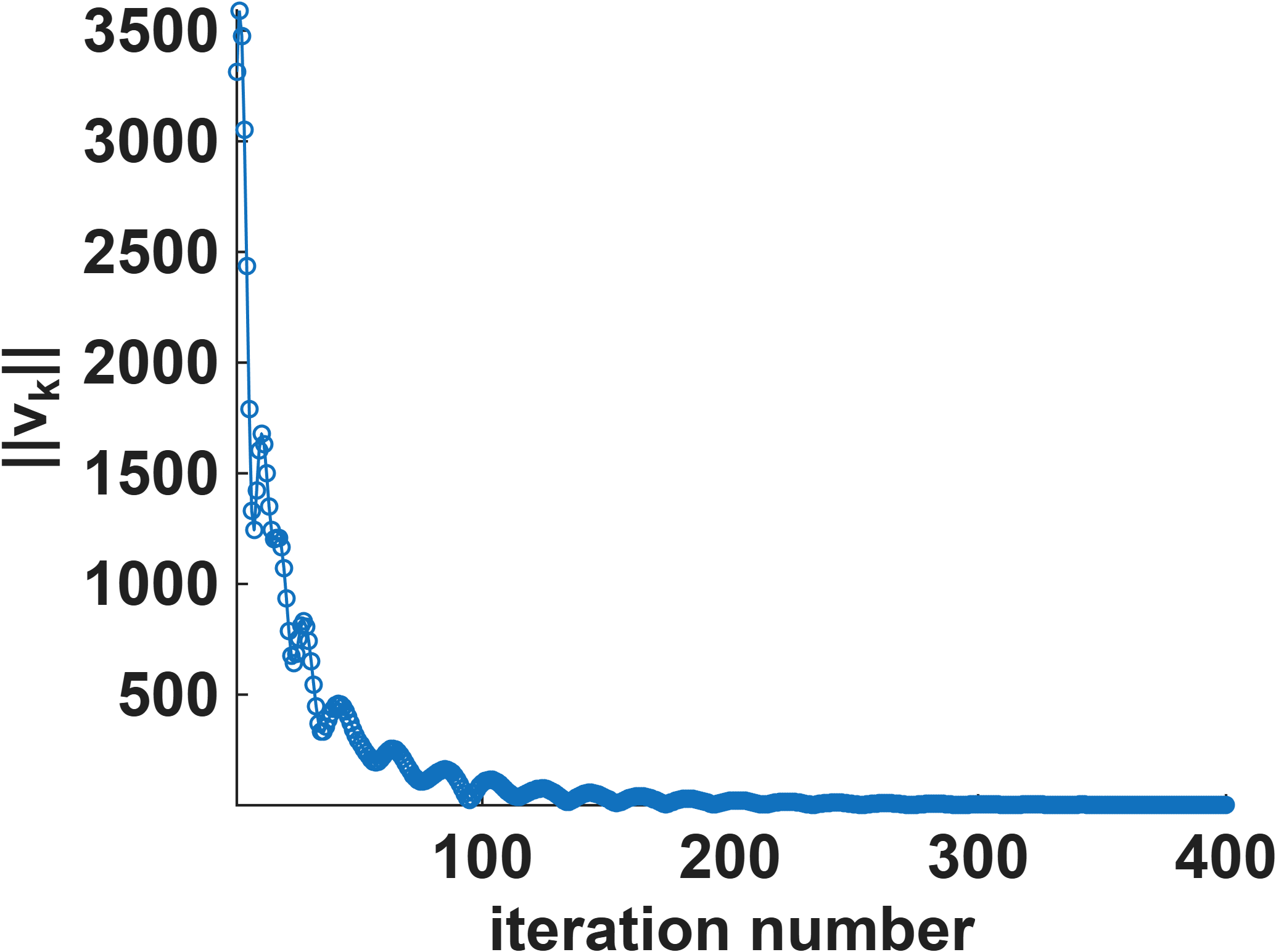"}
    \caption{Nesterov}
  \end{subfigure}\hfill
  \begin{subfigure}{0.24\textwidth}\centering
    \includegraphics[width=0.8\textwidth]{"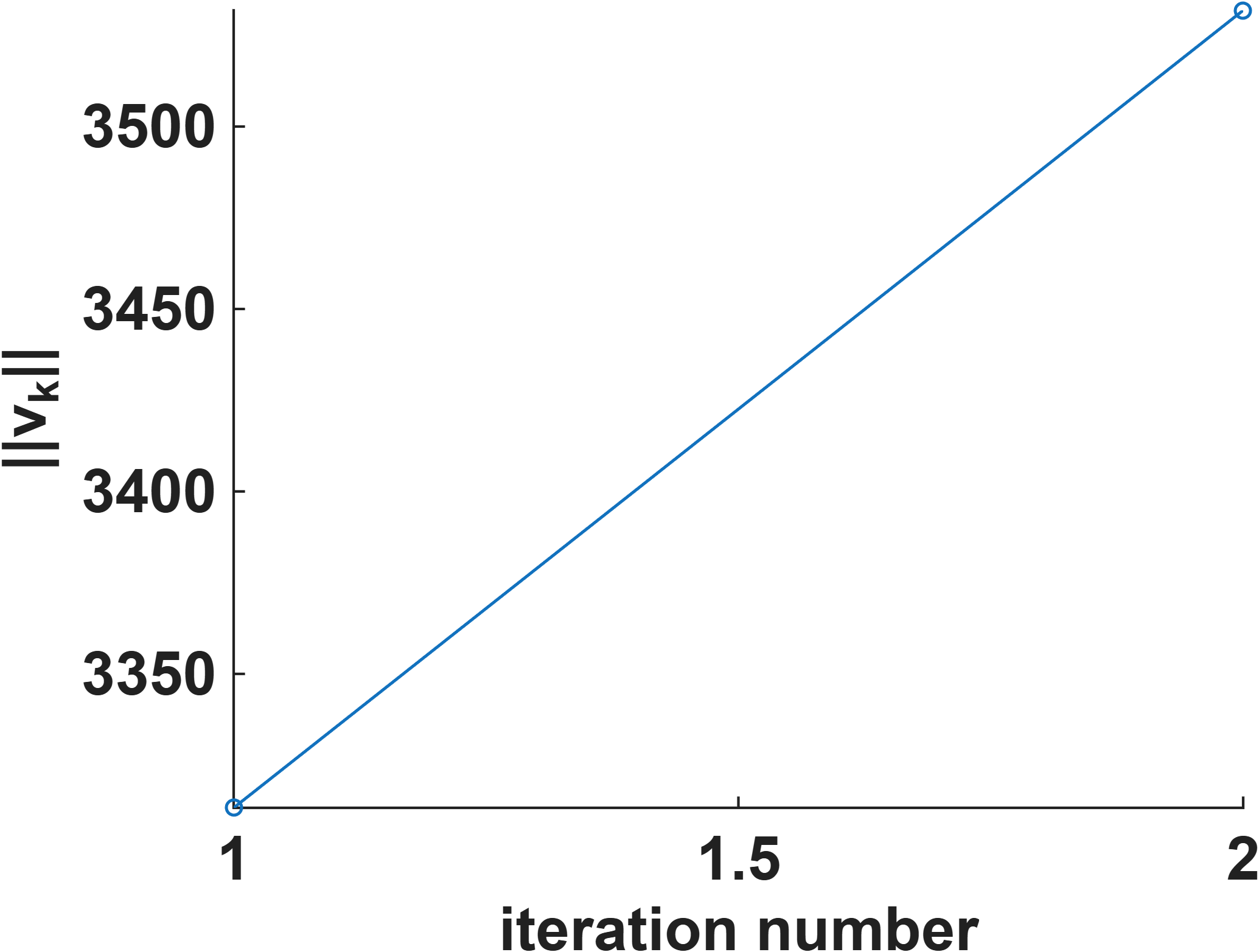"}
    \caption{GD}
  \end{subfigure}

  \caption{Comparison of $\|v_k\|_2$ across methods for the 10D Rotated Hyper-Ellipsoid Function (\(h=1/16\)). Each subplot is one method (left-to-right, top-to-bottom).}
  \label{fig:kinetic_comparison}
\end{figure}

\subsection{All tables of convergence rates on higher-dimensional testing functions}

\begin{table}[H]
  \centering
  \caption{1D Sphere Function }
  \label{tab:1d_sphere_compact}
  {\fontsize{9}{11}\selectfont
  \setlength{\tabcolsep}{4pt}
  \begin{adjustbox}{max width=\textwidth}
  \begin{tabular}{l
                  r c
                  r c
                  r c
                  r c
                  r c
                  r c
                  r c}
    \toprule
    \makecell[l]{h} &
    \makecell{FB\\$\bar{p}$} & \makecell{FB\\S/F} &
    \makecell{Semi\\$\bar{p}$} & \makecell{Semi\\S/F} &
    \makecell{FD\\$\bar{p}$} & \makecell{FD\\S/F} &
    \makecell{IMEX\\$\bar{p}$} & \makecell{IMEX\\S/F} &
    \makecell{IPAHD\\$\bar{p}$} & \makecell{IPAHD\\S/F} &
    \makecell{NM\\$\bar{p}$} & \makecell{NM\\S/F} &
    \makecell{GD\\$\bar{p}$} & \makecell{GD\\S/F} \\
    \midrule
    1          & 4.6697 & 1 & 4.6697 & 1 & 4.6697 & 1 & 4.6697 & 1 & 4.6697 & 1 & 241.7813 & 1 & NaN      & 1 \\
    $1/2^{1}$  & 4.5371 & 1 & 4.5372 & 1 & 4.5371 & 1 & 4.5371 & 1 & 4.5371 & 1 & 10.7844  & 1 & NaN      & 1 \\
    $1/2^{2}$  & 4.3740 & 1 & 4.3740 & 1 & 4.3740 & 1 & 4.3740 & 1 & 4.3754 & 1 & 2.0145   & 1 & NaN      & 1 \\
    $1/2^{3}$  & 4.2230 & 1 & 4.2230 & 1 & 4.2230 & 1 & 4.2229 & 1 & 4.2238 & 1 & 4.5747   & 1 & NaN      & 1 \\
    $1/2^{4}$  & 4.0242 & 1 & 4.0242 & 1 & 4.0242 & 1 & 4.0242 & 1 & 4.0271 & 1 & 10.8744  & 1 & NaN      & 1 \\
    $1/2^{5}$  & 3.8859 & 1 & 3.8852 & 1 & 3.8859 & 1 & 3.8859 & 1 & 3.8882 & 1 & 4.3273   & 1 & NaN      & 1 \\
    $1/2^{6}$  & 3.6615 & 1 & 3.6610 & 1 & 3.6615 & 1 & 3.6615 & 1 & 3.6645 & 1 & 5128.0000& 1 & NaN      & 1 \\
    $1/2^{7}$  & 3.5659 & 1 & 3.5659 & 1 & 3.5659 & 1 & 3.5655 & 1 & 3.5683 & 1 & 5.9379   & 1 & NaN      & 1 \\
    \bottomrule
  \end{tabular}
  \end{adjustbox}
  }
\end{table}

\begin{table}[H]
  \centering
  \caption{2D Sphere Function }
  \label{tab:2d_sphere_compact}
  {\fontsize{9}{11}\selectfont
  \setlength{\tabcolsep}{4pt}
  \begin{adjustbox}{max width=\textwidth}
  \begin{tabular}{l
                  r c
                  r c
                  r c
                  r c
                  r c
                  r c
                  r c}
    \toprule
    \makecell[l]{h} &
    \makecell{FB\\$\bar{p}$} & \makecell{FB\\S/F} &
    \makecell{Semi\\$\bar{p}$} & \makecell{Semi\\S/F} &
    \makecell{FD\\$\bar{p}$} & \makecell{FD\\S/F} &
    \makecell{IMEX\\$\bar{p}$} & \makecell{IMEX\\S/F} &
    \makecell{IPAHD\\$\bar{p}$} & \makecell{IPAHD\\S/F} &
    \makecell{NM\\$\bar{p}$} & \makecell{NM\\S/F} &
    \makecell{GD\\$\bar{p}$} & \makecell{GD\\S/F} \\
    \midrule
    1          & 4.7276 & 1 & 4.7276 & 1 & 4.7276 & 1 & 4.7276 & 1 & 4.7270 & 1 & 241.7813 & 1 & NaN      & 1 \\
    $1/2^{1}$  & 4.6175 & 1 & 4.6175 & 1 & 4.6175 & 1 & 4.6176 & 1 & 4.6159 & 1 & 10.2989  & 1 & NaN      & 1 \\
    $1/2^{2}$  & 4.4820 & 1 & 4.4829 & 1 & 4.4820 & 1 & 4.4820 & 1 & 4.4793 & 1 & 27.6623  & 1 & NaN      & 1 \\
    $1/2^{3}$  & 4.3533 & 1 & 4.3532 & 1 & 4.3533 & 1 & 4.3525 & 1 & 4.3499 & 1 & 83.2620  & 1 & NaN      & 1 \\
    $1/2^{4}$  & 4.1873 & 1 & 4.1880 & 1 & 4.1873 & 1 & 4.1873 & 1 & 4.1822 & 1 & 10.4080  & 1 & NaN      & 1 \\
    $1/2^{5}$  & 4.0641 & 1 & 4.0646 & 1 & 4.0641 & 1 & 4.0641 & 1 & 4.0585 & 1 & 4.1936   & 1 & NaN      & 1 \\
    $1/2^{6}$  & 3.8757 & 1 & 3.8761 & 1 & 3.8757 & 1 & 3.8757 & 1 & 3.8679 & 1 & 4183.8   & 1 & NaN      & 1 \\
    $1/2^{7}$  & 3.7793 & 1 & 3.7793 & 1 & 3.7793 & 1 & 3.7792 & 1 & 3.7729 & 1 & 5.6631   & 1 & NaN      & 1 \\
    \bottomrule
  \end{tabular}
  \end{adjustbox}
  }
\end{table}

\begin{table}[H]
  \centering
  \caption{10D Sphere Function}
  \label{tab:10d_sphere_compact}
  {\fontsize{9}{11}\selectfont
  \setlength{\tabcolsep}{4pt}
  \begin{adjustbox}{max width=\textwidth}
  \begin{tabular}{l
                  r c
                  r c
                  r c
                  r c
                  r c
                  r c
                  r c}
    \toprule
    \makecell[l]{h} &
    \makecell{FB\\$\bar{p}$} & \makecell{FB\\S/F} &
    \makecell{Semi\\$\bar{p}$} & \makecell{Semi\\S/F} &
    \makecell{FD\\$\bar{p}$} & \makecell{FD\\S/F} &
    \makecell{IMEX\\$\bar{p}$} & \makecell{IMEX\\S/F} &
    \makecell{IPAHD\\$\bar{p}$} & \makecell{IPAHD\\S/F} &
    \makecell{NM\\$\bar{p}$} & \makecell{NM\\S/F} &
    \makecell{GD\\$\bar{p}$} & \makecell{GD\\S/F} \\
    \midrule

    1          & 4.7511 & 1 & 4.7510 & 1 & 4.7511 & 1 & 4.7511 & 1 & 4.7511 & 1 & 241.7813 & 1 & 1.2460 & 1 \\
    $1/2^{1}$  & 4.6497 & 1 & 4.6497 & 1 & 4.6497 & 1 & 4.6497 & 1 & 4.6490 & 1 & 10.2989  & 1 & 1.2294 & 1 \\
    $1/2^{2}$  & 4.5244 & 1 & 4.5243 & 1 & 4.5244 & 1 & 4.5244 & 1 & 4.5237 & 1 & 26.4058  & 1 & 1.2195 & 1 \\
    $1/2^{3}$  & 4.4042 & 1 & 4.4049 & 1 & 4.4042 & 1 & 4.4043 & 1 & 4.4031 & 1 & 83.2620  & 1 & 1.2669 & 1 \\
    $1/2^{4}$  & 4.2507 & 1 & 4.2507 & 1 & 4.2507 & 1 & 4.2502 & 1 & 4.2480 & 1 & 10.4080  & 1 & 1.2297 & 1 \\
    $1/2^{5}$  & 4.1347 & 1 & 4.1346 & 1 & 4.1347 & 1 & 4.1347 & 1 & 4.1319 & 1 & 3.8423   & 1 & 1.2788 & 1 \\
    $1/2^{6}$  & 3.9590 & 1 & 3.9590 & 1 & 3.9590 & 1 & 3.9590 & 1 & 3.9553 & 1 & 4183.8000 & 1 & 1.2419 & 1 \\
    $1/2^{7}$  & 3.8653 & 1 & 3.8653 & 1 & 3.8653 & 1 & 3.8653 & 1 & 3.8623 & 1 & 5.4302   & 1 & 1.3061 & 1 \\
    \bottomrule
  \end{tabular}
  \end{adjustbox}
  }
\end{table}

\begin{table}[H]
  \centering
  \caption{50D Sphere Function }
  \label{tab:50d_sphere_compact}
  {\fontsize{9}{11}\selectfont
  \setlength{\tabcolsep}{4pt}
  \begin{adjustbox}{max width=\textwidth}
  \begin{tabular}{l
                  r c
                  r c
                  r c
                  r c
                  r c
                  r c
                  r c}
    \toprule
    \makecell[l]{h} &
    \makecell{FB\\$\bar{p}$} & \makecell{FB\\S/F} &
    \makecell{Semi\\$\bar{p}$} & \makecell{Semi\\S/F} &
    \makecell{FD\\$\bar{p}$} & \makecell{FD\\S/F} &
    \makecell{IMEX\\$\bar{p}$} & \makecell{IMEX\\S/F} &
    \makecell{IPAHD\\$\bar{p}$} & \makecell{IPAHD\\S/F} &
    \makecell{NM\\$\bar{p}$} & \makecell{NM\\S/F} &
    \makecell{GD\\$\bar{p}$} & \makecell{GD\\S/F} \\
    \midrule
    1          & 4.8034 & 1 & 4.8035 & 1 & 4.8034 & 1 & 4.8031 & 1 & 4.8042 & 1 & 241.7813 & 1 & NaN      & 1 \\
    $1/2^{1}$  & 4.7235 & 1 & 4.7253 & 1 & 4.7235 & 1 & 4.7220 & 1 & 4.7231 & 1 & 10.2989  & 1 & NaN      & 1 \\
    $1/2^{2}$  & 4.6272 & 1 & 4.6320 & 1 & 4.6272 & 1 & 4.6244 & 1 & 4.6220 & 1 & 24.0752  & 1 & NaN      & 1 \\
    $1/2^{3}$  & 4.5385 & 1 & 4.5444 & 1 & 4.5385 & 1 & 4.5337 & 1 & 4.5233 & 1 & 66.9773  & 1 & NaN      & 1 \\
    $1/2^{4}$  & 4.4231 & 1 & 4.4289 & 1 & 4.4231 & 1 & 4.4186 & 1 & 4.3978 & 1 & 10.4080  & 1 & NaN      & 1 \\
    $1/2^{5}$  & 4.3301 & 1 & 4.3337 & 1 & 4.3301 & 1 & 4.3271 & 1 & 4.2997 & 1 & 4.1737   & 1 & NaN      & 1 \\
    $1/2^{6}$  & 4.1991 & 1 & 4.2016 & 1 & 4.1991 & 1 & 4.1967 & 1 & 4.1554 & 1 & 4183.8   & 0 & NaN      & 1 \\
    $1/2^{7}$  & 4.1090 & 1 & 4.1101 & 1 & 4.1090 & 1 & 4.1079 & 1 & 4.0703 & 1 & 5.4302   & 0 & NaN      & 1 \\
    \bottomrule
  \end{tabular}
  \end{adjustbox}
  }
\end{table}

\begin{table}[H]
  \centering
  \caption{1D Modified Sphere Function}
  \label{tab:modsphere1d_compact}
  {\fontsize{9}{11}\selectfont
  \setlength{\tabcolsep}{4pt}
  \begin{adjustbox}{max width=\textwidth}
  \begin{tabular}{l
                  r c
                  r c
                  r c
                  r c
                  r c
                  r c
                  r c}
    \toprule
    \makecell[l]{h} &
    \makecell{FB\\$\bar{p}$} & \makecell{FB\\S/F} &
    \makecell{Semi\\$\bar{p}$} & \makecell{Semi\\S/F} &
    \makecell{FD\\$\bar{p}$} & \makecell{FD\\S/F} &
    \makecell{IMEX\\$\bar{p}$} & \makecell{IMEX\\S/F} &
    \makecell{IPAHD\\$\bar{p}$} & \makecell{IPAHD\\S/F} &
    \makecell{NM\\$\bar{p}$} & \makecell{NM\\S/F} &
    \makecell{GD\\$\bar{p}$} & \makecell{GD\\S/F} \\
    \midrule
    1          & 4.7396 & 1 & 4.7397 & 1 & 4.7396 & 1 & 4.7392 & 1 & 4.7395 & 1 & 1.4140 & 1 & 1.0042 & 1 \\
    $1/2^{1}$  & 4.8544 & 1 & 4.8546 & 1 & 4.8544 & 1 & 4.8551 & 1 & 4.8535 & 1 & 1.0704 & 1 & 1.0034 & 1 \\
    $1/2^{2}$  & 5.5196 & 1 & 5.5163 & 1 & 5.5196 & 1 & 5.5229 & 1 & 5.5627 & 1 & 1.0235 & 1 & 1.0029 & 1 \\
    $1/2^{3}$  & 3.3144 & 1 & 3.3120 & 1 & 3.3144 & 1 & 3.3168 & 1 & 3.3520 & 1 & 1.0573 & 1 & 1.0052 & 1 \\
    $1/2^{4}$  & 2.7222 & 1 & 2.7222 & 1 & 2.7222 & 1 & 2.7222 & 1 & 2.7169 & 1 & 1.1546 & 1 & 1.0033 & 1 \\
    $1/2^{5}$  & 2.5157 & 1 & 2.5156 & 1 & 2.5157 & 1 & 2.5159 & 1 & 2.5228 & 1 & 1.1642 & 1 & 1.0056 & 1 \\
    $1/2^{6}$  & 2.3680 & 1 & 2.3680 & 1 & 2.3680 & 1 & 2.3681 & 1 & 2.3711 & 1 & 1.0080 & 1 & 1.0035 & 1 \\
    $1/2^{7}$  & 2.4007 & 1 & 2.4007 & 1 & 2.4007 & 1 & 2.3917 & 1 & 2.4066 & 1 & 1.0074 & 1 & 1.0058 & 1 \\
    \bottomrule
  \end{tabular}
  \end{adjustbox}
  }
\end{table}

\begin{table}[H]
  \centering
  \caption{2D Modified Sphere Function}
  \label{tab:modsphere2d_compact}
  {\fontsize{9}{11}\selectfont
  \setlength{\tabcolsep}{4pt}
  \begin{adjustbox}{max width=\textwidth}
  \begin{tabular}{l
                  r c
                  r c
                  r c
                  r c
                  r c
                  r c
                  r c}
    \toprule
    \makecell[l]{h} &
    \makecell{FB\\$\bar{p}$} & \makecell{FB\\S/F} &
    \makecell{Semi\\$\bar{p}$} & \makecell{Semi\\S/F} &
    \makecell{FD\\$\bar{p}$} & \makecell{FD\\S/F} &
    \makecell{IMEX\\$\bar{p}$} & \makecell{IMEX\\S/F} &
    \makecell{IPAHD\\$\bar{p}$} & \makecell{IPAHD\\S/F} &
    \makecell{NM\\$\bar{p}$} & \makecell{NM\\S/F} &
    \makecell{GD\\$\bar{p}$} & \makecell{GD\\S/F} \\
    \midrule
    1          & 4.7063 & 1 & 4.7063 & 1 & 4.7063 & 1 & 4.7057 & 1 & 4.7062 & 1 & 1.4244 & 1 & 1.0056 & 1 \\
    $1/2^{1}$  & 4.6001 & 1 & 4.6001 & 1 & 4.6001 & 1 & 4.6000 & 1 & 4.6002 & 1 & 1.0777 & 1 & 1.0049 & 1 \\
    $1/2^{2}$  & 4.4736 & 1 & 4.4736 & 1 & 4.4736 & 1 & 4.4730 & 1 & 4.4756 & 1 & 1.0286 & 1 & 1.0044 & 1 \\
    $1/2^{3}$  & 4.3614 & 1 & 4.3614 & 1 & 4.3614 & 1 & 4.3622 & 1 & 4.3629 & 1 & 1.0609 & 1 & 1.0065 & 1 \\
    $1/2^{4}$  & 4.2202 & 1 & 4.2200 & 1 & 4.2202 & 1 & 4.2201 & 1 & 4.2224 & 1 & 1.1572 & 1 & 1.0048 & 1 \\
    $1/2^{5}$  & 4.0563 & 1 & 4.0562 & 1 & 4.0563 & 1 & 4.0566 & 1 & 4.0614 & 1 & 1.1660 & 1 & 1.0069 & 1 \\
    $1/2^{6}$  & 3.8290 & 1 & 3.8289 & 1 & 3.8290 & 1 & 3.8292 & 1 & 3.8358 & 1 & 1.0093 & 1 & 1.0049 & 1 \\
    $1/2^{7}$  & 3.6278 & 1 & 3.6277 & 1 & 3.6278 & 1 & 3.6279 & 1 & 3.6366 & 1 & 1.0083 & 1 & 1.0071 & 1 \\
    \bottomrule
  \end{tabular}
  \end{adjustbox}
  }
\end{table}

\begin{table}[H]
  \centering
  \caption{10D Modified Sphere Function}
  \label{tab:modsphere10d_compact}
  {\fontsize{9}{11}\selectfont
  \setlength{\tabcolsep}{4pt}
  \begin{adjustbox}{max width=\textwidth}
  \begin{tabular}{l
                  r c
                  r c
                  r c
                  r c
                  r c
                  r c
                  r c}
    \toprule
    \makecell[l]{h} &
    \makecell{FB\\$\bar{p}$} & \makecell{FB\\S/F} &
    \makecell{Semi\\$\bar{p}$} & \makecell{Semi\\S/F} &
    \makecell{FD\\$\bar{p}$} & \makecell{FD\\S/F} &
    \makecell{IMEX\\$\bar{p}$} & \makecell{IMEX\\S/F} &
    \makecell{IPAHD\\$\bar{p}$} & \makecell{IPAHD\\S/F} &
    \makecell{NM\\$\bar{p}$} & \makecell{NM\\S/F} &
    \makecell{GD\\$\bar{p}$} & \makecell{GD\\S/F} \\
    \midrule
    1          & 4.6499 & 1 & 4.6513 & 1 & 4.6499 & 1 & 4.6499 & 1 & 4.6580 & 1 & 1.4548 & 1 & 0.2500 & 0 \\
    $1/2^{1}$  & 4.5149 & 1 & 4.5212 & 1 & 4.5149 & 1 & 4.5116 & 1 & 4.5164 & 1 & 1.0983 & 1 & -0.4444 & 0 \\
    $1/2^{2}$  & 4.3619 & 1 & 4.3699 & 1 & 4.3619 & 1 & 4.3538 & 1 & 4.3441 & 1 & 1.0430 & 1 & 2.0000 & 0 \\
    $1/2^{3}$  & 4.2237 & 1 & 4.2319 & 1 & 4.2237 & 1 & 4.2167 & 1 & 4.1827 & 1 & 1.0724 & 1 & 0.9630 & 0 \\
    $1/2^{4}$  & 4.0406 & 1 & 4.0464 & 1 & 4.0406 & 1 & 4.0358 & 1 & 3.9773 & 1 & 1.1649 & 1 & 0.7873 & 0 \\
    $1/2^{5}$  & 3.9177 & 1 & 3.9212 & 1 & 3.9177 & 1 & 3.9143 & 1 & 3.8365 & 1 & 1.1723 & 1 & 0.7224 & 0 \\
    $1/2^{6}$  & 3.7113 & 1 & 3.7133 & 1 & 3.7113 & 1 & 3.7093 & 1 & 3.6081 & 1 & 1.0135 & 1 & 0.6936 & 0 \\
    $1/2^{7}$  & 3.6329 & 1 & 3.6339 & 1 & 3.6329 & 1 & 3.6317 & 1 & 3.5173 & 1 & 1.0114 & 1 & 0.6799 & 0 \\
    \bottomrule
  \end{tabular}
  \end{adjustbox}
  }
\end{table}

\begin{table}[H]
  \centering
  \caption{1D Sum Squares Function}
  \label{tab:sumsq1d_compact}
  {\fontsize{9}{11}\selectfont
  \setlength{\tabcolsep}{4pt}
  \begin{adjustbox}{max width=\textwidth}
  \begin{tabular}{l
                  r c
                  r c
                  r c
                  r c
                  r c
                  r c
                  r c}
    \toprule
    \makecell[l]{h} &
    \makecell{FB\\$\bar{p}$} & \makecell{FB\\S/F} &
    \makecell{Semi\\$\bar{p}$} & \makecell{Semi\\S/F} &
    \makecell{FD\\$\bar{p}$} & \makecell{FD\\S/F} &
    \makecell{IMEX\\$\bar{p}$} & \makecell{IMEX\\S/F} &
    \makecell{IPAHD\\$\bar{p}$} & \makecell{IPAHD\\S/F} &
    \makecell{NM\\$\bar{p}$} & \makecell{NM\\S/F} &
    \makecell{GD\\$\bar{p}$} & \makecell{GD\\S/F} \\
    \midrule
    1          & 4.7179 & 1 & 4.7178 & 1 & 4.7179 & 1 & 4.7179 & 1 & 4.7180 & 1 & 241.7813 & 1 & NaN & 1 \\
    $1/2^{1}$  & 4.6044 & 1 & 4.6044 & 1 & 4.6044 & 1 & 4.6035 & 1 & 4.6026 & 1 & 10.2989  & 1 & NaN & 1 \\
    $1/2^{2}$  & 4.4637 & 1 & 4.4647 & 1 & 4.4637 & 1 & 4.4637 & 1 & 4.4617 & 1 & 27.7666  & 1 & NaN & 1 \\
    $1/2^{3}$  & 4.3309 & 1 & 4.3317 & 1 & 4.3309 & 1 & 4.3310 & 1 & 4.3289 & 1 & 83.2620  & 1 & NaN & 1 \\
    $1/2^{4}$  & 4.1603 & 1 & 4.1603 & 1 & 4.1603 & 1 & 4.1596 & 1 & 4.1554 & 1 & 10.4080  & 1 & NaN & 1 \\
    $1/2^{5}$  & 4.0338 & 1 & 4.0343 & 1 & 4.0338 & 1 & 4.0338 & 1 & 4.0295 & 1 & 4.1936   & 1 & NaN & 1 \\
    $1/2^{6}$  & 3.8394 & 1 & 3.8394 & 1 & 3.8394 & 1 & 3.8394 & 1 & 3.8332 & 1 & 4183.8   & 1 & NaN & 1 \\
    $1/2^{7}$  & 3.7426 & 1 & 3.7425 & 1 & 3.7426 & 1 & 3.7421 & 1 & 3.7377 & 1 & 5.6631   & 1 & NaN & 1 \\
    \bottomrule
  \end{tabular}
  \end{adjustbox}
  }
\end{table}

\begin{table}[H]
  \centering
  \caption{2D Sum Squares Function}
  \label{tab:sumsq2d_compact}
  {\fontsize{9}{11}\selectfont
  \setlength{\tabcolsep}{4pt}
  \begin{adjustbox}{max width=\textwidth}
  \begin{tabular}{l
                  r c
                  r c
                  r c
                  r c
                  r c
                  r c
                  r c}
    \toprule
    \makecell[l]{h} &
    \makecell{FB\\$\bar{p}$} & \makecell{FB\\S/F} &
    \makecell{Semi\\$\bar{p}$} & \makecell{Semi\\S/F} &
    \makecell{FD\\$\bar{p}$} & \makecell{FD\\S/F} &
    \makecell{IMEX\\$\bar{p}$} & \makecell{IMEX\\S/F} &
    \makecell{IPAHD\\$\bar{p}$} & \makecell{IPAHD\\S/F} &
    \makecell{NM\\$\bar{p}$} & \makecell{NM\\S/F} &
    \makecell{GD\\$\bar{p}$} & \makecell{GD\\S/F} \\
    \midrule
    1          & 4.7138 & 1 & 4.7138 & 1 & 4.7138 & 1 & 4.7138 & 1 & 4.7145 & 1 & 16.2964 & 1 & 0.8723 & 0 \\
    $1/2^{1}$  & 4.5967 & 1 & 4.5967 & 1 & 4.5967 & 1 & 4.5977 & 1 & 4.5986 & 1 & 6.0346  & 1 & -0.8553 & 0 \\
    $1/2^{2}$  & 4.4535 & 1 & 4.4535 & 1 & 4.4535 & 1 & 4.4535 & 1 & 4.4555 & 1 & 17.5578 & 1 & 3.8178 & 0 \\
    $1/2^{3}$  & 4.3180 & 1 & 4.3180 & 1 & 4.3180 & 1 & 4.3188 & 1 & 4.3209 & 1 & 4.9300  & 1 & 2.8488 & 0 \\
    $1/2^{4}$  & 4.1415 & 1 & 4.1408 & 1 & 4.1415 & 1 & 4.1415 & 1 & 4.1465 & 1 & 1.8354  & 1 & 2.2949 & 0 \\
    $1/2^{5}$  & 4.0142 & 1 & 4.0142 & 1 & 4.0142 & 1 & 4.0142 & 1 & 4.0190 & 1 & 1.2223  & 1 & 2.1021 & 0 \\
    $1/2^{6}$  & 3.8139 & 1 & 3.8135 & 1 & 3.8139 & 1 & 3.8138 & 1 & 3.8210 & 1 & 1.1073  & 1 & 2.0177 & 0 \\
    $1/2^{7}$  & 3.7196 & 1 & 3.7196 & 1 & 3.7196 & 1 & 3.7198 & 1 & 3.7251 & 1 & 1.0436  & 1 & 1.9778 & 0 \\
    \bottomrule
  \end{tabular}
  \end{adjustbox}
  }
\end{table}

\begin{table}[H]
  \centering
  \caption{10D Sum Squares Function}
  \label{tab:sumsq10d_compact}
  {\fontsize{9}{11}\selectfont
  \setlength{\tabcolsep}{4pt}
  \begin{adjustbox}{max width=\textwidth}
  \begin{tabular}{l
                  r c
                  r c
                  r c
                  r c
                  r c
                  r c
                  r c}
    \toprule
    \makecell[l]{h} &
    \makecell{FB\\$\bar{p}$} & \makecell{FB\\S/F} &
    \makecell{Semi\\$\bar{p}$} & \makecell{Semi\\S/F} &
    \makecell{FD\\$\bar{p}$} & \makecell{FD\\S/F} &
    \makecell{IMEX\\$\bar{p}$} & \makecell{IMEX\\S/F} &
    \makecell{IPAHD\\$\bar{p}$} & \makecell{IPAHD\\S/F} &
    \makecell{NM\\$\bar{p}$} & \makecell{NM\\S/F} &
    \makecell{GD\\$\bar{p}$} & \makecell{GD\\S/F} \\
    \midrule
    1          & 4.7945 & 1 & 4.7956 & 1 & 4.7945 & 1 & 4.7937 & 1 & 4.7915 & 1 & 0.9990  & 0 & 0.2500 & 0 \\
    $1/2^{1}$  & 4.7132 & 1 & 4.7150 & 1 & 4.7132 & 1 & 4.7114 & 1 & 4.7054 & 1 & 4.0115  & 1 & -0.4444 & 0 \\
    $1/2^{2}$  & 4.6146 & 1 & 4.6174 & 1 & 4.6146 & 1 & 4.6123 & 1 & 4.5985 & 1 & 3.5929  & 1 & 2.0000 & 0 \\
    $1/2^{3}$  & 4.5193 & 1 & 4.5216 & 1 & 4.5193 & 1 & 4.5166 & 1 & 4.4941 & 1 & 1.2750  & 1 & 0.9630 & 0 \\
    $1/2^{4}$  & 4.3996 & 1 & 4.4019 & 1 & 4.3996 & 1 & 4.3976 & 1 & 4.3614 & 1 & 1.1237  & 1 & 0.7873 & 0 \\
    $1/2^{5}$  & 4.3031 & 1 & 4.3044 & 1 & 4.3031 & 1 & 4.3017 & 1 & 4.2586 & 1 & 1.0518  & 1 & 0.7224 & 0 \\
    $1/2^{6}$  & 4.1647 & 1 & 4.1656 & 1 & 4.1647 & 1 & 4.1640 & 1 & 4.1064 & 1 & 1.0250  & 1 & 0.6936 & 0 \\
    $1/2^{7}$  & 4.0754 & 1 & 4.0760 & 1 & 4.0754 & 1 & 4.0753 & 1 & 4.0187 & 1 & 1.0099  & 1 & 0.6799 & 0 \\
    \bottomrule
  \end{tabular}
  \end{adjustbox}
  }
\end{table}

\begin{table}[H]
  \centering
  \caption{50D Sum Squares Function}
  \label{tab:sumsq50d_compact}
  {\fontsize{9}{11}\selectfont
  \setlength{\tabcolsep}{4pt}
  \begin{adjustbox}{max width=\textwidth}
  \begin{tabular}{l
                  r c
                  r c
                  r c
                  r c
                  r c
                  r c
                  r c}
    \toprule
    \makecell[l]{h} &
    \makecell{FB\\$\bar{p}$} & \makecell{FB\\S/F} &
    \makecell{Semi\\$\bar{p}$} & \makecell{Semi\\S/F} &
    \makecell{FD\\$\bar{p}$} & \makecell{FD\\S/F} &
    \makecell{IMEX\\$\bar{p}$} & \makecell{IMEX\\S/F} &
    \makecell{IPAHD\\$\bar{p}$} & \makecell{IPAHD\\S/F} &
    \makecell{NM\\$\bar{p}$} & \makecell{NM\\S/F} &
    \makecell{GD\\$\bar{p}$} & \makecell{GD\\S/F} \\
    \midrule
    1          & 4.8622 & 1 & 4.8704 & 1 & 4.8622 & 1 & 4.8537 & 1 & 4.8304 & 1 & 0.9990  & 0 & 0.2500 & 0 \\
    $1/2^{1}$  & 4.8245 & 1 & 4.8371 & 1 & 4.8245 & 1 & 4.8134 & 1 & 4.7599 & 1 & 0.9989  & 0 & -0.4444 & 0 \\
    $1/2^{2}$  & 4.7809 & 1 & 4.7973 & 1 & 4.7809 & 1 & 4.7690 & 1 & 4.6717 & 1 & 0.9989  & 0 & 2.0000 & 0 \\
    $1/2^{3}$  & 4.7366 & 1 & 4.7506 & 1 & 4.7366 & 1 & 4.7270 & 1 & 4.5845 & 1 & 2.0213  & 1 & 0.9630 & 0 \\
    $1/2^{4}$  & 4.6820 & 1 & 4.6934 & 1 & 4.6820 & 1 & 4.6745 & 1 & 4.4744 & 1 & 1.3968  & 1 & 0.7873 & 0 \\
    $1/2^{5}$  & 4.6275 & 1 & 4.6343 & 1 & 4.6275 & 1 & 4.6231 & 1 & 4.3860 & 1 & 1.1722  & 1 & 0.7224 & 0 \\
    $1/2^{6}$  & 4.5598 & 1 & 4.5639 & 1 & 4.5598 & 1 & 4.5575 & 1 & 4.2586 & 1 & 1.0650  & 1 & 0.6936 & 0 \\
    $1/2^{7}$  & 4.3752 & 0 & 4.3746 & 0 & 4.3752 & 0 & 4.3695 & 0 & 4.1791 & 0 & 1.0273  & 1 & 0.6799 & 0 \\
    \bottomrule
  \end{tabular}
  \end{adjustbox}
  }
\end{table}

\subsection{Energy evolution in Swarm-based Algorithm test}\label{Appendix:sb:energy:plot}
\subsubsection{Ackley Function}
\begin{figure}[H]
  \centering
  \captionsetup[subfigure]{justification=centering}

  \begin{subfigure}{0.24\textwidth}\centering
    \includegraphics[width=\textwidth]{"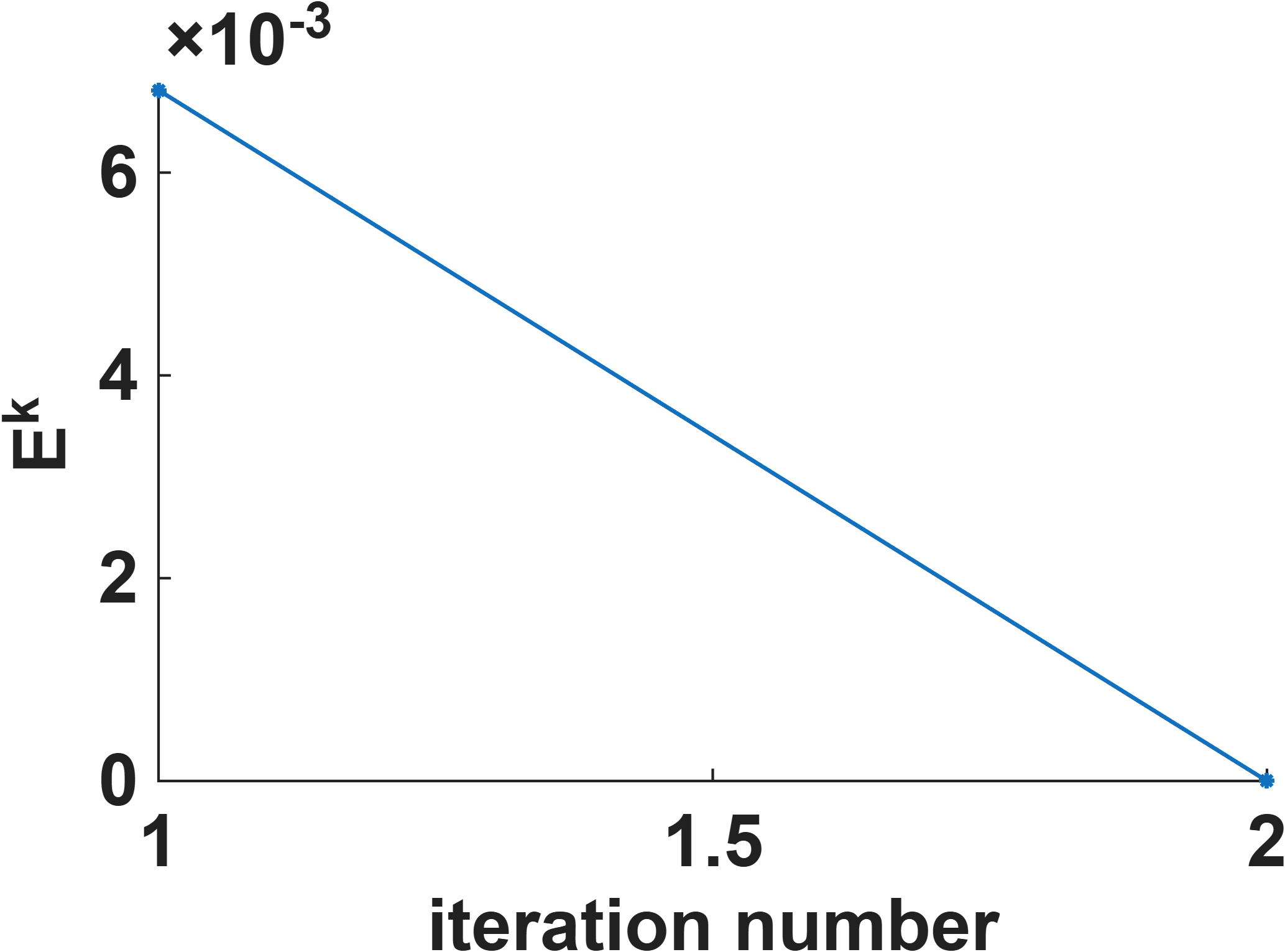"}
    \caption{SB-FB}
  \end{subfigure}\hfill
  \begin{subfigure}{0.24\textwidth}\centering
    \includegraphics[width=\textwidth]{"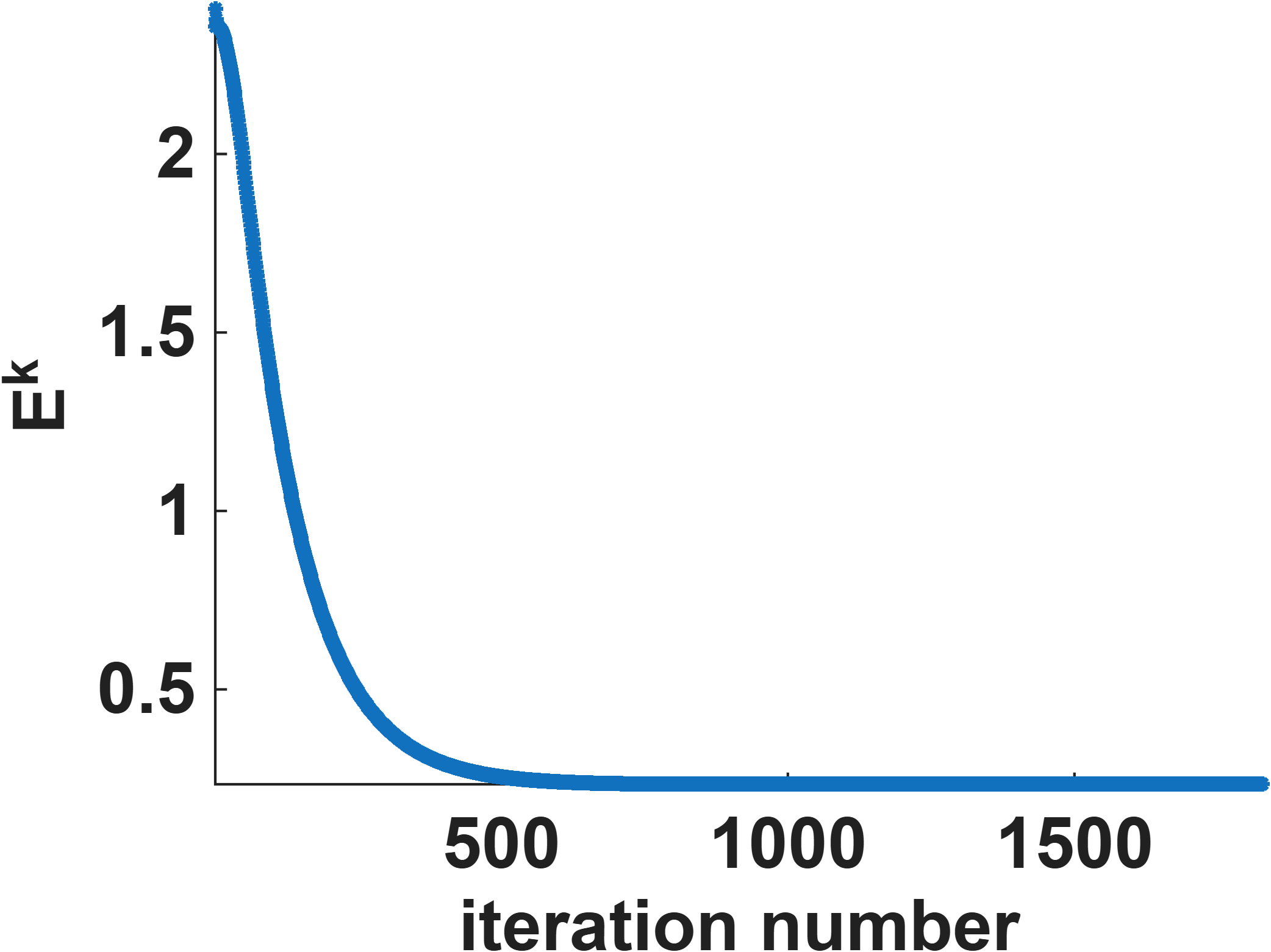"}
    \caption{SB-IMEX-RB}
  \end{subfigure}\hfill
  \begin{subfigure}{0.24\textwidth}\centering
    \includegraphics[width=\textwidth]{"1D_Ack_B0FB_Energy_i1000.png"}
    \caption{SB-Semi}
  \end{subfigure}\hfill
  \begin{subfigure}{0.24\textwidth}\centering
    \includegraphics[width=\textwidth]{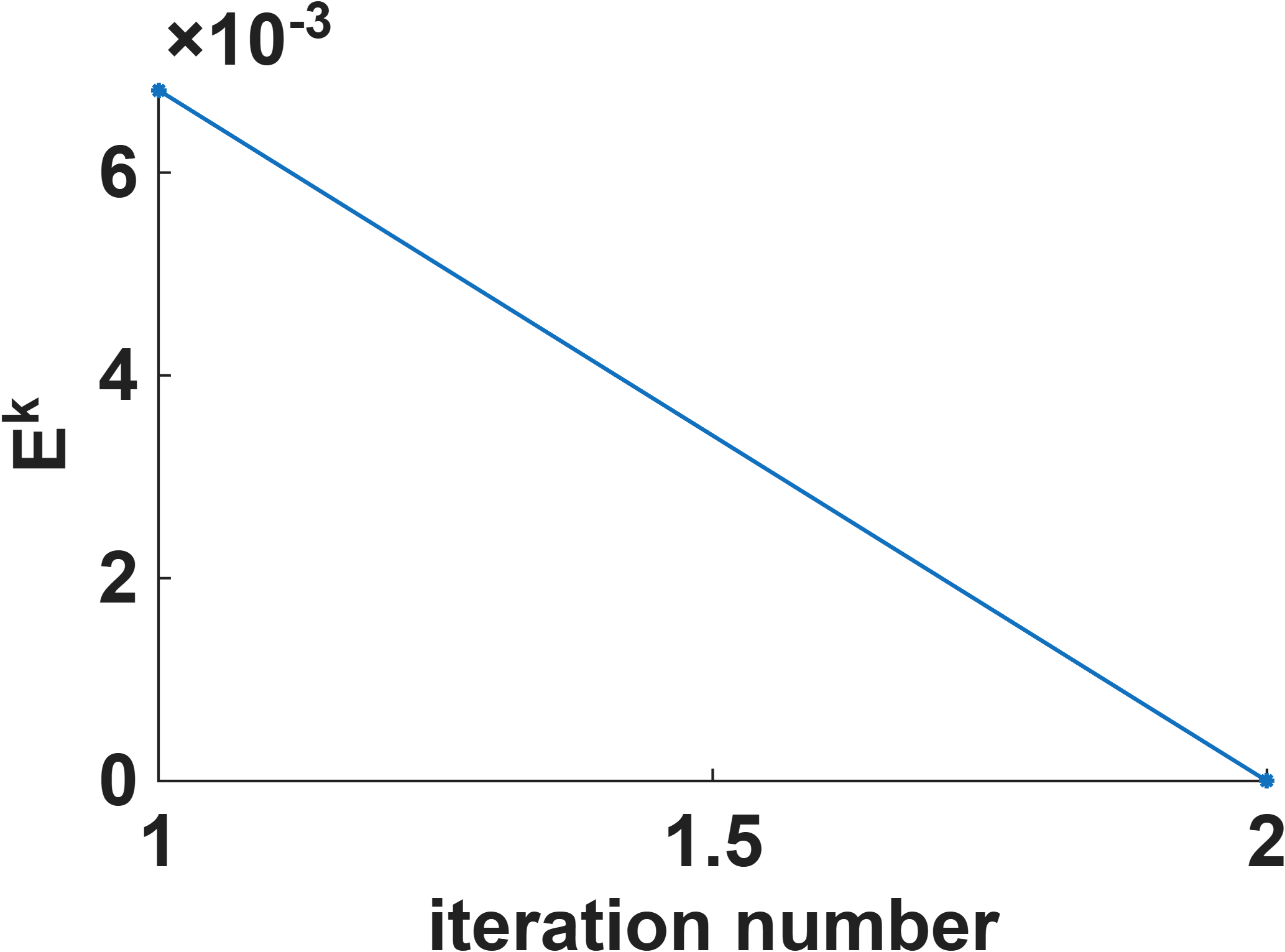}
    \caption{SB-IPAHD}
  \end{subfigure}

  \vspace{0.45em}

  \begin{subfigure}{0.24\textwidth}\centering
    \includegraphics[width=\textwidth]{"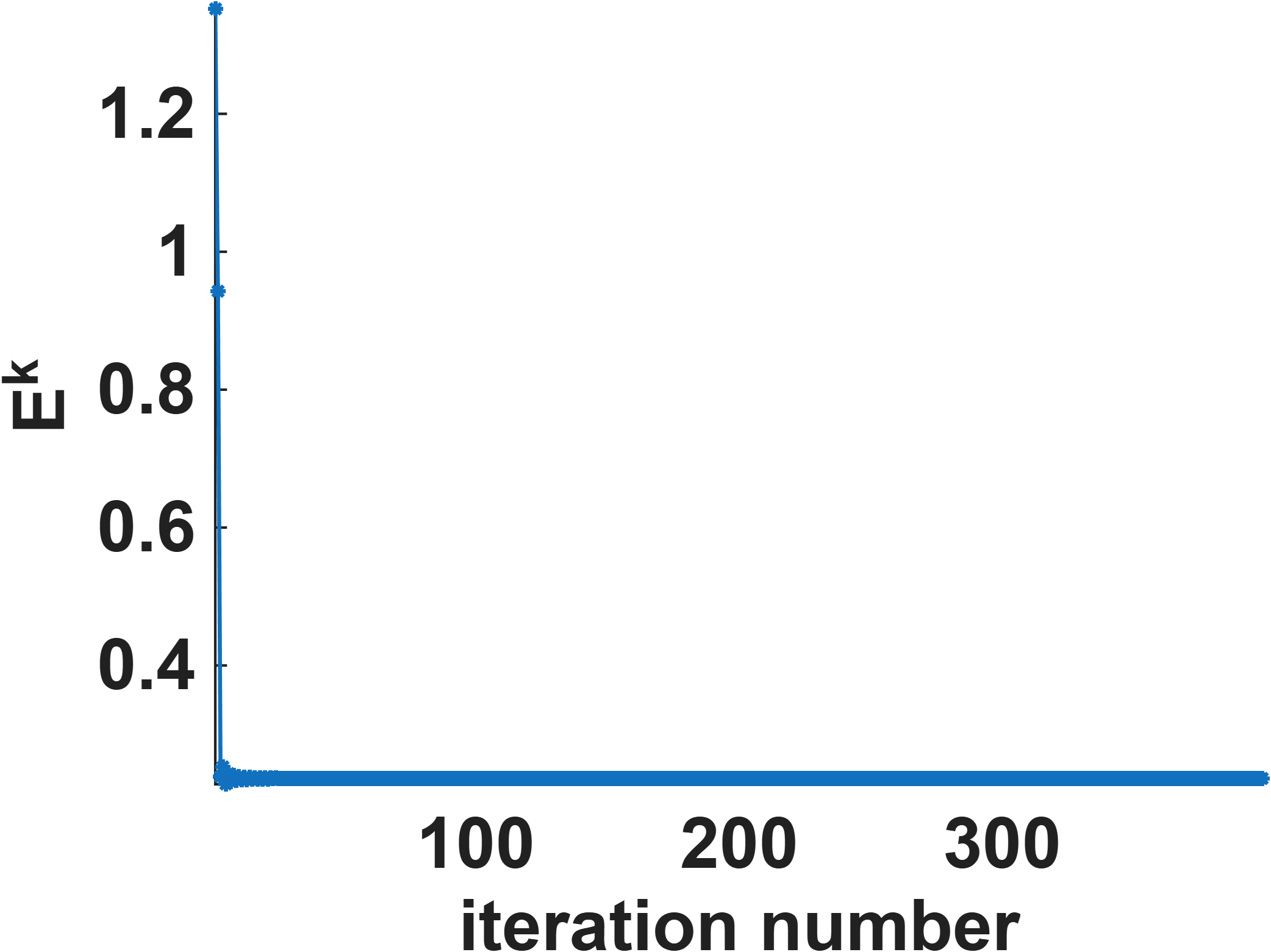"}
    \caption{SB-FD}
  \end{subfigure}\hfill
  \begin{subfigure}{0.24\textwidth}\centering
    \includegraphics[width=\textwidth]{"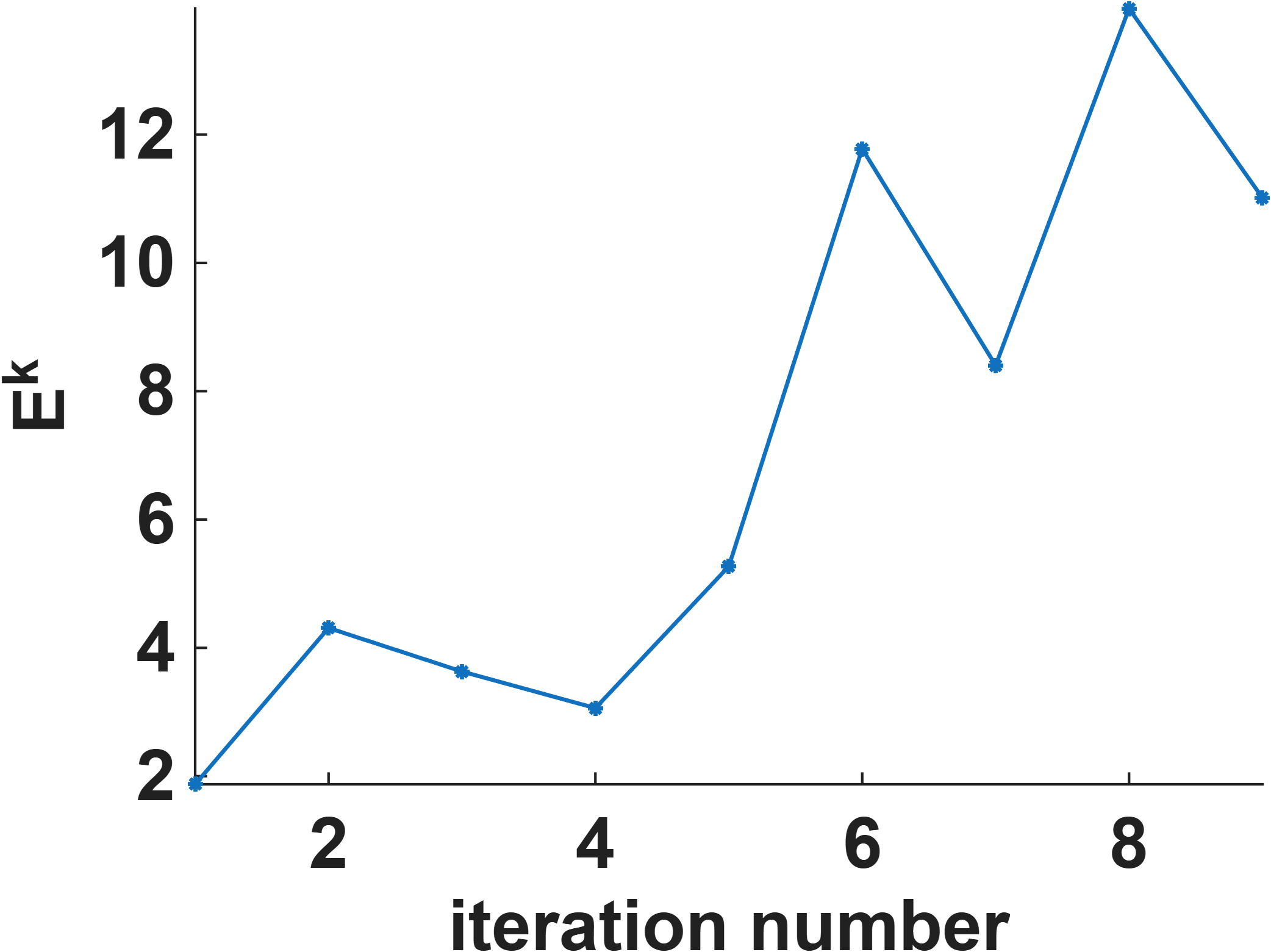"}
    \caption{SB-NM}
  \end{subfigure}\hfill
  \begin{subfigure}{0.24\textwidth}\centering
    \includegraphics[width=\textwidth]{"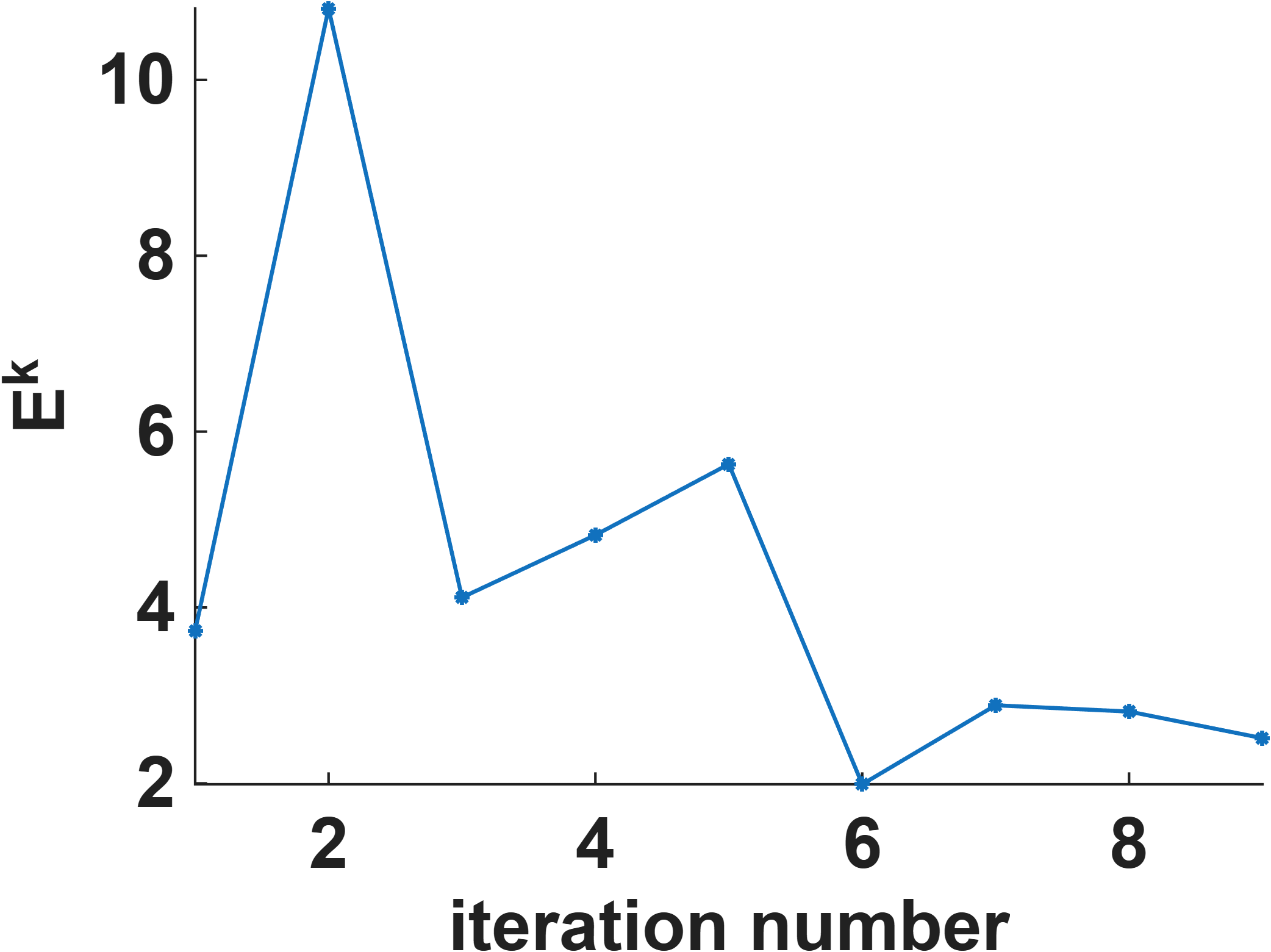"}
    \caption{SB-GD}
  \end{subfigure}

  \caption{Comparison of $E^k$ across methods for the 1D Ackley Function with $B=0$. Each subplot is one method (left-to-right, top-to-bottom).}
  \label{fig:swar:1DAck}
\end{figure}

\begin{figure}[H]
  \centering
  \captionsetup[subfigure]{justification=centering}

  \begin{subfigure}{0.24\textwidth}\centering
    \includegraphics[width=\textwidth]{"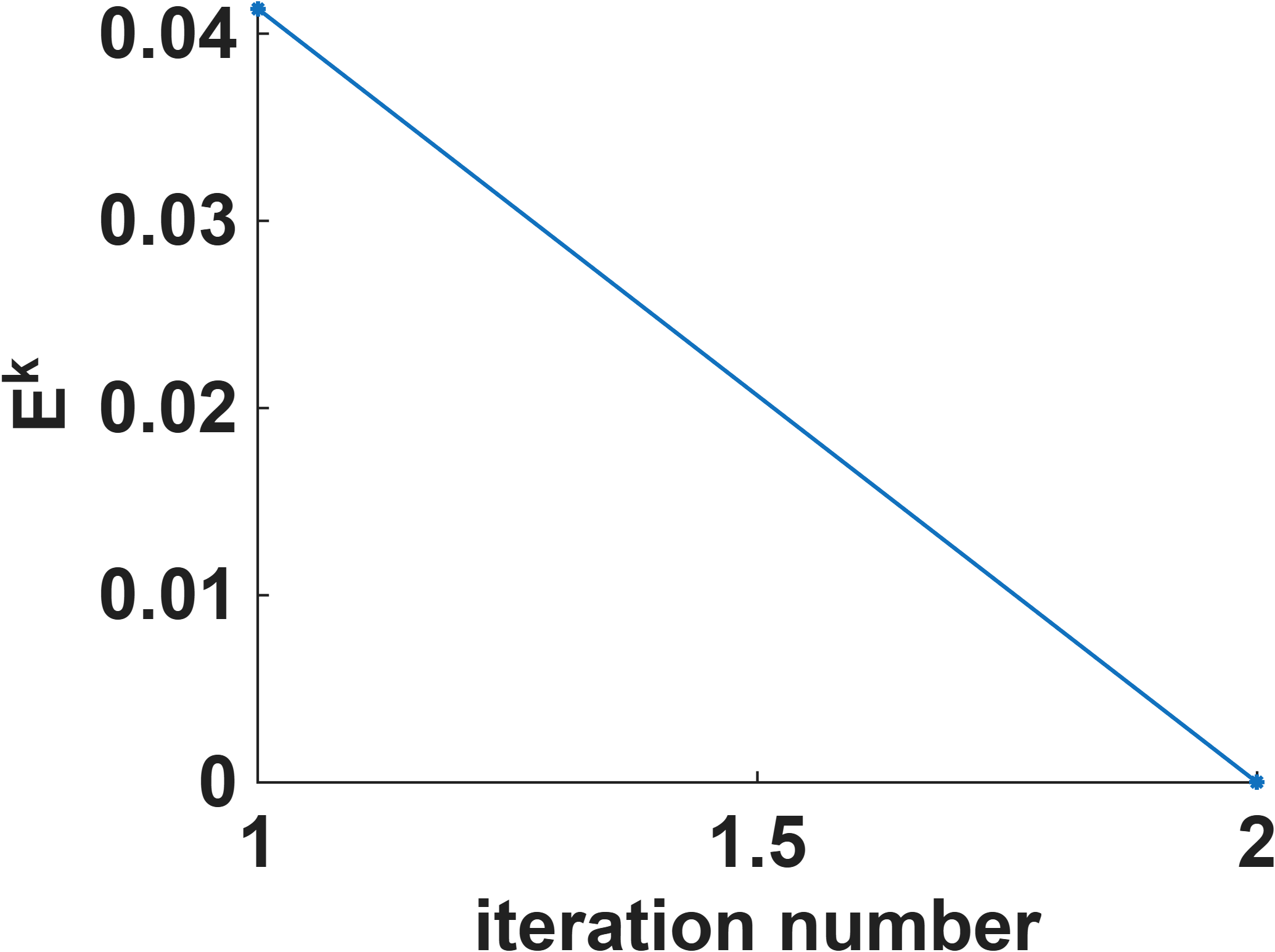"}
    \caption{SB-FB}
  \end{subfigure}\hfill
  \begin{subfigure}{0.24\textwidth}\centering
    \includegraphics[width=\textwidth]{"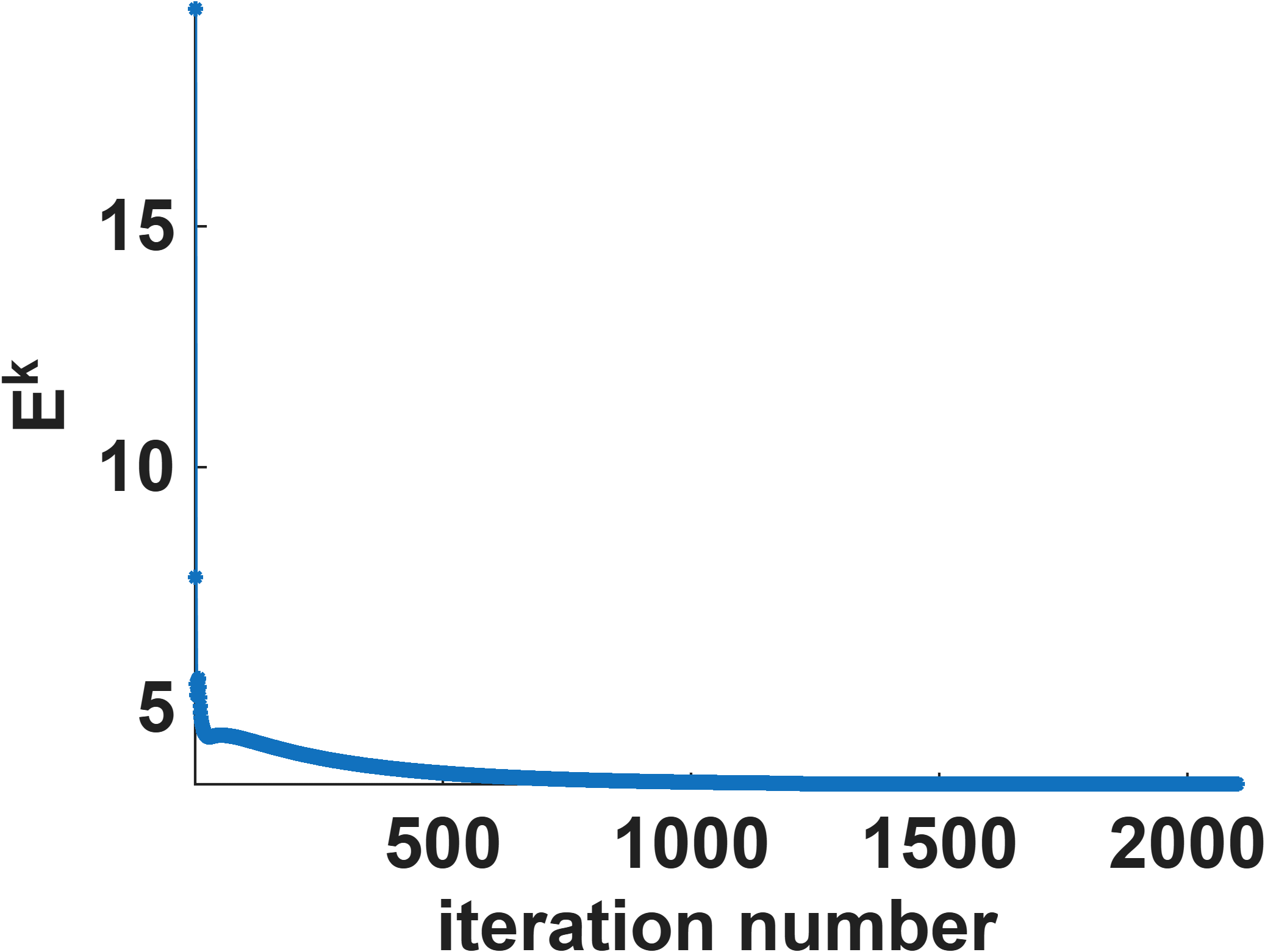"}
    \caption{SB-IMEX-RB}
  \end{subfigure}\hfill
  \begin{subfigure}{0.24\textwidth}\centering
    \includegraphics[width=\textwidth]{"2D_Ack_B0FB_Energy_i1000.png"}
    \caption{SB-Semi}
  \end{subfigure}\hfill
  \begin{subfigure}{0.24\textwidth}\centering
    \includegraphics[width=\textwidth]{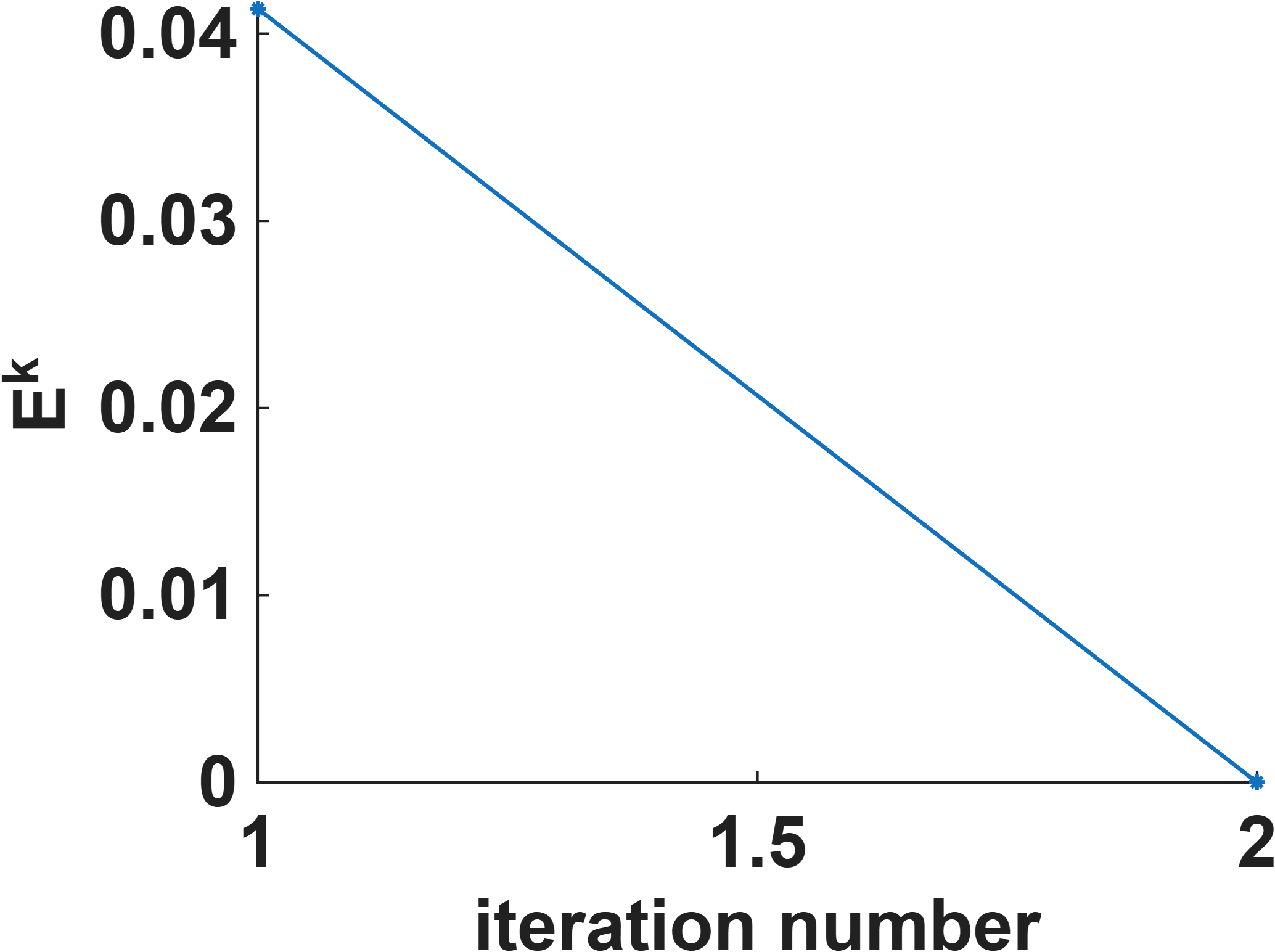}
    \caption{SB-IPAHD}
  \end{subfigure}

  \vspace{0.45em}

  \begin{subfigure}{0.24\textwidth}\centering
    \includegraphics[width=\textwidth]{"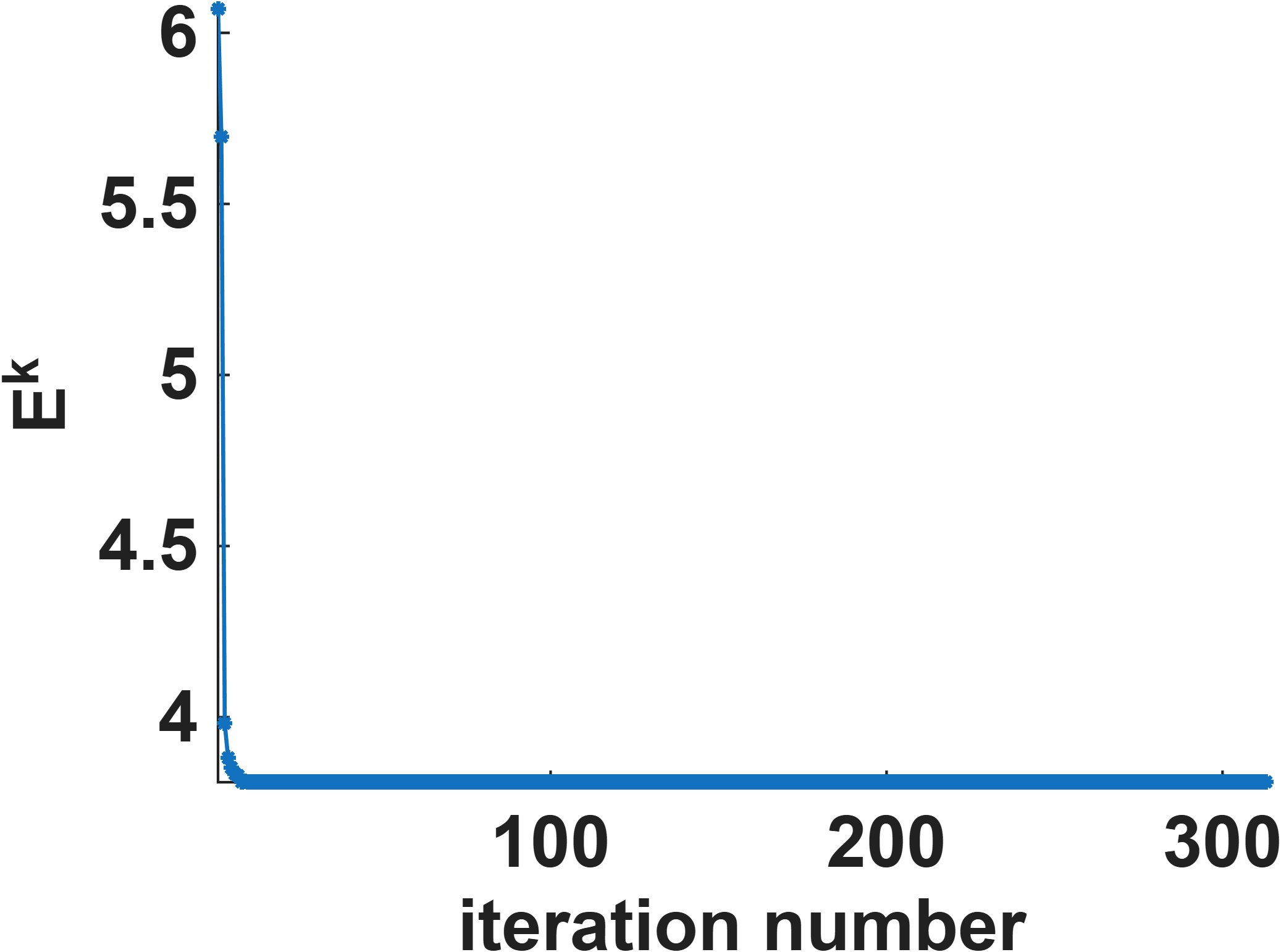"}
    \caption{SB-FD}
  \end{subfigure}\hfill
  \begin{subfigure}{0.24\textwidth}\centering
    \includegraphics[width=\textwidth]{"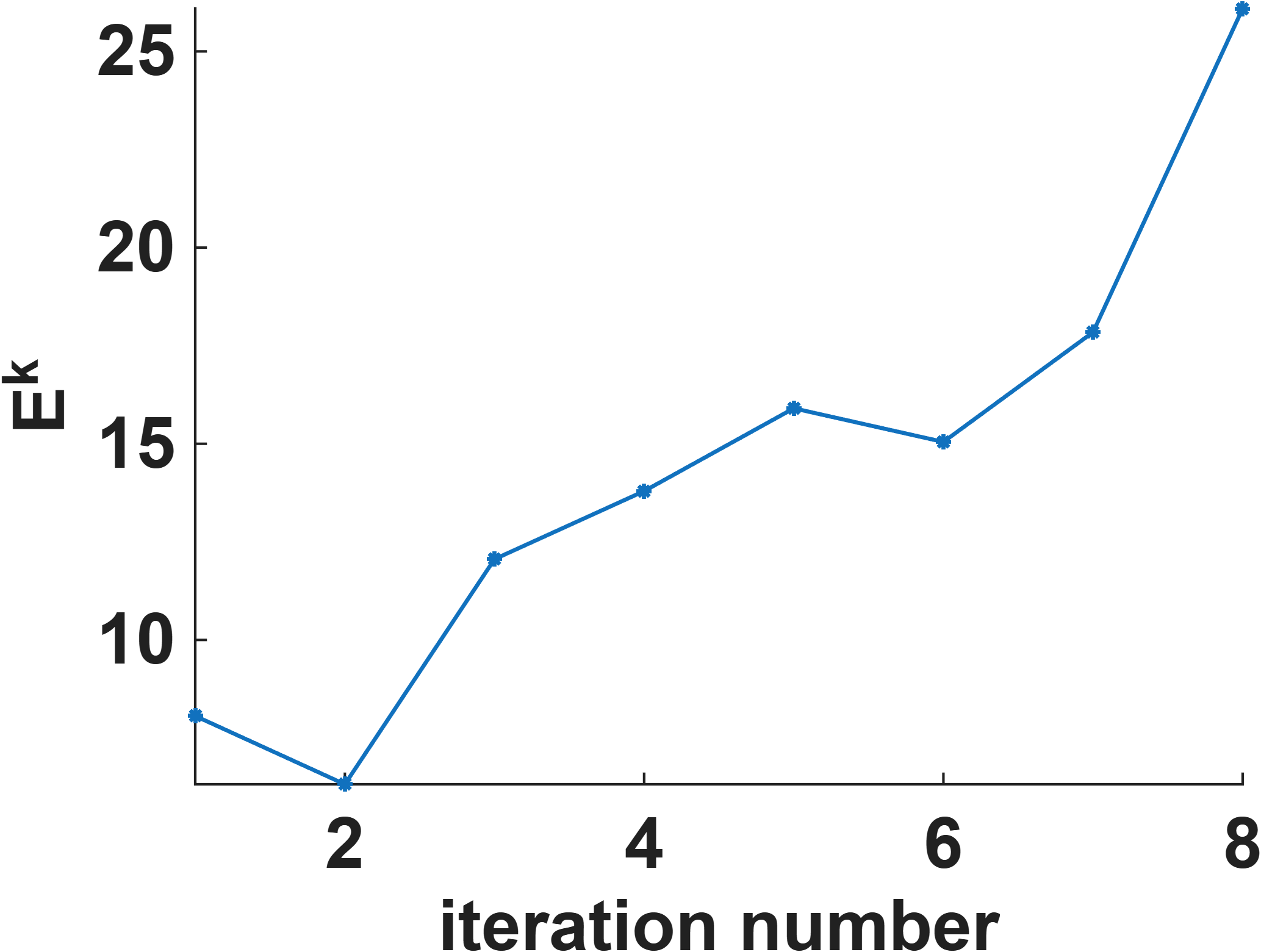"}
    \caption{SB-NM}
  \end{subfigure}\hfill
  \begin{subfigure}{0.24\textwidth}\centering
    \includegraphics[width=\textwidth]{"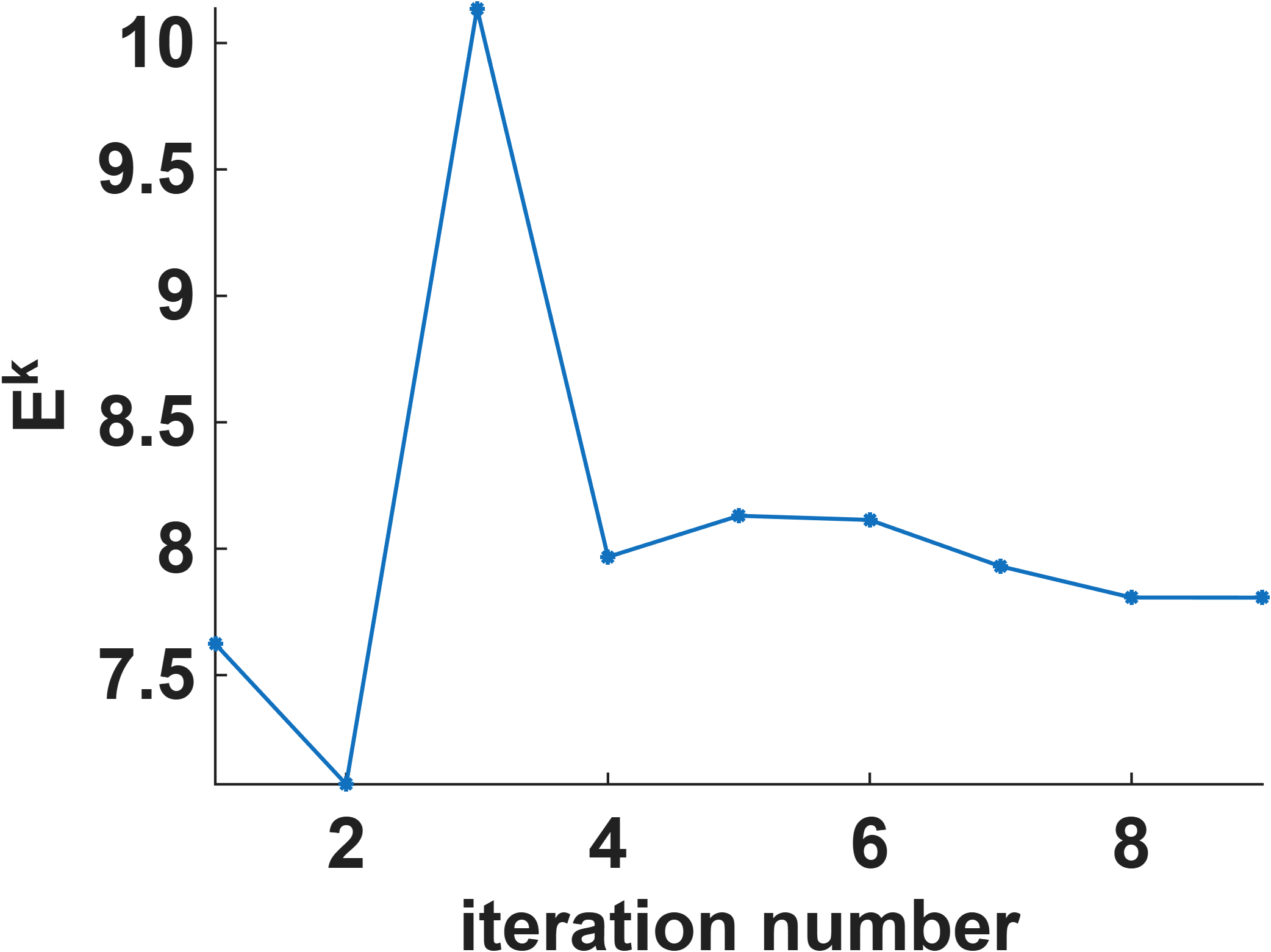"}
    \caption{SB-GD}
  \end{subfigure}

  \caption{Comparison of $E^k$ across methods for the 2D Ackley Function with $B=0$. Each subplot is one method (left-to-right, top-to-bottom).}
  \label{fig:swar:2DAck}
\end{figure}

\begin{figure}[H]
  \centering
  \captionsetup[subfigure]{justification=centering}

  \begin{subfigure}{0.24\textwidth}\centering
    \includegraphics[width=\textwidth]{"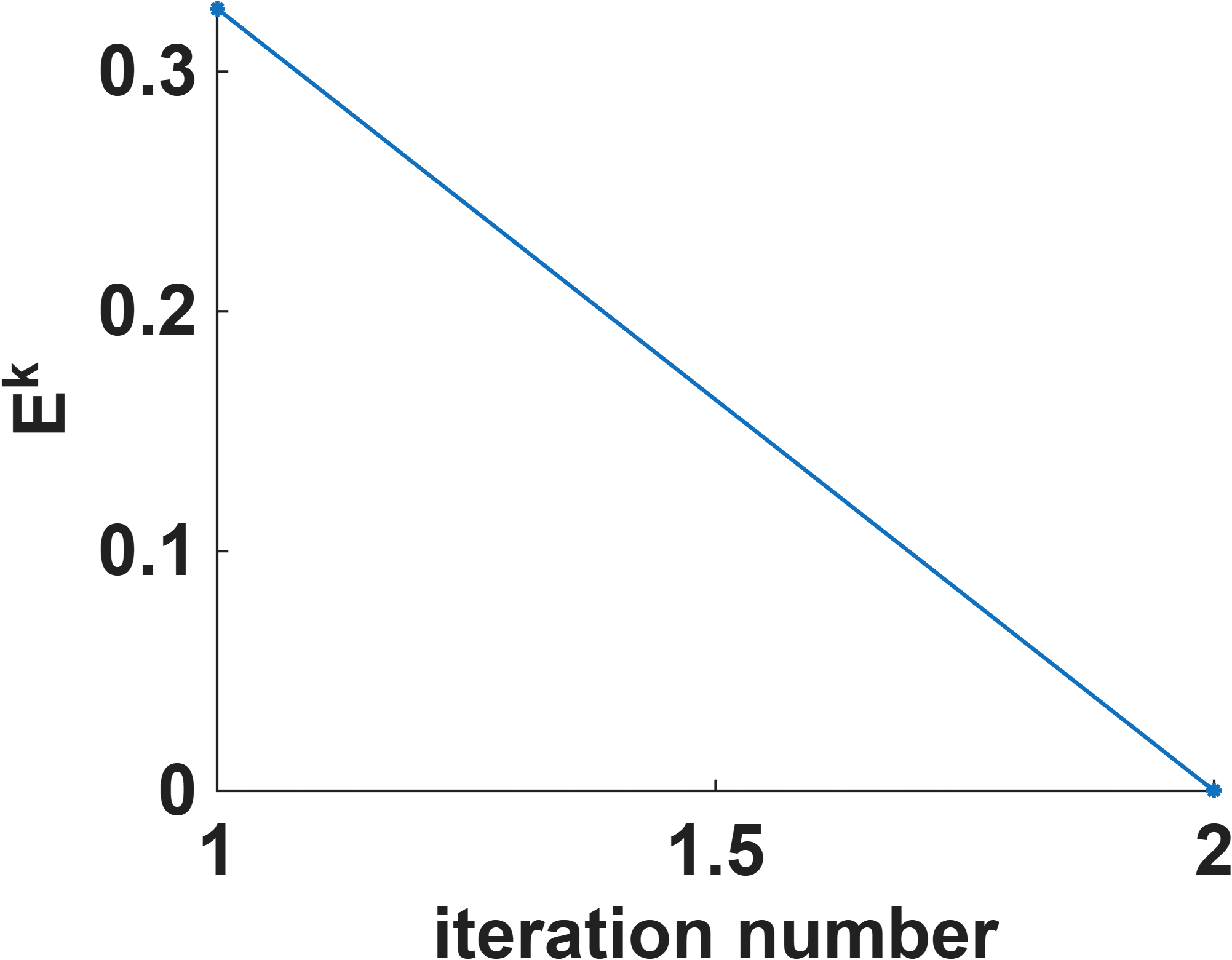"}
    \caption{SB-FB}
  \end{subfigure}\hfill
  \begin{subfigure}{0.24\textwidth}\centering
    \includegraphics[width=\textwidth]{"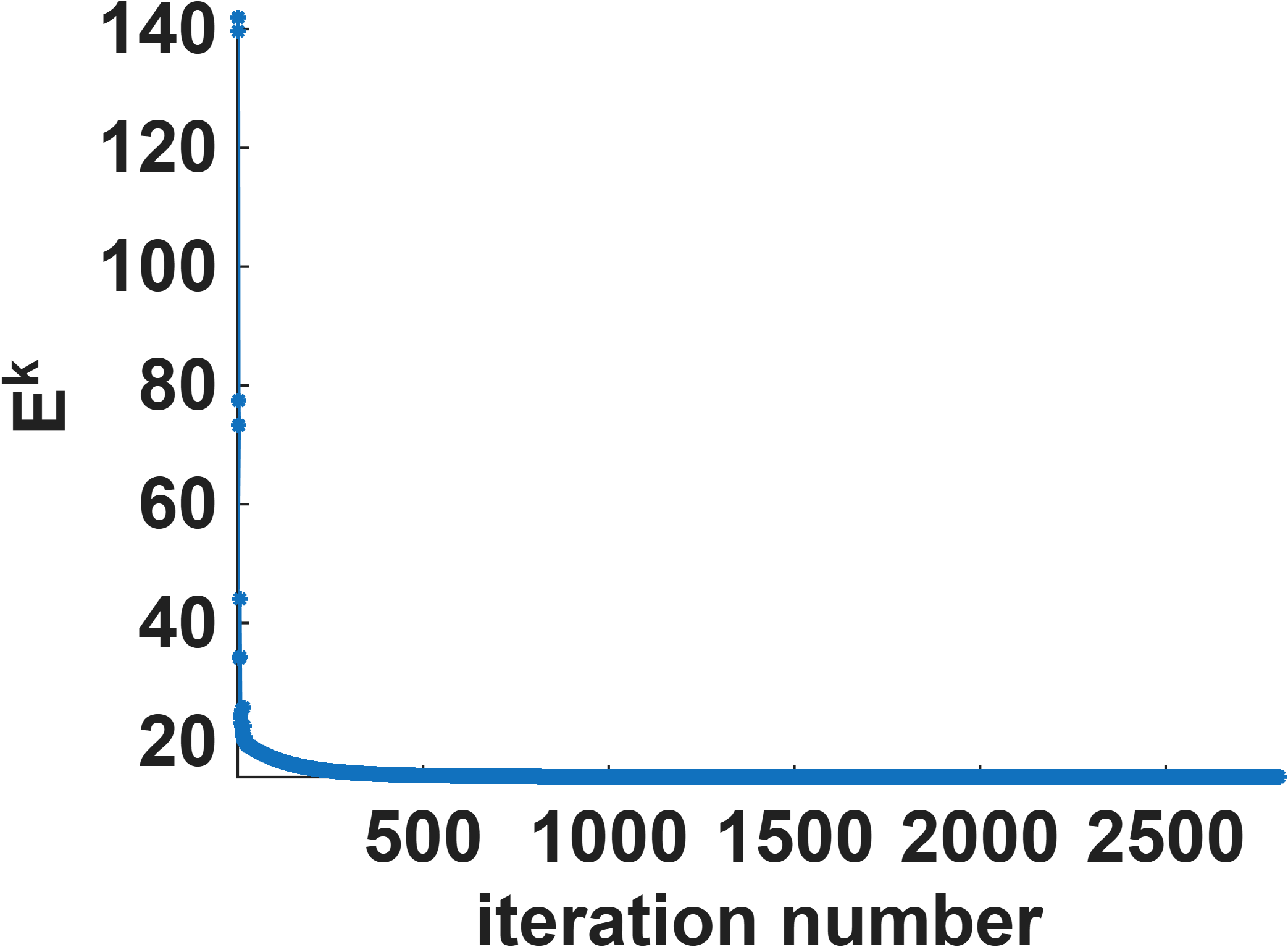"}
    \caption{SB-IMEX-RB}
  \end{subfigure}\hfill
  \begin{subfigure}{0.24\textwidth}\centering
    \includegraphics[width=\textwidth]{"10D_Ack_B0FB_Energy_i1000.png"}
    \caption{SB-Semi}
  \end{subfigure}\hfill
  \begin{subfigure}{0.24\textwidth}\centering
    \includegraphics[width=\textwidth]{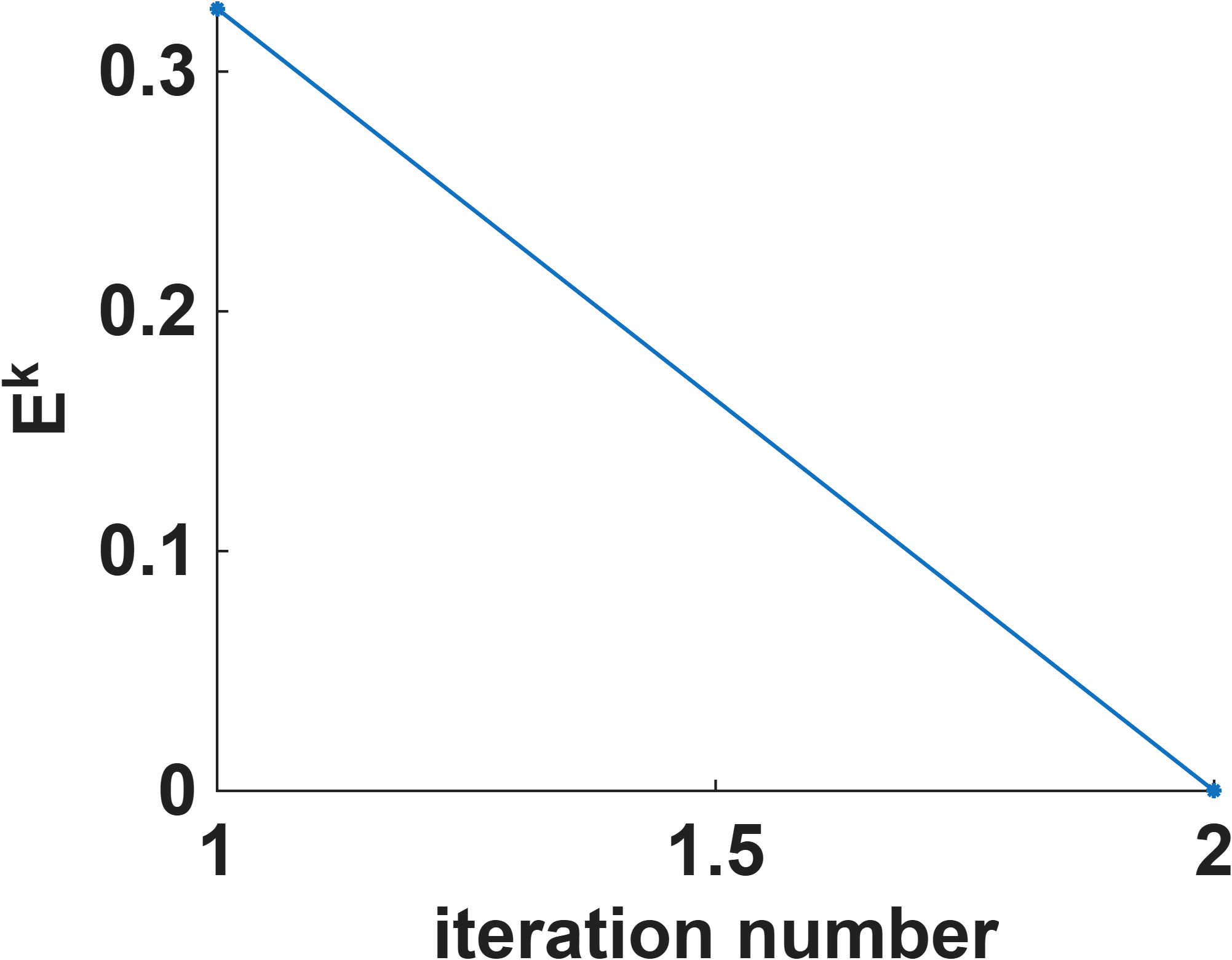}
    \caption{SB-IPAHD}
  \end{subfigure}

  \vspace{0.45em}

  \begin{subfigure}{0.24\textwidth}\centering
    \includegraphics[width=\textwidth]{"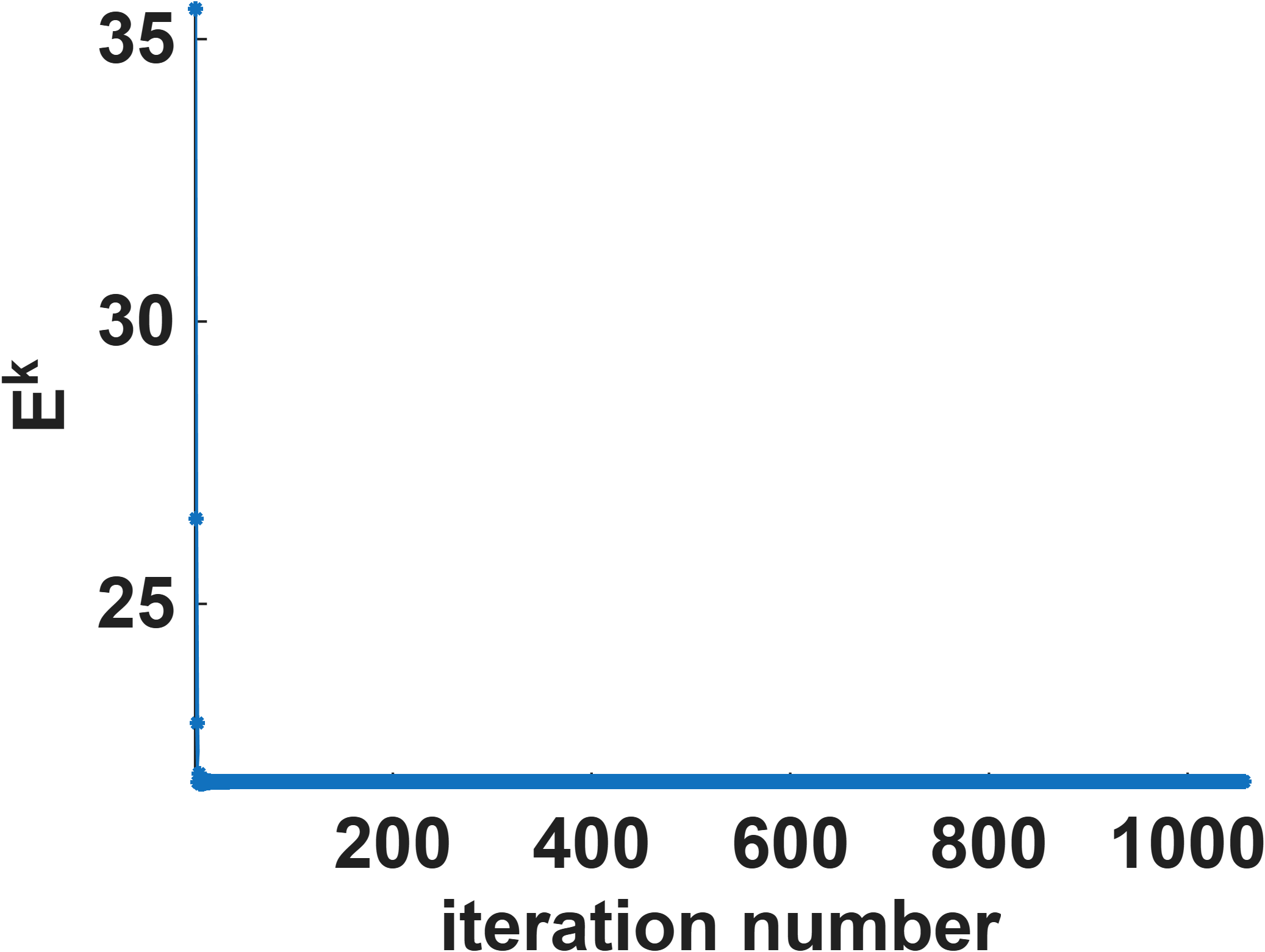"}
    \caption{SB-FD}
  \end{subfigure}\hfill
  \begin{subfigure}{0.24\textwidth}\centering
    \includegraphics[width=\textwidth]{"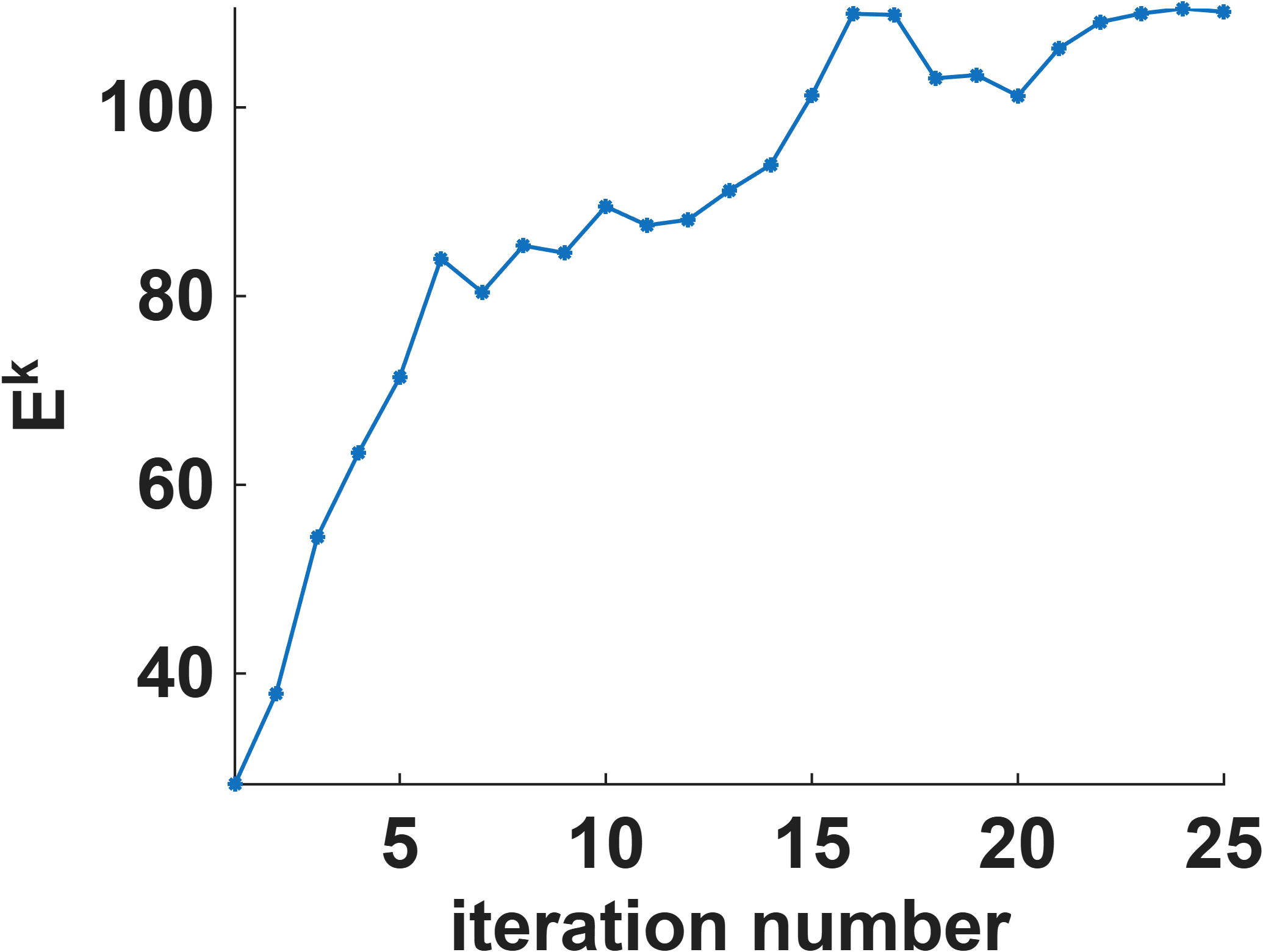"}
    \caption{SB-NM}
  \end{subfigure}\hfill
  \begin{subfigure}{0.24\textwidth}\centering
    \includegraphics[width=\textwidth]{"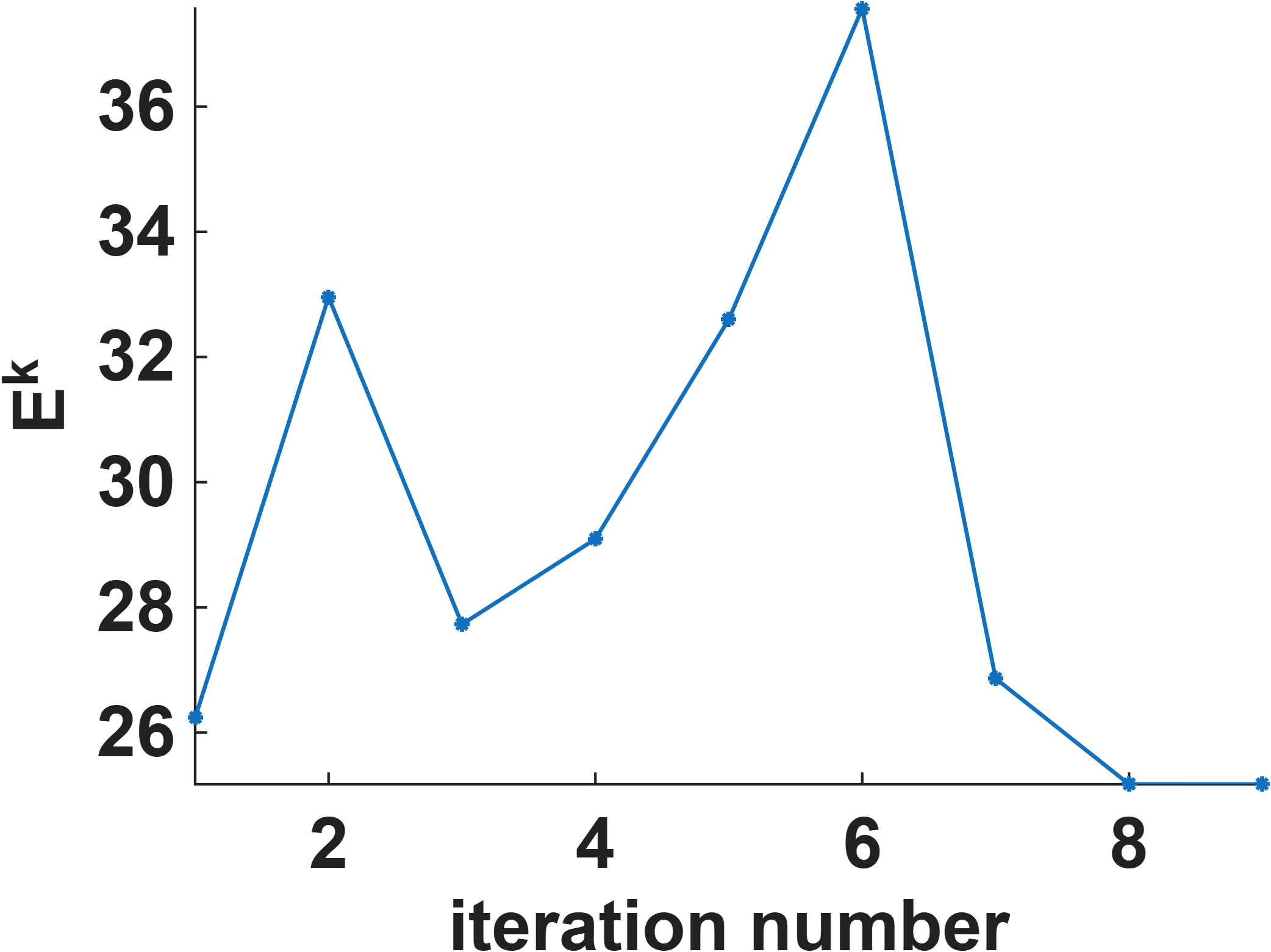"}
    \caption{SB-GD}
  \end{subfigure}

  \caption{Comparison of $E^k$ across methods for the 10D Ackley Function with $B=0$. Each subplot is one method (left-to-right, top-to-bottom).}
  \label{fig:swar:10DAck}
\end{figure}

\subsubsection{Rastrigin Function}
\begin{figure}[H]
  \centering
  \captionsetup[subfigure]{justification=centering}

  \begin{subfigure}{0.24\textwidth}\centering
    \includegraphics[width=\textwidth]{"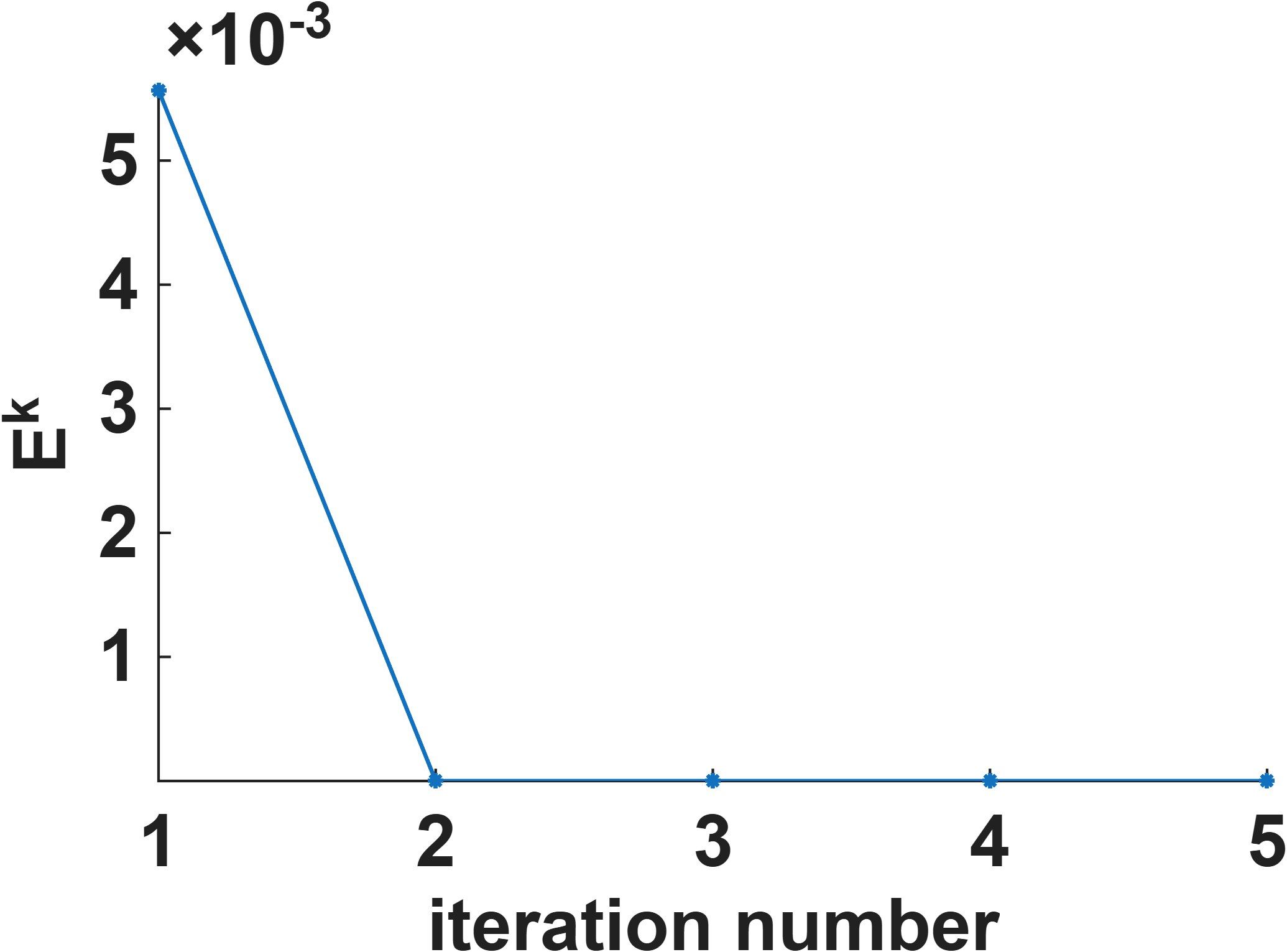"}
    \caption{SB-FB}
  \end{subfigure}\hfill
  \begin{subfigure}{0.24\textwidth}\centering
    \includegraphics[width=\textwidth]{"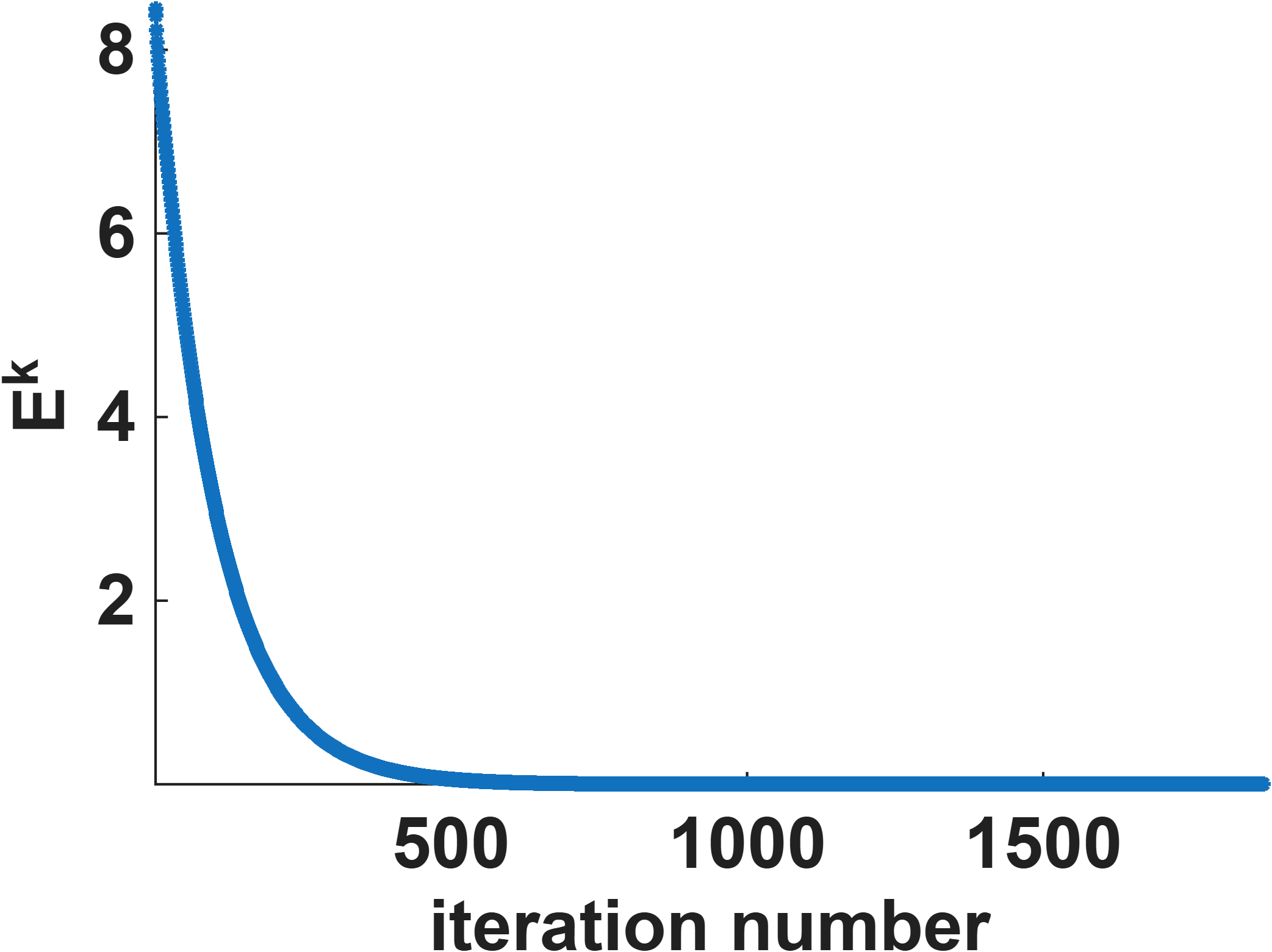"}
    \caption{SB-IMEX-RB}
  \end{subfigure}\hfill
  \begin{subfigure}{0.24\textwidth}\centering
    \includegraphics[width=\textwidth]{"1D_Ras_B0FB_Energy_i1000.png"}
    \caption{SB-Semi}
  \end{subfigure}\hfill
  \begin{subfigure}{0.24\textwidth}\centering
    \includegraphics[width=\textwidth]{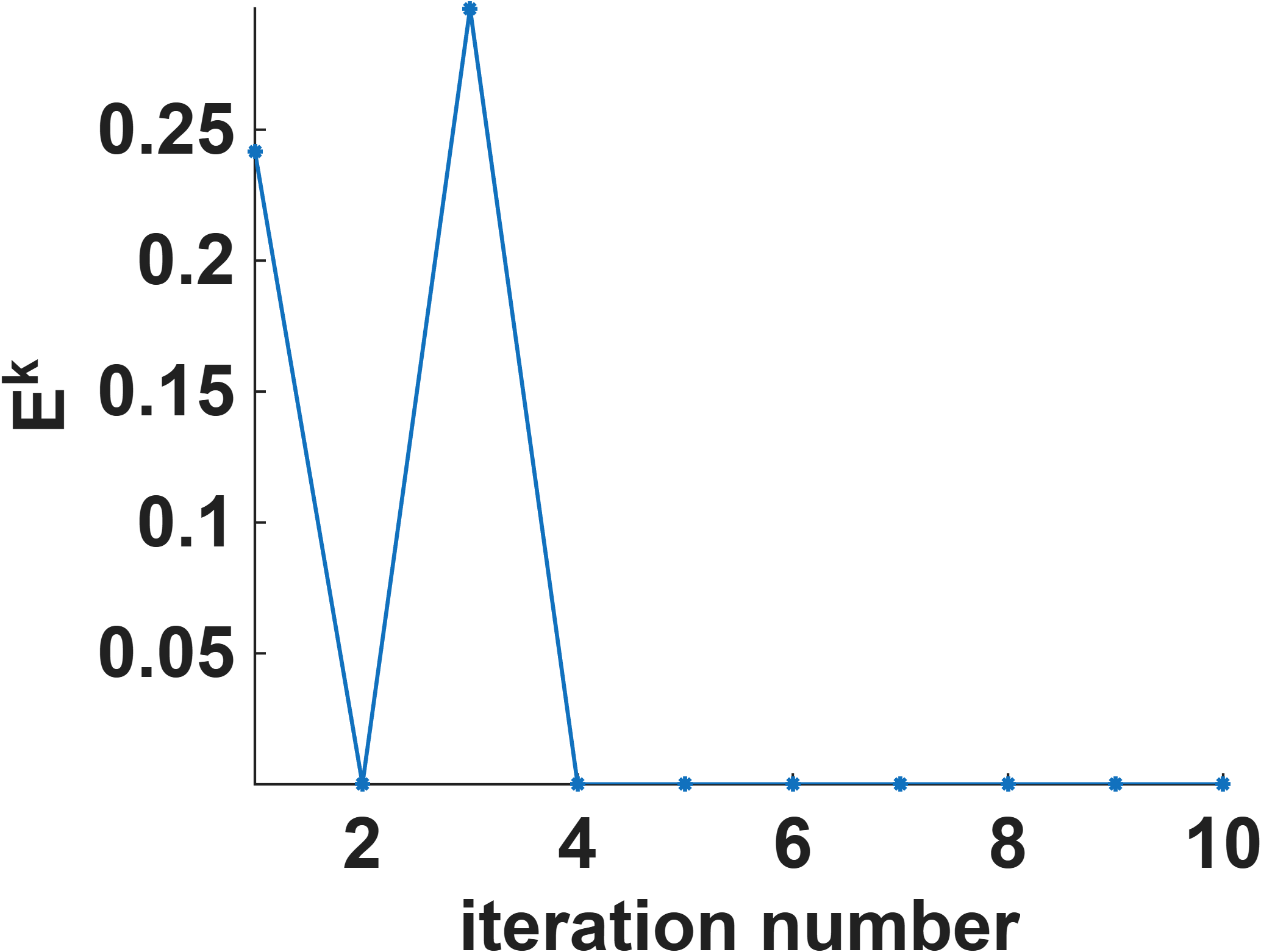}
    \caption{SB-IPAHD}
  \end{subfigure}

  \vspace{0.45em}

  \begin{subfigure}{0.24\textwidth}\centering
    \includegraphics[width=\textwidth]{"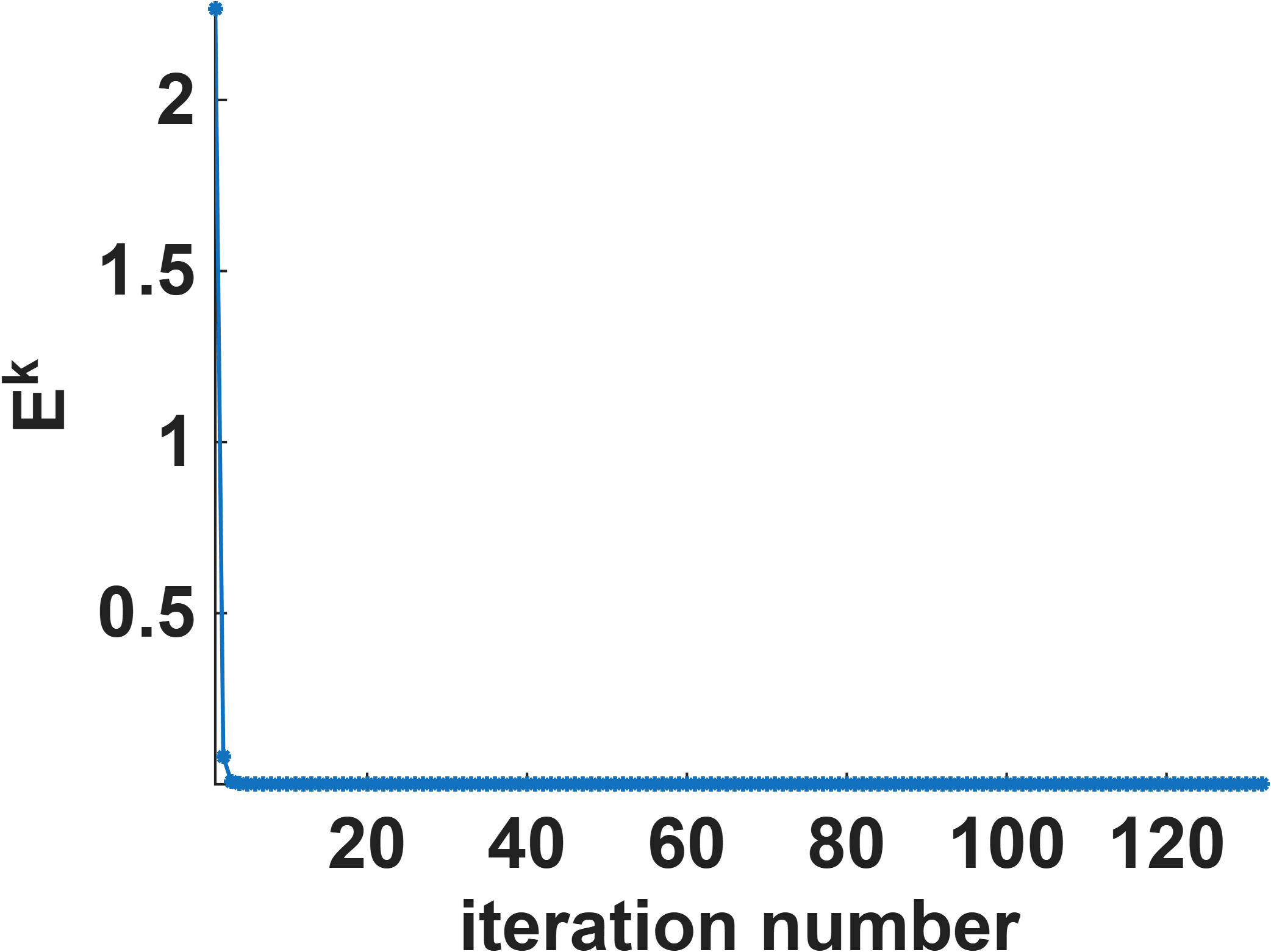"}
    \caption{SB-FD}
  \end{subfigure}\hfill
  \begin{subfigure}{0.24\textwidth}\centering
    \includegraphics[width=\textwidth]{"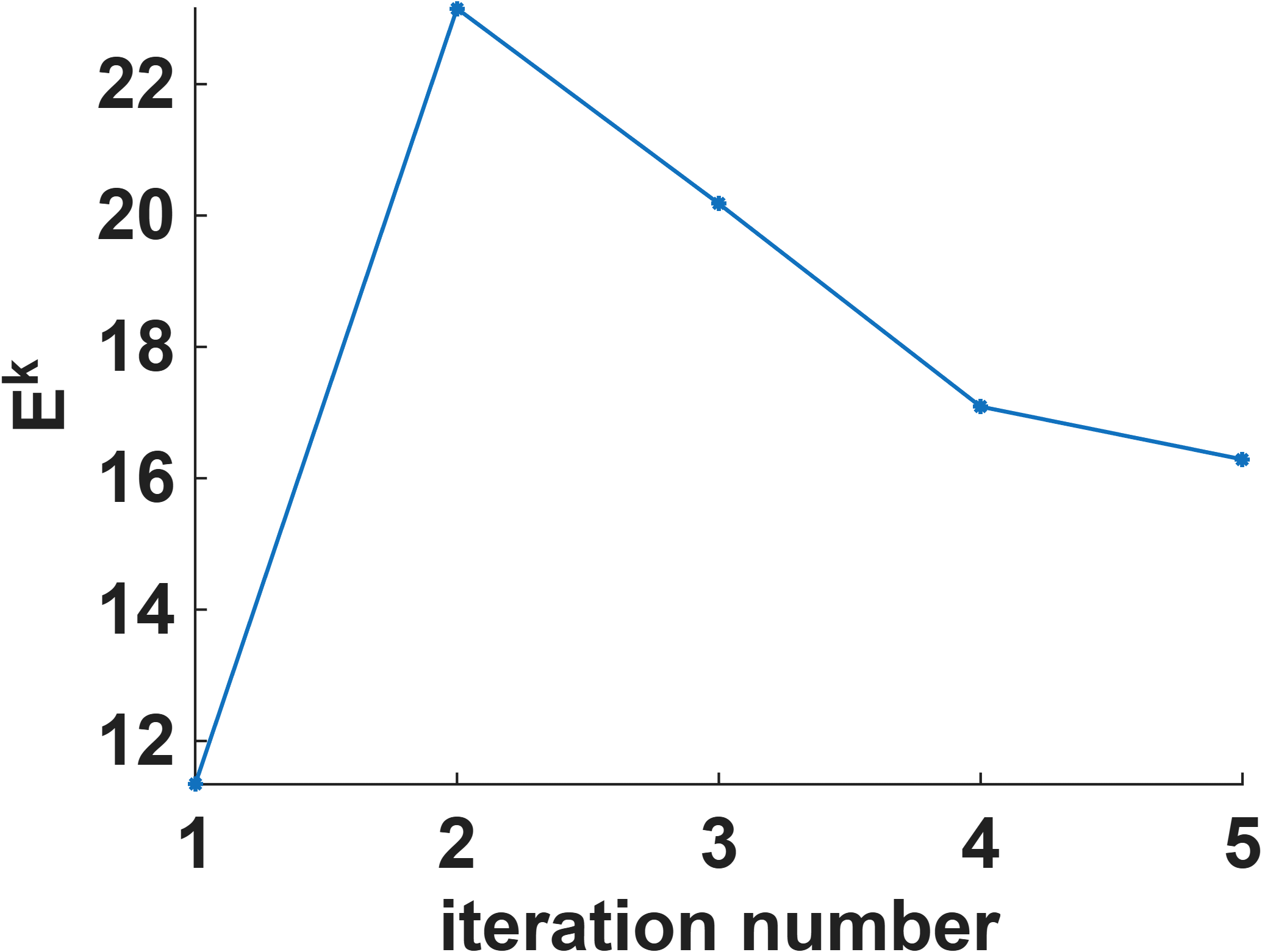"}
    \caption{SB-NM}
  \end{subfigure}\hfill
  \begin{subfigure}{0.24\textwidth}\centering
    \includegraphics[width=\textwidth]{"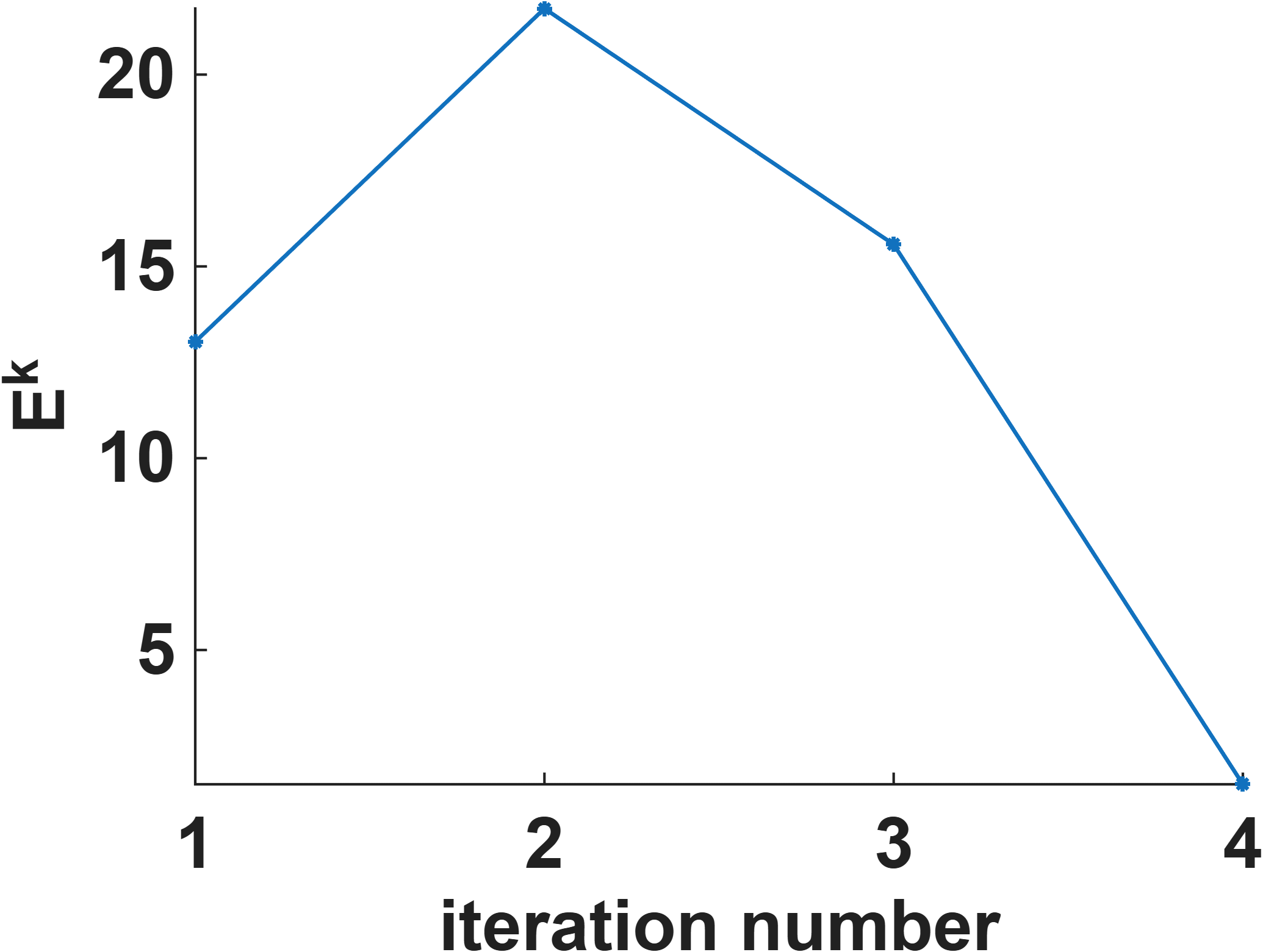"}
    \caption{SB-GD}
  \end{subfigure}

  \caption{Comparison of $E^k$ across methods for the 1D Rastrigin Function with $B=0$. Each subplot is one method (left-to-right, top-to-bottom).}
  \label{fig:swar:1DRas}
\end{figure}

\begin{figure}[H]
  \centering
  \captionsetup[subfigure]{justification=centering}

  \begin{subfigure}{0.24\textwidth}\centering
    \includegraphics[width=\textwidth]{"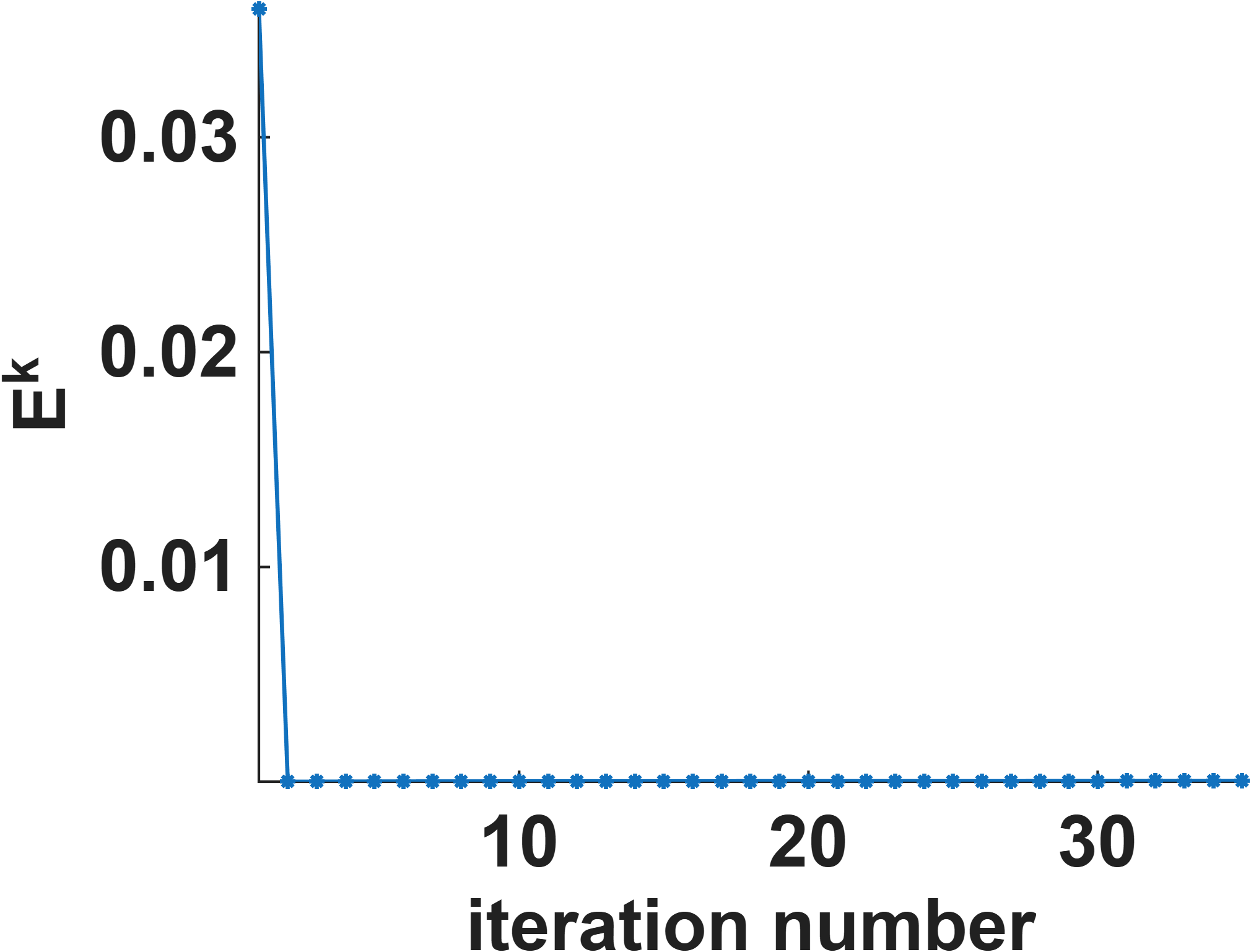"}
    \caption{SB-FB}
  \end{subfigure}\hfill
  \begin{subfigure}{0.24\textwidth}\centering
    \includegraphics[width=\textwidth]{"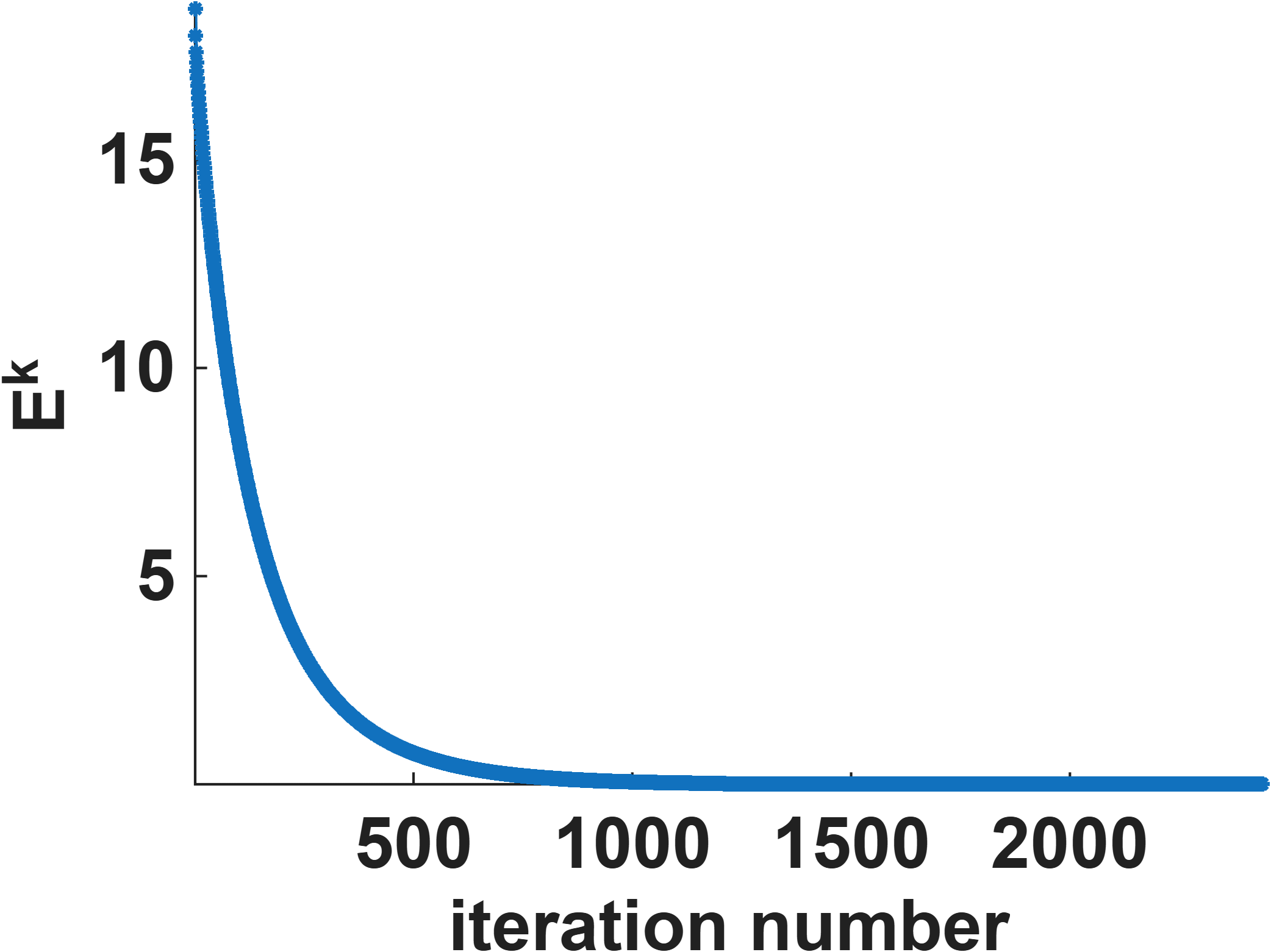"}
    \caption{SB-IMEX-RB}
  \end{subfigure}\hfill
  \begin{subfigure}{0.24\textwidth}\centering
    \includegraphics[width=\textwidth]{"2D_Ras_B0FB_Energy_i100.png"}
    \caption{SB-Semi}
  \end{subfigure}\hfill
  \begin{subfigure}{0.24\textwidth}\centering
    \includegraphics[width=\textwidth]{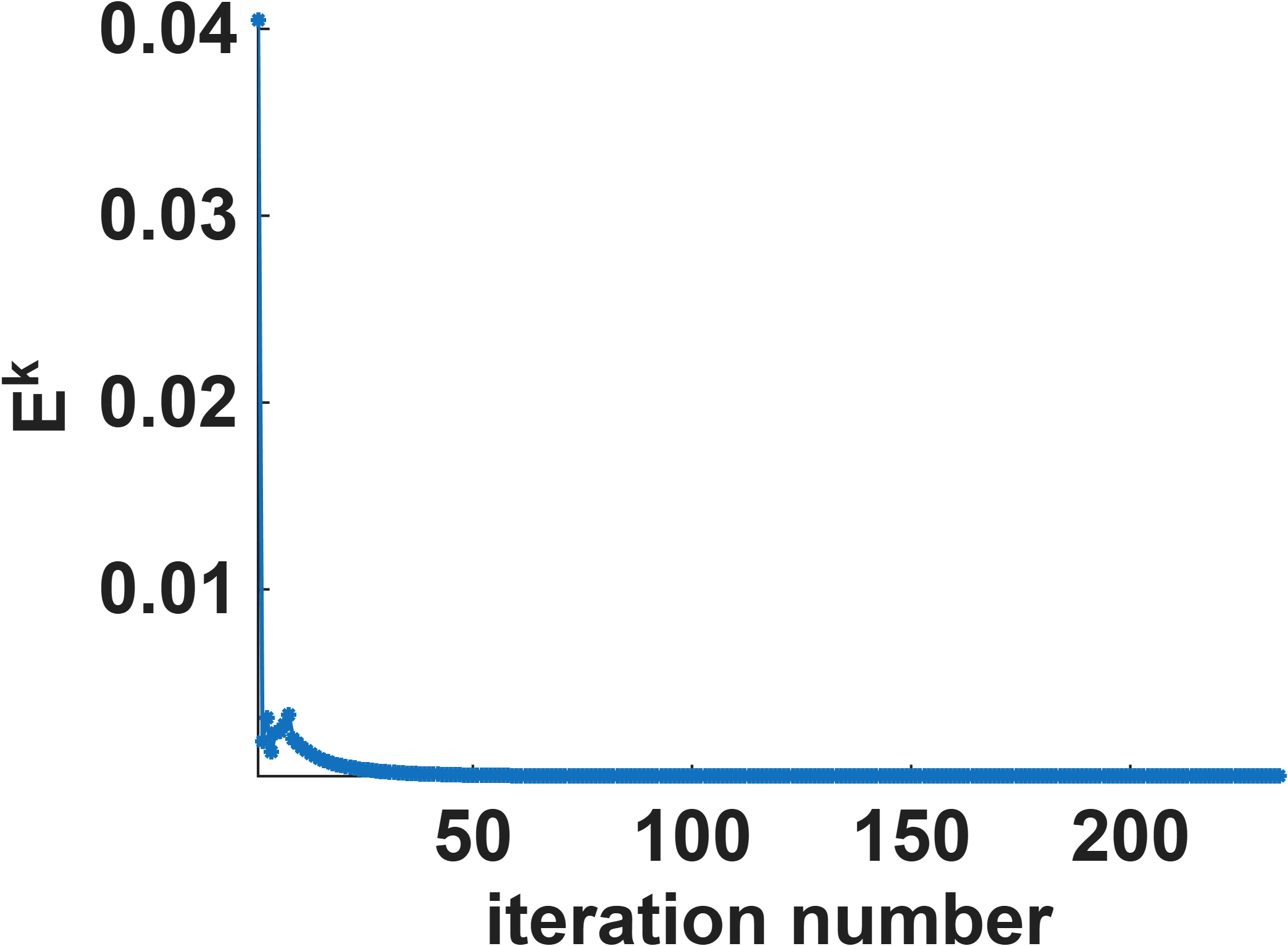}
    \caption{SB-IPAHD}
  \end{subfigure}

  \vspace{0.45em}

  \begin{subfigure}{0.24\textwidth}\centering
    \includegraphics[width=\textwidth]{"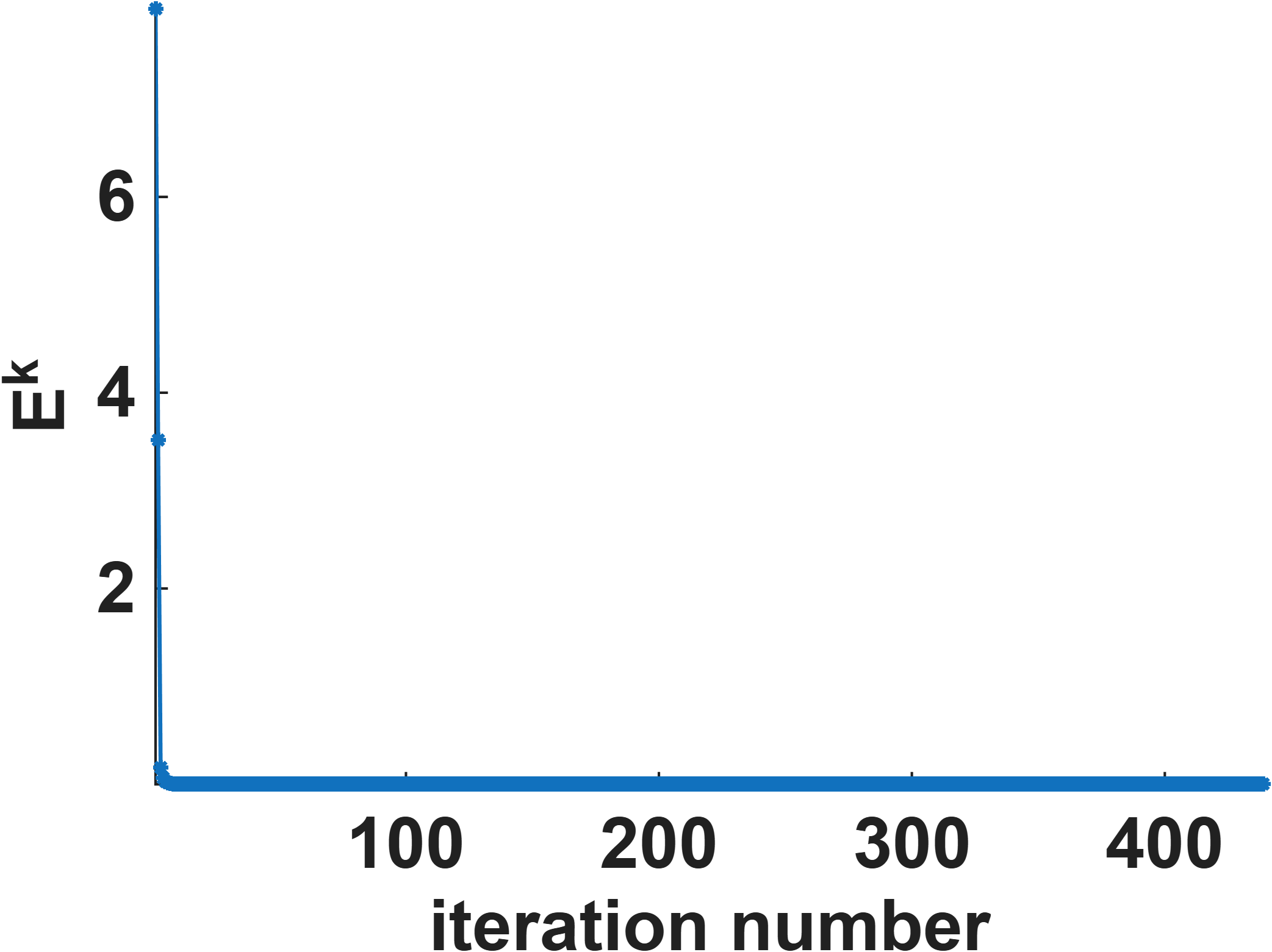"}
    \caption{SB-FD}
  \end{subfigure}\hfill
  \begin{subfigure}{0.24\textwidth}\centering
    \includegraphics[width=\textwidth]{"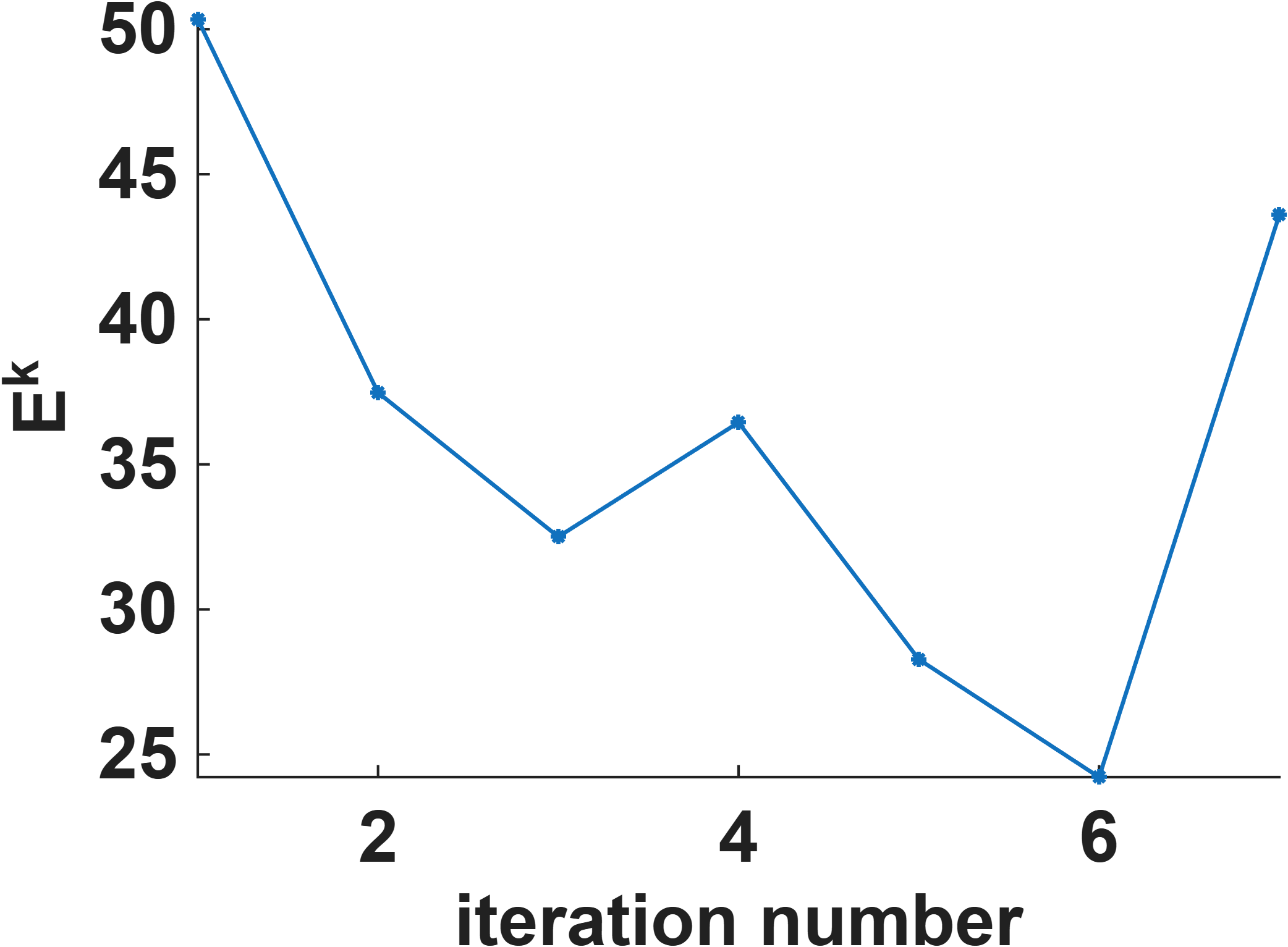"}
    \caption{SB-NM}
  \end{subfigure}\hfill
  \begin{subfigure}{0.24\textwidth}\centering
    \includegraphics[width=\textwidth]{"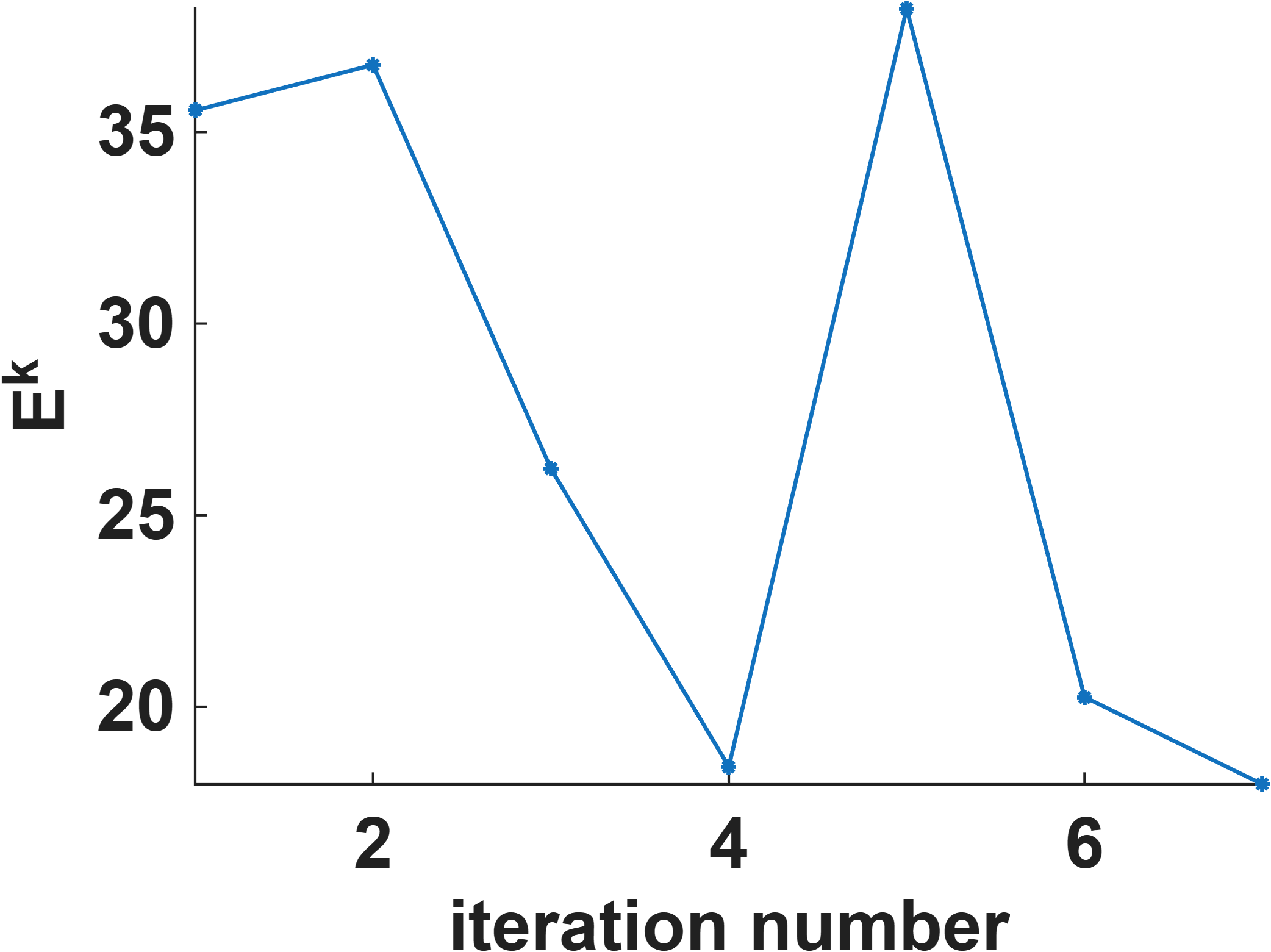"}
    \caption{SB-GD}
  \end{subfigure}

  \caption{Comparison of $E^k$ across methods for the 2D Rastrigin Function with $B=0$. Each subplot is one method (left-to-right, top-to-bottom).}
  \label{fig:swar:2DRas}
\end{figure}

\begin{figure}[H]
  \centering
  \captionsetup[subfigure]{justification=centering}

  \begin{subfigure}{0.24\textwidth}\centering
    \includegraphics[width=\textwidth]{"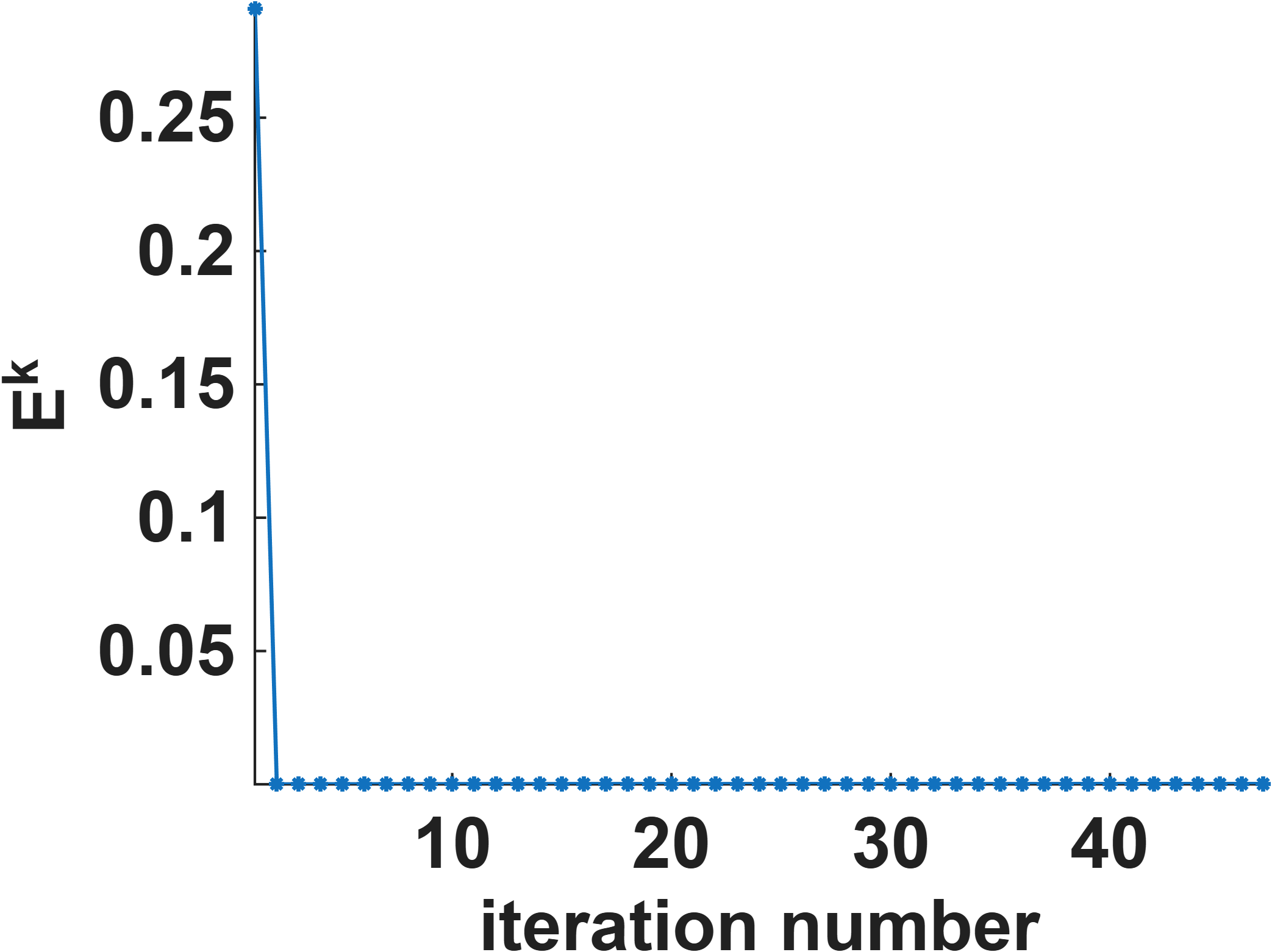"}
    \caption{SB-FB}
  \end{subfigure}\hfill
  \begin{subfigure}{0.24\textwidth}\centering
    \includegraphics[width=\textwidth]{"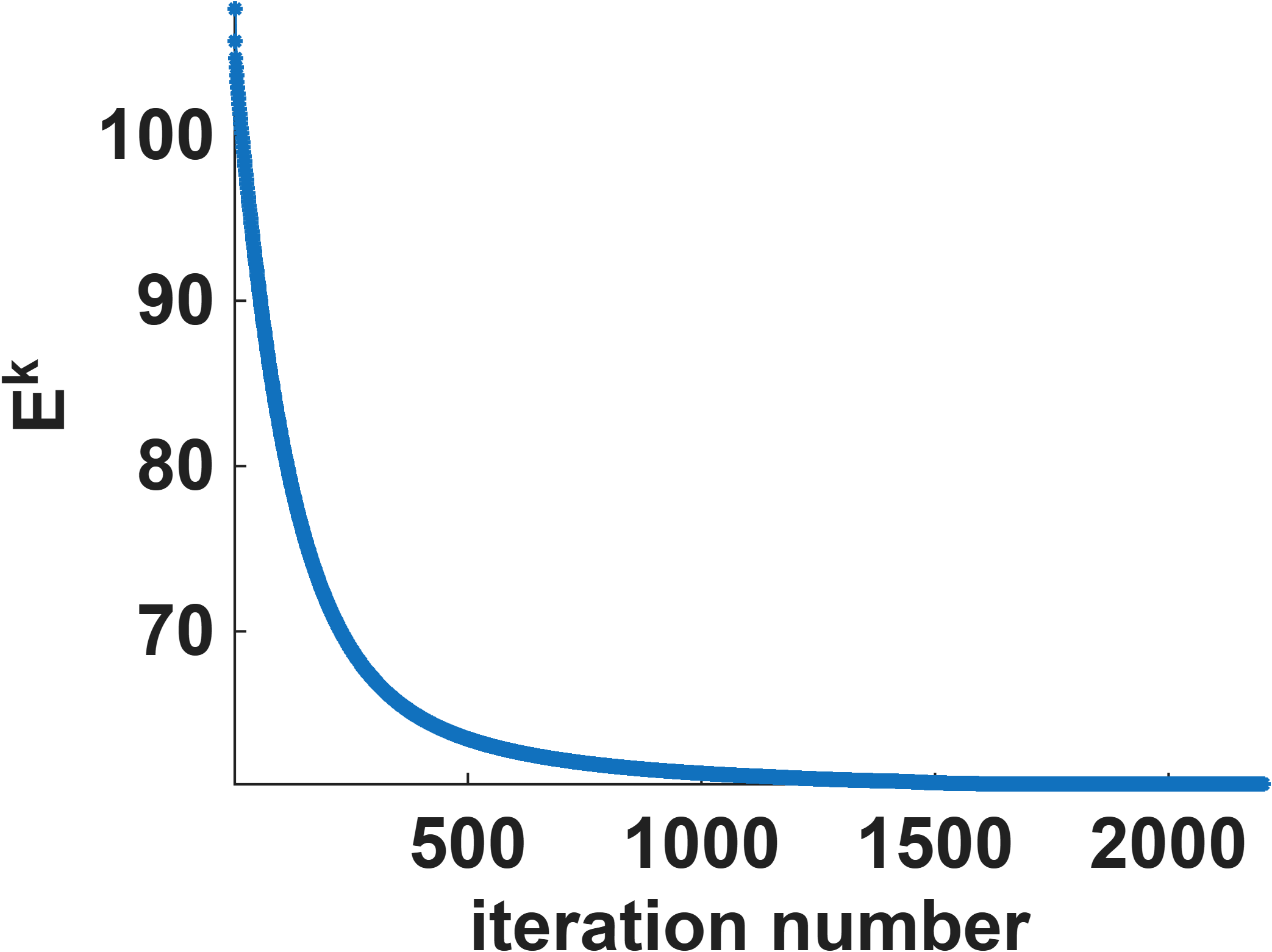"}
    \caption{SB-IMEX-RB}
  \end{subfigure}\hfill
  \begin{subfigure}{0.24\textwidth}\centering
    \includegraphics[width=\textwidth]{"10D_Ras_B0FB_Energy_i1000.png"}
    \caption{SB-Semi}
  \end{subfigure}\hfill
  \begin{subfigure}{0.24\textwidth}\centering
    \includegraphics[width=\textwidth]{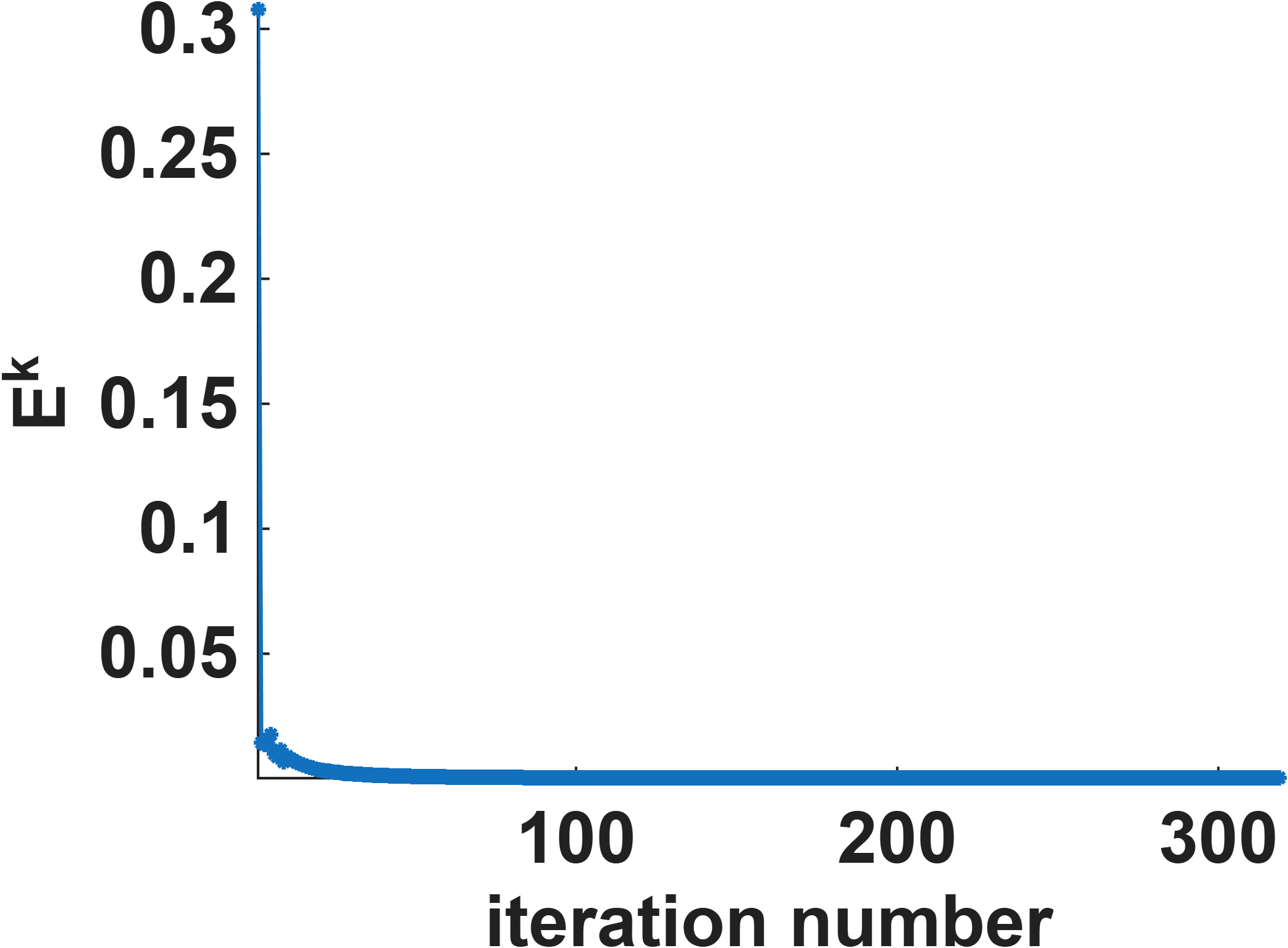}
    \caption{SB-IPAHD}
  \end{subfigure}

  \vspace{0.45em}

  \begin{subfigure}{0.24\textwidth}\centering
    \includegraphics[width=\textwidth]{"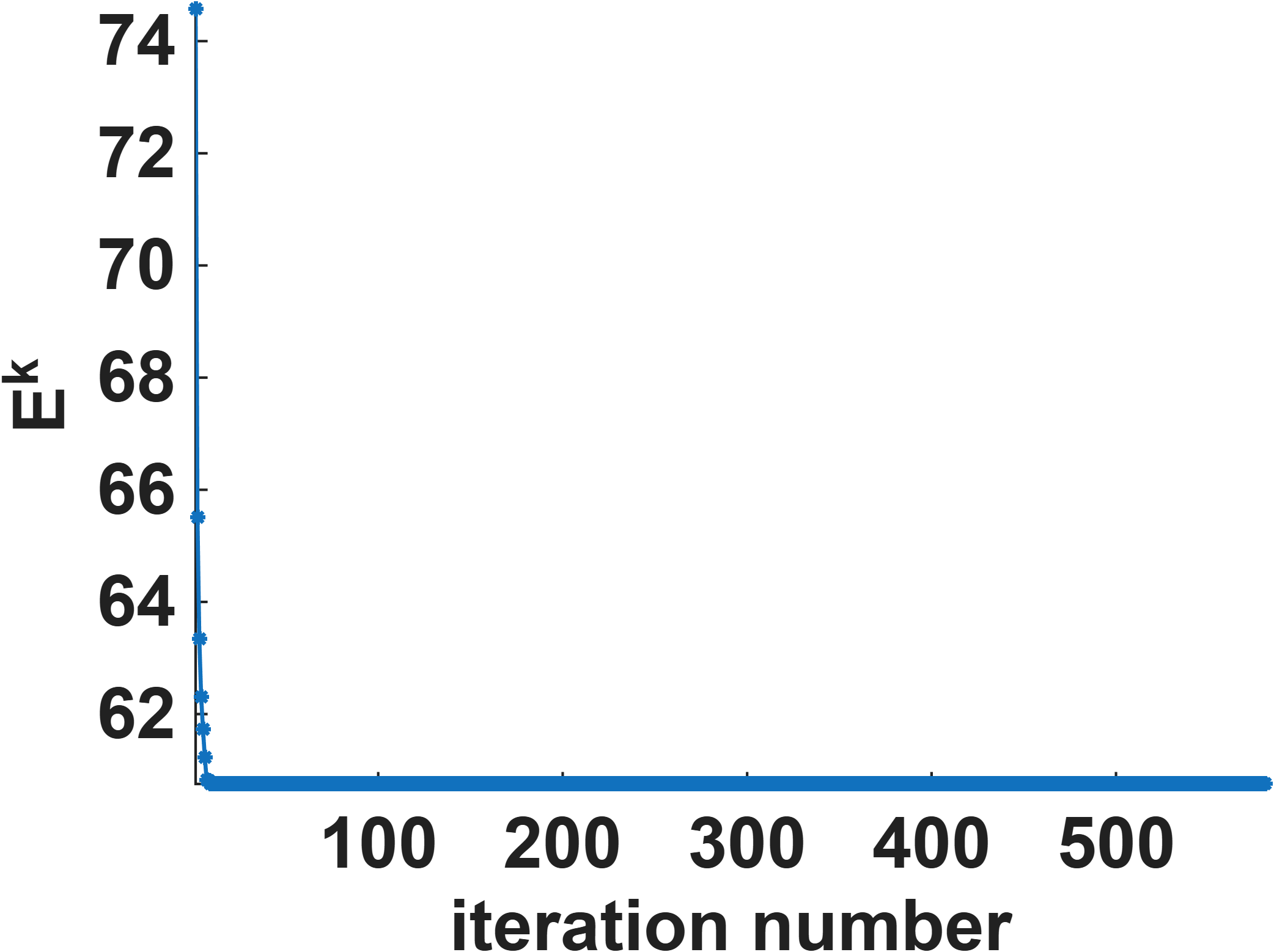"}
    \caption{SB-FD}
  \end{subfigure}\hfill
  \begin{subfigure}{0.24\textwidth}\centering
    \includegraphics[width=\textwidth]{"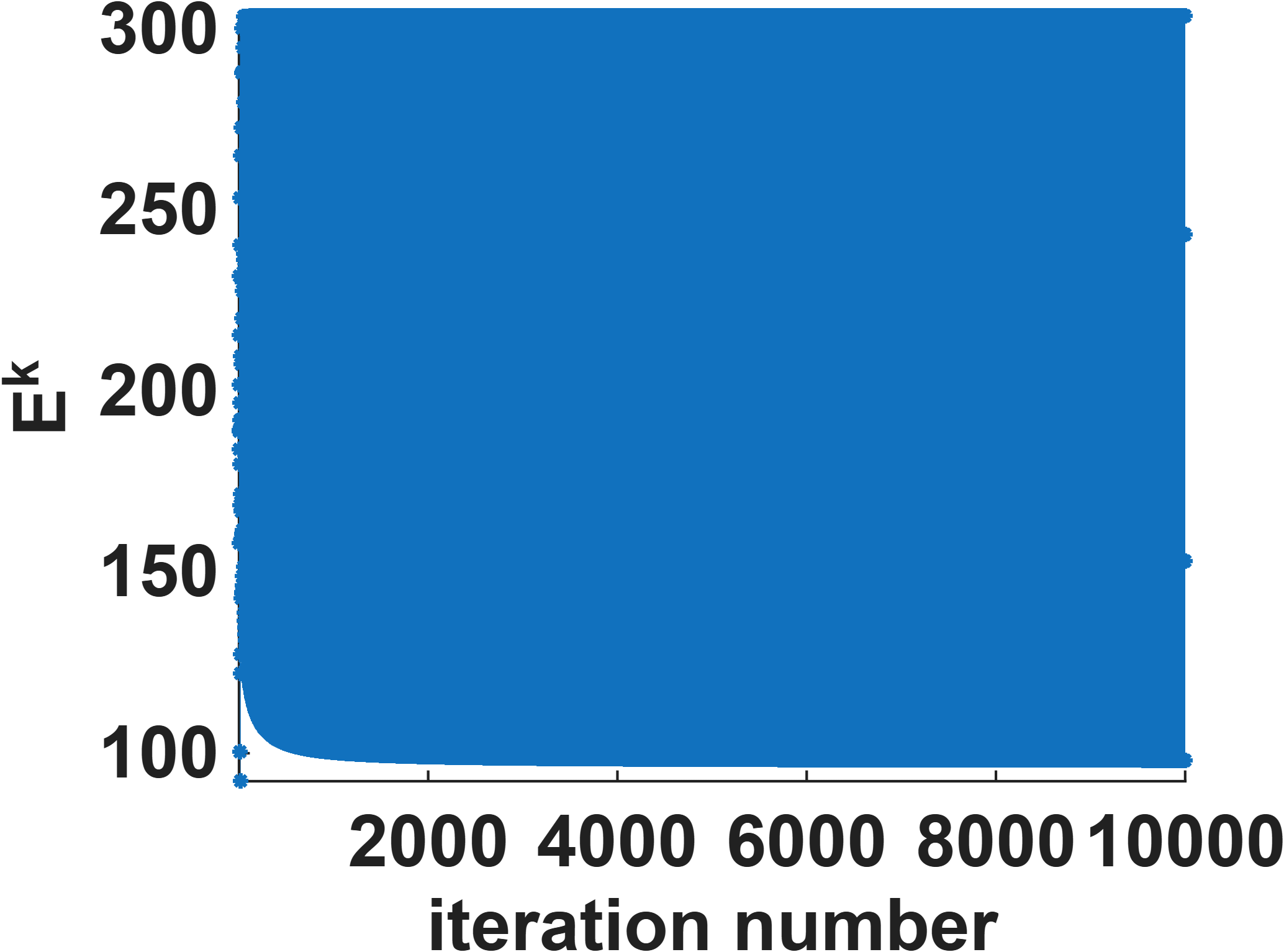"}
    \caption{SB-NM}
  \end{subfigure}\hfill
  \begin{subfigure}{0.24\textwidth}\centering
    \includegraphics[width=\textwidth]{"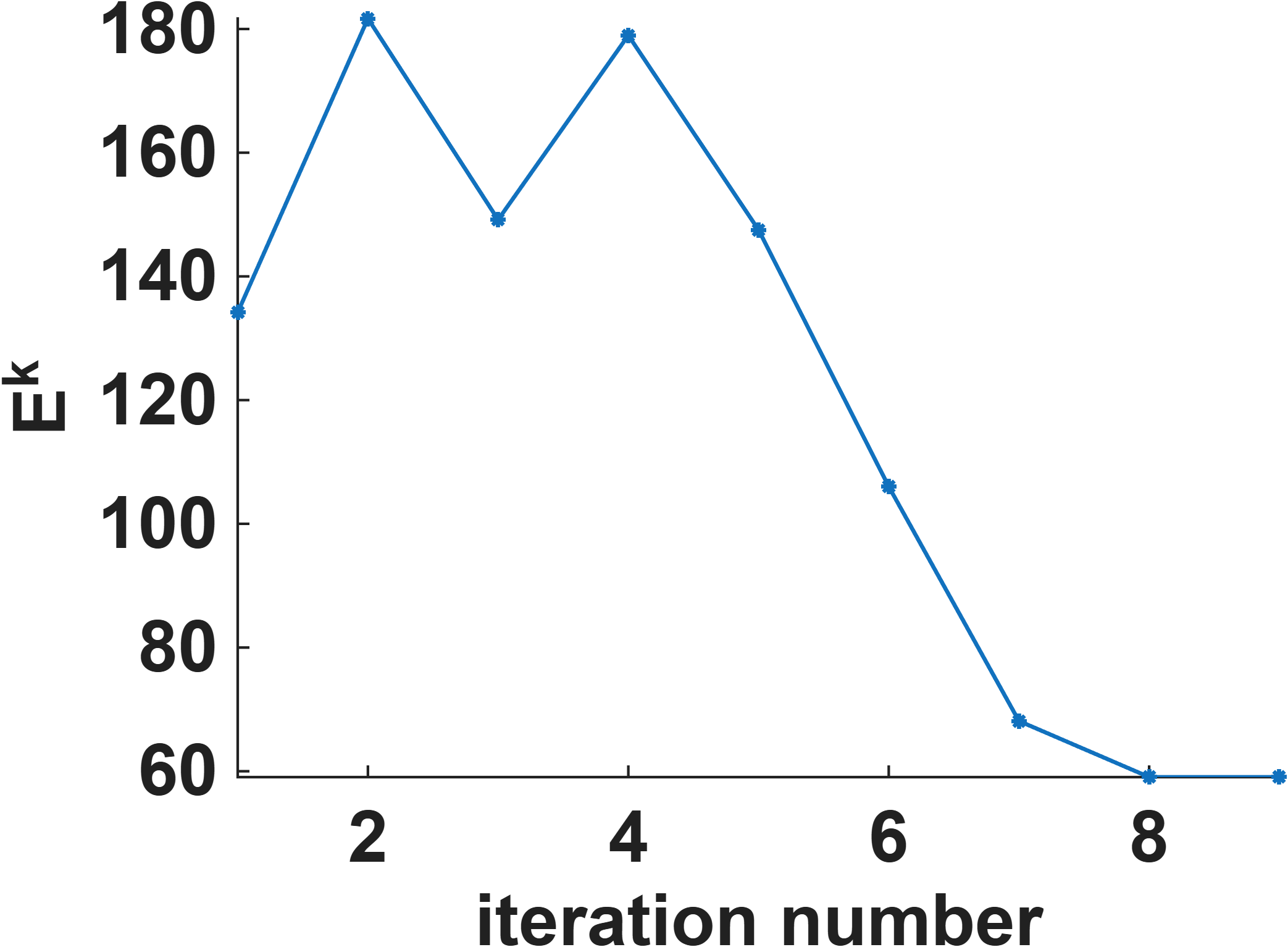"}
    \caption{SB-GD}
  \end{subfigure}

  \caption{Comparison of $E^k$ across methods for the 10D Rastrigin Function with $B=0$. Each subplot is one method (left-to-right, top-to-bottom).}
  \label{fig:swar:10DRas}
\end{figure}

\end{appendices}

\end{document}